\documentclass[11pt,a4paper]{article}
\usepackage[margin=1in]{geometry}
\usepackage{tabularx, enumerate, authblk,soul}
\usepackage[dvipsnames]{xcolor}
\usepackage{amsmath,amssymb,dsfont,subcaption}
\usepackage{mathtools, graphicx,color,dsfont,amsthm,bbm}
\usepackage{lipsum,bm,comment}
\usepackage[round]{natbib}
\usepackage{longtable}
\usepackage{array}
\usepackage{caption}
\usepackage{url}
\usepackage[colorlinks=true,linkcolor=blue,anchorcolor=blue,citecolor=blue,filecolor=black,menucolor=black,runcolor=black,urlcolor=black]{hyperref}
\usepackage{xurl}
\hypersetup{breaklinks=true}
\usepackage{cleveref}
\usepackage{float}
\usepackage{tablefootnote}
\usepackage{tikz}
\usepackage{pgfplots}
\pgfplotsset{compat=1.10}
\usepgfplotslibrary{fillbetween}
\usetikzlibrary{patterns}
\apptocmd{\sloppy}{\hbadness 10000\relax}{}{}

\newcommand{\Var}[1]{
    \mathrm{Var}
}

\newcommand{\vecbf}[1]{
    \mathbf{#1}
}

\newcommand{\vecbfm}[1]{
    \bm{#1}
}

\newcommand{\RN}[1]{%
  (\textup{\uppercase\expandafter{\romannumeral#1}})%
}
\newcolumntype{H}{>{\setbox0=\hbox\bgroup}c<{\egroup}@{}}

\theoremstyle{plain}
\newtheorem{theorem}{Theorem}
\newtheorem{lemma}[theorem]{Lemma}

\newtheorem{proposition}[theorem]{Proposition}

\theoremstyle{definition}
\newtheorem{definition}{Definition}

\newtheorem{remark}{Remark}

\DeclareMathOperator*{\argmax}{arg\,max}

\newcommand*\diff{\mathop{}\!\mathrm{d}}

\allowdisplaybreaks

\newcommand{\beamevar}{1}

\newcommand{\beame}[3]{%
    \if1\beamevar{#1}\fi
    \if2\beamevar{#2}\fi
    \if3\beamevar{#3}\fi
}

\title{Locally Differentially Private Two-Sample Testing}
\author{Alexander Kent}
\author{Thomas B.~Berrett}
\author{Yi Yu}
\affil{Department of Statistics, University of Warwick}
\date{\today}

\begin{document}

\maketitle

\begin{abstract}
    We consider the problem of two-sample testing under a local differential privacy constraint where a permutation procedure is used to calibrate the tests. We develop testing procedures which are optimal up to logarithmic factors, for general discrete distributions and continuous distributions subject to a smoothness constraint. Both non-interactive and interactive tests are considered, and we show allowing interactivity results in an improvement in the minimax separation rates. Our results show that permutation procedures remain feasible in practice under local privacy constraints, despite the inability to perturb the non-private data directly and only the private views. Further, through a refined theoretical analysis of the permutation procedure, we are able to avoid an equal sample size assumption which has been made in the permutation testing literature regardless of the presence of the privacy constraint. Lastly, we conduct numerical experiments which demonstrate the performance of our proposed test and verify the theoretical findings, especially the improved performance enabled by allowing interactivity.

\end{abstract}

\section{Introduction}\label{sec1}
    \subsection{Background}
        As large volumes of data become increasingly easy to store and access, the rapid development of techniques for data analysis is followed by growing concerns regarding the data owners' privacy. This is motivated by both the personal concerns of individuals and legal considerations. In particular, earlier \emph{ad-hoc} methods of data privacy such an \beame{anonymisation}{anonymization}{anonymization} can often be defeated, even in settings that might seem challenging, for example recovering individual surnames from genome data \citep[e.g.][]{Gymrek:2013:Genomes} or recovering training data from a trained neural network \citep[e.g.][]{Haim:2022:Reconstructing}.
        
        The issue of reidentification from \beame{anonymised}{anonymized}{anonymized} data has led to an interest in methods of \beame{anonymisation}{anonymization}{anonymization} and \beame{privatisation}{privatization}{privatization} which have provable guarantees. One framework enjoying considerable interest from both applied and theoretical standpoints is that of \emph{differential privacy}. Here, the aim is to obscure the value of an individual's data by injecting carefully calibrated noise, so that the true value of any one data point cannot be deduced with a high degree of confidence, but also so that the influence of the noise is diminished in the aggregate, enabling approximate recovery of quantities of interest. Whilst first \beame{formalised}{formalized}{formalized} as differential privacy in \cite{Dwork:2006:CalibratingNoise}, the idea of adding random noise to provide privacy, and hence avoid evasive responses in surveys of sensitive questions, dates back to \cite{Warner:1965:RR}.
    
        Since these early developments, differential privacy has evolved into many branches, including \emph{central} differential privacy, where a private output statistic is calculated by a trusted data aggregator who has access to the original data, and \emph{local} differential privacy, where data are \beame{privatised}{privatized}{privatized} at the individual level before being collected for analysis. These frameworks have also since been \beame{generalised}{generalized}{generalized} \citep[e.g.][]{Levy:2021:UserLevel} and similar but distinct definitions have also been proposed \citep[e.g.][]{Mironov:2017:Renyi, Dong:2022:GaussianDP}, allowing differential privacy both to be applied to broader cases than initially possible, and take advantage of problem-specific structures.

        One topic of particular interest in the setting of local differential privacy is that of non-interactive and interactive privacy. In the non-interactive setting, each individual transmits their data through a given noisy channel, without any communication between individuals, whereas in the interactive setting, an individual may use the \beame{privatised}{privatized}{privatized} output of previous individuals, in some ordering, to inform their method of \beame{privatisation}{privatization}{privatization}. It has been shown in some problems such as quadratic functional estimation \citep{Butucea:2023:Interactive}, goodness-of-fit testing \citep{Berrett:2020:Faster} and spectral density estimation \citep{Butucea:2025:InterativeNonParametric} that allowing interactivity can enable improved performance, such as reducing the severity of the curse of dimensionality. As such, \beame{characterising}{characterizing}{characterizing} problems where there is a gap between the performance of non-interactive and interactive procedures is of interest.

        The statistics community's interest in differential privacy has led to a growing understanding of how the framework of differential privacy affects classical statistical procedures, such as regarding convergence properties and finite-sample minimax guarantees, in both the central and local models \citep{Wasserman:2010:Statistical, Duchi:2018:DJW}. There is also a growing understanding of how differential privacy interacts with both a range of estimation problems such as linear models \citep[e.g.][]{Duchi:2018:DJW, Arora:2022:PrivateGLM}; functional data analysis \citep[e.g.][]{Hall:2013:PrivateFunctions, Xue:2024:PrivateFDA}; dynamic pricing \citep[e.g.][]{Tang:2020:PrivateDynamicPricing, Zhao:2024:PrivateDynamicPricing}; and also broader statistical frameworks such as robustness \citep[e.g.][]{Dwork:2009:Robust, Liu:2022:RobustCentralPrivacy, Li:2023:Robust}. 
        
        However, whilst estimation problems have been studied in great depth, hypothesis testing is relatively less explored, particularly in the local model. Though some literature successfully identifying the minimax sample complexities of various testing problems exists \citep[e.g.][]{Acharya:2019:TestWithoutTrust,  Acharya:2019:Hadamard, Acharya:2020:I, Dubois:2019:GOFtest, Berrett:2020:Faster}, there is a lack of practical methodology, with the developed tests usually depending on conservative constants for determining the rejection region, potentially resulting in unnecessarily low power in practice.

        We in this paper consider the problem of two-sample testing. To calibrate our tests, we \beame{utilise}{utilize}{utilize} a permutation procedure. Such procedures date back to works of Fisher and Pitman \citep{Fisher:1935:Book, Pitman:1938:Significance}. Permutation tests are desirable as they do not rely on exact knowledge of the distribution of the test statistic under the null, which can be difficult to deduce. They also do not employ asymptotic arguments which only hold in the limit, and which are prone to failure under even slight deviations from the setting they were established in (see, for example, Appendix A(a)~in \citealt{Berrett:2021:USP} for a demonstration). Instead, permutation tests rely on typically milder assumptions, under which the tests are guaranteed to have the desired level. Further, classical asymptotic analyses also suggest the strong power properties of permutation methods \citep[e.g.][Chapter~15]{Lehmann:2005:Testing}.

        Despite their positive characteristics, theoretical understanding of permutation tests in terms of finite-sample guarantees is sparse, owing to the difficulty of the analysis. The first works to prove finite-sample guarantees for permutation tests \citep{Berrett:2021:IndepPerm, Kim:2022:MinimaxPermutationTests} develop general methods to prove finite-sample bounds on the performance of the permutation procedures. In particular, these works show that, for two-sample-testing and independence-testing, permutation tests are minimax optimal for multinomial distributions and continuous distributions under non-parametric smoothness assumptions. This has since been extended to the central model of privacy \citep{Kim:2023:PrivatePermutationTests} with kernel-based test statistics including maximum mean discrepancy and the Hilbert--Schmidt independence criterion.

        The application of permutation testing procedures under non-interactive local differential privacy constraints has been considered in the recent paper of \cite{Mun:2025:LocalPerm}. We share some results with theirs, discuss comparisons between our methods and results for the non-interactive setting in the relevant sections, and compare the results in the non-interactive case to our interactive rates.

    \subsection{Contributions}  

    In this paper, we aim to investigate two-sample testing under local privacy constraints using permutation test procedures in both the non-interactive and interactive settings. We obtain positive results in both these cases for different testing settings, and \beame{summarise}{summarize}{summarize} our contributions as follows:
        
        \begin{enumerate}
            \item Broadly, we demonstrate the feasibility of permutation test procedures under local privacy constraints in the non-interactive and interactive settings. In particular, these positive results in the interactive setting are, to the best of our knowledge, new to the literature.
        
            \item Considering general multinomial distributions on $d \in \mathbb{N}$ outcomes, we construct non-interactive and interactive tests and observe the minimax separation rate with respect to $L_p$-separation for $p \in [1,2]$. In particular, we observe an improvement in the dependence on $d$ with the interactive tests. We obtain a further novel contribution in that the upper bounds have only poly-logarithmic dependence on the type-I and type-II errors $\alpha, \beta$ without requiring the assumption that $n_1 \asymp n_2$ which prior literature on permutation testing required. These results are contained in \Cref{sec4}. Another important contribution key to obtaining these results is demonstrating the negative association \citep{Joag:1983:NA} of the popular unary-encoding method for privatising discrete data. In particular, though the entries of the vector obtained via unary encoding are not independent, the negative association property shows that their sum concentrates at least as well as if they were independent, allowing us to obtain stronger concentration inequalities which are vital for the improved rates we obtain. More details may be found in \Cref{app:misc}.

            \item For continuous distributions satisfying a Sobolev-type regularity assumption, we construct non-interactive and interactive tests and observe the minimax separation rate with respect to $L_p$-separation for $p \in [1,2]$. To the best of our knowledge, the use of a Sobolev regularity assumption in testing under LDP is a first for problems of this type, which instead often use a H\"{o}lder or Besov condition \citep[e.g.][]{Dubois:2019:GOFtest, Mun:2025:LocalPerm}. The Sobolev regularity assumption allows us to consider a testing procedure via a basis expansion, as opposed to binning procedures considered in prior literature, and this novel private testing methodology may be of independent interest beyond the permutation testing setting. Further, whilst in some cases Sobolev spaces are a special case of Besov spaces, related literature consider Besov spaces defined via an ellipsoid condition on a Haar wavelet expansion, which prevents the transferring of results for Besov spaces to Sobolev spaces. These results, a further discussion on comparisons to results for Besov spaces, and extensions to adaptive tests, can be found in \Cref{sec5}.

            \item We implement our constructed test procedures and carry out simulations in \Cref{sec6}. We explore the sensitivity of the testing procedures to tuning parameters, empirically verify the theoretical minimax separation rates, and investigate the performance. In particular, we observe the improved performance with the interactive procedures in practice in addition to their theoretical improvements.
        \end{enumerate}

        We \beame{summarise}{summarize}{summarize} non-private separation rates from prior literature and our results in the following table.
        \begin{table}[H]
        \centering
            \begin{tabular}{lccc}
            \hline\hline
            Setting & Non-private & Non-interactive & Interactive\\ \hline
            Multinomial($d$), $L_1$  & $\frac{d^{1/2}}{n_1^{1/2}n_2^{1/4}} \vee \frac{d^{1/4}}{n_1^{1/2}}$ & $\frac{d^{3/4}}{(n_1\varepsilon^2)^{1/2}}$ & $\frac{d^{1/2}}{(n_1\varepsilon^2)^{1/2}}$ \\
            Multinomial($d$), $L_2$  & $\frac{1}{n_1^{1/2}}$ & \begin{tabular}{@{}l@{}}$\lesssim \frac{d^{1/4}}{(n_1\varepsilon^2)^{1/2}}$\\$\gtrsim \frac{d^{1/4}}{(n_1\varepsilon^2)^{1/2}} \wedge \frac{1}{d^{1/2}}$\end{tabular} & $\frac{1}{(n_1\varepsilon^2)^{1/2}}$ \\
            Continuous (Sobolev)  & $(n_1\varepsilon^2)^{-2s/(4s+d)}$ & $(n_1\varepsilon^2)^{-2s/(4s+3d)}$ & $(n_1\varepsilon^2)^{-2s/(4s+2d)}$ \\ \hline
            \end{tabular}
            \caption{Summary of separation radii from two-sample testing with sample sizes $n_1 \leq n_2$ and privacy parameter $\varepsilon \in (0, 1]$ with logarithmic factors suppressed. Non-private rates can be found in, for example, \cite{Bhattacharya:2015:l1TwoSample} for Multinomial($d$) under~$L_1$-separation, \cite{Kim:2022:MinimaxPermutationTests} for Multinomial($d$) under~$L_2$-separation, and \cite{Li:2024:NonParametricTests} for $d$-dimensional continuous distributions under a Sobolev regularity assumption of soothness parameter $s > 0$.} \label{tab:summary}
        \end{table}

        \subsection{Notation}
            For $n \in \mathbb{N}$, let $[n] = \{1, \hdots, n\}$. For $a, b \in \mathbb{R}$, let $a \wedge b = \min(a, b)$ and $a \vee b = \max(a, b)$. For non-negative real sequences $\{a_n\}_{n \in \mathbb{N}}$, $\{b_n\}_{n \in \mathbb{N}}$, we write $a_n \lesssim b_n$ when there exists a constant $C > 0$ such that $\limsup_{n \rightarrow \infty} a_n/b_n \leq C$, and $a_n \gtrsim b_n$ denotes $b_n \lesssim a_n$. We write $a_n \asymp b_n$ if $a_n \lesssim b_n$ and $b_n \lesssim a_n$. For $a, b \in \mathbb{R}$ with $a \leq b$, we write $\Pi_{[a, b]}$ for the projection onto the interval $[a, b]$. A random variable $X$ is $\sigma^2$-sub-Gaussian if it satisfies $\mathbb{E}[\exp\{\lambda (X - \mathbb{E}[X])\}] \leq \exp(\lambda^2\sigma^2/2)$ for all $\lambda \in \mathbb{R}$. A random variable $X$ is $\sigma$-sub-exponential if it satisfies $\mathbb{E}[\exp\{\lambda (X - \mathbb{E}[X])\}] \leq \exp(\lambda^2\sigma^2)$ for all $\lambda$ such that $|\lambda| \leq 1/\sigma$. We denote these as $\mathrm{SG}(\sigma^2)$ and $\mathrm{SE}(\sigma)$ respectively. We say a random variable $X$ has the standard Laplace distribution if it has, with respect to the Lebesgue measure on $\mathbb{R}$, the density $f_X(x) = \exp(-|x|)/2$. Throughout this work, for simplicity we omit the taking of the floor or ceiling of quantities for which an integer value is expected. For $n \geq 2$, let $\mathcal{I}_2^{n} = \{(i,j) : i,j \in [n],\; i \neq j\}$. For a vector $\vecbf{x} \in \mathbb{R}^d$ for some $d \in \mathbb{N}$, denote for $p \in \{1, 2\}$ the norm $\|x\|_p = (\sum_{j=1}^d |x_j|^p)^{1/p}$, which we refer to as the $L_p$-norm. We similarly denote the $L_p$-norm for functions, for a function $f:[0,1] \rightarrow \mathbb{R}$, as the norms $\|f\|_p = (\int_0^1 |f(x)|^p\diff{x})^{1/p}$. We denote the type-I error and type-II error of a test by $\alpha$ and $\beta$ respectively, and require $\alpha + \beta \leq 1/2$ throughout. Note $\alpha$ and $\beta$ are not treated as universal constants, and so dependence on them will not be suppressed along with other constants. 
        
    \section{Problem Setup}\label{sec2}
        \subsection{Local Differential Privacy}
            We first introduce the framework of local differential privacy (LDP). Given data $\{X_i\}_{i = 1}^n \subset \mathcal{X}^n$ and some output space $\mathcal{Z}$, consider a family of conditional distributions $\{Q_i\}_{i = 1}^n$ where $Q_1 : \sigma(\mathcal{Z}) \times \mathcal{X} \rightarrow [0,1]$ and $Q_i : \sigma(\mathcal{Z}) \times \mathcal{X} \times \mathcal{Z}^{i - 1} \rightarrow [0,1]$ for $i > 1$, where $\sigma(\mathcal{Z})$ denotes a sigma algebra on $\mathcal{Z}$. Here, the $i$-th individual \beame{privatises}{privatizes}{privatizes} their datum $X_i$ into a private view $Z_i$ via some random mapping from $\mathcal{X} \times \mathcal{Z}^{i - 1}$ to $\mathcal{Z}$, inducing the conditional distribution $Q_i$.

            For $\varepsilon \geq 0$, the collection $\{Q_i\}_{i = 1}^n$ is called $\varepsilon$-LDP if, for all $i \in [n]$,
            \begin{equation} \label{sec2:eq:LDPdef}
                \begin{aligned}
                    &\sup_{S \in \sigma(\mathcal{Z})} \frac{Q_i(Z_i \in S \mid X_i = x_i, Z_1 = z_1, \hdots, Z_{i -1} = z_{i-1})}{Q_i(Z_i \in S \mid X_i = x_i', Z_1 = z_1, \hdots, Z_{i -1} = z_{i-1})} \leq \exp(\varepsilon), \\
                    &\hspace{8cm} \forall x_i, x_i' \in \mathcal{X} \mbox{ and } \forall z_1, \hdots, z_{i-1} \in \mathcal{Z}. 
                \end{aligned}
            \end{equation}
            More specifically, such a collection is termed \emph{sequentially interactive}. If instead each $Q_i$ depends only on the value of $X_i$, and is independent of the collection $\{Z_j\}_{j \neq i}$, then the collection is referred to as \emph{non-interactive} and the conditional distributions simplify to take the form $Q_i(Z_i \in \cdot\; \mid X_i = x_i)$. We denote the collection $Q = \{Q_i\}_{i = 1}^n$ and refer to $Q$ as an $\varepsilon$-LDP mechanism when the collection $\{Q_i\}_{i = 1}^n$ satisfies $\varepsilon$-LDP, and denote by $\mathcal{Q}_\varepsilon$ the collection of all $\varepsilon$-LDP mechanisms.

            An important result in differential privacy is that of \emph{post-processing}, where the output of an arbitrary data-independent function applied to $\varepsilon$-LDP data remains $\varepsilon$-LDP.
            \begin{proposition}[\citealt{Dwork:2006:CalibratingNoise}] \label{sec2:prop:postprocessing}
                Fix $n \in \mathbb{N}$ and let $\{Z_1, \hdots, Z_n\} \subset \mathcal{Z}^n$ be $\varepsilon$-LDP views of data $\{X_1, \hdots, X_n\} \subset \mathcal{X}^n$. Let $f : \mathcal{Z}^{n} \rightarrow \mathcal{A}$ for any output space $\mathcal{A}$ be a data-independent function. The random variable $f(Z_1, \hdots, Z_n)$ satisfies $\varepsilon$-LDP.
            \end{proposition}

        \subsection{Local Differential Privacy with Two Samples} 
            In the setting of two-sample testing, to be introduced later, the data of interest are the combined collection of two distinct samples.
            To extend sequential interactivity to the case of two distinct samples, denote two samples $\{X_i\}_{i = 1}^{n_1} \subset \mathcal{X}^{n_1}$ and $\{Y_i\}_{i = 1}^{n_2} \subset \mathcal{Y}^{n_2}$ for $n_1, n_2 \in \mathbb{N}$ the sample sizes and $\mathcal{X}, \mathcal{Y}$ some (possibly distinct) state spaces. We illustrate graphically the model of interactivity in \Cref{fig:interactive}. In words, for each $i$, we have that each pair $Z_i$ and $W_i$ may depend on all previous observations $\{(Z_j, W_j)\}_{j < i}$. We now \beame{formalise}{formalize}{formalize} as follows. We assume without loss of generality that $n_1 \leq n_2$. We consider the family of conditional distributions $\{Q_{X, i}\}_{i = 1}^{n_1} \cup \{Q_{Y, i}\}_{i = 1}^{n_2}$ where 
            \begin{equation*}
                \begin{aligned}
                    &Q_{X, 1} : \sigma(\mathcal{Z}) \times \mathcal{X} \rightarrow [0,1];\quad
                    Q_{Y, 1} : \sigma(\mathcal{W}) \times \mathcal{Y}\rightarrow [0,1]; \\
                    &Q_{X, i} : \sigma(\mathcal{Z}) \times \mathcal{X} \times (\mathcal{Z} \times \mathcal{W})^{i - 1}, \mbox{ for } 1 < i \leq n_1; \\
                    &Q_{Y, i} : \sigma(\mathcal{W}) \times \mathcal{Y} \times (\mathcal{Z} \times \mathcal{W})^{i - 1}, \mbox{ for } 1 < i \leq n_1; \mbox{ and}\\
                    & Q_{Y, j} : \sigma(\mathcal{W}) \times \mathcal{Y} \times \mathcal{Z}^{n_1} \times \mathcal{W}^{j - 1}, \mbox{ for } n_1 < j \leq n_2.
                \end{aligned}
            \end{equation*}
            As before, these conditional distributions are induced by a random mappings from $\mathcal{X}$ and $\mathcal{Y}$ to $\mathcal{Z}$ and $\mathcal{W}$ respectively, resulting in the collections of private views $\{Z_i\}_{i = 1}^{n_1}$, and $\{W_i\}_{i = 1}^{n_2}$. We describe such a collection of conditional distributions as \emph{two-sample sequentially interactive} $\varepsilon$-LDP if, for all $i \in [n_1]$,
            \begin{equation} \label{sec2:eq:LDPdefTwo}
                \begin{aligned}
                    &\sup_{S \in \sigma(\mathcal{Z})} \frac{Q_{X, i}(Z_i \in S \mid X_i = x_i, Z_1 = z_1, \hdots, Z_{i -1} = z_{i-1}, W_1 = w_1, \hdots, W_{i-1} = w_{i - 1})}{Q_{X, i}(Z_i \in S \mid X_i = x_i', Z_1 = z_1, \hdots, Z_{i -1} = z_{i-1}, W_1 = w_1, \hdots, W_{i-1} = w_{i - 1})} \leq \exp(\varepsilon), \\
                    &\hspace{6.5cm}\forall x_i, x_i' \in \mathcal{X},\; \forall z_1, \hdots, z_{i-1} \in \mathcal{Z}, \mbox{ and } \forall w_1, \hdots w_{i - 1} \in \mathcal{W},
                \end{aligned}
            \end{equation}
            and likewise for $Q_{Y, j}$ for all $j \in [n_2]$, with the natural adjustment when $j > n_1$. If instead each $Q_{X, i}$ and $Q_{Y, j}$ depends only on the value of $X_i$ and $Y_j$ respectively, then the collection is referred to as \emph{two-sample non-interactive}. As in this paper we exclusively focus on the setting where there are two samples, we will henceforth suppress the term \emph{two-sample} when referring to privacy mechanisms for brevity, and only make the distinction when necessary.
            \begin{figure}
                \centering
                \includegraphics[width=0.3\linewidth]{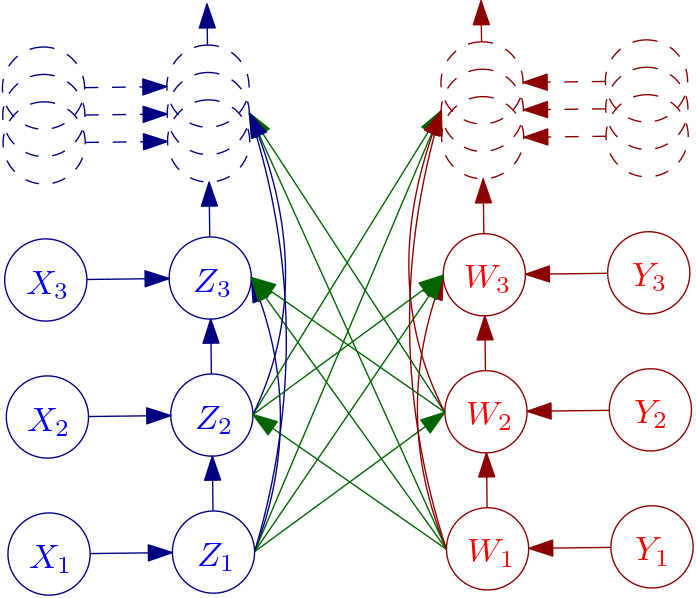}
                \caption{Illustration of the interactive model in the two-sample setting.}
                \label{fig:interactive}
            \end{figure}
            
        \subsection{Permutation Tests} \label{sec2:sec:permutationtest}
            We briefly introduce and review relevant properties of permutation tests; more detailed treatments can be found in, for example, \cite{Lehmann:2005:Testing}, \cite{Pesarin:2010:PermutationTests} and \cite{Hemerik:2018:ExactTestingPermutations}.
    
            Suppose we have a sample of $n$ observations denoted $\mathcal{D}_n = \{X_1, \hdots, X_n\} \subset \mathcal{X}^n$. Given a permutation $\pi \in S_n$, the symmetric group on $[n]$, we denote a permuted version\footnote{More generally, permutation tests can be formulated in terms of any group action on the sample space, enabling a broader range of possible tests. See e.g.~\cite{Hemerik:2018:ExactTestingPermutations} for such a formulation.} of the sample $\mathcal{D}_n$ by $\mathcal{D}_n^\pi = \{X_{\pi(1)}, \hdots, X_{\pi(n)}\}$. We write $T_n = T_n(\mathcal{D}_n)$ for a test statistic computed on the original sample, and write $T_n^\pi = T_n(\mathcal{D}_n^\pi)$ for the same test statistic computed on the sample permuted by $\pi$. We say that the distribution of $\mathcal{D}_n$ is permutation invariant if, under the null hypothesis, $\mathcal{D}_n$ and $\mathcal{D}_n^\pi$ are equal in distribution for all $\pi \in S_n$.
    
            We can then define the $p$-value of the permutation test with test statistic $T_n$ as
            \begin{equation*}
                p = p(T_n, S_n) = \frac{\sum_{\pi \in S_n} \mathbbm{1}\{T_n \leq T_n^\pi\}}{|S_n|}.
            \end{equation*}
            When invariance holds and the probability of a tie $\{T_n = T_n^\pi\}$ is zero, excepting the identity permutation, we have that $p$ is uniformly distributed over the set $\{1/|S_n|, 2/|S_n|, \hdots, 1\}$, and hence that a test that rejects the null hypothesis when $p \leq \alpha$ for some $\alpha \in (0, 1)$ has level $\alpha$. If $\alpha \notin \{1/|S_n|, 2/|S_n|, \hdots, 1\}$ or the probability of a tie $\{T_n = T_n^\pi\}$ is non-zero, then the test will be conservative, but a suitable \beame{randomisation}{randomization}{randomization} of the test can be employed to recover exact size $\alpha$ if desired.
    
            For all but small values of $n$, the number of permutations $|S_n|$ can grow prohibitively large, and so in practice a constant number $B$ of permutations $\Pi_B = \{\pi_1, \hdots, \pi_B\} \subset S_n$ are sampled uniformly at random, giving the $p$-value
            \begin{equation} \label{sec2:def:permpval}
                p_B = p(T_n, \Pi_B) = \frac{1 + \sum_{b=1}^B \mathbbm{1}\{T_n \leq T_n^{\pi_b}\}}{1 + B},
            \end{equation}
            and a test rejecting when $p_B \leq \alpha$ similarly has level $\alpha$.

    \subsection{Locally Differentially Private Two-Sample Testing} \label{sec2:sec:problemdef}

        Let random variables $X \in \mathcal{X}$ and $Y \in \mathcal{Y}$ be generated from distributions $P_X$ and $P_Y$.  Given two samples $\mathcal{D}_{X, n_1} = \big\{X_1, \hdots, X_{n_1}\big\} \overset{\mathrm{i.i.d.}}{\sim} P_X$ and $\mathcal{D}_{Y, n_2} = \big\{Y_{1}, \hdots, Y_{n_2}\big\} \overset{\mathrm{i.i.d.}}{\sim} P_Y$, we are to test
        \begin{equation} \label{sec2:eq:Testdef}
            \mathrm{H}_0: P_X = P_Y \quad \mbox{vs} \quad \mathrm{H}_1: D(P_X, P_Y) \geq \rho,
        \end{equation}
        for some metric $D$ on the space of distributions and some level of separation $\rho > 0$.
    
        Rather than having direct access to the raw samples $\mathcal{D}_{X, n_1} \cup \mathcal{D}_{Y, n_2}$, we instead construct an $\varepsilon$-LDP privacy mechanism $Q$, and observe the induced samples $\tilde{\mathcal{D}}_{X, n_1} = \{ Z_i \}_{i \in [n_1]}$ and $\tilde{\mathcal{D}}_{Y, n_2} = \{ W_{i'} \}_{i' \in [n_2]}$ where, for each $i \in [n_1]$ and $i' \in [n_2]$, $Z_i \in \mathcal{Z}$ and $W_{i'} \in \mathcal{W}$ are $\varepsilon$-LDP views satisfying \eqref{sec2:eq:LDPdefTwo}, for some output spaces $\mathcal{Z}$ and $\mathcal{W}$, of $X_i$ and $Y_{i'}$ respectively.
    
        To measure the performance of a testing procedure, we introduce the following minimax framework. Denote the combined samples by $\tilde{\mathcal{D}}_{n_1, n_2} = \tilde{\mathcal{D}}_{X, n_1} \cup \tilde{\mathcal{D}}_{Y, n_2}$, and consider a test $\phi : \tilde{\mathcal{D}}_{n_1, n_2} \rightarrow \{0, 1\}$ where a value of $\phi = 1$ indicates that we reject the null hypothesis of \eqref{sec2:eq:Testdef}. We then define the space of pairs of distributions satisfying the null and alternative as
        \begin{equation*}
            \mathcal{P}_0 = \big\{(P_X, P_Y) \in \mathcal{P} \times \mathcal{P} : P_X = P_Y \big\} \quad \mbox{and} \quad \mathcal{P}_1(D, \rho) = \big\{(P_X, P_Y) \in \mathcal{P} \times \mathcal{P} : D(P_X, P_Y) \geq \rho \big\}.
        \end{equation*}
        For a fixed level $\alpha \in (0, 1)$, privacy level $\varepsilon \in (0, 1]$ and separation $\rho > 0$, we define the private minimax testing risk as
        \begin{align*}
             \mathcal{R}_{n_1, n_2, \varepsilon, \alpha, D, \rho}
             = \alpha + \inf_{Q \in \mathcal{Q}_\varepsilon}\inf_{\phi \in \Phi_{n_1, n_2, Q}(\alpha)} \sup_{(P_X, P_Y) \in \mathcal{P}_1(D, \rho)} \mathbb{P}_{P_X, P_Y, Q}(\phi = 0),
        \end{align*}
        where
        \begin{equation*}
            \Phi_{n_1, n_2, Q}(\alpha) = \bigg\{\phi: \sup_{(P_X, P_Y) \in \mathcal{P}_0} \mathbb{P}_{P_X, P_Y, Q} (\phi = 1) \leq \alpha \bigg\},
        \end{equation*}
        denotes the set of level-$\alpha$ tests for a fixed $\varepsilon$-LDP mechanism $Q$, and $\mathcal{Q}_\varepsilon$ denotes the class of all non-interactive or interactive mechanisms. Finally, fixing a desired type-II error probability $\beta \in (0, 1 - \alpha)$, we define the private minimax separation radius
        \begin{equation} \label{sec2:eq:minimaxsep}
            \rho^\ast(n_1, n_2, \varepsilon, \alpha, \beta, D) = \inf\{\rho > 0 : \mathcal{R}_{n_1, n_2, \varepsilon, \alpha, D, \rho} \leq \alpha + \beta\}.
        \end{equation}
        To distinguish between the non-interactive and interactive cases we use the subscript $\rho_{\mathrm{NI}}^\ast$ and $\rho_{\mathrm{I}}^\ast$ respectively.

    \section{Bounds on Separation Radii} \label{sec3}
        \subsection{Upper Bounds on Separation Radii} \label{sec3:sec:separationcriteria}
            In this section, we outline results which provide convenient criteria for establishing upper bounds on the minimax separation radii of different testing problems. We bound the type-I and type-II errors of the permutation tests we construct, and as the tests by the permutation construction will have exact type-I error control, our goal is to show that the tests are sufficiently powerful so as to have suitable type-II error control.
    
            We first start with a general result, which controls the type-II error through controlling the deviations from the means of test statistics via a sub-Gaussian/exponential property. In particular, we obtain two conditions in the differing cases, with an improved dependence on the type-I error $\alpha$ in the sub-Gaussian case.
            \begin{proposition} \label{sec3:prop:permtestcontrol}
                Consider a test statistic $T$ and its permuted version $T^\pi$ for $\pi$ a permutation sampled uniformly at random. Assume that $\mathbb{E}[T^\pi] = 0$, $T$ is $\mathrm{SE}(\Sigma)$ for some $\Sigma > 0$, and the number of sampled permutations $B \geq 4/(\alpha \beta) - 1$.
    
                (i) If $T^\pi$ is $\mathrm{SE}(\widetilde{\Sigma})$ for some $\widetilde{\Sigma} > 0$, and
                \begin{equation} \label{sec3:eq:permtestcontrolSE}
                    \mathbb{E}[T] \geq C\widetilde{\Sigma}\log\{1/(\alpha\beta)\} + C\Sigma\log(1/\beta),
                \end{equation}
                for $C > 0$ some sufficiently large absolute constant, then the permutation $p$-value \eqref{sec2:def:permpval} satisfies $\mathbb{P}(p_B \geq \alpha) \leq \beta$.
    
                (ii) If $T^\pi$ is $\mathrm{SG}(\widetilde{\Sigma}'^2)$ for some $\widetilde{\Sigma}' > 0$, the condition
                \begin{equation} \label{sec2:eq:permtestcontrolSG}
                    \mathbb{E}[T] \geq C'\widetilde{\Sigma}'[\log\{1/(\alpha\beta)\}]^{1/2} + C'\Sigma\log(1/\beta),
                \end{equation}
                for $C' > 0$ some sufficiently large absolute constant, then the permutation $p$-value \eqref{sec2:def:permpval} satisfies $\mathbb{P}(p_B \geq \alpha) \leq \beta$.
            \end{proposition}
            \noindent
            The proof can be found in \Cref{Sec:appA}.
        
            In multiple settings we consider, the test statistic employed takes the form of the $U$-statistic
            \begin{equation} \label{sec3:eq:Ustatistic}
                U_{n_1, n_2}(\widetilde{\mathcal{D}}_{X, n_1}, \widetilde{\mathcal{D}}_{Y, n_2}) = \frac{1}{n_1(n_1 - 1)} \frac{1}{n_2(n_2 - 1)} \sum_{(i,j) \in \mathcal{I}_2^{n_1}} \sum_{(k,l) \in \mathcal{I}_2^{n_2}} (\vecbf{Z}_i - \vecbf{W}_k)^T(\vecbf{Z}_j - \vecbf{W}_l),
            \end{equation}
            where we denote the samples $\widetilde{\mathcal{D}}_{X, n_1} = \{\vecbf{Z}_1, \hdots, \vecbf{Z}_{n_1} \} \subset \mathbb{R}^d$ and $\widetilde{\mathcal{D}}_{Y, n_2} = \{\vecbf{W}_1, \hdots, \vecbf{W}_{n_2} \} \subset \mathbb{R}^d$ respectively. For the permutation, write the joint sample $\widetilde{\mathcal{D}}_{n_1, n_2} = \{\widetilde{\vecbf{D}}_1, \hdots, \widetilde{\vecbf{D}}_{n_1 + n_2}\}$ where, for $i \in [n_1]$ and $i' \in [n_2]$, we denote $\widetilde{\vecbf{D}}_{i} = \vecbf{Z}_i$ and $\widetilde{\vecbf{D}}_{n_1 + i'} = \vecbf{W}_{i'}$. Then, for a permutation $\pi \in S_{n_1 + n_2}$, we calculate the $U$-statistic on the permuted data
            \begin{equation} \label{sec3:eq:permutedUstatistic}
                U_{n_1, n_2}^\pi(\mathcal{D}_{n_1, n_2}) = \frac{1}{n_1(n_1 - 1)} \frac{1}{n_2(n_2 - 1)} \sum_{(i,j) \in \mathcal{I}_2^{n_1}} \sum_{(k,l) \in \mathcal{I}_2^{n_2}} (\widetilde{\vecbf{D}}_{\pi(i)} - \widetilde{\vecbf{D}}_{\pi(n_1 + k)})^T(\widetilde{\vecbf{D}}_{\pi(j)} - \widetilde{\vecbf{D}}_{\pi(n_1 + l)}).
            \end{equation}
    
            Applying the permutation test procedure as in \Cref{sec2:sec:permutationtest}, we obtain the following theorem, the proof of which is given in \Cref{Sec:appA}.
            \begin{theorem} \label{sec3:thm:sepcondU}
                Take $n_1 \leq n_2$ without loss of generality. Fix type-I and type-II errors $\alpha, \beta \in (0, 1)$ satisfying $\alpha + \beta < 1/2$. Let the test statistic $U_{n_1, n_2}$ as defined in \Cref{sec3:eq:Ustatistic}. Suppose that $\{\vecbf{Z}_i\}_{i \in [n_1]} \cup \{\vecbf{W}_{i'}\}_{i' \in [n_2]}$ are mutually independent and $\vecbf{Z}_i, \vecbf{W}_{i'}$ are $\mathrm{SG}(\sigma^2)$, for $i \in [n_1]$ and $i' \in [n_2]$.
    
                Carrying out the permutation test procedure as in \eqref{sec2:def:permpval} with the permuted test-statistic \eqref{sec3:eq:permutedUstatistic}, if the number of sampled permutations $B \geq 4/(\alpha\beta) - 1$, and $U_{n_1, n_2}$ satisfies
                \begin{equation} \label{sec2:eq:sepcriteria}
                    \mathbb{E}[U_{n_1, n_2}] \geq \frac{Cd^{1/2}\sigma^2}{n_1}[\log\{1/(\alpha\beta)\}]^2 \quad \mbox{and} \quad
                    n_1 \geq C\log\{1/(\alpha\beta)\},
                \end{equation}
                for $C > 0$ some absolute constant, then the $p$-value satisfies $\mathbb{P}(p_{B} \geq \alpha) \leq \beta$.  
            \end{theorem}
            
            In \Cref{sec3:thm:sepcondU}, we see that the sufficient conditions in \eqref{sec2:eq:sepcriteria} for the test to have the desired power depend only logarithmically on the specified type-I and type-II errors without imposing an additional assumption that $n_1 \asymp n_2$. This is an improvement on prior work on permutation tests which we discuss in more detail shortly. Later applications of this theorem will demonstrate it allows us to attain tight minimax rates. We also remark that although we apply the above result only in the setting of privacy, it holds more generally for any samples satisfying the requirements of the theorem, and so can be applied in non-private settings as well, enabling this improvement in a broad range of statistical settings.
            
            Our results of this section build on earlier work of \cite{Berrett:2021:IndepPerm} and \cite{Kim:2022:MinimaxPermutationTests}, who establish the first general techniques of \beame{analysing}{analysing}{analyzing} the optimality of permutation tests. Both \cite{Berrett:2021:IndepPerm} and \cite{Kim:2022:MinimaxPermutationTests} obtain separation conditions on $D(P_X, P_Y)$ which depend polynomially on $\alpha^{-1}$ and $\beta^{-1}$ via applications of Markov's inequality. However, in some cases this polynomial dependence may be undesirable when, for example, carrying out multiple tests and applying a correction as in the case of adaptive testing which we consider in \Cref{sec5:UpperBoundNonIntAdapt}. \cite{Kim:2023:PrivatePermutationTests} show that for two-sample testing, under the assumption $n_1 \asymp n_2$, it is possible to improve these dependencies from polynomial to logarithmic. Our main result of this section instead provides control of the type-II error whilst incurring only logarithmic dependencies, without requiring the assumption that $n_1 \asymp n_2$. This is achieved by a combinatorial argument which provides a finite-sample high probability bound on the difference of the mean of the two groups after permutation, providing stronger concentration results for the permuted test statistic.
    
        \subsection{Lower Bounds on Separation Radii}      
            In the non-private setting, as shown in \cite{Arias:2018:GoodnessofFitDimensionality}, lower bounds on the minimax testing rate of the two-sample testing problem follow immediately from the goodness-of-fit testing problem with one sample. This can be understood from the perspective that the goodness-of-fit problem can be viewed as a limiting case of the two-sample problem where one sample is arbitrarily large. We prove such a correspondence also holds in the setting of local differential privacy. Interestingly, the result holds even in the interactive setting, where the \beame{privatised}{privatized}{privatized} data may no longer be independent or identically distributed, in contrast to the non-private setting.

            In what follows, we denote the minimax separation radii of the goodness-of-fit testing problem, defined formally in \Cref{app:sec:lower}, as $\rho_{\mathrm{NI}}^\ast(n_1, P_0, \varepsilon, D, \alpha, \beta)$ and $\rho_{\mathrm{I}}^\ast(n_1, P_0, \varepsilon, D, \alpha, \beta)$ in the non-interactive and interactive settings respectively.
    
            \begin{lemma} \label{sec3:lem:lowerboundlem}
                Suppose, without loss of generality, that $n_1 \leq n_2$ and let $P_0 \in \mathcal{P}$ be any distribution. Then, for any space of distributions $\mathcal{P}$, metric $D$ on $\mathcal{P}$, $\varepsilon \in (0,1]$ and $\alpha \in (0,1), \beta \in (0, 1 - \alpha)$, the minimax separation radii of the goodness-of-fit and two-sample testing problems satisfy that
                \begin{equation*}
                    \rho_{\mathrm{NI}}^\ast(n_1, P_0, \varepsilon, D, \alpha, \beta)
                    \leq \rho_{\mathrm{NI}}^\ast(n_1, n_2, \varepsilon, D, \alpha, \beta) \mbox{ and }
                    \rho_{\mathrm{I}}^\ast(n_1, P_0, \varepsilon, D, \alpha, \beta)
                    \leq \rho_{\mathrm{I}}^\ast(n_1, n_2, \varepsilon, D, \alpha, \beta).
                \end{equation*}
            \end{lemma}
            \noindent
            The proof of the lemma can be found in \Cref{app:sec:lower}.
            \begin{remark}
                Inspecting the proof of \Cref{sec3:lem:lowerboundlem}, we note that no specific property of the LDP setting is used except that the private observations are generated by passing data from an underlying ``true'' distribution through a channel $Q$. Hence, \Cref{sec3:lem:lowerboundlem} extends to others notions of privacy, such as user-level privacy \citep[e.g.][]{Levy:2021:UserLevel} and Gaussian Differential Privacy \citep{Dong:2022:GaussianDP}. More generally, this can apply to other information constraint frameworks, such as communication constraints. Indeed, there is a growing appreciation of unified results across differing families of information constraints \citep[e.g.][]{Acharya:2023:UnifiedLB}, and general results such as these may be of interest and motivate further study.
            \end{remark}
    
            We use this result to inherit existing lower bounds for goodness-of-fit testing under local privacy. However, we note that most existing lower bounds for goodness-of-fit testing under local privacy constraints only ask for the minimum separation such that $\inf\{\rho' > 0 : \mathcal{R}_{n_1, P_0, \varepsilon, D, \rho'}' \leq \gamma\}$ for some fixed constant $\gamma \in (0,1)$. If we restrict $\alpha$ and $\beta$ such that $\alpha + \beta \leq 1/2$, generally a very mild condition, then we have that
            \begin{align*}
                \rho^\ast(n_1, P_0, \varepsilon, D, \alpha, \beta)
                = \inf\{\rho' > 0 : \mathcal{R}_{n_1, P_0, \varepsilon, D, \rho'}' \leq \alpha + \beta\}
                \geq \inf\{\rho' > 0 : \mathcal{R}_{n_1, P_0, \varepsilon, D, \rho'}' \leq 1/2\}.
            \end{align*}
            Hence, we assume throughout this work that $\alpha + \beta \leq 1/2$.
            
    \section{Discrete Case} \label{sec4}
        In this section, we consider the case of discrete data. For some $d \in \mathbb{N}$, we let $\mathcal{X} = \mathcal{Y} = [d]$. For probability vectors $\vecbf{p}_X, \vecbf{p}_Y \in [0,1]^d$, we have that $P_X = \mathrm{Multinom}(\vecbf{p}_X)$ and $P_Y = \mathrm{Multinom}(\vecbf{p}_Y)$. Hence, the testing problem is
        \begin{equation} \label{sec4:eq:Testdef}
            \mathrm{H}_0: \vecbf{p}_X = \vecbf{p}_Y \quad \mbox{vs} \quad \mathrm{H}_1: \|\vecbf{p}_X - \vecbf{p}_Y\|_p \geq \rho,
        \end{equation}
        for $p \in [1,2]$.
        
        We have the following result on the minimax separation radii.
        \begin{theorem}\label{sec4:thm:main}
            Take $n_1 \leq n_2$ without loss of generality and let $p \in [1, 2]$.
            
            (i) Assume $n_1 \geq C\log\{1/(\alpha\beta)\}$ for $C > 0$ some absolute constant. The non-interactive minimax separation radii as defined in \eqref{sec2:eq:minimaxsep} for the testing problem \eqref{sec4:eq:Testdef} satisfy
            \begin{align*}
                \frac{d^{1/p - 1/4}}{(n_1 \varepsilon^2)^{1/2}} \wedge \frac{1}{d^{1 - 1/p}\{\log(d)\}^{1/2}}
                \lesssim
                \rho_{\mathrm{NI}}^\ast(n_1, n_2, \varepsilon, L_p , \alpha, \beta) &\lesssim \frac{d^{1/p - 1/4}}{(n_1 \varepsilon^2)^{1/2}} [\log\{1/(\alpha\beta)\}]^2.
            \end{align*}
            
            (ii)
            Assume $n_1\varepsilon^2 \geq C'\log\{1/(\alpha\beta)\}$ for $C' > 0$ some absolute constant. The interactive minimax separation radii as defined in \eqref{sec2:eq:minimaxsep} for the testing problem \eqref{sec4:eq:Testdef} satisfy
            \begin{align*}
                \frac{d^{1/p - 1/2}}{(n_1 \varepsilon^2)^{1/2}}
                \wedge(d^{1/p - 1}&\mathbbm{1}\{p \notin \{1, 2\}\} + \mathbbm{1}\{p \in \{1, 2\}\}) \\
                &\lesssim \rho_{\mathrm{I}}^\ast(n_1, n_2, \varepsilon, L_p , \alpha, \beta) \lesssim \frac{d^{1/p - 1/2}}{(n_1 \varepsilon^2)^{1/2}} (\{\log(1/\beta)\}^2 + [\log\{1/(\alpha\beta)\}]^{1/2}).
            \end{align*}
        \end{theorem}

        The upper bounds of the theorem follow from the testing procedures we construct in \Cref{sec4:UpperBoundNonInt} and \Cref{sec4:UpperBoundInt}, the theoretical performance of which are \beame{analysed}{analysed}{analyzed} in \Cref{app:sec:UBDiscrete}. Proofs of the lower bounds are contained in \Cref{app:sec:disclbproof}.

        Comparing the upper and lower bounds, we see that our proposed tests are minimax rate-optimal for interactive tests when $p \in \{1, 2\}$ and non-interactive tests when $p = 1$ up to the logarithmic factor in the type-I and type-II errors $\alpha, \beta$. This logarithmic dependence holds for arbitrary $n_1, n_2$, an improvement on both the requirement $n_1 \asymp n_2$ \citep{Kim:2023:PrivatePermutationTests}, or the polynomial dependence incurred by relaxing this assumption \citep{Berrett:2021:USP, Kim:2022:MinimaxPermutationTests}. We observe that the dependence on the type-I error $\alpha$ is improved from $O(\{\log(1/\alpha)\}^2)$ to $O(\{\log(1/\alpha)\}^{1/2})$ under interactivity. This improvement follows from the stronger concentration properties of the permuted test statistic under interactivity.

        We also note the improvement in the dependence on $d$ between the non-interactive and interactive settings, in particular the dimensional penalty being completely eliminated in the case of $L_2$-separation in the interactive setting. This matches the behaviour seen in goodness-of-fit testing problems under interactive LDP as in \cite{Berrett:2020:Faster}, and public-coin settings as in \cite{Acharya:2020:I, Acharya:2021:III}. Interestingly, this match between the two problems differs from what is observed in the non-private setting which we now discuss.

        Indeed, in the non-private setting the separation rate with respect to the $L_1$-norm is different between the goodness-of-fit and two-sample settings. There, the rates are $d^{1/4}/n_1^{1/2}$ and $\max\{d^{1/2}/(n_1^{1/2}n_2^{1/4})$, $d^{1/4}/n_1^{1/2}\}$ respectively \citep{Bhattacharya:2015:l1TwoSample, Diakonikolas:2016:l1TwoSample}, with the two-sample problem being strictly harder and depending crucially on the size of the larger sample. We see from our results that in neither the non-interactive nor interactive setting do we observe this behaviour, with the problems of LDP goodness-of-fit testing and two-sample testing remaining equally difficult in a minimax sense.

        Regarding results with respect to the $L_2$-norm in the non-private setting, the optimal separation rate without privacy constraints is $n_1^{-1/2}$ in the goodness-of-fit setting \citep[Proposition~11]{Chan:2014:DiscreteTesting}. The interactive rate under $L_2$-separation matches this up to a reduction of the effective sample size from $n_1$ to $n_1\varepsilon^2$, showing that the dimensional penalty of the problem does not worsen under privacy in contrast to the other private separation radii in \Cref{sec4:thm:main}.

        Regarding the lower bound in the interactive case, we observe that the test is optimal, up to logarithmic factors, for $p \in \{1, 2\}$, provided that $n_1\varepsilon^2 \geq d^{1/p - 1/2}$ which is natural as this is required for the lower bound to not be bounded away from zero. The lower bound is more delicate in the case $p \in (1, 2)$, as it is possible for $d^{1/p - 1} \leq d^{1/p - 1/2}/(n_1\varepsilon^2)^{1/2}$, and hence the second term in the minimum is selected. This term arises when ensuring the lower bound construction is valid, and the failure of the standard constructions for $p \in \{1, 2\}$ employed in prior literature (which only consider $p \in \{1, 2\}$) to provide tight rates is an interesting topic for future research. Nevertheless, this gap is resolved under the mild condition $n_1\varepsilon^2 \geq d$.

        We end by noting that the lower bound on the $L_p$-separation in the non-interactive setting with $p \neq 1$ does not match the upper bound in all regimes of $n_1\varepsilon^2$ and $d$. This phenomenon when $p = 2$ is an outstanding problem of goodness-of-fit testing problems under local differential privacy in the literature, and one that we inherit through \Cref{sec3:lem:lowerboundlem}, with the same gap arising across two different approaches to obtain the lower bound as in \cite{Berrett:2020:Faster} and \cite{Acharya:2020:I}. Whilst these works only consider $p \in \{1, 2\}$, our more general results for $p \in [1, 2]$ show that this in fact arises for all $p \neq 1$. Nevertheless, this gap is also resolved under the mild condition that $n_1\varepsilon^2 \geq d^{3/2}/\log(d)$.

        \begin{remark}[Comparisons with \cite{Mun:2025:LocalPerm}]
            Comparing with \cite{Mun:2025:LocalPerm}, in terms of the $L_2$-separation results in the non-interactive case, we are able to obtain logarithmic dependence on the type-I and type-II errors without the further assumptions made in \cite{Mun:2025:LocalPerm}. It is also worth noting that our results with respect to general $L_p$-separation in the non-interactive setting, and the results for the interactive setting, are new to the literature.
        \end{remark}

        \subsection{Testing Procedure (Non-Interactive)} \label{sec4:UpperBoundNonInt}
            We provide a non-interactive testing procedure which uses data \beame{privatised}{privatized}{privatized} by the unary-encoding mechanism (see, for example, \citealt{Wang:2017:UnaryEncoding} for a summary of this technique) to obtain a scaled estimate of $\|\vecbf{p}_X - \vecbf{p}_Y\|_2$.

            To implement unary-encoding, define for $j \in \mathbb{N}$ the functions $R_j : \mathbb{N} \times [0,1] \rightarrow \{0, 1\}$ where
            \begin{equation} \label{sec4:eq:UEFunc}
                \begin{aligned}
                    R_j(x, u)
                    = \mathbbm{1}\{X_i = j\}\mathbbm{1}&\{u \leq \exp(\varepsilon/2)/(\exp(\varepsilon/2) + 1)\} \\
                    &+ \mathbbm{1}\{X_i \neq j\}\mathbbm{1}\{u > \exp(\varepsilon/2)/(\exp(\varepsilon/2) + 1)\}.
                \end{aligned}
            \end{equation}
            The algorithm is then as follows.
            \begin{enumerate}
                \item
                    Set desired type-I error $\alpha$. For $i \in [n_1]$, $i' \in [n_2]$ and $j \in [d]$, let $V_{i,j}^{(X)}, V_{i',j}^{(Y)} \overset{\mathrm{i.i.d.}}{\sim} \mathrm{Unif}[0,1]$ and let $\vecbf{Z}_i, \vecbf{W}_{i'} \in \mathbb{R}^d$ with $j$-th co-ordinates
                    \begin{equation} \label{sec4:eq:UE}
                        Z_{i,j} = R_j(X_i, V_{i, j}^{(X)}), \mbox{ and } W_{i',j} = R_j(Y_{i'}, V_{i', j}^{(Y)}).
                    \end{equation}
                \item
                    Calculate the test statistic $U_{n_1, n_2} = U_{n_1, n_2}(\widetilde{\mathcal{D}}_{X, n_1}, \widetilde{\mathcal{D}}_{Y, n_2})$ as in \eqref{sec3:eq:Ustatistic}.
                \item
                    Sample $B \in \mathbb{N}$ permutations $\Pi_B = \{\pi_1, \hdots, \pi_B\}$ from $S_{n_1 + n_2}$ uniformly at random. Write $\widetilde{\mathcal{D}}_{n_1, n_2} = \{\widetilde{\vecbf{D}}_1, \hdots, \widetilde{\vecbf{D}}_{n_1}, \widetilde{\vecbf{D}}_{n_1 + 1}, \hdots, \widetilde{\vecbf{D}}_{n_1 + n_2}\}$ where, for $i \in [n_1]$ and $i' \in [n_2]$, $\widetilde{\vecbf{D}}_{i} = \vecbf{Z}_i$ and $\widetilde{\vecbf{D}}_{n_1 + i'} = \vecbf{W}_{i'}$. For $b \in [B]$, calculate the test statistic on the permuted data $U_{n_1, n_2}^{\pi_b}(\widetilde{\mathcal{D}}_{n_1, n_2})$ as in \eqref{sec3:eq:permutedUstatistic}.
                \item
                    Return the $p$-value $p(U_{n_1, n_2}, \Pi_B)$ as in \eqref{sec2:def:permpval}, and denote the test outcome $\mathbbm{1}\{p(U_{n_1, n_2}, \Pi_B) \leq \alpha\}$.
            \end{enumerate}
            The \beame{privatisation}{privatization}{privatization} in \eqref{sec4:eq:UE} satisfies $\varepsilon$-LDP by, for example, \citet[Section~3.2]{Duchi:2013:Minimax}. Additionally, by the post-processing property \Cref{sec2:prop:postprocessing}, no privacy leakage occurs when permuting the \beame{privatised}{privatized}{privatized} data.

            As the test statistic $U_{n_1, n_2}$ takes the form of a $U$-statistic, for the proof of the performance of this testing procedure we \beame{utilise}{utilize}{utilize} our general result \Cref{sec3:thm:sepcondU} for tests based on $U$-statistics. In particular, the statistic $U_{n_1, n_2}$ constitutes an unbiased estimator of the rescaled $L_2$-norm $[\{\exp(\varepsilon/2) - 1\}/\{\exp(\varepsilon/2) + 1\}]\|\vecbf{p}_X - \vecbf{p}_Y \|_2^2$, and so will be able to detect deviation from the null hypothesis, and importantly is precisely zero under the null.

            We remark that we do not use the common \beame{privatisation}{privatization}{privatization} method for discrete data of \beame{Randomised}{Randomized}{Randomized} Response \citep{Warner:1965:RR, Christofides:2003:GRR} which would constitute \beame{privatising}{privatizing}{privatizing} a value in $[d]$ by mapping it to some value, possibly different, also in $[d]$. This would result in sub-optimal dependence on the dimension scaling as $O(d^{5/4})$ instead of the optimal $O(d^{3/4})$, in the goodness-of-fit testing problem \citep[see][]{Sheffet:2018:LDPTesting, Acharya:2021:III}. This motivates our decision to use the more sophisticated, and harder to \beame{analyse}{analyse}{analyze} method of unary-encoding as in \eqref{sec4:eq:UE}.

            Finally, numerical simulations exploring the performance of this testing procedure are conducted in \Cref{sec6:disc}.
            
        \subsection{Testing Procedure (Interactive)} \label{sec4:UpperBoundInt}
            We now provide an interactive testing procedure, first introducing the necessary notation. For convenience, we assume that the total number of observations in each of the two samples is even, in particular $2n_1$ and $2n_2$ respectively. This does not change the final minimax rate except up to a constant. Denote the following samples
            \begin{equation}
                \begin{aligned} \label{sec4:eq:sampledefs}
                    &\mathcal{D}_{X, n_1} = \{X_1, \hdots, X_{n_1}\}, \quad
                    \mathcal{D}_{X, n_1}' = \{X_1', \hdots, X_{n_1}'\}, \\
                    &\mathcal{D}_{Y, n_2} = \{Y_{1}, \hdots, Y_{n_2}\}, \quad
                    \mathcal{D}_{Y, n_2}' = \{Y_{1}', \hdots, Y_{n_2}'\},
                \end{aligned}
            \end{equation}
            where the samples of \beame{privatised}{privatized}{privatized} data are denoted $\widetilde{\mathcal{D}}_{X, n_1}, \widetilde{\mathcal{D}}_{X, n_1}', \widetilde{\mathcal{D}}_{Y, n_2}, \widetilde{\mathcal{D}}_{Y, n_2}'$ analogously.

            To implement Randomised Response, define the function $R_0: \mathbb{R} \times [0, \infty) \times [0,1] \rightarrow \mathbb{R}$ where
            \begin{equation} \label{sec4:eq:RRFunc}
                R_0(x, t, u) = t c_\varepsilon \mathbbm{1}\bigg\{u \leq \frac{1}{2}\bigg(1 + \frac{\Pi_{[-t, t]}(x)}{t c_\varepsilon}\bigg)\bigg\} - t c_\varepsilon \mathbbm{1}\bigg\{u > \frac{1}{2}\bigg(1 + \frac{\Pi_{[-t, t]}(x)}{t c_\varepsilon}\bigg)\bigg\},
            \end{equation}
            where $c_\varepsilon = \{\exp(\varepsilon) + 1\}/\{\exp(\varepsilon) - 1\}$.
            The algorithm is then as follows.
            \begin{enumerate}
                \item
                    Set desired type-I error $\alpha$. For $i \in [n_1]$, $i' \in [n_2]$ and $j \in [d]$, let $V_{i,j}'^{(X)}, V_{i',j}'^{(Y)} \overset{\mathrm{i.i.d.}}{\sim} \mathrm{Unif}[0,1]$ and let $\vecbf{Z}_{i}', \vecbf{W}_{i'}' \in \mathbb{R}^d$ with $j$-th co-ordinates
                    \begin{equation} \label{sec4:eq:UEint}
                        Z_{i,j}' = R_j(X_i', V_{i, j}'^{(X)}),\; \mbox{and } W_{i',j}' = R_j(Y_{i'}', V_{i', j}'^{(Y)}),
                    \end{equation}
                    where $R_j$ is as in \eqref{sec4:eq:UEFunc}.
                \item
                    Calculate the estimates
                    \begin{equation} \label{sec4:eq:UEpmfests}
                        \begin{aligned}
                            &\hat{\vecbf{p}}_X = \frac{1}{n_1} \frac{\exp(\varepsilon/2) + 1}{\exp(\varepsilon/2) - 1}\sum_{i = 1}^{n_1} \bigg( \vecbf{Z}_{i}' - \frac{\vecbfm{1}}{\exp(\varepsilon/2) + 1} \bigg),\; \mbox{and} \\
                            &\hat{\vecbf{p}}_Y = \frac{1}{n_2} \frac{\exp(\varepsilon/2) + 1}{\exp(\varepsilon/2) - 1}\sum_{i' = 1}^{n_2} \bigg( \vecbf{W}_{i'}' - \frac{\vecbfm{1}}{\exp(\varepsilon/2) + 1} \bigg),
                        \end{aligned}
                    \end{equation}
                    where $\vecbfm{1}$ is the vector of ones.
                \item 
                    For $i \in [n_1]$ and $i' \in [n_2]$, let $V_{i}^{(X)}, V_{i'}^{(Y)} \overset{\mathrm{i.i.d.}}{\sim} \mathrm{Unif}[0,1]$ and construct
                    \begin{equation} \label{sec4:eq:UEint2}
                        Z_i = R_0(\hat{p}_{X, X_i} - \hat{p}_{Y, X_i}, \tau, V_{i}^{(X)}),\; \mbox{and } W_{i'} = R_0(\hat{p}_{X, Y_{i'}} - \hat{p}_{Y, Y_{i'}}, \tau, V_{i'}^{(Y)})
                    \end{equation}
                    where $\tau = 1/(n_1\varepsilon^2)^{1/2}$. Denote the \beame{privatised}{privatized}{privatized} samples $\widetilde{\mathcal{D}}_{X, n_1} = \{Z_1, \hdots, Z_{n_1}\}$ and $\widetilde{\mathcal{D}}_{Y, n_2} = \{W_1, \hdots, W_{n_2}\}$.
                \item 
                    Calculate the test statistic
                    \begin{equation}
                        T_{n_1, n_2}(\widetilde{\mathcal{D}}_{X, n_1}, \widetilde{\mathcal{D}}_{Y, n_2}) = \frac{1}{n_1} \sum_{i = 1}^{n_1} Z_i - \frac{1}{n_2} \sum_{i' = 1}^{n_2} W_{i'}. \label{sec4:eq:intstatdisc}
                    \end{equation}
                \item
                   Sample $B \in \mathbb{N}$ permutations $\Pi_B = \{\pi_1, \hdots, \pi_B\}$ from $S_{n_1 + n_2}$ uniformly at random. Write $\widetilde{\mathcal{D}}_{n_1, n_2} = \{\widetilde{D}_1, \hdots, \widetilde{D}_{n_1}, \widetilde{D}_{n_1 + 1}, \hdots, \widetilde{D}_{n_1 + n_2}\}$ where, for $i \in [n_1]$ and $i' \in [n_2]$, we denote $\widetilde{D}_{i} = Z_i, \widetilde{D}_{n_1 + i'} = W_{i'}$. For $b \in [B]$, calculate the test statistic on the permuted data
                    \begin{equation}
                        T_{n_1, n_2}^{\pi_b}(\widetilde{\mathcal{D}}_{n_1, n_2}) =  \frac{1}{n_1} \sum_{i = 1}^{n_1} \widetilde{D}_{\pi_b(i)} - \frac{1}{n_2} \sum_{i' = 1}^{n_2} \widetilde{D}_{\pi_b(n_1 + i')}. \label{sec4:eq:intstatdiscperm}
                    \end{equation}
                \item
                    Return the $p$-value $p(T_{n_1, n_2}, \Pi_B)$ as in \eqref{sec2:def:permpval}, and denote the test outcome $\mathbbm{1}\{p(T_{n_1, n_2}, \Pi_B) \leq \alpha\}$.
            \end{enumerate}
            The privatisation in \eqref{sec4:eq:UEint} satisfies $\varepsilon$-LDP by, for example, \citet[Section~3.2]{Duchi:2013:Minimax}, and \eqref{sec4:eq:UEint2} satisfies $\varepsilon$-LDP as an instance of binary \beame{randomised}{randomized}{randomized} response \citep[e.g.~Lemma~2.5][]{Gopi:2020:LDPHypoSelection}. Additionally, by the post-processing property \Cref{sec2:prop:postprocessing}, no privacy leakage occurs when permuting the \beame{privatised}{privatized}{privatized} data.

            The interactive procedure first uses an initial fold of the data to obtain crude estimates of the distributions of the two samples. Then, these estimates are used to inform the \beame{privatisation}{privatization}{privatization} mechanism for the second fold, yielding univariate responses which can sufficiently estimate the separation of the distributions using only a rescaled binary output. This reduction in dimension is the key to the improved rate, significantly reducing the magnitude of injected noise necessary to ensure privacy. 
            
            This procedure is similar to the interactive goodness-of-fit test in \cite{Berrett:2020:Faster}, but the two-sample setting significantly complicates the analysis, as does our goal to obtain logarithmic dependence on the type-I and type-II errors. Key to attaining this logarithmic dependence is a result which shows that the co-ordinates of the vectors arising from the unary-encoding \beame{privatisation}{privatization}{privatization} procedure \eqref{sec4:eq:UEint} satisfy the property of negative association \citep{Joag:1983:NA}. In particular, the co-ordinates of the resulting vector are not independent, making it difficult to \beame{analyse}{analyse}{analyze} the combined behaviour of the estimates \eqref{sec4:eq:UEpmfests}. Demonstrating the negative association property for this \beame{privatisation}{privatization}{privatization} procedure in \Cref{app:lem:UENA} lets us bypass this dependence issue by providing a bound on the expectation of a product that is at least as good as what could be attained under independence, allowing us to show that the test statistic $T_{n_1, n_2}$ is sub-exponential. More generally, this negative association result enables the proof of strong concentration properties of the unary-encoding \beame{privatisation}{privatization}{privatization} procedure which may be of independent interest.

            Numerical simulations exploring the performance of this testing procedure are conducted in \Cref{sec6:disc}. Further, in practice the truncation width $\tau$ in \eqref{sec4:eq:UEint2} can be varied up to a constant. The sensitivity of the testing procedure to changes in this truncation width is explored in \Cref{sec6:sensitivity}.
            
        \section{Continuous Case} \label{sec5}
            In this section, we consider the case of continuous observations where we take $\mathcal{X} = \mathcal{Y} = [0,1]^d$ for some $d \in \mathbb{N}$ and where the distributions $P_X, P_Y$ of interest have corresponding density functions $f_X, f_Y$ satisfying a Sobolev regularity condition. We consider an ellipsoid in a Sobolev space \citep[e.g.~Equation~(1.91) in][]{Tsybakov:2009:Book} defined as follows. For $R > 0$, let
            \begin{equation} \label{sec5:eq:VectorBall}
                \mathbb{N}_0 = \{0\} \cup \mathbb{N} \quad \mbox{and} \quad
                \mathbb{N}_0^d(R) = \{\vecbf{l} \in \mathbb{N}_0^d : 0 < \| \vecbf{l} \|_2 \leq R\}.
            \end{equation}
            Denote the orthonormal trigonometric basis
            \begin{equation} \label{sec5:eq:trigbasis}
                \varphi_0(x) = 1, \quad \varphi_{2j-1}(x) = 2^{1/2}\sin(2 \pi j x) \quad \mbox{and} \quad \varphi_{2j}(x) = 2^{1/2}\cos(2 \pi j x), \quad j \in \mathbb{N},
            \end{equation}
            and, for $\vecbf{l} \in \mathbb{N}_0^d(R)$, write
            \begin{equation} \label{sec5:eq:trigbasisalldim}
                \varphi_{\vecbf{l}}(\vecbf{x}) = \prod_{j = 1}^d \varphi_{l_j}(x_j).
            \end{equation}
            Given a smoothness parameter $s > 0$ and a constant radius $r > 0$, a Sobolev class of functions of smoothness $s$ and radius $r$ is given by
            \begin{equation} \label{sec5:eq:sobolev}
                \mathcal{S}_{s, r, d} = \bigg\{ f \in L^2([0,1]^d) : \, f = 1 + \sum_{\vecbf{l} \in \mathbb{N}_0^d} \theta_{\vecbf{l}} \varphi_{\vecbf{l}},\; \sum_{\vecbf{l} \in \mathbb{N}_0^d} \| \vecbf{l} \|_2^{2s} \theta_{\vecbf{l}}^2 \leq r^2,\; \theta_{\vecbf{l}} \in \mathbb{R} \; \forall \, \vecbf{l} \in \mathbb{N}_0^d(R) \bigg\}.
            \end{equation}
            Restricting to density functions, consider the sub-class of Sobolev densities
            \begin{align*}
                \mathcal{F}_{s, r, d} = \bigg\{ f \in \mathcal{S}_{s, r, d} : \, f \geq 0, \, \int f(\vecbf{x}) \diff{\vecbf{x}} = 1 \bigg\}.
            \end{align*}
            Hence, the testing problem is
            \begin{equation} \label{sec5:eq:Testdef}
                \mathrm{H}_0: f_X = f_Y \quad \mbox{vs} \quad \mathrm{H}_1: \|f_X - f_Y\|_p \geq \rho,
            \end{equation}
            for $p \in [1, 2]$.
            We have the following result.
            \begin{theorem}\label{sec5:thm:main}
                Take $n_1 \leq n_2$ without loss of generality and fix constant $s, r > 0$. Assume that $n_1\varepsilon^2 \geq C\log\{1/(\alpha\beta)\}$ for $C > 0$ some absolute constant. The following results hold for $p \in [1, 2]$.
                
                (i) The non-interactive minimax separation radii as defined in \eqref{sec2:eq:minimaxsep} for the testing problem \eqref{sec5:eq:Testdef} satisfy
                \begin{align*}
                    c_{r, s, d}\{\log(n_1\varepsilon^2)\}^{-1}\bigg(\frac{1}{n_1\varepsilon^2}\bigg)^{2s/(4s+3d)}
                    \leq \rho_{\mathrm{NI}}^\ast(&n_1, n_2, \varepsilon, L_p , \alpha, \beta) \\
                    &\leq C_{r, s, d}\bigg(\frac{[\log\{1/(\alpha\beta)\}]^2}{n_1\varepsilon^2}\bigg)^{2s/(4s+3d)},
                \end{align*}
                where $c_{r, s, d}, C_{r, s, d} > 0$ are constants depending on $r, s$ and $d$.

                (ii) The interactive minimax separation radii as defined in \eqref{sec2:eq:minimaxsep} for the testing problem \eqref{sec5:eq:Testdef} satisfy
                \begin{align*}
                    c_{r, s, d}'\{\log(n_1\varepsilon^2)\}^{-1} \bigg(\frac{1}{n_1\varepsilon^2}\bigg)^{2s/(4s+2d)}
                    \leq \rho_{\mathrm{I}}^\ast(&n_1, n_2, \varepsilon, L_p , \alpha, \beta) \\
                    &\leq C_{r, s, d}'\bigg(\frac{[\log(n_2/\beta)\log\{1/(\alpha\beta)\}]^2}{n_1\varepsilon^2}\bigg)^{2s/(4s+2d)},
                \end{align*}
                where $c_{r, s, d}', C_{r, s, d}' > 0$ are constants depending on $r, s$ and $d$.
            \end{theorem}

            The upper bounds in the theorem follow from the testing procedures we construct in \Cref{sec5:UpperBoundNonInt} and \Cref{sec5:UpperBoundInt}, the performance of which are \beame{analysed}{analysed}{analyzed} in \Cref{app:sec:ContUB}. The proof of the lower bounds are contained in \Cref{app:sec:contlbproof}.

            Inspecting the upper and lower bounds, we see that we obtain matching upper and lower bounds on both the non-interactive and interactive separation radii with respect to the $L_p$-norm, up to the logarithmic factor of the type-I and type-II errors $\alpha, \beta$, and a factor logarithmic in $n_2$ in the interactive setting. As with the discrete setting in \Cref{sec4}, this demonstrates an improvement on existing work on permutation tests, private or non-private, which requires the assumption $n_1 \asymp n_2$ when aiming for logarithmic dependence on the type-I and type-II errors.
            
            We see that in both the non-interactive and interactive settings, the $L_p$-separation rates are the same for all $p \in [1, 2]$, unlike in the setting with discrete data. Considering the upper bounds in particular, this follows from the fact that, by H\"{o}lder's inequality, the $L_2$-norm of a function is lower bounded by the $L_1$-norm for functions supported on $[0,1]$. Further, this inequality still holds, up to some constant in $d$, for functions supported on any bounded subset of $\mathbb{R}^d$. This assumption that the function is supported on a bounded subset of $\mathbb{R}^d$ is common in both LDP non-parametric estimation \citep[e.g.][]{Duchi:2018:DJW, Li:2023:Robust} and testing problems \citep[e.g.][]{Lam-Weil:2022:GOFtest}. Nevertheless, there is some literature for testing under local privacy constraints which does not make this assumption \citep{Dubois:2019:GOFtest}, and exploring the difference, if any, of $L_p$-separation radii for various values of $p$ under this setting remains an interesting open problem, as only $L_1$-separation is considered therein.

            Focusing on the difference between the non-interactive and interactive rates, we observe an improved polynomial exponent. Such an improvement is also observed in the goodness-of-fit setting under a H\"{o}lder assumption \citep{Dubois:2019:GOFtest}. However, the tests we construct differ significantly, and the way interactivity is implemented differs between the two settings. We note the additional logarithmic dependence on the larger sample $n_2$ in the upper bound for the interactive rate. This arises during the analysis due to a union bound argument controlling $O(n_1 + n_2) = O(n_2)$ many events using sub-Gaussian concentration arguments, leading to this logarithmic factor. As the benefits of a larger sample will be generally exhausted before the size of the larger sample is exponentially larger than the smaller one, the presence of this term does will not cause issues with unbalanced samples in practice.
            
        \begin{remark}[Comparisons with \cite{Mun:2025:LocalPerm}]
            Comparing with \cite{Mun:2025:LocalPerm}, we first note that we make differing smoothness assumptions. In \cite{Mun:2025:LocalPerm}, both H\"{o}lder and Besov conditions are considered, which enable a binning argument after which one may apply the result for discrete distributions. We consider a Sobolev condition which does not lend itself to such a \beame{discretisation}{discretization}{discretization} approach. Nevertheless, we are again able to obtain logarithmic dependence on the type-I and type-II errors without the further assumptions made in \cite{Mun:2025:LocalPerm}. We further obtain results with respect to general $L_p$-separation and in the interactive setting which are new to the literature.
        \end{remark}

        \begin{remark}
            We consider comparing our results for densities satisfying the Sobolev condition \eqref{sec5:eq:sobolev} to those based on a Besov condition as in \cite{Lam-Weil:2022:GOFtest} and also \cite{Mun:2025:LocalPerm}. The definition of a Besov space using moduli of smoothness can be shown to include a Sobolev space as a special case \citep[see e.g.~Section~4.3 in][]{Gine:2021:InfiniteDimensionalBook}. Hence, in some cases results for Besov spaces are adapted for Sobolev spaces \citep[e.g.~Section~4.1][]{Li:2023:Robust}. However, in \cite{Lam-Weil:2022:GOFtest} and \cite{Mun:2025:LocalPerm} the Besov condition considered is based on an ellipsoid condition for the coefficients of an expansion using Haar wavelets, for which an equivalence with the definition based on moduli of smoothness holds for only a small range of the smoothness parameter. To be precise, the Haar wavelet basis is $1$-regular \citep[Definition~4.2.14 in][]{Gine:2021:InfiniteDimensionalBook}, and so equivalence only holds for smoothness $s < 1$ \citep[Theorem~4.3.2 in][]{Gine:2021:InfiniteDimensionalBook}. Therefore, our results under the Sobolev condition \eqref{sec5:eq:sobolev} with smoothness parameters $s \geq 1$ are not immediately comparable to work based on Besov-type conditions, even in the common setting of $s = 1$ analogous to a Lipschitz condition.
        \end{remark}
        
        \subsection{Testing Procedure (Non-Interactive)}  \label{sec5:UpperBoundNonInt}
            Via the Sobolev regularity assumption, we are able to reduce from an infinite dimensional to finite dimensional problem as follows: letting $R > 0$ be specified later, denote $V = |\mathbb{N}_0^d(R)|$. We then write $\vecbfm{\theta}_{X, 1:V} = (\theta_{X, \vecbf{l}} : \vecbf{l} \in \mathbb{N}_0^d(R))^T$ and likewise for $\vecbfm{\theta}_{Y, 1:V}$ with some arbitrary fixed indexing on the $\vecbf{l}$. We have via the orthonormality of the basis \eqref{sec5:eq:trigbasis} that
            \begin{align}
                \|f_X - f_Y\|_2^2
                = \sum_{\vecbf{l} \in \mathbb{N}_0^d} (\theta_{X, \vecbf{l}} - \theta_{Y, \vecbf{l}})^2
                &= \sum_{\vecbf{l} \in \mathbb{N}_0^d(R)} (\theta_{X, \vecbf{l}} - \theta_{Y,\vecbf{l}})^2 + \sum_{\vecbf{l} \in \mathbb{N}_0^d \setminus \mathbb{N}_0^d(R)} (\theta_{X, \vecbf{l}} - \theta_{Y, \vecbf{l}})^2 \nonumber \\
                &\leq \|\vecbfm{\theta}_{X, 1:V} - \vecbfm{\theta}_{Y, 1:V} \|_2^2 + \frac{4r^2}{R^{2s}}, \label{sec5:eq:SobolevError}
            \end{align}
            where the last inequality follows from the Sobolev smoothness condition \eqref{sec5:eq:sobolev}. We see that the second term in the final expression corresponds to a bias term arising due to the truncation of the sequence of coefficients, whilst a greater value of $R$ would reduce this bias term at the expense of a larger variance when estimating the first term. We explore the dependence on $R$ in numerical simulations in \Cref{sec6:cont}. We further observe that under the alternative hypothesis where $\|f_X - f_Y\|_2^2$ is large, the term $\|\vecbfm{\theta}_{X, 1:V} - \vecbfm{\theta}_{Y, 1:V} \|_2^2$ will also be large. Hence, we first seek to privately estimate the value of the multivariate quantities $\vecbfm{\theta}_{X, 1:V}$ and $\vecbfm{\theta}_{Y, 1:V}$, then apply a similar test statistic to \Cref{sec4:UpperBoundNonInt} which provides an estimate of $\|\vecbfm{\theta}_{X, 1:V} - \vecbfm{\theta}_{Y, 1:V} \|_2^2$.
            
            In the following algorithm, we use \beame{privatisation}{privatization}{privatization} mechanisms for vectors confined to a sup-norm ball of fixed radius from \cite{Duchi:2018:DJW} and \cite{Li:2023:Thesis}. Note for any $\vecbf{l} \in \mathbb{N}_0^d(R)$, we have $\sup_{\vecbf{x}} |\varphi_{\vecbf{l}}(\vecbf{x})| \leq 2^{d/2}$, which allows us to apply the aforementioned \beame{privatisation}{privatization}{privatization} mechanisms. In what follows, we make use of the following function. We define $R_0^\ast: \mathbb{R} \times [0, \infty) \times [0,1] \rightarrow \mathbb{R}$ where
            \begin{equation} \label{sec5:eq:RRFunc}
                R_0^\ast(x, t, u) = tc_\varepsilon \mathbbm{1}\bigg\{u \leq \frac{1}{2}\bigg(1 + \frac{x}{t}\bigg)\bigg\} - t c_\varepsilon \mathbbm{1}\bigg\{u > \frac{1}{2}\bigg(1 + \frac{x}{t}\bigg)\bigg\},
            \end{equation}

            The algorithm is \beame{formalised}{formalized}{formalized} as follows.
            \begin{enumerate}
                \item
                    Set $R = \{n_1\varepsilon^2/[\log\{1/(\alpha\beta)\}]\}^{1/(2s + 3d/2)}$. Then, for $i \in [n_1], i' \in [n_2]$ and $\vecbf{l} \in \mathbb{N}_0^d(R)$, construct the vectors $\hat{\vecbfm{\theta}}_{X, 1:V}^{(i)}$ and $\hat{\vecbfm{\theta}}_{Y, 1:V}^{(i')}$ with co-ordinates as
                    \begin{equation} \label{sec4:eq:basistermestimate}
                        \hat{\theta}_{X, \vecbf{l}}^{(i)} = \varphi_\vecbf{l}(\vecbf{X}_i), \mbox{ and } \hat{\theta}_{Y, \vecbf{l}}^{(i')} = \varphi_\vecbf{l}(\vecbf{Y}_{i'}),
                    \end{equation}
                    where $\varphi_{\vecbf{l}}$ is as in \eqref{sec5:eq:trigbasis}.
                \item 
                    For $i \in [n_1]$ and $i' \in [n_2]$, publish $\varepsilon$-private views $\vecbf{Z}_i$, $\vecbf{W}_{i'}$ of $\hat{\vecbfm{\theta}}_{X, 1:V}^{(i)}$ and $\hat{\vecbfm{\theta}}_{Y, 1:V}^{(i')}$ respectively via the following procedure:

                    For $\vecbf{l} \in \mathbb{N}_0^d(R)$, let $V_{i, \vecbf{l}}^{(X)}, V_{i', \vecbf{l}}^{(Y)} \overset{\mathrm{i.i.d.}}{\sim} \mathrm{Unif}[0,1]$ and let $\tilde{\vecbf{Z}}_i, \tilde{\vecbf{W}}_i \in \{-2^{d/2}, 2^{d/2}\}^V$ with co-ordinates
                    \begin{equation} \label{sec4:eq:nonintcoefviews}
                        \tilde{Z}_{i, \vecbf{l}} = R_0^\ast(\hat{\theta}_{X, \vecbf{l}}^{(i)}, 2^{d/2}, V_{i,\vecbf{l}}^{(X)}), \mbox{ and } \tilde{W}_{i', \vecbf{l}} = R_0^\ast(\hat{\theta}_{Y, \vecbf{l}}^{(i')}, 2^{d/2}, V_{i',\vecbf{l}}^{(Y)}),
                    \end{equation}
                    with the function $R_0^\ast$ as in \eqref{sec5:eq:RRFunc}.

                    Obtain the \beame{privatised}{privatized}{privatized} samples $\widetilde{\mathcal{D}}_{X, n_1} = \{\vecbf{Z}_1, \hdots, \vecbf{Z}_{n_1}\}$ and $\widetilde{\mathcal{D}}_{Y, n_2} = \{\vecbf{W}_1, \hdots, \vecbf{W}_{n_2}\}$ using the sampling mechanisms \eqref{app:eq:oddMsamplemech} and \eqref{app:eq:evenMsamplemech} in the cases of odd $V$ and even $V$ respectively.

                \item
                    Calculate the test statistic $U_{n_1, n_2} = U_{n_1, n_2}(\widetilde{\mathcal{D}}_{X, n_1}, \widetilde{\mathcal{D}}_{Y, n_2})$ as in \eqref{sec3:eq:Ustatistic}.
                \item
                    Sample $B \in \mathbb{N}$ permutations $\Pi_B = \{\pi_1, \hdots, \pi_B\}$ from $S_{n_1 + n_2}$ uniformly at random. Write $\widetilde{\mathcal{D}}_{n_1, n_2} = \{\widetilde{\vecbf{D}}_1, \hdots, \widetilde{\vecbf{D}}_{n_1}, \widetilde{\vecbf{D}}_{n_1 + 1}, \hdots, \widetilde{\vecbf{D}}_{n_1 + n_2}\}$ where, for $i \in [n_1]$ and $i' \in [n_2]$, $\widetilde{\vecbf{D}}_{i} = \vecbf{Z}_i$ and $\widetilde{\vecbf{D}}_{n_1 + i'} = \vecbf{W}_{i'}$. For $b \in [B]$, calculate the test statistic on the permuted data $U_{n_1, n_2}^{\pi_b}(\widetilde{\mathcal{D}}_{n_1, n_2})$ as in \eqref{sec3:eq:permutedUstatistic}.
                \item
                    Return the $p$-value $p(U_{n_1, n_2}, \Pi_B)$ as in \eqref{sec2:def:permpval}, and denote the test outcome $\mathbbm{1}\{p(U_{n_1, n_2}, \Pi_B) \leq \alpha\}$.
            \end{enumerate}
            We first discuss the choice of \beame{privatisation}{privatization}{privatization} mechanisms in Step~2. The method for odd $V$ \citep[Equation~26 in][]{Duchi:2018:DJW} and even $V$ \citep[Equation~3.18 in][]{Li:2023:Thesis} both satisfy $\varepsilon$-LDP and are unbiased as shown in their respective references. The reason for different mechanisms based on the parity of $V$ is as follows: In the case that $V$ is odd, the mechanism of \cite{Duchi:2018:DJW} is shown to be sub-Gaussian in \cite{Li:2023:Thesis}. For the case of even $V$, the same argument does not hold, and so instead the sub-Gaussianity of a slightly different mechanism is proved, based on that of \cite{Dubois:2021:thesis}. By choosing between these two mechanisms, we have that the private output of Step~2 is sub-Gaussian for all $V$, allowing us to satisfy the conditions of \Cref{sec3:thm:sepcondU}. Finally, by the post-processing property \Cref{sec2:prop:postprocessing}, the permuted data remains $\varepsilon$-LDP.

            We now give intuition for this procedure. We carry out an LDP mean estimation procedure to estimate the basis coefficients. We then use these to estimate the $L_2$-norm and use large values of this estimator as evidence against the null. Regarding the truncation radius $R$ determining the number of coefficients considered, less approximation bias is introduced as $R$ grows, but the variance associated with the higher dimensional mean estimation problem is the greater. The choice of $R$ in Step~1 is chosen to optimally balance these bias and variance terms.

            Numerical simulations exploring the performance of this testing procedure are conducted in \Cref{sec6:cont}, in particular the behaviour as we vary the truncation radius $R$.

        \subsection{Testing Procedure (Interactive)}
            \label{sec5:UpperBoundInt}
            To obtain a testing procedure for the interactive case with the desired performance, we require a combination of multiple tests. We first give an overview of each test before \beame{formalising}{formalizing}{formalizing} the algorithms.

            The first test mimics the interactive procedure in the discrete setting, obtaining a rough estimate of the distributions and using truncation to limit the required magnitude of injected noise for privacy. This enables users in the second half to estimate the $L_2$-norm using univariate quantities, reducing the noise incurred due to \beame{privatisation}{privatization}{privatization}. However, to be able to obtain the desired power guarantee, the unknown $L_2$-norm would have to be used as the truncation threshold. We circumvent this issue by using an adaptive procedure to carry out this test over a grid of candidate values of the unknown $L_2$-norm.
            
            We then introduce a second interactive test to avoid the issue of testing over an arbitrarily large grid in the previous test. This test restricts to some index $\vecbf{l}^\ast \in \mathbb{N}_0^d$ and estimates the difference $|\theta_{X, \vecbf{l}^\ast} - \theta_{Y, \vecbf{l}^\ast}|$, a non-zero value being an indication of departure from the null. The index $\vecbf{l}^\ast$ is chosen as the index corresponding to the maximal difference in an initial rough estimate of all the candidate differences. For this procedure to be successful, at least one difference must be sufficiently large so that it can be detected with high probability. This condition is satisfied once the $L_2$-norm itself is sufficiently large.

            We now \beame{formalise}{formalize}{formalize} these testing procedures. We split our samples used within each individual test as in~\eqref{sec4:eq:sampledefs}.

            \noindent
            \textbf{First Procedure:}
            Given a value $\eta > 0$ and desired type-I error guarantee $\alpha$, we construct the test $\phi_{\mathrm{trunc}, \alpha}^{\eta}$ as follows.
            \begin{enumerate}
                \item
                    Set desired type-I error $\alpha$ and type-II error $\beta$. Set
                    \begin{equation} \label{sec5:eq:trunc1}
                        R = [n_1\varepsilon^2/\{\log(4n_2/\beta)[\log\{1/(\alpha\beta)\}]^2\}]^{1/(2s + d)}.
                    \end{equation}
                    Partition the observations in $\mathcal{D}_{X, n_1}'$ and $\mathcal{D}_{Y, n_1}'$ into $V = |\mathbb{N}_0^d(R)|$ equally sized samples as follows: For $\vecbf{l} \in \mathbb{N}_0^d(R)$, and some bijective indexing $\iota : \mathbb{N}_0^d(R) \rightarrow [V]$, denote the index sets
                    \begin{equation} \label{sec5:eq:SamplePartition}
                        \begin{aligned}
                            N_{X, \vecbf{l}} &= \{
                            (\iota(\vecbf{l}) - 1)\lfloor n_1/V \rfloor + 1, \hdots, \iota(\vecbf{l})\lfloor n_1/V \rfloor\}, \\
                            N_{Y, \vecbf{l}} &= \{(\iota(\vecbf{l}) - 1)\lfloor n_2/V \rfloor + 1, \hdots, \iota(\vecbf{l})\lfloor n_2/V \rfloor\},
                        \end{aligned}
                    \end{equation}
                    and define the partitions $\mathcal{D}_{X, n_1}^{(\vecbf{l})} = \{\vecbf{X}_i'\}_{i \in N_{X,\vecbf{l}}}$ and  $\mathcal{D}_{Y, n_2}^{(\vecbf{l})} = \{\vecbf{Y}_{i'}'\}_{i' \in N_{Y,\vecbf{l}}}$.
                \item
                    For $\vecbf{l} \in \mathbb{N}_0^d(R)$, $i \in N_{X,\vecbf{l}}$, and $i' \in N_{Y,\vecbf{l}}$, let $V_{i}'^{(X)}, V_{i'}'^{(Y)} \overset{\mathrm{i.i.d.}}{\sim} \mathrm{Unif}[0,1]$ and construct
                    \begin{equation} \label{sec5:eq:RRsobcoeffestsProc1}
                        Z_{i}' = R_0(\varphi_{\vecbf{l}}(\vecbf{X}_i'), 2^{d/2}, V_{i}'^{(X)}),\; \mbox{and } W_{i'}' = R_0(\varphi_{\vecbf{l}}(\vecbf{Y}_{i'}'), 2^{d/2} V_{i'}'^{(Y)}),
                    \end{equation}
                    where $R_0$ is as in \eqref{sec4:eq:RRFunc}.
                \item
                    For $\vecbf{l} \in \mathbb{N}_0^d(R)$, calculate the private estimates of the $\vecbf{l}$-th coefficient
                    \begin{equation} \label{sec5:eq:RRsobcoeffestscombProc1}
                        \hat{\theta}_{X,\vecbf{l}} = \frac{V}{n_1} \sum_{i \in N_{X, \vecbf{l}}}  Z_i',\; \mbox{and} \quad
                        \hat{\theta}_{Y,\vecbf{l}} = \frac{V}{n_2} \sum_{i' \in N_{Y, \vecbf{l}}}  W_{i'}'.
                    \end{equation}
                \item 
                    For $i \in [n_1]$, $i' \in [n_2]$, construct the private views
                    \begin{equation} \label{sec5:eq:contMethod1NotTrunc}
                        \begin{aligned} 
                            Z_i &= \Pi_{[-\eta, \eta]}\bigg\{\sum_{\vecbf{l} \in \mathbb{N}_0^d(R)} (\hat{\theta}_{X,\vecbf{l}} - \hat{\theta}_{Y,\vecbf{l}})\varphi_\vecbf{l}(\vecbf{X}_i)\bigg\} + \frac{2\eta}{\varepsilon}\zeta_i,\; \mbox{and} \\
                            W_{i'} &= \Pi_{[-\eta, \eta]}\bigg\{\sum_{\vecbf{l} \in \mathbb{N}_0^d(R)} (\hat{\theta}_{X,\vecbf{l}} - \hat{\theta}_{Y,\vecbf{l}})\varphi_\vecbf{l}(\vecbf{Y}_i)\bigg\} + \frac{2\eta}{\varepsilon}\xi_{i'},
                        \end{aligned}
                    \end{equation}
                    where $\zeta_i$ and $\xi_{i'}$ are i.i.d.~standard Laplace random variables. Denote the \beame{privatised}{privatized}{privatized} samples $\widetilde{\mathcal{D}}_{X, n_1} = \{Z_1, \hdots, Z_{n_1}\}$ and $\widetilde{\mathcal{D}}_{Y, n_2} = \{W_1, \hdots, W_{n_2}\}$.
                \item 
                    Calculate the test statistic $T_{n_1, n_2}(\widetilde{\mathcal{D}}_{X, n_1}, \widetilde{\mathcal{D}}_{Y, n_2})$ as in \eqref{sec4:eq:intstatdisc}.
                \item
                    Sample $B \in \mathbb{N}$ permutations $\Pi_B = \{\pi_1, \hdots, \pi_B\}$ from $S_{n_1 + n_2}$ uniformly at random. Write $\widetilde{\mathcal{D}}_{n_1, n_2} = \{\widetilde{D}_1, \hdots, \widetilde{D}_{n_1}, \widetilde{D}_{n_1 + 1}, \hdots, \widetilde{D}_{n_1 + n_2}\}$ where, for $i \in [n_1]$ and $i' \in [n_2]$, we denote $\widetilde{D}_{i} = Z_i, \widetilde{D}_{n_1 + i'} = W_{i'}$. For $b \in [B]$, calculate the test statistic on the permuted data $T_{n_1, n_2}^{\pi_b}(\widetilde{\mathcal{D}}_{n_1, n_2})$ as in \eqref{sec4:eq:intstatdiscperm}.
                \item
                    Return the $p$-value $p(T_{n_1, n_2}, \Pi_B)$ as in     \eqref{sec2:def:permpval}, and denote the outcome of the test by $\phi_{\mathrm{trunc}, \alpha}^{\eta} = \mathbbm{1}\{p(T_{n_1, n_2}, \Pi_B) \leq \alpha\}$.
            \end{enumerate}
            The \beame{privatisation}{privatization}{privatization} in \eqref{sec5:eq:RRsobcoeffestsProc1} satisfies $\varepsilon$-LDP by Equation~4.1 in \cite{Rohde:2020:Geometrizing}, and note that truncation does not occur as $\sup_{\vecbf{x}} |\varphi_{\vecbf{l}}(\vecbf{x})| \leq 2^{d/2}$. The \beame{privatisation}{privatization}{privatization} in \eqref{sec5:eq:contMethod1NotTrunc} satisfies $\varepsilon$-LDP as an instance of the Laplace mechanism \citep[e.g.][]{Dwork:2006:CalibratingNoise}. The post-processing property \Cref{sec2:prop:postprocessing} ensures the permuted data remains $\varepsilon$-LDP.

            Comparing this procedure to the interactive procedure in the discrete case in \Cref{sec4:UpperBoundInt}, a key difference is that the initial rough estimate is only over a finite truncation of the basis expansion as per \eqref{sec5:eq:sobolev}. Further, the estimator produced by each user in the second stage is a sum of dependent terms, complicating concentration arguments and precluding an easy analysis. Unlike the discrete setting, we do not have an analogous negative association result. Nevertheless, we are able to produce an optimal testing procedure by \beame{utilising}{utilizing}{utilizing} truncation and calibrating the noise added for privacy to this truncation level. 

            \noindent
            \textbf{Second Procedure:}
            For a desired type-I error guarantee $\alpha$, we construct the test $\phi_{\mathrm{local}, \alpha}$.
            \begin{enumerate}
                \item
                    Set desired type-I error $\alpha$ and type-II error $\beta$. Set
                    \begin{equation} \label{sec5:eq:trunc2}
                        R = [n_1\varepsilon^2/\{\log(4n_2/\beta)[\log\{1/(\alpha\beta)\}]^2\}]^{1/(2s + 1)}.
                    \end{equation}
                    For $\vecbf{l} \in \mathbb{N}_0^d(R)$, partition the observations in $\mathcal{D}_{X, n_1}'$ and $\mathcal{D}_{Y, n_1}'$ into $V = |\mathbb{N}_0^d(R)|$ equally sized partitions as in \eqref{sec5:eq:SamplePartition} yielding $N_{X, \vecbf{l}}$, $N_{Y,\vecbf{l}}$.
                \item
                    For $\vecbf{l} \in \mathbb{N}_0^d(R)$, $i \in N_{X, \vecbf{l}}$, and $i' \in N_{Y, \vecbf{l}}$, let $V_{i}'^{(X)}, V_{i'}'^{(Y)} \overset{\mathrm{i.i.d.}}{\sim} \mathrm{Unif}[0,1]$ and construct
                    \begin{equation} \label{sec5:eq:RRsobcoeffests}
                        Z_{i}' = R_0(\varphi_{\vecbf{l}}(X_i'), 2^{d/2}, V_{i}'^{(X)}),\; \mbox{and } W_{i'}' = R_0(\varphi_{\vecbf{l}}(Y_{i'}'), 2^{d/2}, V_{i'}'^{(Y)}),
                    \end{equation}
                    where $R_0$ is as in \eqref{sec4:eq:RRFunc}.
                \item
                    For $\vecbf{l} \in \mathbb{N}_0^d(R)$, calculate the private estimates of the $\vecbf{l}$-th coefficient as in \eqref{sec5:eq:RRsobcoeffestscombProc1}, yielding $\hat{\theta}_{X,\vecbf{l}}$, $\hat{\theta}_{Y,\vecbf{l}}$. Let $\vecbf{l}^\ast = \argmax_{\vecbf{l} \in \mathbb{N}_0^d(R)}\{|\hat{\theta}_{X,\vecbf{l}} - \hat{\theta}_{Y,\vecbf{l}}|\}$.
                \item 
                    For $i \in [n_1]$, $i' \in [n_2]$, let $V_{i}^{(X)}, V_{i'}^{(Y)} \overset{\mathrm{i.i.d.}}{\sim} \mathrm{Unif}[0,1]$ and construct
                    \begin{equation} \label{sec5:eq:localProcSelected}
                        Z_{i} = R_0(\varphi_{\vecbf{l}^\ast}(X_i), 2^{d/2}, V_{i}^{(X)}),\; \mbox{and } W_{i'} = R_0(\varphi_{\vecbf{l}^\ast}(Y_{i'}), 2^{d/2}, V_{i'}^{(Y)}),
                    \end{equation}
                    Denote the \beame{privatised}{privatized}{privatized} samples $\widetilde{\mathcal{D}}_{X, n_1} = \{Z_1, \hdots, Z_{n_1}\}$ and $\widetilde{\mathcal{D}}_{Y, n_2} = \{W_1, \hdots, W_{n_2}\}$.
                \item
                    Calculate the test statistic $U_{n_1, n_2} = U_{n_1, n_2}(\widetilde{\mathcal{D}}_{X, n_1}, \widetilde{\mathcal{D}}_{Y, n_2})$ as in \eqref{sec3:eq:Ustatistic}.
                \item
                    Sample $B \in \mathbb{N}$ permutations $\Pi_B = \{\pi_1, \hdots, \pi_B\}$ from $S_{n_1 + n_2}$ uniformly at random. Write $\widetilde{\mathcal{D}}_{n_1, n_2} = \{\widetilde{\vecbf{D}}_1, \hdots, \widetilde{\vecbf{D}}_{n_1}, \widetilde{\vecbf{D}}_{n_1 + 1}, \hdots, \widetilde{\vecbf{D}}_{n_1 + n_2}\}$ where, for $i \in [n_1]$ and $i' \in [n_2]$, $\widetilde{\vecbf{D}}_{i} = \vecbf{Z}_i$ and $\widetilde{\vecbf{D}}_{n_1 + i'} = \vecbf{W}_{i'}$. For $b \in [B]$, calculate the test statistic on the permuted data $U_{n_1, n_2}^{\pi_b}(\widetilde{\mathcal{D}}_{n_1, n_2})$ as in \eqref{sec3:eq:permutedUstatistic}.
                \item
                    Return the $p$-value $p(U_{n_1, n_2}, \Pi_B)$ as in     \eqref{sec2:def:permpval}, and denote the outcome of the test by $\phi_{\mathrm{local}, \alpha} = \mathbbm{1}\{p(U_{n_1, n_2}, \Pi_B) \leq \alpha\}$.
            \end{enumerate}
            The \beame{privatisation}{privatization}{privatization} in both \eqref{sec5:eq:RRsobcoeffests} and \eqref{sec5:eq:localProcSelected} satisfy $\varepsilon$-LDP by Equation~4.1 in \cite{Rohde:2020:Geometrizing}, and note that truncation never occurs as $\sup_\vecbf{x} |\varphi_{\vecbf{l}}(\vecbf{x})| \leq 2^{d/2}$. The permuted data remains $\varepsilon$-LDP by \Cref{sec2:prop:postprocessing}.

            With these tests constructed, we now combine them to form the following interactive test.
            
            \noindent
            \textbf{Combined Test:}
            \begin{enumerate}
                \item Set desired type-I error $\alpha$ and type-II error $\beta$. Denote $J = \{1, \hdots, \log_2(V)\}$. For $j \in J$, let
                \begin{equation} \label{sec5:eq:truncvals}
                    \eta_j = \bigg\{ \frac{C_{r, s, d} 2^{j} V^2[\log\{2|J|/(\alpha\beta)\}]^2\{\log(4n_2/\beta)\}^2}{n_1\varepsilon^2/(2|J|)} \bigg\}^{1/2},
                \end{equation}
                for $C_{r, s, d} > 0$ some sufficiently large constant depending on $r, s$ and $d$.

                Partition the observations in $\mathcal{D}_{X, n_1}$ and $\mathcal{D}_{Y, n_2}$ into $|J|$ equally sized samples as follows: For $l \in [|J|]$, denote the index sets
                \begin{equation}
                    \label{sec5:eq:combsamplesplit}
                    \begin{aligned}
                        N_{X, l} &= \{\lfloor n_1/2 \rfloor + 1 + (l-1)\lfloor n_1/|J| \rfloor + 1, \hdots, \lfloor n_1/2 \rfloor + 1 + l\lfloor n_1/|J| \rfloor\}, \\
                        N_{Y, l} &= \{\lfloor n_2/2 \rfloor + 1 + (l-1)\lfloor n_2/|J| \rfloor + 1, \hdots, \lfloor n_2/2 \rfloor + 1 + l\lfloor n_2/|J| \rfloor\},
                    \end{aligned}
                \end{equation}
                and
                \begin{equation*}
                        N_{X, 0} = \{1, \hdots, \lfloor n_1/2 \rfloor\}, \;
                        N_{Y, 0} = \{1, \hdots, \lfloor n_2/2 \rfloor\}.
                \end{equation*}
                For $k \in \{0, 1, \hdots, |J|\}$, define the partitions $\mathcal{D}_{X, n_1}^{(k)} = \{X_i\}_{i \in N_{X,\vecbf{l}}}$ and  $\mathcal{D}_{Y, n_2}^{(k)} = \{Y_{i'}\}_{i' \in N_{Y,\vecbf{l}}}$.

                \item
                For $j \in J$, using the samples $\mathcal{D}_{X, n_1}^{(j)}$ and $\mathcal{D}_{Y, n_2}^{(j)}$, carry out the first procedure with type-I error level $\alpha/(2|J|)$ and truncation level $\eta_j$ to obtain $\phi_{\mathrm{trunc}, \alpha/2}^{\eta_j}$.
                
                Using the samples $\mathcal{D}_{X, n_1}^{(0)}$ and $\mathcal{D}_{Y, n_2}^{(0)}$, carry out the second procedure with type-I error level $\alpha/2$ to obtain $\phi_{\mathrm{local}, \alpha/2}$.

                \item
                Define the test $\phi_{\mathrm{trunc}, \alpha/2} = \max_{j \in J}\{ \phi_{\mathrm{trunc}, \alpha/(2|J|)}^{\eta_j} \}$. Then, denote the combined test $\phi_{\mathrm{final}, \alpha} = \max\{\phi_{\mathrm{trunc}, \alpha/2}, \phi_{\mathrm{local}, \alpha/2}\}$.
            \end{enumerate}

            Comparing the two sub-procedures developed, we note that the second is comparatively simpler compared to the first one, which must test over a grid of candidates for the value $\eta$ and therefore apply a Bonferroni-type correction. Hence, the second test will in many cases perform better with higher power. Nevertheless, the power guarantee fails when the separation lies in a small interval close to the lower bound, and so the first procedure is vital for constructing an overall test that has power in all the desired regimes of the separation and attaining the matching minimax bounds.

            Numerical simulations exploring the performance of this testing procedure are conducted in \Cref{sec6:cont}. The truncation width in \eqref{sec5:eq:truncvals} depends on a constant $C_{r, s, d} > 0$ where for the theoretical analysis, this is taken sufficiently large. However, in practice we may wish to set this constant smaller to reduce the magnitude of the noise we must inject for privacy. Hence, we explore the sensitivity of the testing procedure to changes in this constant in \Cref{sec6:sensitivity}.

        \subsection{Adaptive Tests}
            \label{sec5:UpperBoundNonIntAdapt}
            The tests developed in the previous section make use of and fix a smoothness parameter which is generally unknown in practice. To circumvent this, we develop tests which are adaptive to this unknown parameter, testing over a grid of candidate smoothness values and \beame{utilising}{utilizing}{utilizing} a Bonferroni correction-type argument as in \cite{Ingster:2000:AdaptiveTests}. In what follows, we construct \beame{generalisations}{generalizations}{generalizations} of the non-interactive method of \Cref{sec5:UpperBoundNonInt} and the interactive method of \Cref{sec5:UpperBoundInt}.

            \begin{theorem}\label{sec5:thm:adaptive}
                Suppose $n_1 \leq n_2$ without loss of generality. Suppose further $n_1\varepsilon^2 \geq C\log\{1/(\alpha\beta)\}$ for $C > 0$ some absolute constant and fix $r > 0$.
                
                (i) For the testing problem \eqref{sec5:eq:Testdef}, there exists a non-interactive test which is adaptive to the unknown parameter $s > 0$ and has type-I and type-II error controlled by $\alpha$ and $\beta$ respectively, provided
                \begin{align*}
                    \|f_X - f_Y\|_p \geq C_{r, s, d} \bigg(\frac{\log(n_1\varepsilon^2)[\log\{\log(n_1\varepsilon^2)/(\alpha\beta)\}]^2}{n_1\varepsilon^2}\bigg)^{2s/(4s+3d)},
                \end{align*}
                where $C_{r, s, d} > 0$ is a constant depending on $r, s$ and $d$ and $p \in [1, 2]$.
                
                (ii) For the testing problem \eqref{sec5:eq:Testdef}, there exists an interactive test which is adaptive to the unknown parameter $s > 0$ and has type-I and type-II error controlled by $\alpha$ and $\beta$ respectively, provided
                \begin{align*}
                    \|f_X - f_Y\|_p \geq C_{r, s, d}' \bigg(\frac{\log(n_1\varepsilon^2)[\log(n_2/\beta)\log\{\log(n_1\varepsilon^2)/(\alpha\beta)\}]^2}{n_1\varepsilon^2}\bigg)^{2s/(4s+2d)},
                \end{align*}
                where $C_{r, s, d}' > 0$ is a constant depending on $r, s$ and $d$ and $p \in [1, 2]$.
            \end{theorem}
            The proof of the above theorem can be found in \Cref{app:sec:adaptivetestproofs}. Before discussing this result in detail, we introduce the adaptive tests, which are \beame{formalised}{formalized}{formalized} as follows.

            \begin{enumerate}
                \item
                    Set desired type-I error $\alpha$. Denote the integers $k_{\mathrm{max}}^{\mathrm{NI}} = (2/3)\log_2(n_1\varepsilon^2) + 1$, and $k_{\mathrm{max}}^{\mathrm{I}} = \log_2(n_1\varepsilon^2) + 1$. Set $k_{\mathrm{max}}$ to be $k_{\mathrm{max}}^{\mathrm{NI}}$ or $k_{\mathrm{max}}^{\mathrm{I}}$ in the non-interactive and interactive settings respectively.
                \item
                    Split the samples $\mathcal{D}_{X, n_1}$, $\mathcal{D}_{Y, n_1}$ into $k_{\mathrm{max}}$ folds, where, for $k \in [k_{\mathrm{max}}]$, we define
                    \begin{align*}
                        &\mathcal{D}_{X, n_1}^{(k)} = \{X_{\lfloor n_1/k_{\max} \rfloor(k-1)+1}, \hdots, X_{\lfloor n_1/k_{\max} \rfloor k}\},
                        \mbox{ and } \\
                        &\mathcal{D}_{Y, n_1}^{(k)} = \{Y_{\lfloor n_2/k_{\max} \rfloor(k-1)+1}, \hdots, Y_{\lfloor n_2/k_{\max} \rfloor k}\}.
                    \end{align*}
                \item
                    For each $k \in [k_{\mathrm{max}}]$, carry out the test of \Cref{sec5:UpperBoundNonInt} or \Cref{sec5:UpperBoundInt} in the non-interactive and interactive settings respectively, with samples $\mathcal{D}_{X, n_1}^{(k)}, \mathcal{D}_{Y, n_2}^{(k)}$ and significance level $\alpha/k_{\mathrm{max}}$, fixing $R = 2^k$ therein. Denote $\phi_k$ the outcome of this test with a value $1$ for rejection of the null hypothesis and $0$ otherwise.
                \item 
                    Define the adaptive test $\phi_{\mathrm{adapt}} = \max_{k \in [k_{\mathrm{max}}]} \phi_k$, where a value of $1$ corresponds to rejection of the null hypothesis.
            \end{enumerate}   

            Inspecting the upper bound on the separation radii given by the above theorem, we see one can adapt to an unknown smoothness parameter $s$ at the cost of an extra poly-logarithmic penalty in $n_1\varepsilon^2$ in both the non-interactive and interactive settings.
            
            Focusing on the additional logarithmic terms, these arise as both a logarithmic prefactor, and as an iterated logarithm term. The logarithmic prefactor arises from splitting the sample $O(\log(n\varepsilon^2))$ times, reducing the sample size in each fold by this amount, leading directly to this term. The double-logarithm term arises due to the Bonferroni correction, where $\alpha$ is divided by $O(\log(n\varepsilon^2))$ to correct for this many tests, with $\alpha$ itself appearing inside a logarithm, resulting in the iterated logarithm. Unlike with adaptive tests without privacy as in \cite{Ingster:2000:AdaptiveTests}, where an iterated logarithmic penalty $O([\log\{\log(n_1)\}]^{1/2})$ is incurred for adaptivity, we see that in the LDP setting we instead incur a logarithmic penalty. This occurs as we cannot immediately repeat a test with the same users without risking privacy leakage unlike in the non-private setting. Attempting to avoid sample splitting by instead increase the privacy guarantee to $\varepsilon/k_{\mathrm{max}}$, incurs a penalty of $k_{\mathrm{max}}^2 = \{O(\log(n_1\varepsilon^2))\}^2$ instead. As splitting the sample incurs a smaller penalty, we adopt this strategy. The lack of a square root in our results compared to the square root on the iterated logarithm in \cite{Ingster:2000:AdaptiveTests} arises due to the weaker concentration properties of our test statistics, in particular the sub-exponential concentration as opposed to sub-Gaussian.

    \section{Simulations} \label{sec6}
        \subsection{Discrete Setting} \label{sec6:disc}
            In this subsection, we investigate the performance of the tests constructed in \Cref{sec4}, and empirically validate the minimax rates. We consider the testing problem \eqref{sec4:eq:Testdef}. We also compare against the test of \cite{Canonne:2024:Heterogeneous} given as Algorithm~2 therein, which we implement, to the best of our understanding, as in the original paper wherein the test is calibrated using a critical value, as well as our own implementation where we calibrate the test statistic via a permutation procedure. The first sample will be drawn from $P_X \sim \mathrm{Multinom}(\vecbf{p}_X)$ with $\vecbf{p}_X = (1/d, \hdots, 1/d)^T \in [0,1]^d$. For the second sample, let $\gamma \in [0,1]$. We consider two distributions $P_{Y, \gamma}^{(1)} \sim \mathrm{Multinom}(\vecbf{p}_{Y, \gamma}^{(1)})$ and $P_{Y, \gamma}^{(2)} \sim \mathrm{Multinom}(\vecbf{p}_{Y, \gamma}^{(2)})$, where
            \begin{equation} \label{app:eq:discsimdistributions}
                \begin{aligned}
                    \vecbf{p}_{Y, \gamma}^{(1)} &= (1/d + \gamma/d, 1/d - \gamma/d, 1/d + \gamma/d, 1/d - \gamma/d, \hdots, 1/d + \gamma/d, 1/d - \gamma/d)^T, \\
                    \vecbf{p}_{Y, \gamma}^{(2)} &= (1/d + (d-1)\gamma/d, (1-\gamma)/d, \hdots, (1-\gamma)/d)^T.
                \end{aligned}
            \end{equation}
            These are constructed so that for $\gamma \in [0,1]$, we have $\|P_X - P_{Y, \gamma}^{(1)}\|_1 = \gamma$, and $\|P_X - P_{Y, \gamma}^{(2)}\|_2 = \{(d-1)/d\}^{1/2}\gamma$, and at $\gamma = 0$, the null hypotheses $P_X = P_{Y, 0}^{(1)}$ and $P_X = P_{Y, 0}^{(2)}$ hold.

            In all simulations which follow, we set the desired type-I error guarantee $\alpha = 0.05$, carry out permutation tests with $B = 199$, and take the average over $2000$ repetitions. For brevity, we refer to testing between $P_X$ and $P_{Y,\gamma}^{(1)}$ as the $L_1$-problem, and $P_X$ and $P_{Y,\gamma}^{(2)}$ as the $L_2$-problem. We run simulations varying model parameters in the following ways:
            \begin{enumerate}
                \item \textbf{Varying privacy}: Fix $n_1 = n_2 = 250$, $d = 8$. Vary $\varepsilon \in \{0.5, 1, 2, 4\}$. We also compare to a minimax optimal $\chi^2$ test in the non-private setting \citep{Chan:2014:DiscreteTesting}, denoted by $\varepsilon = \infty$. The results are contained in the first row of \Cref{sec6:fig:discretepriv}.
            
                \item \textbf{Varying sample size}: Fix $n_1 = 250$, $d = 8$, $\varepsilon = 2$. Vary $n_2 \in \{250, 500, 750, 1000\}$. The results are contained in the second row of \Cref{sec6:fig:discretepriv}.
                
                \item \textbf{Varying dimension}: Fix $n_1 = n_2 = 250$, $\varepsilon = 2$. Vary $d \in \{8, 16, 24, 32\}$. The results are contained in the third row of \Cref{sec6:fig:discretepriv}.
            \end{enumerate}

            \begin{figure}[ht]
                \captionsetup{aboveskip=0pt}
                \centering
                \begin{subfigure}[t]{0.2\textwidth}
                    \centering
                    \includegraphics[height=0.11\textheight]{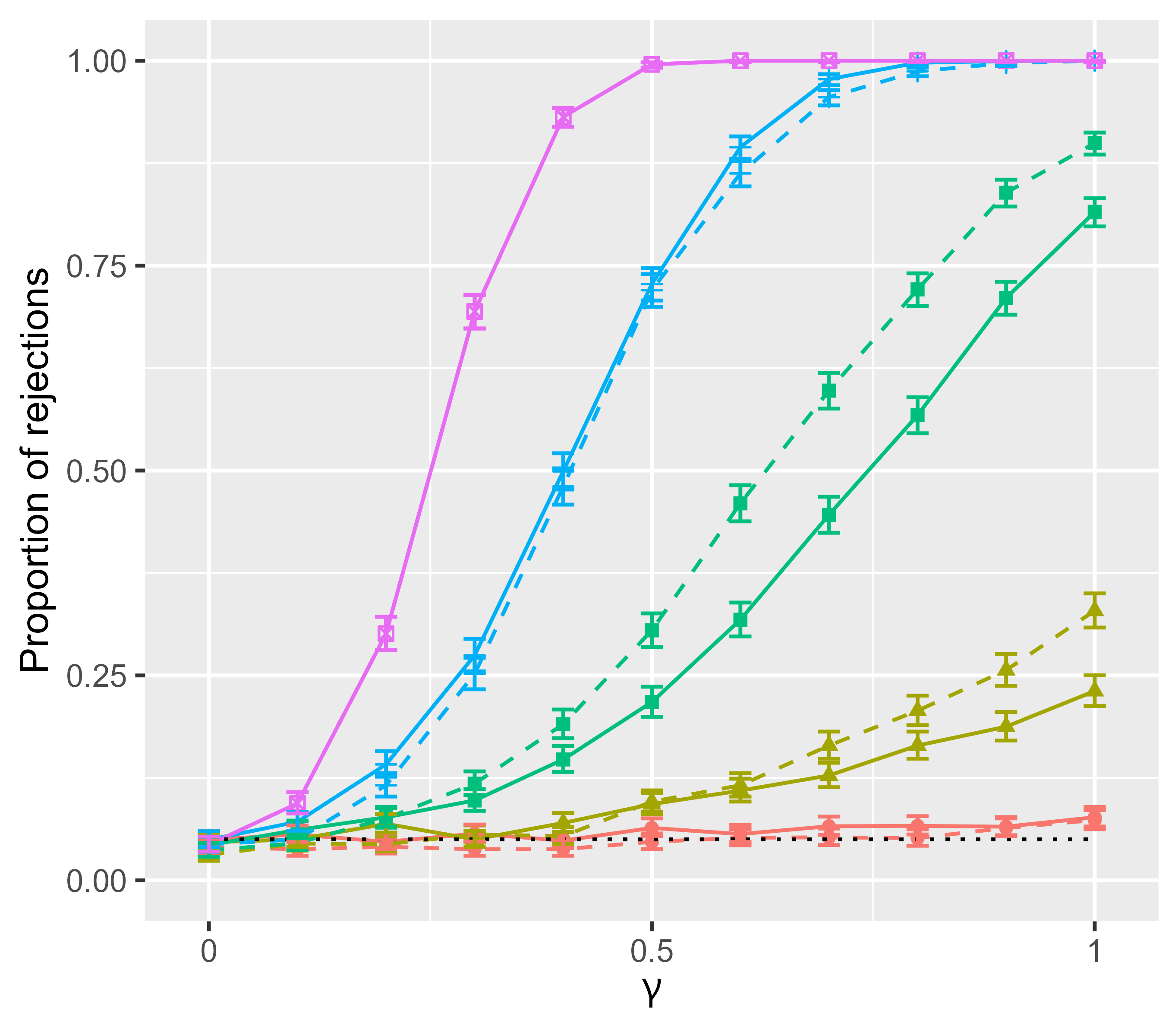}
                \end{subfigure}%
                \begin{subfigure}[t]{0.3\textwidth}
                    \centering
                    \includegraphics[height=0.11\textheight]{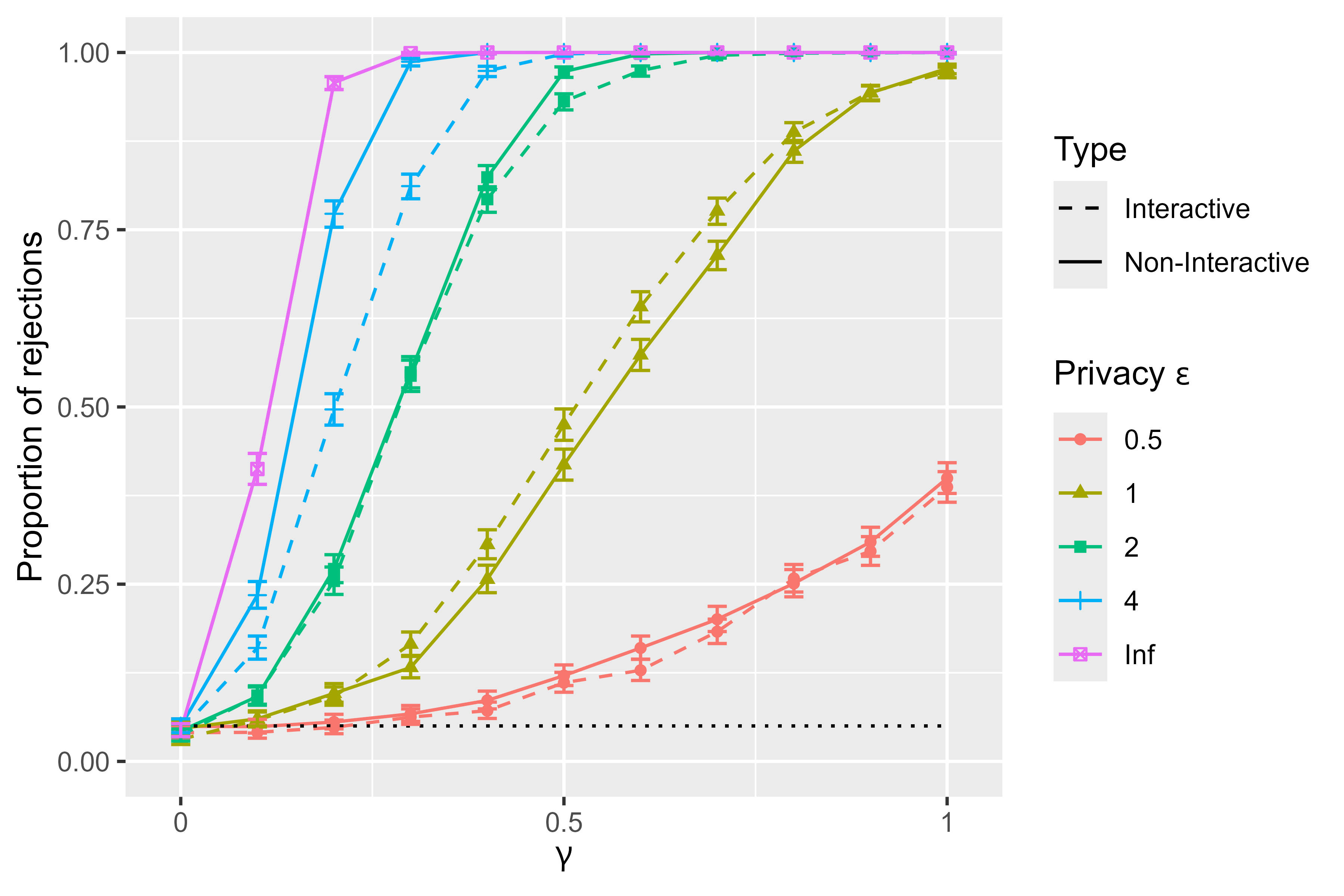}
                \end{subfigure}%
                \begin{subfigure}[t]{0.2\textwidth}
                    \centering
                    \includegraphics[height=0.11\textheight]{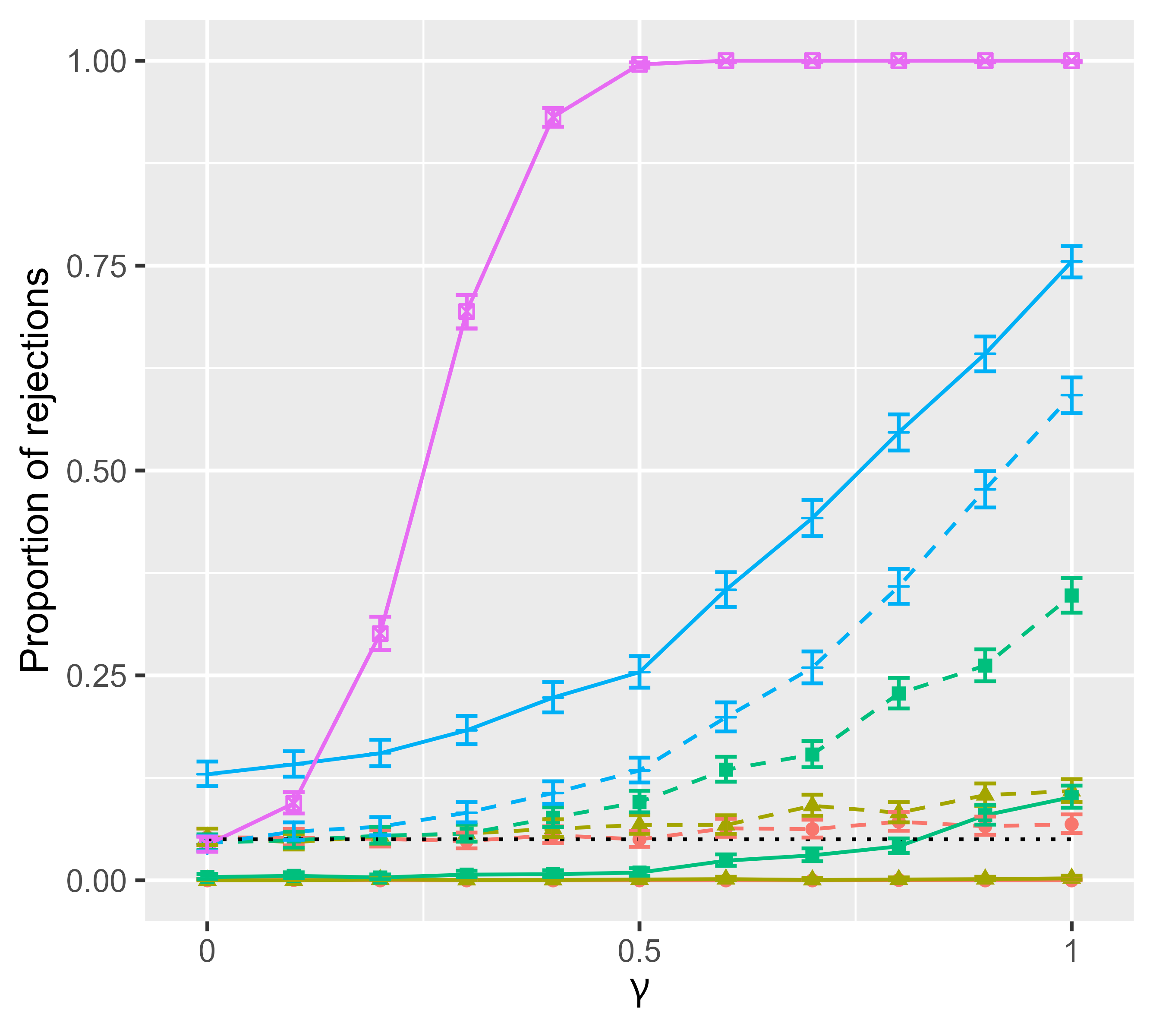}
                \end{subfigure}%
                \begin{subfigure}[t]{0.3\textwidth}
                    \centering
                    \includegraphics[height=0.11\textheight]{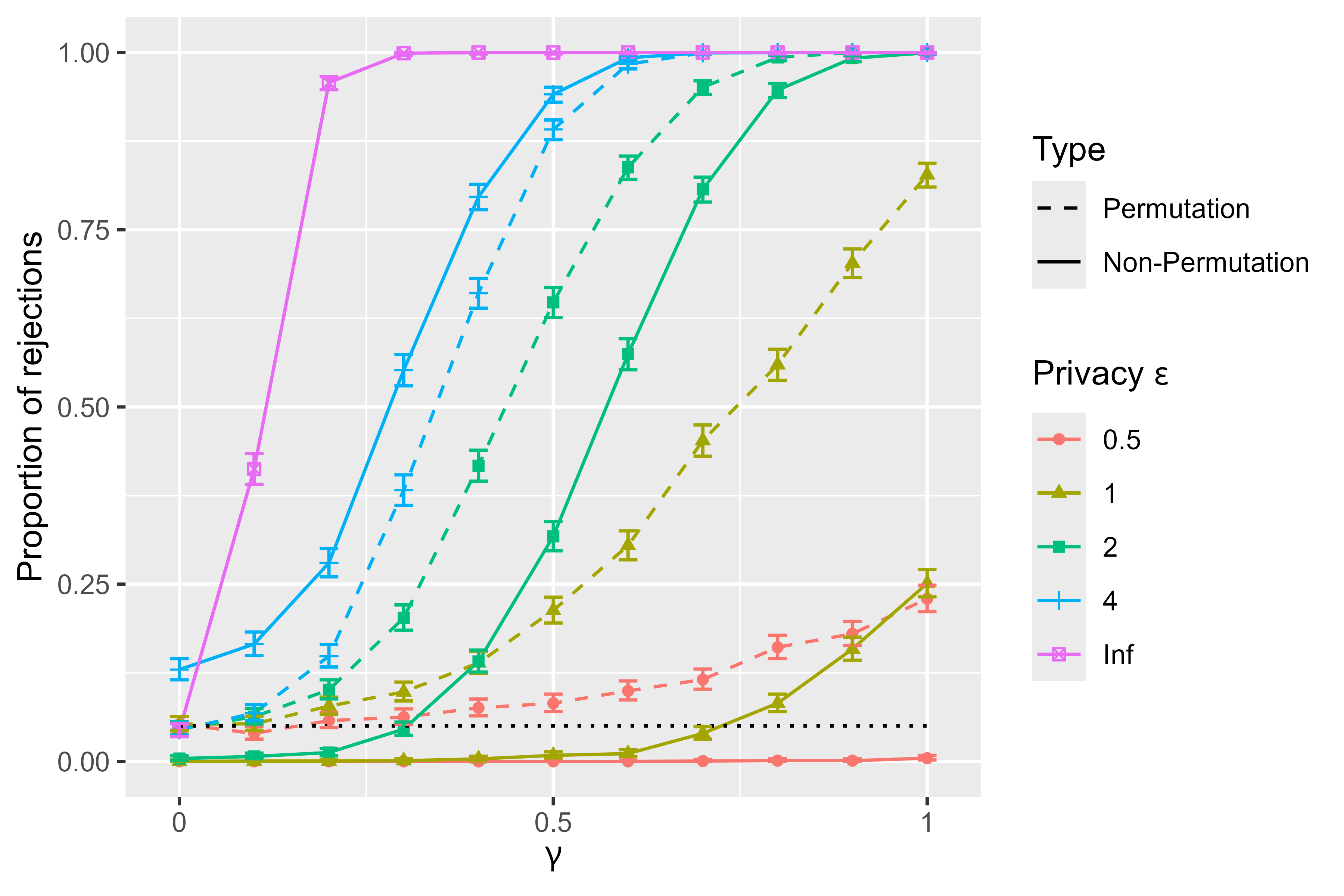}
                \end{subfigure}

                \vspace{1em}

                \begin{subfigure}[t]{0.2\textwidth}
                    \centering
                    \includegraphics[height=0.11\textheight]{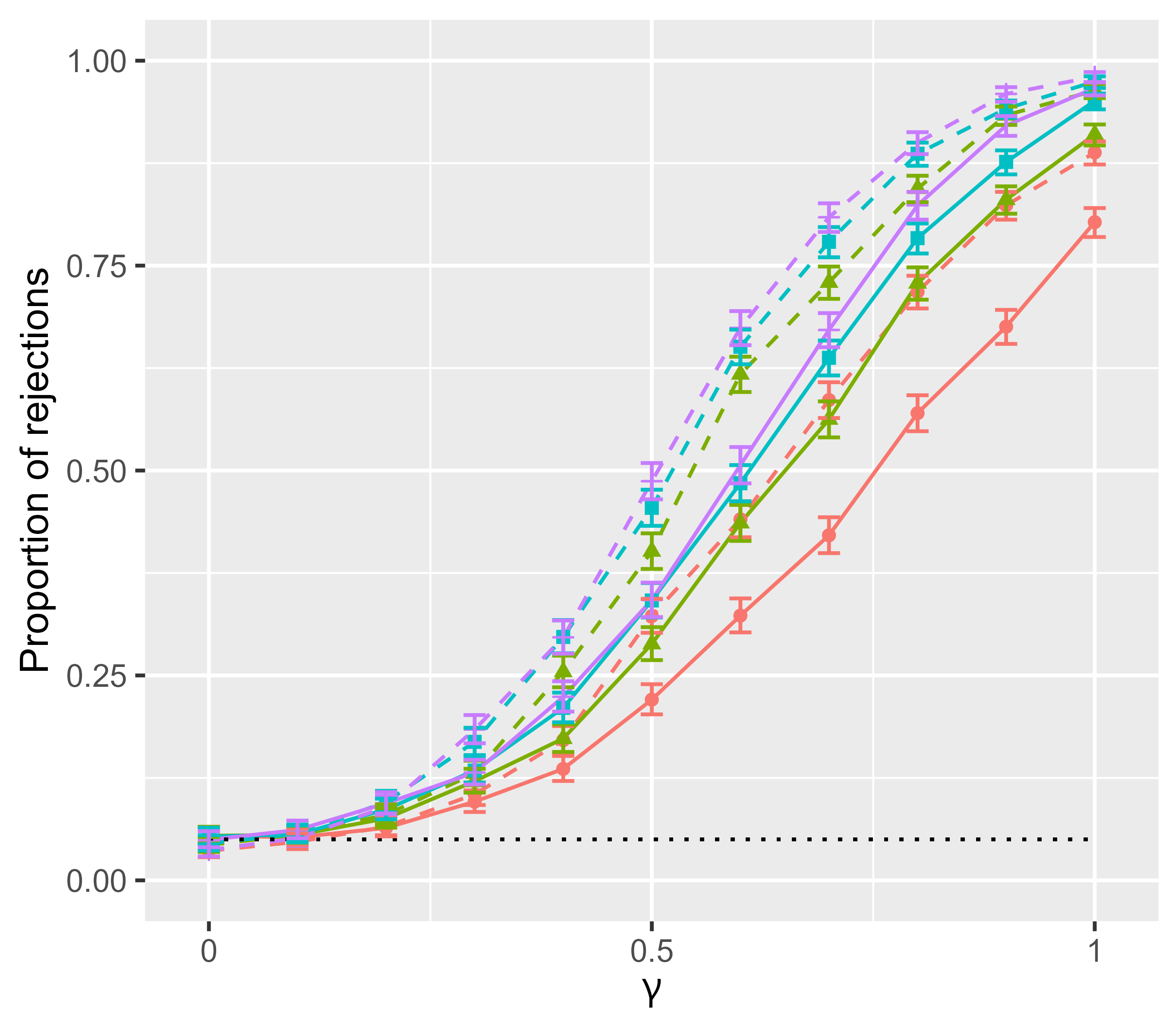}
                \end{subfigure}%
                \begin{subfigure}[t]{0.3\textwidth}
                    \centering
                    \includegraphics[height=0.11\textheight]{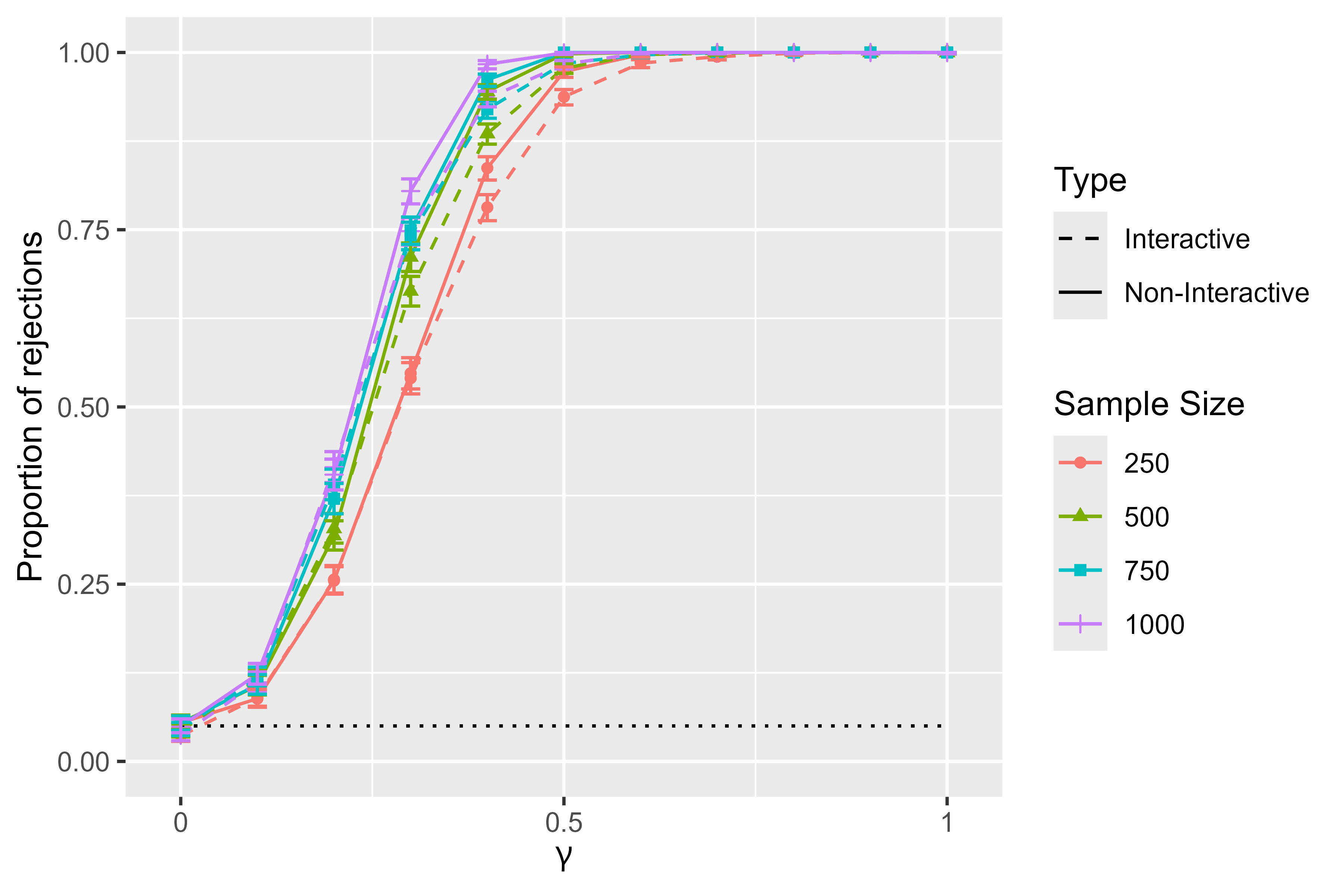}
                \end{subfigure}%
                \begin{subfigure}[t]{0.2\textwidth}
                    \centering
                    \includegraphics[height=0.11\textheight]{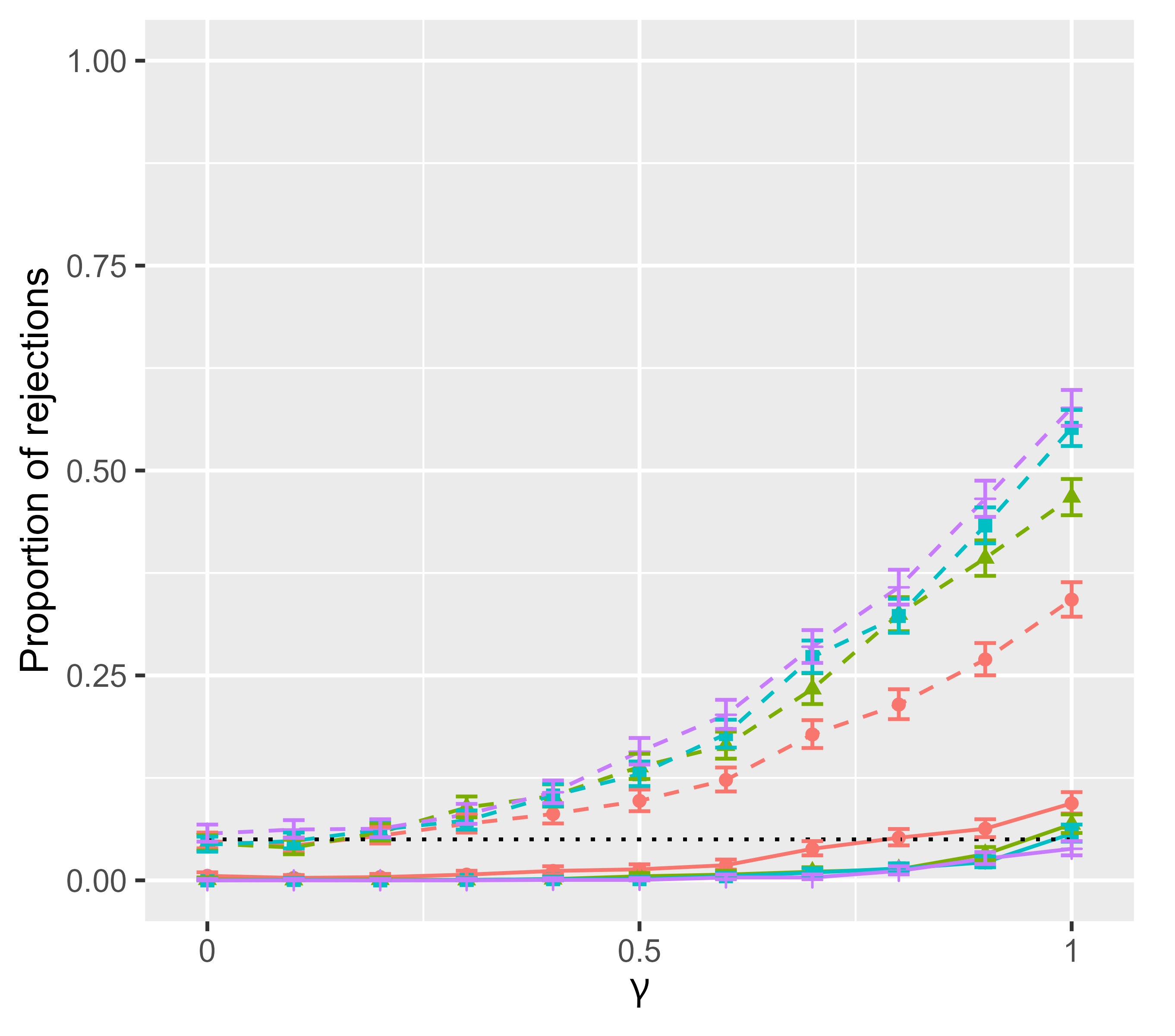}
                \end{subfigure}%
                \begin{subfigure}[t]{0.3\textwidth}
                    \centering
                    \includegraphics[height=0.11\textheight]{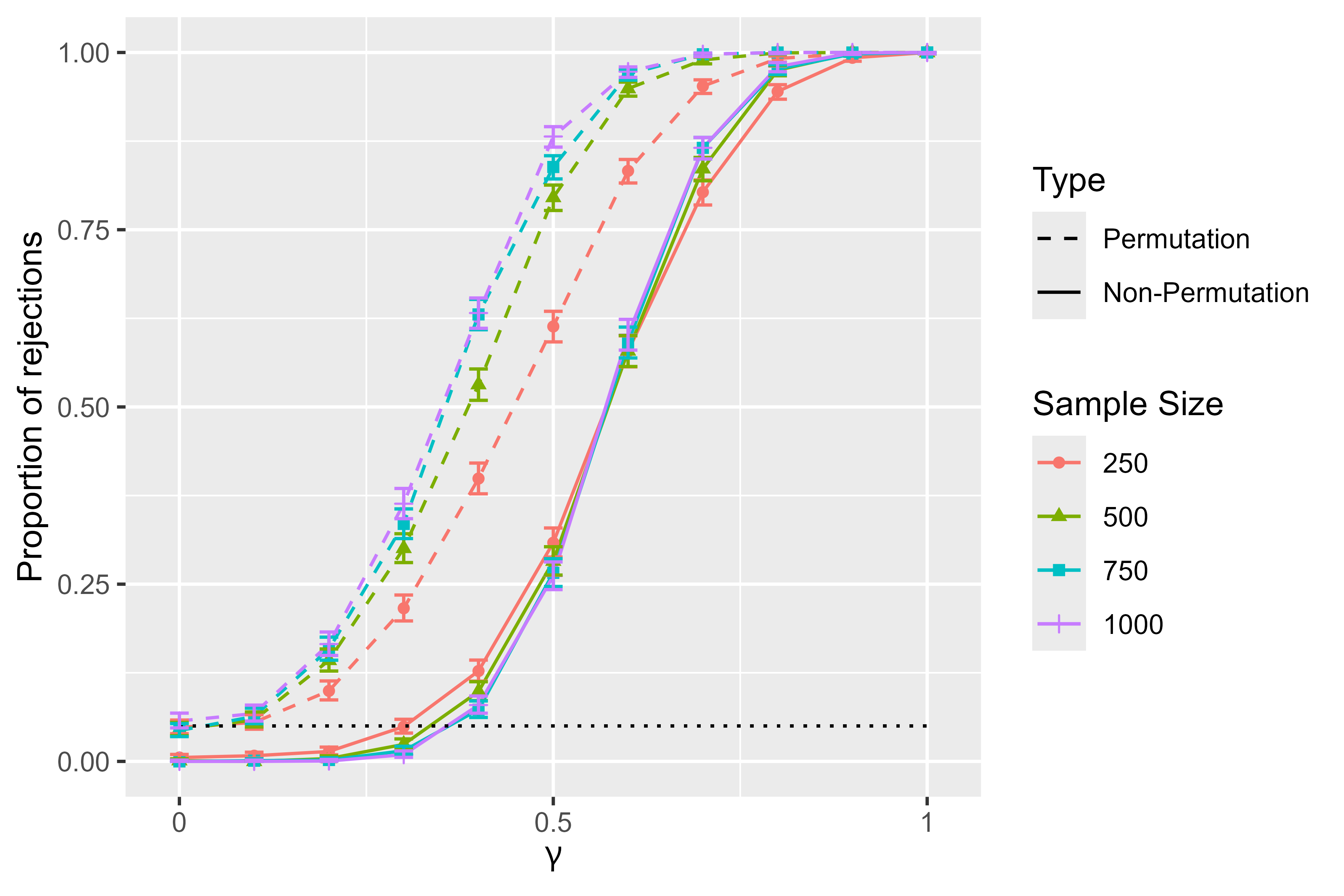}
                \end{subfigure}  

                \vspace{1em}

                \begin{subfigure}[t]{0.2\textwidth}
                    \centering
                    \includegraphics[height=0.11\textheight]{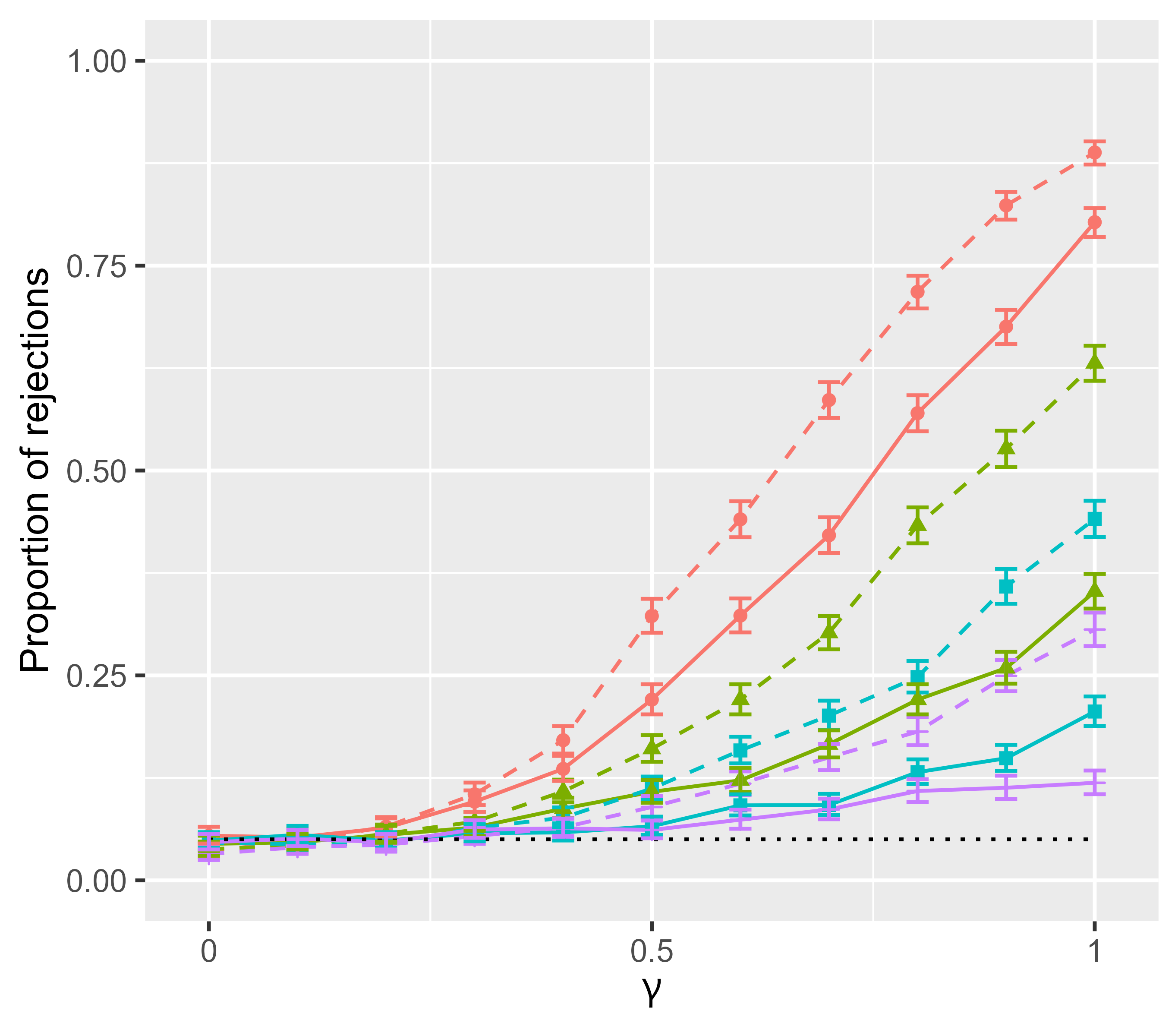}
                \end{subfigure}%
                \begin{subfigure}[t]{0.3\textwidth}
                    \centering
                    \includegraphics[height=0.11\textheight]{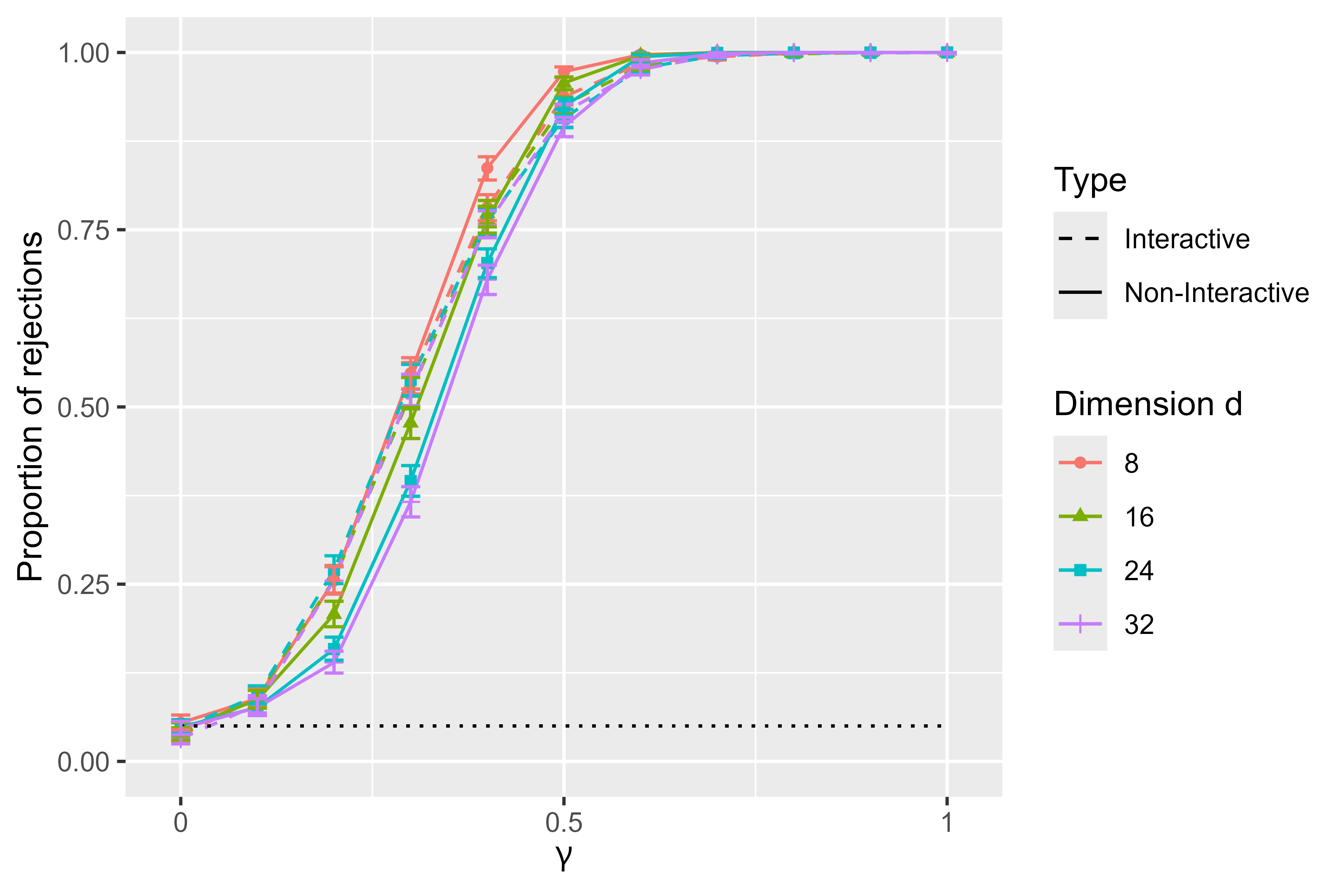}
                \end{subfigure}%
                \begin{subfigure}[t]{0.2\textwidth}
                    \centering
                    \includegraphics[height=0.11\textheight]{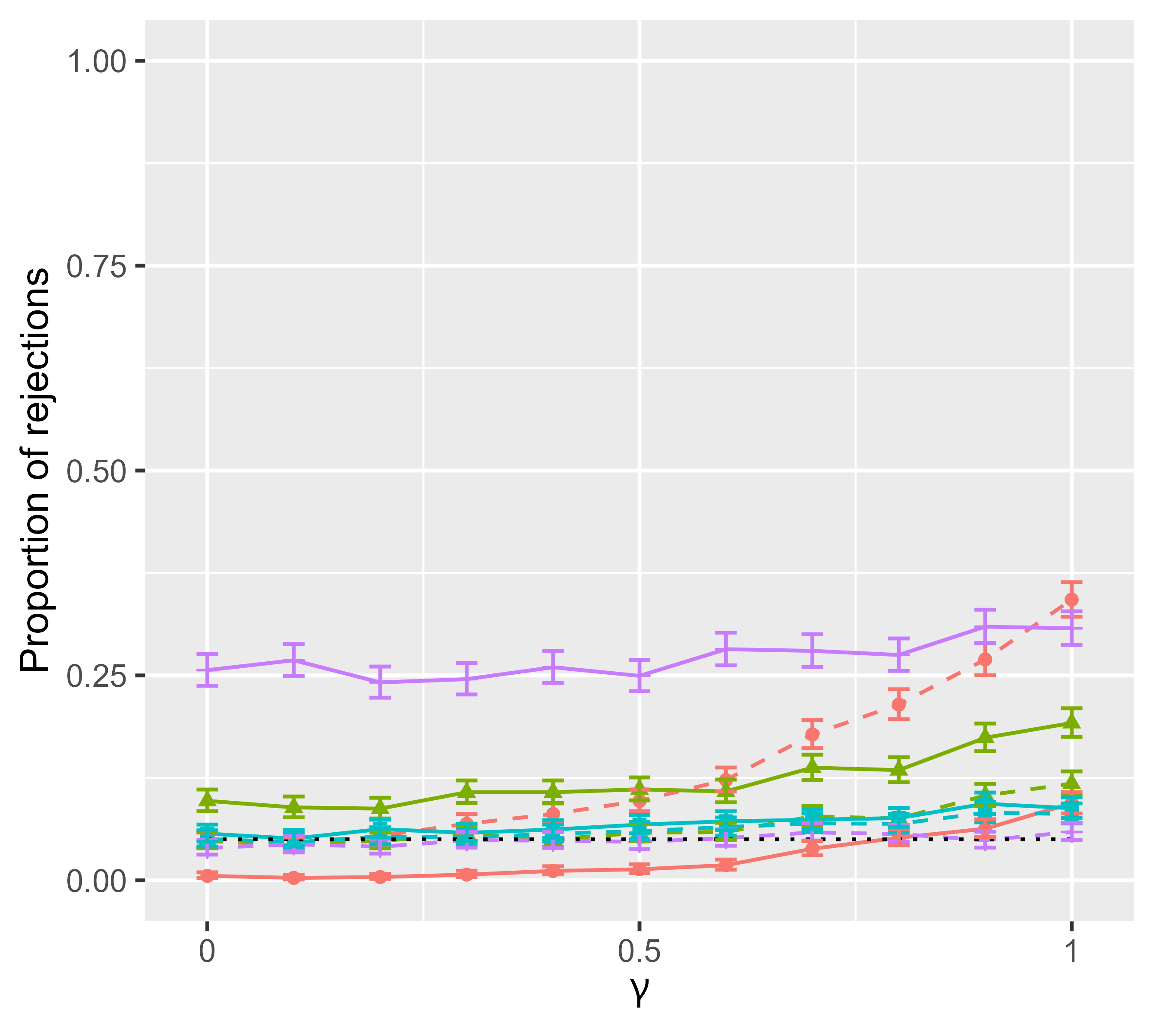}
                \end{subfigure}%
                \begin{subfigure}[t]{0.3\textwidth}
                    \centering
                    \includegraphics[height=0.11\textheight]{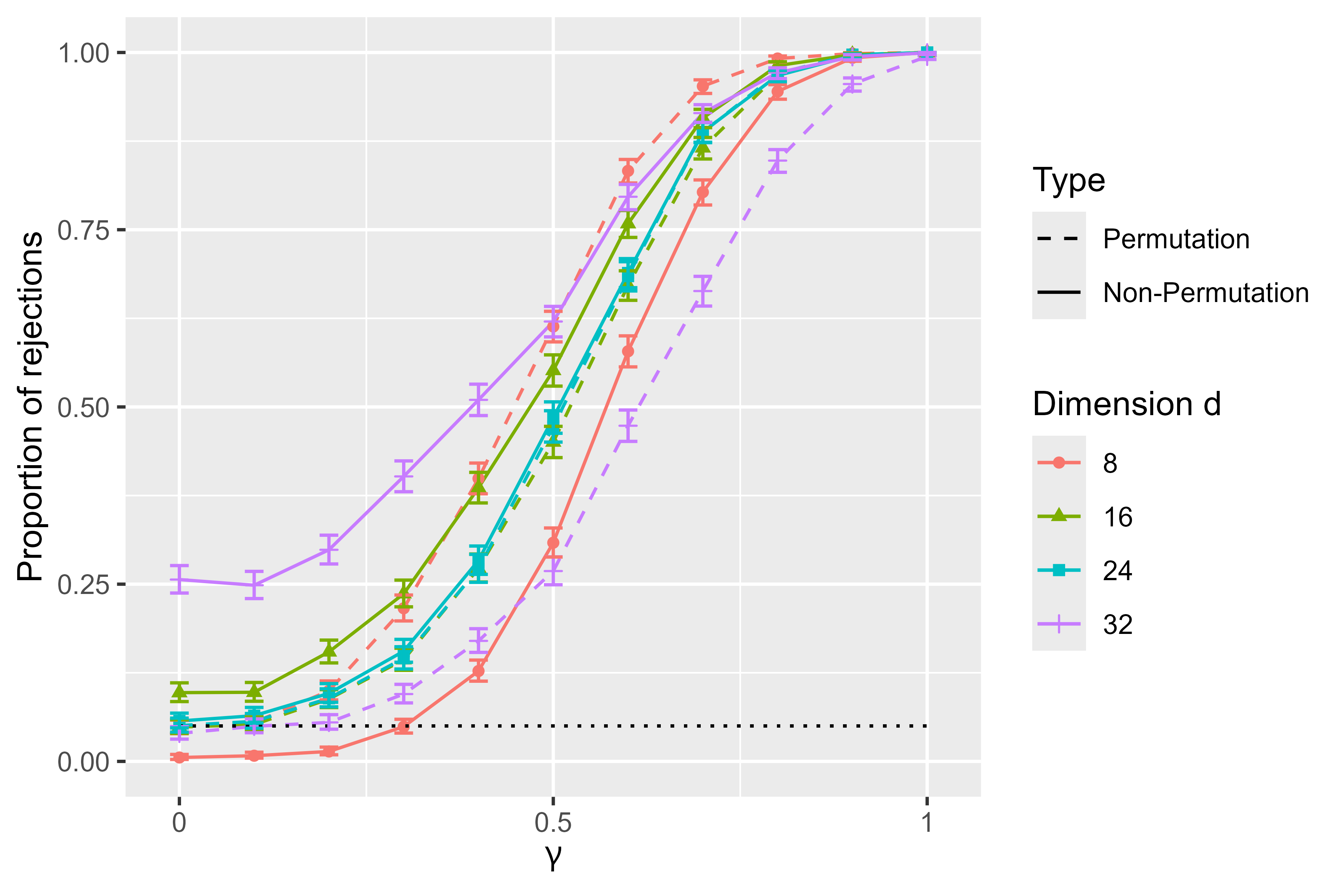}
                \end{subfigure}  
                
                \caption{Power curves for the $L_1$-problem (first and third columns), and $L_2$-problem (second and fourth columns). The first two columns show results for our test procedures, the last two columns show results for test of \cite{Canonne:2024:Heterogeneous}. First, second and third rows correspond to settings of varying privacy, sample size and dimension respectively  Bars indicate pointwise 95\% confidence intervals. Privacy level $\varepsilon = \infty$ corresponds to a centred $\chi^2$-test applied to the un\beame{privatised}{privatized}{privatized} data.}
                \label{sec6:fig:discretepriv}
            \end{figure}

            \noindent
            \textbf{Discussion}: 
            We start by comparing our tests against that of \cite{Canonne:2024:Heterogeneous}. A key focus of \cite{Canonne:2024:Heterogeneous} is developing a test that is able to handle heterogeneous privacy requirements across the two samples, which our test does not account for. Further, for calibrating the test, \cite{Canonne:2024:Heterogeneous} choose a critical value so that the type-I and type-II errors are bounded above by at most $1/3$. When implemented to the best of our understanding, the test exhibits more irregular type-I error compared to our permutation tests which are constructed with the aim of controlling the type-I error exactly. However, we emphasise that this merely due to the choice of calibration for the test, rather than a fundamental weakness of the test statistic or procedure in \cite{Canonne:2024:Heterogeneous}.
            
            To help facilitate comparison, we additionally carry out the test of \cite{Canonne:2024:Heterogeneous} wherein we calibrate via a permutation procedure. Here, the test performs more similarly to ours, controlling the type-I error exactly at the target level. However, our tests, designed specifically with permutation testing in mind, outperform in all settings considered.

            The second row of \Cref{sec6:fig:discretepriv} shows the performance as the size of the second sample is grown. We observe that even as the larger sample size becomes many multiples of the smaller, there are still improvements in power. This demonstrates potential benefits in practice of using the entire samples in the case of unbalanced sample sizes.

            The third row of \Cref{sec6:fig:discretepriv} shows the performance as the dimension of the problem is increased. We see in the $L_1$-problem that, for example, the interactive procedure with $d = 32$ performs comparably to the non-interactive method with half the dimension $d = 16$, demonstrating significant advantages of the interactive procedure as the dimension grows. We see similar behaviour in the $L_2$-problem. In fact, in the $L_2$-problem with the interactive procedure, we see that the power appears almost independent of the dimension, as suggested by the minimax testing radius in \Cref{sec4:thm:main} being independent of $d$.

            \noindent
            \textbf{Validating minimax rates}: We now explore empirically the dependence of the separation radii on the model parameters. In what follows, we use a binary search procedure to find the value $\gamma$ which results in a type-II error rate of $0.5$, up to a tolerance of $0.001$. Details of the implementation are given in \Cref{app:sec:binarysearch}. We transform this $\gamma$ such that based on \Cref{sec4:thm:main} we expect a straight line when plotting this transformation of $\gamma$ against relevant model parameters.

            We first observe in \Cref{sec6:fig:nondimseparation} the resulting curves for the $L_1$-problem as we vary privacy and sample size respectively. The two samples are grown together so that both have the same size. In both cases we generally notice a linear trend, supporting the \beame{theorised}{theorized}{theorized} separation rate. For the interactive methods, we notice a flattening off of the curves as $\varepsilon$ increases, which we \beame{hypothesise}{hypothesize}{hypothesize} is due to a transition from a high-privacy to low-privacy regime as $\varepsilon$ is increased. Figures for the $L_2$-problem are omitted for brevity, but the same trend is observed.

            \begin{figure}[ht]
                \captionsetup{aboveskip=0pt}
                \centering               
                \begin{subfigure}[t]{0.235\textwidth}
                    \centering
                    \includegraphics[width=\textwidth]{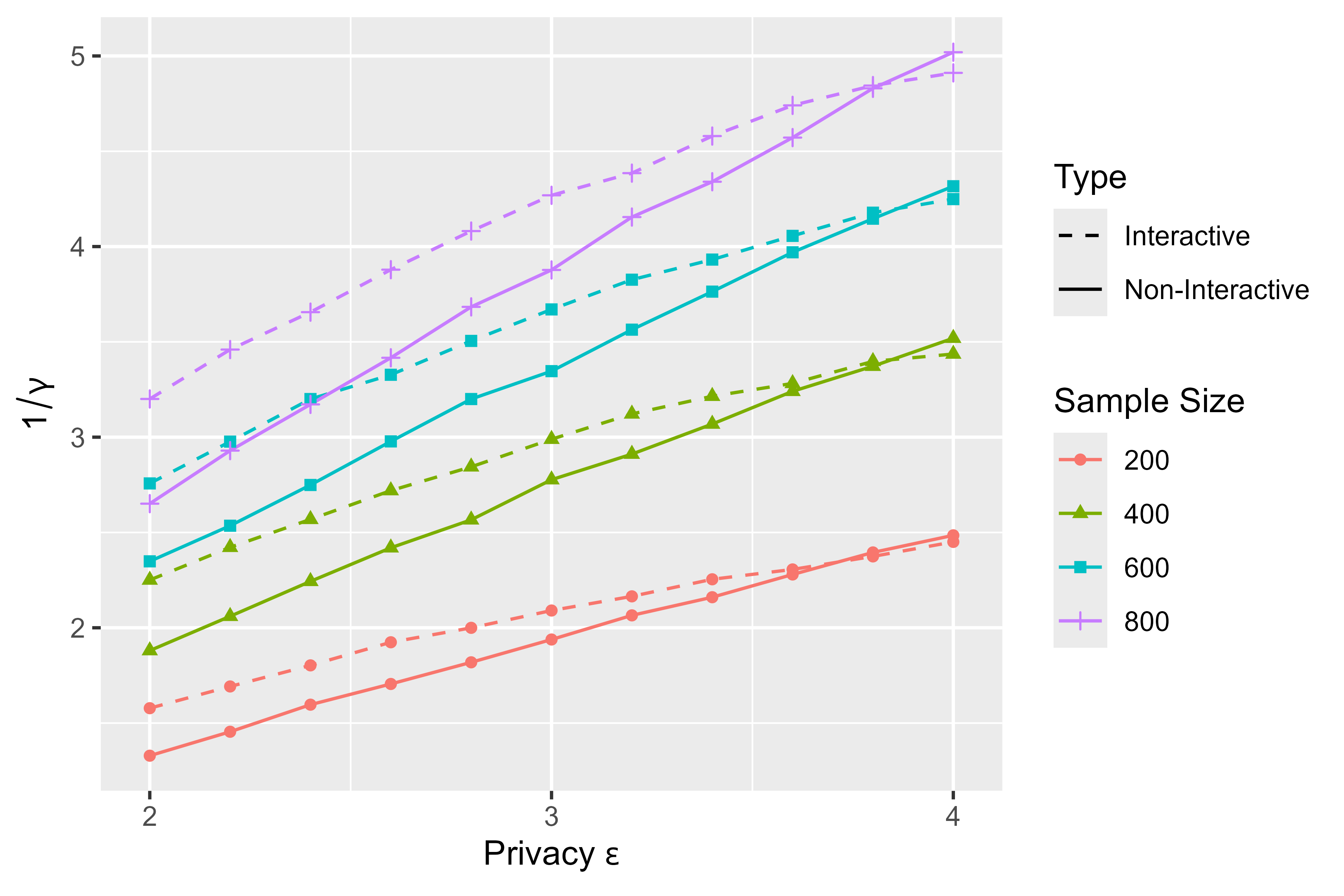}
                    \caption{}
                \end{subfigure}%
                \begin{subfigure}[t]{0.235\textwidth}
                    \centering
                    \includegraphics[width=\textwidth]{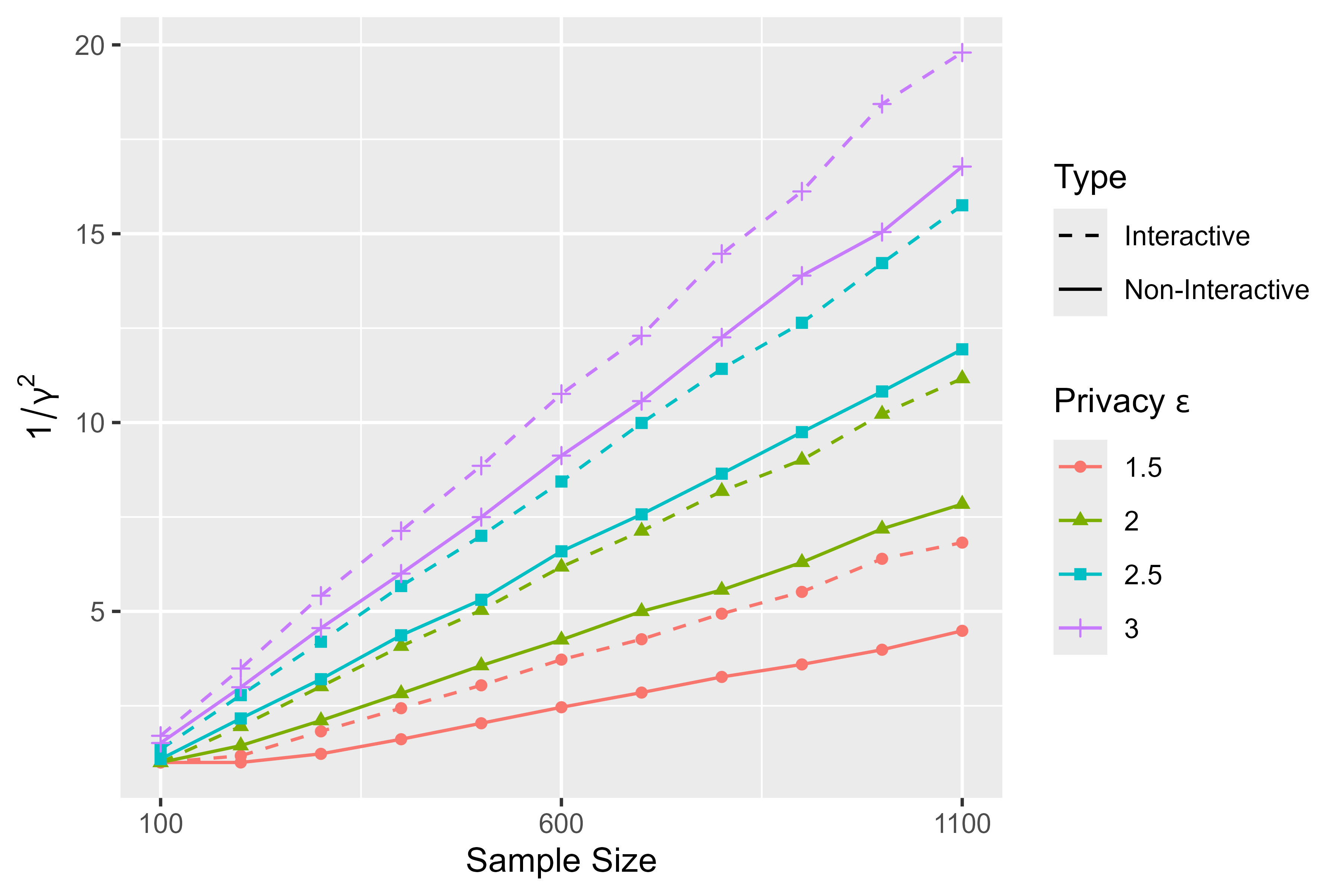}
                    \caption{}
                \end{subfigure}%
                \begin{subfigure}[t]{0.235\textwidth}
                    \centering
                    \includegraphics[width=\textwidth]{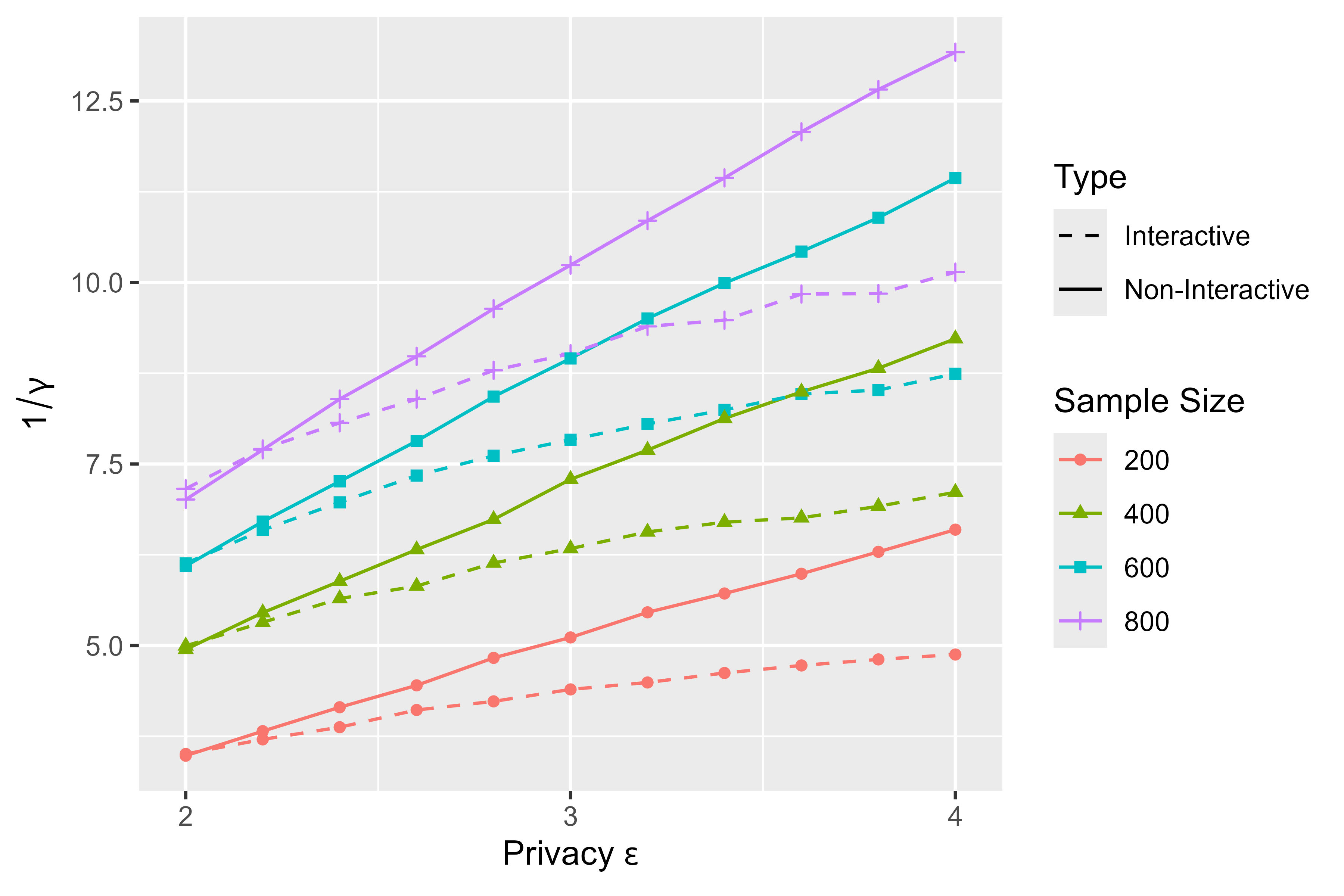}
                    \caption{}
                \end{subfigure}%
                \begin{subfigure}[t]{0.235\textwidth}
                    \centering
                    \includegraphics[width=\textwidth]{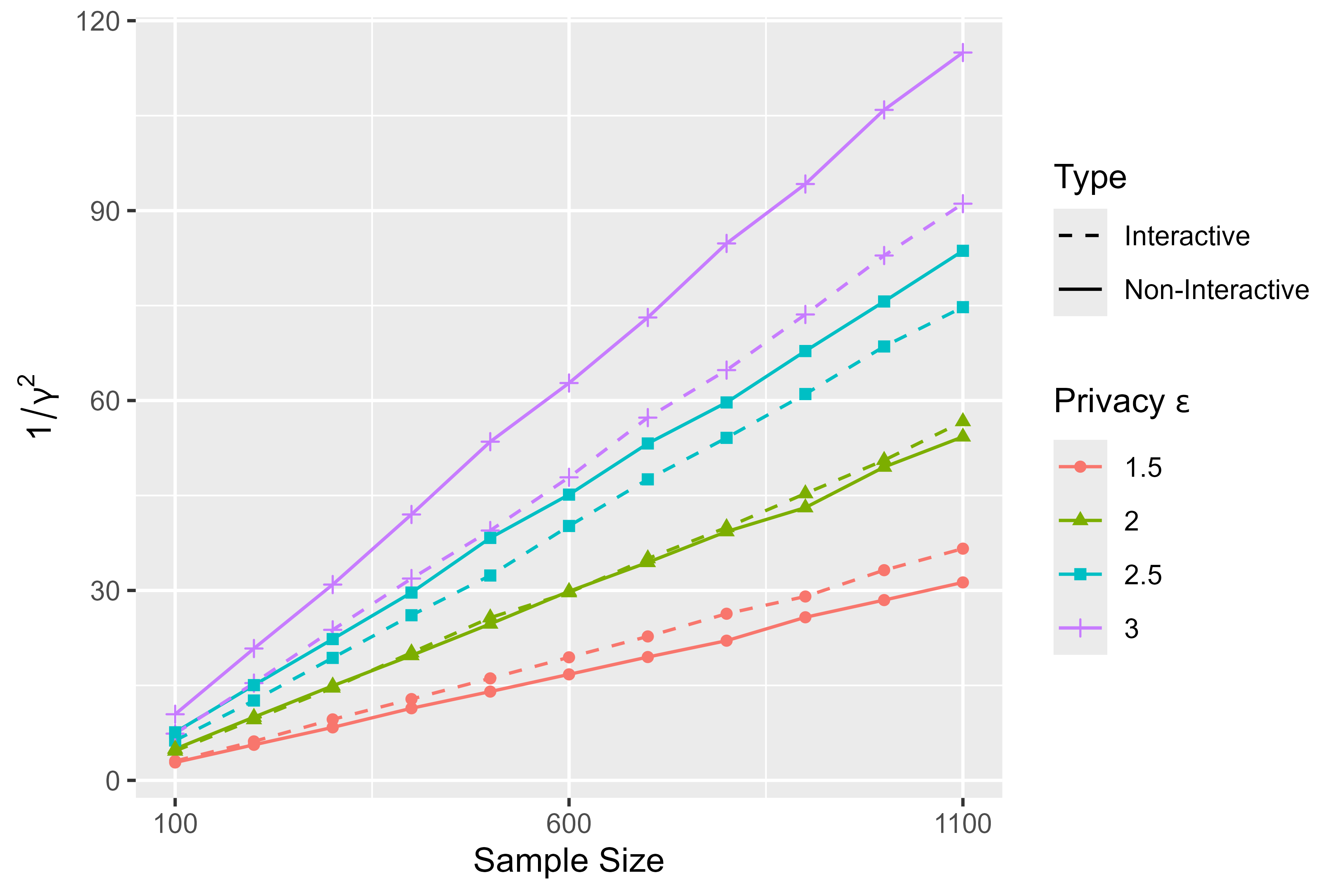}
                    \caption{}
                \end{subfigure}  
                \caption{Plots of transformed separation against dimension for the non-interactive and interactive tests in the $L_1$-problem (a, b) and $L_2$-problem (c, d). $n_1 = n_2 = 1000$, $d = 8$.}
                \label{sec6:fig:nondimseparation}
            \end{figure}
            
            We now consider the more interesting case of the dependence of the separation on the dimensionality, for which the theoretical rate varies greatly between the non-interactive and interactive cases and the $L_1$- and $L_2$-problems. In \Cref{sec6:fig:dimseparation}, after a suitable transformation of $\gamma$, we observe approximately linear trends between the transformed separations and the dimension. In particular, for the interactive procedure in the $L_2$-problem the curves are approximately horizontal lines, matching the theoretical separation which is independent of the dimension.

            \begin{figure}[ht]
                \captionsetup{aboveskip=0pt}
                \centering               
                \begin{subfigure}[t]{0.235\textwidth}
                    \centering
                    \includegraphics[width=\textwidth]{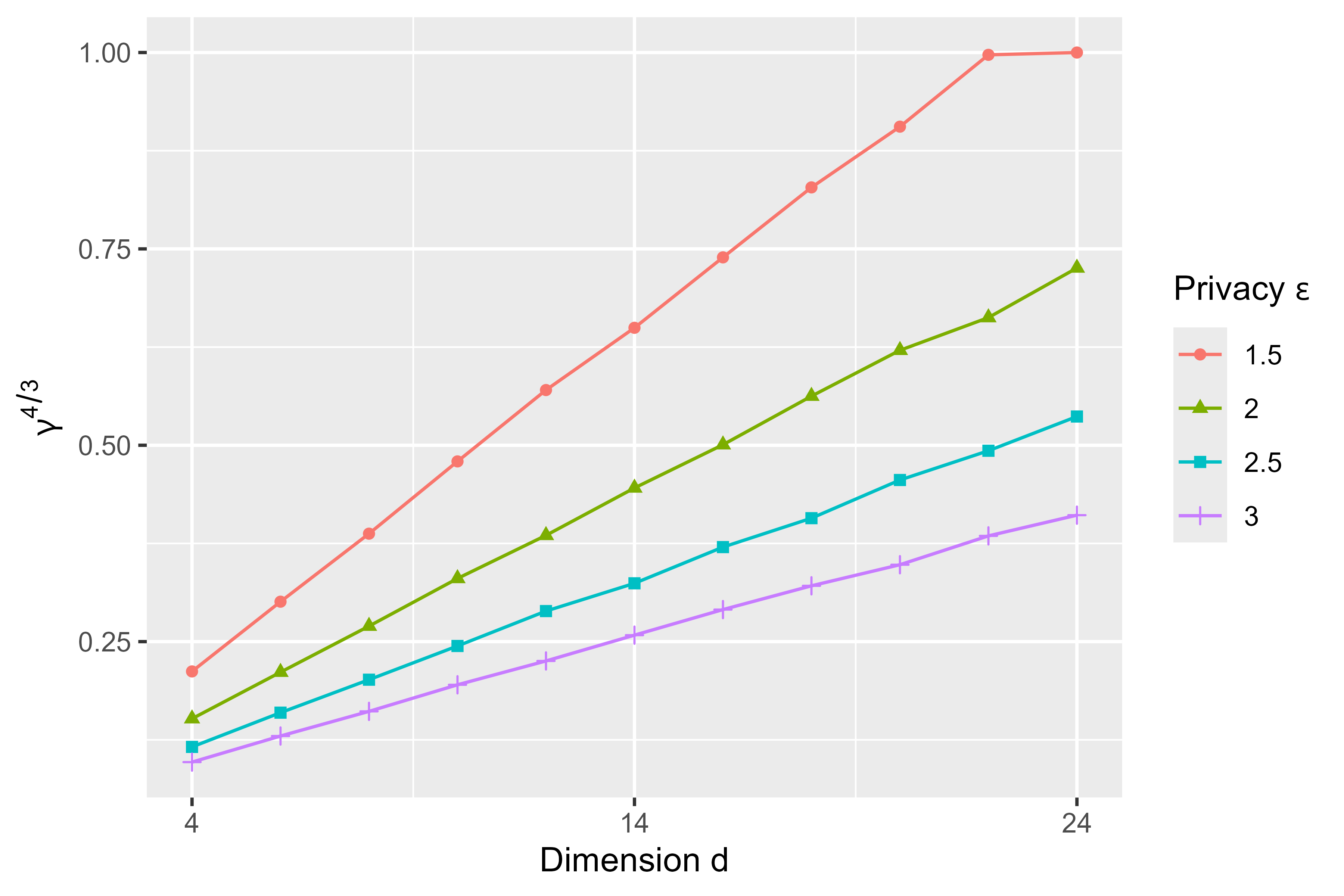}
                    \caption{}
                \end{subfigure}%
                \begin{subfigure}[t]{0.235\textwidth}
                    \centering
                    \includegraphics[width=\textwidth]{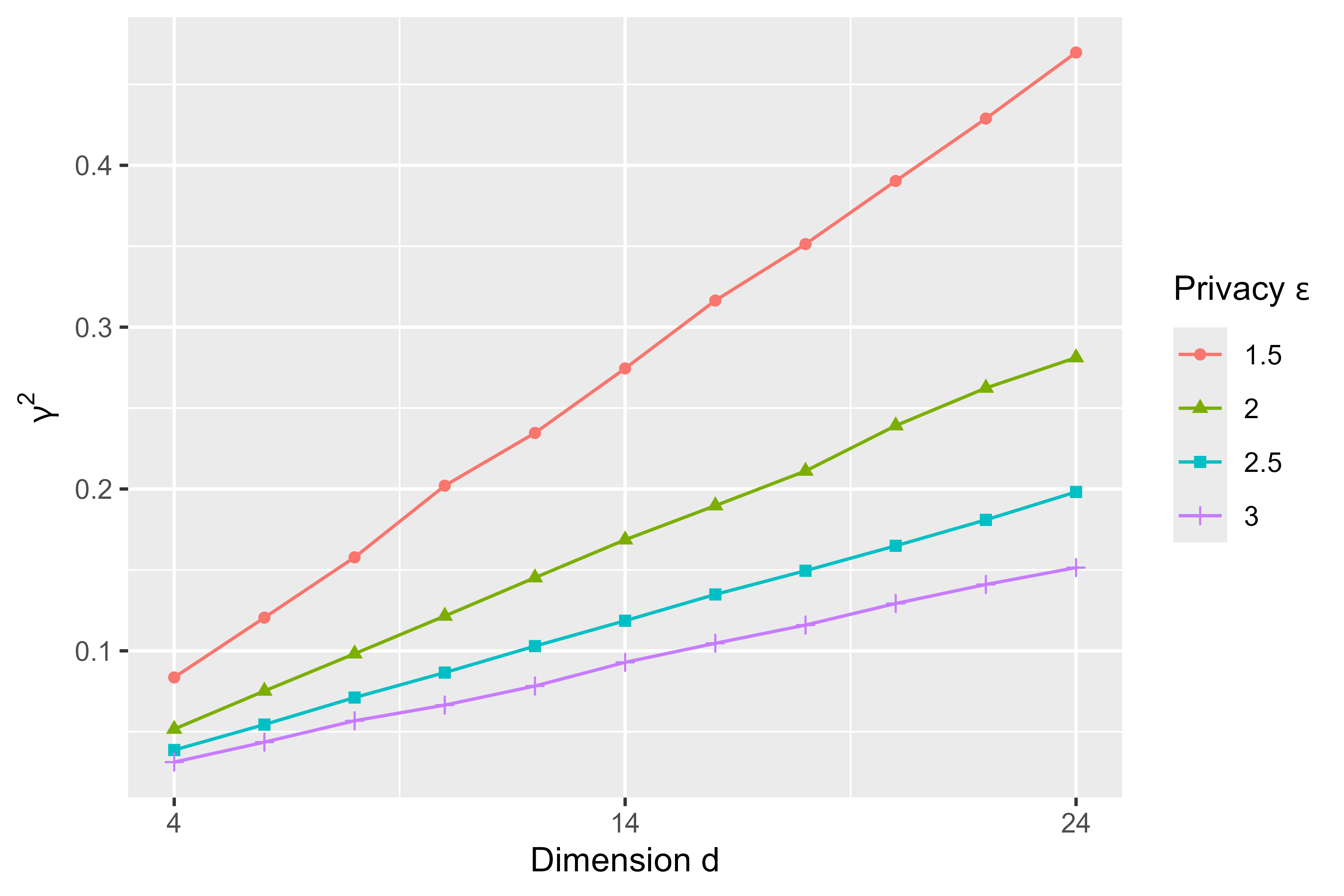}
                    \caption{}
                \end{subfigure}%
                \begin{subfigure}[t]{0.235\textwidth}
                    \centering
                    \includegraphics[width=\textwidth]{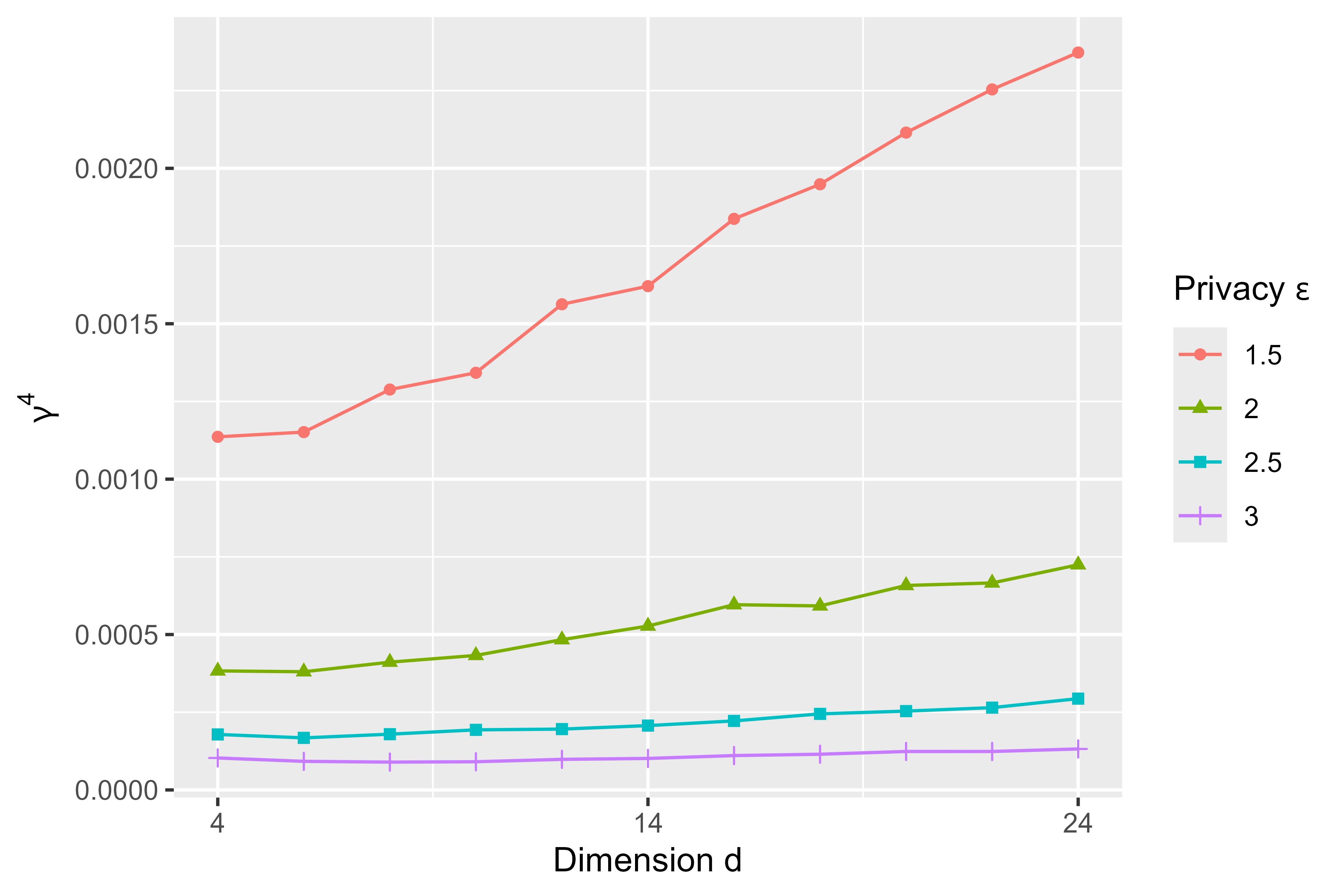}
                    \caption{}
                \end{subfigure}%
                \begin{subfigure}[t]{0.235\textwidth}
                    \centering
                    \includegraphics[width=\textwidth]{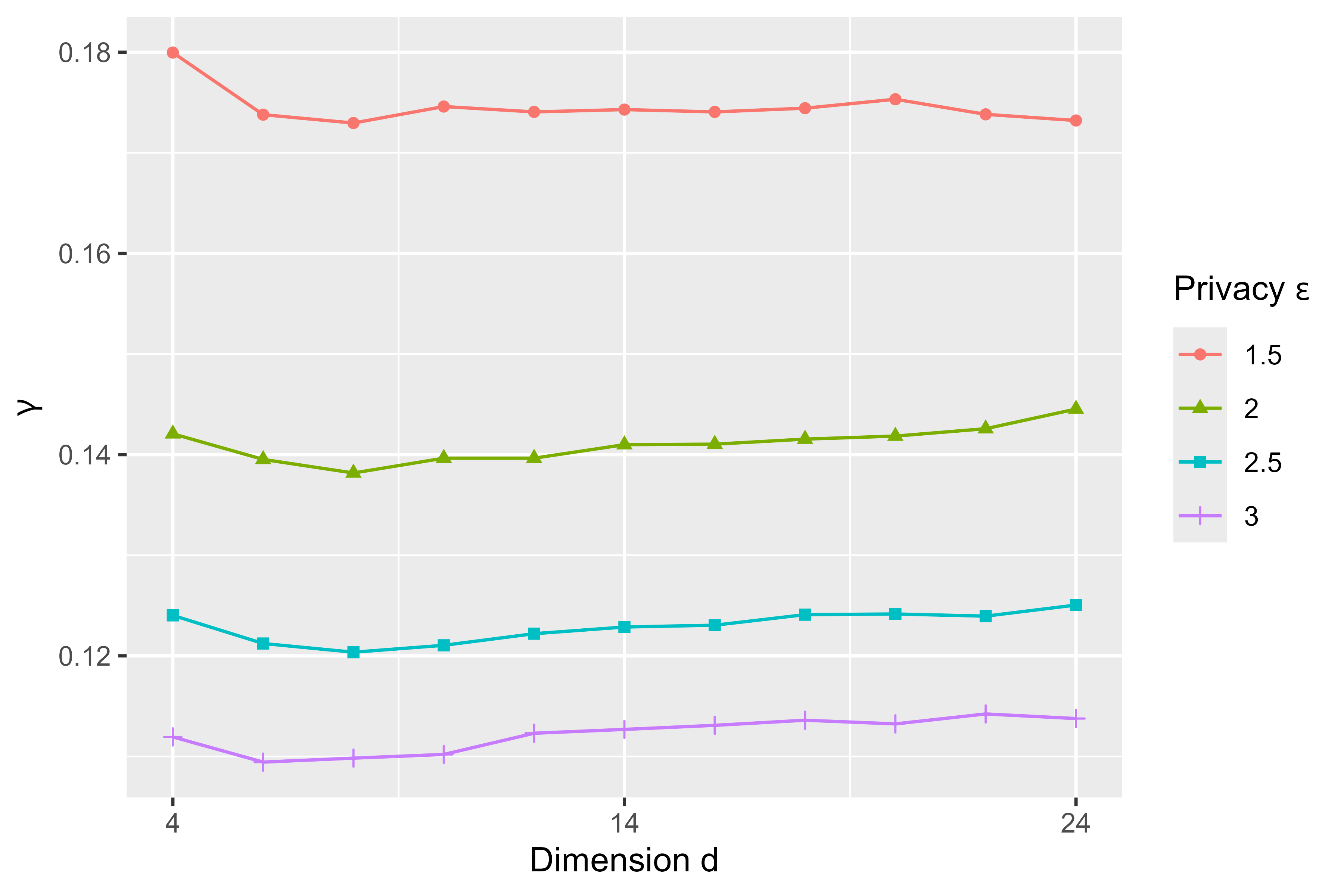}
                    \caption{}
                \end{subfigure}  
                \caption{Plots of transformed separation against dimension for the non-interactive and interactive tests in the $L_1$-problem (a, b) and $L_2$-problem (c, d). $n_1 = n_2 = 1000$.}
                \label{sec6:fig:dimseparation}
            \end{figure}
            
        \subsection{Continuous Setting} \label{sec6:cont}
        
            In this subsection, we investigate the performance of both our non-interactive and interactive testing procedures for continuous data. We also compare against a method which \beame{privatises}{privatizes}{privatizes} the data by adding suitably scaled Laplace noise, and applies an MMD based kernel test implemented via the R package \texttt{euMMD} \citep{Bodenham:2023:eummd}. We first introduce the distributions we consider. The first sample is obtained from i.i.d.~draws from $P_X = \mathrm{Unif}([0,1])$. For the second sample, let $\gamma \in [0, 1]$. We consider three families of distributions $P_{Y, \gamma}^{\mathrm{Beta}}$, $P_{Y, \gamma}^{\mathrm{Tri}}$ and $P_{Y, \gamma, k}^{\mathrm{Cos}}$ induced respectively, for $x \in [0, 1]$ and $k \in \mathbb{N}$, by the following densities supported on $[0,1]$.
            \begin{equation} \label{app:eq:contdistributions}
                \begin{aligned}
                    &f_{\gamma}^{\mathrm{Beta}}(x) = (1 - \gamma) + 5\gamma(1-x)^4, \\
                    &f_{\gamma}^{\mathrm{Tri}}(x)  =  (1 - \gamma) + \gamma(2x\mathbbm{1}\{x \in [0, 1/2)\} + 2(1-x)\mathbbm{1}\{x \in [1/2, 1]\}), \\
                    &f_{\gamma, k}^{\mathrm{Cos}}(x) = 1 + \gamma\cos(2k\pi x).
                \end{aligned}
            \end{equation}
            These densities are constructed to interpolate between a uniform distribution and some alternative such that at $\gamma = 0$ the null hypothesis holds and at $\gamma = 1$ the separation is \beame{maximised}{maximized}{maximized}.
        
            In all simulations which follow, we set the desired type-I error guarantee $\alpha = 0.05$, carry out permutation tests with $B = 199$, and take the average over $2000$ repetitions. For the truncation level of the basis expansion via the formulae \eqref{sec5:eq:trunc2} and \eqref{sec5:eq:trunc2}, we set $\beta = 0.05$, assume $s = 1$ analogous to a Lipschitz condition, and use $\min\{\varepsilon^2, 1\}$ in place of $\varepsilon^2$ therein as we consider the high privacy regime $\varepsilon^2 \lesssim 1$. We run simulations varying model parameters in the following ways.
            \begin{enumerate}
                \item \textbf{Varying privacy}: Fix $n_1 = n_2 = 500$. Vary $\varepsilon \in \{0.5, 1, 2, 4\}$. We also compare with the non-private setting using the raw data and the euMMD test, implemented via the R package \texttt{euMMD} \citep{Bodenham:2023:eummd}, denoted by $\varepsilon = \infty$. The results are contained in the first row of \Cref{sec6:fig:continuouspriv}.
            
                \item \textbf{Varying sample size}: Fix $n_1 = 500$, $\varepsilon = 2$. Vary $n_2 \in \{500, 750, 1000, 1250, 1500\}$. The results are contained in the second row of \Cref{sec6:fig:continuouspriv}.
                
                \item \textbf{Varying truncation}: Fix $n_1 = n_2 = 500$, $\varepsilon = 2$. Rather than choosing the truncation level of the basis expansion via the formulae \eqref{sec5:eq:trunc1} and \eqref{sec5:eq:trunc2}, we fix values of $R$ independently of the other model parameters. The results are contained in the third row of \Cref{sec6:fig:continuouspriv}.
            \end{enumerate}

            \begin{figure}[ht]
                \captionsetup{aboveskip=0pt}
                \centering       
                \begin{subfigure}[t]{0.3\textwidth}
                    \centering
                    \includegraphics[height=0.16\textheight]{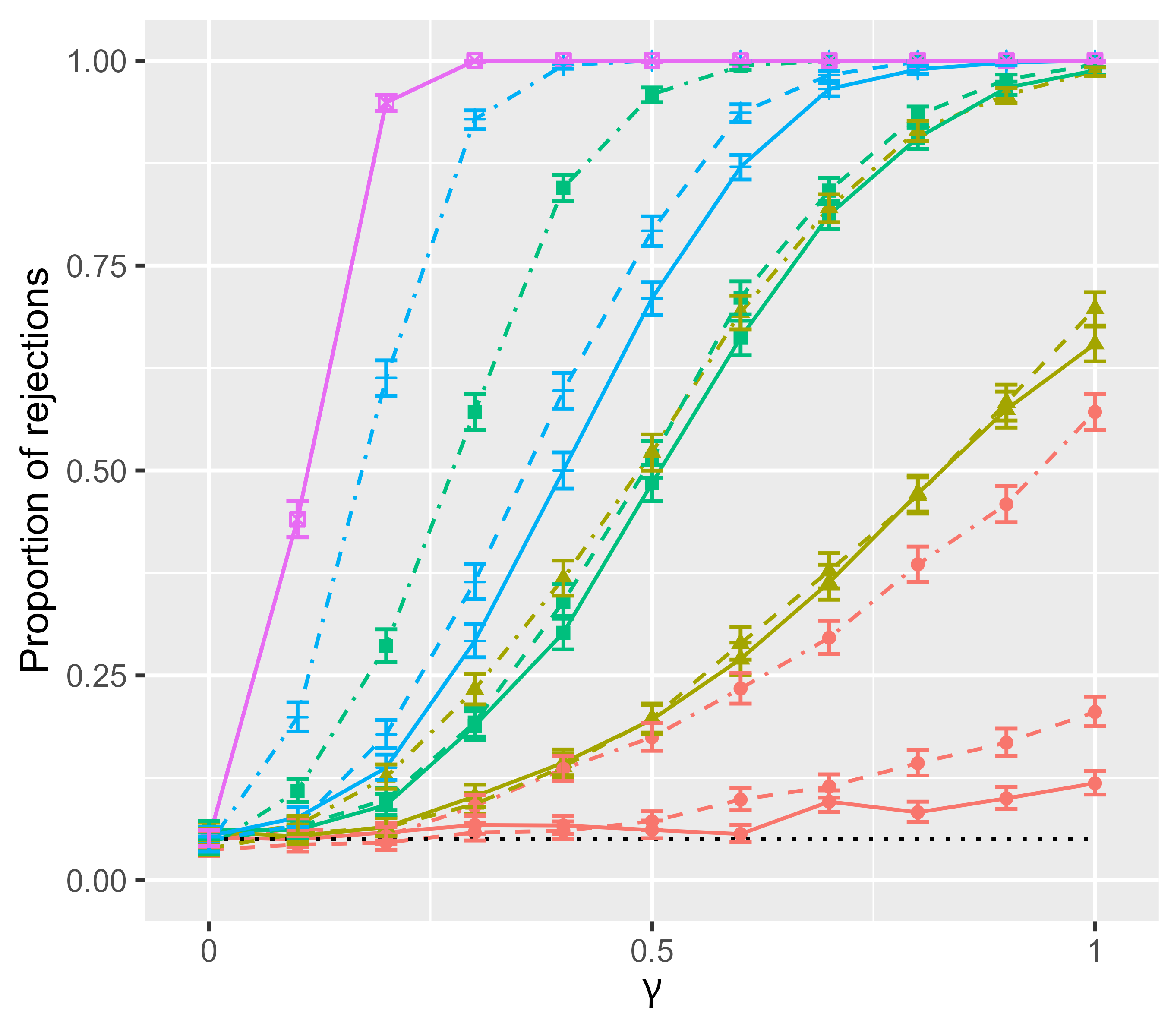}
                \end{subfigure}%
                \begin{subfigure}[t]{0.3\textwidth}
                    \centering
                    \includegraphics[height=0.16\textheight]{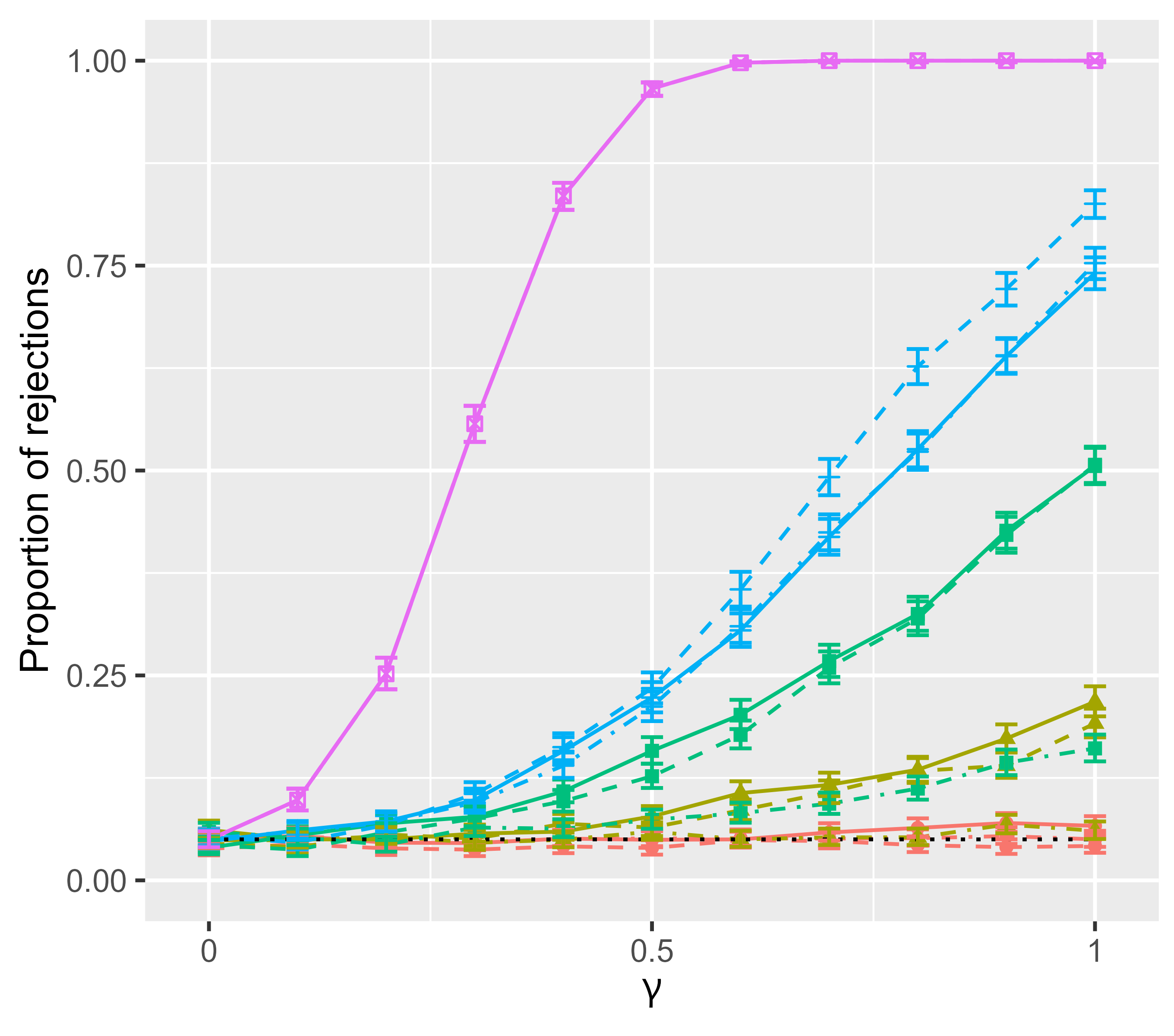}
                \end{subfigure}%
                \begin{subfigure}[t]{0.38\textwidth}
                    \centering
                    \includegraphics[height=0.16\textheight]{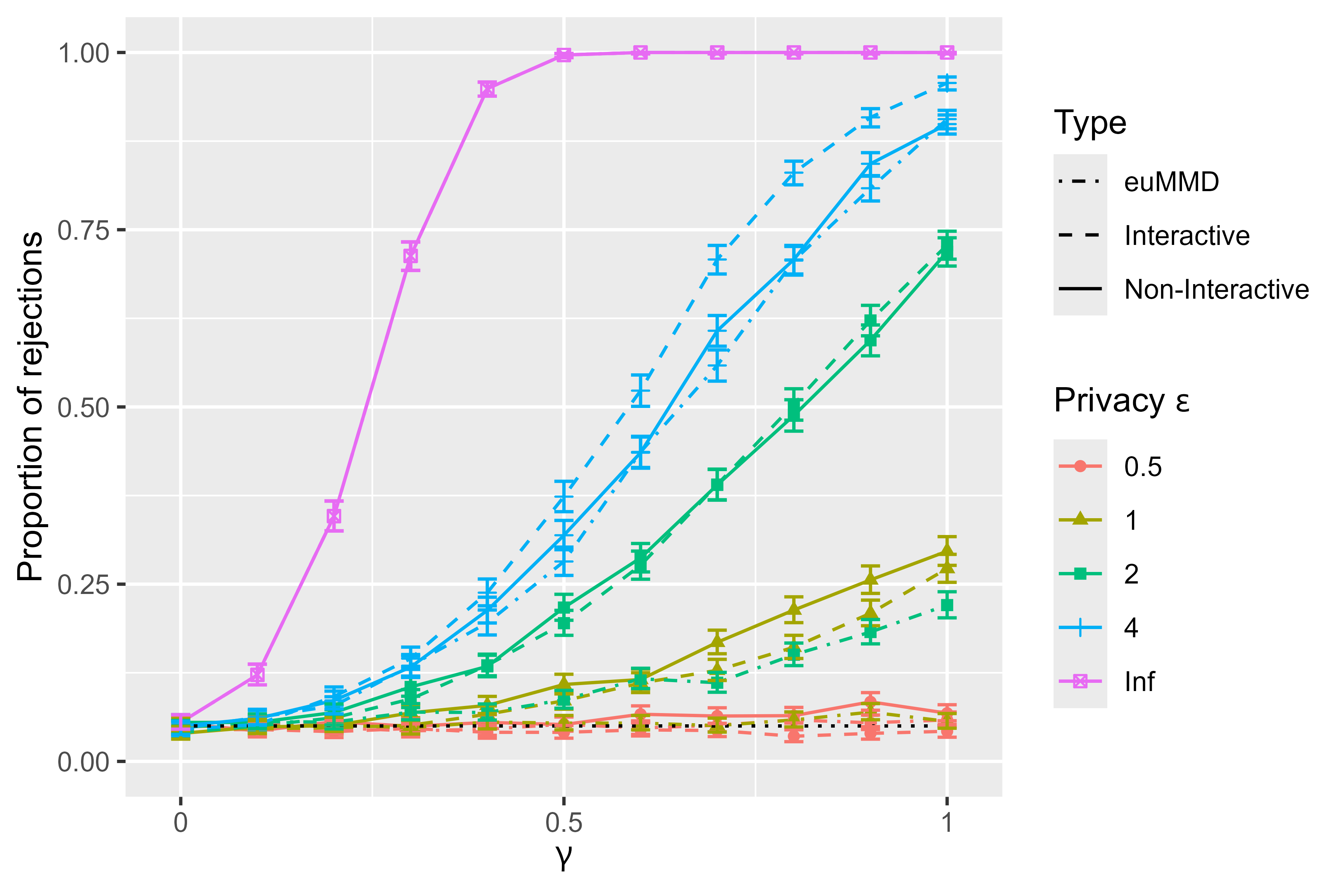}
                \end{subfigure}%

                \vspace{1em}
                
                \begin{subfigure}[t]{0.3\textwidth}
                    \centering
                    \includegraphics[height=0.16\textheight]{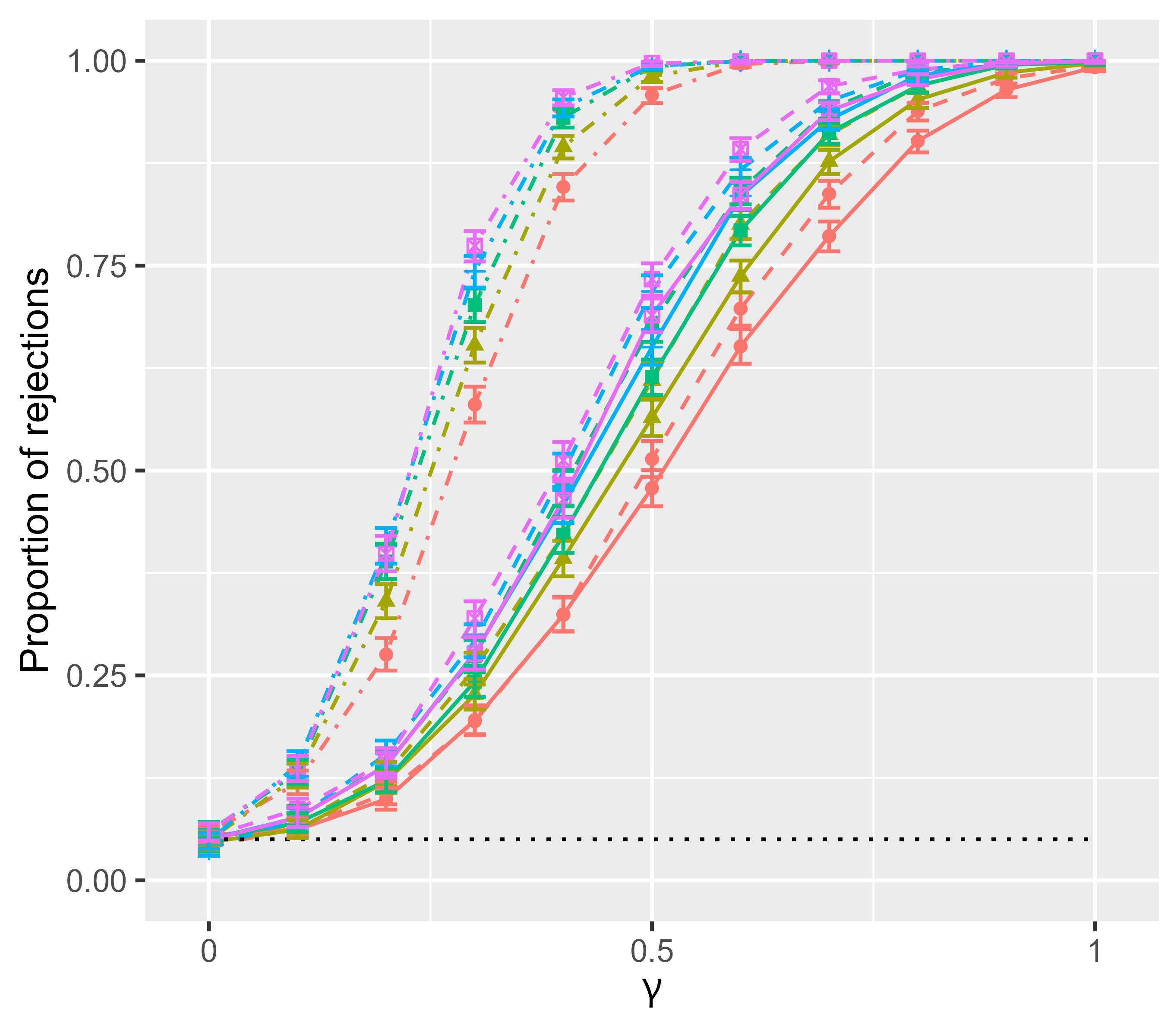}
                \end{subfigure}%
                \begin{subfigure}[t]{0.3\textwidth}
                    \centering
                    \includegraphics[height=0.16\textheight]{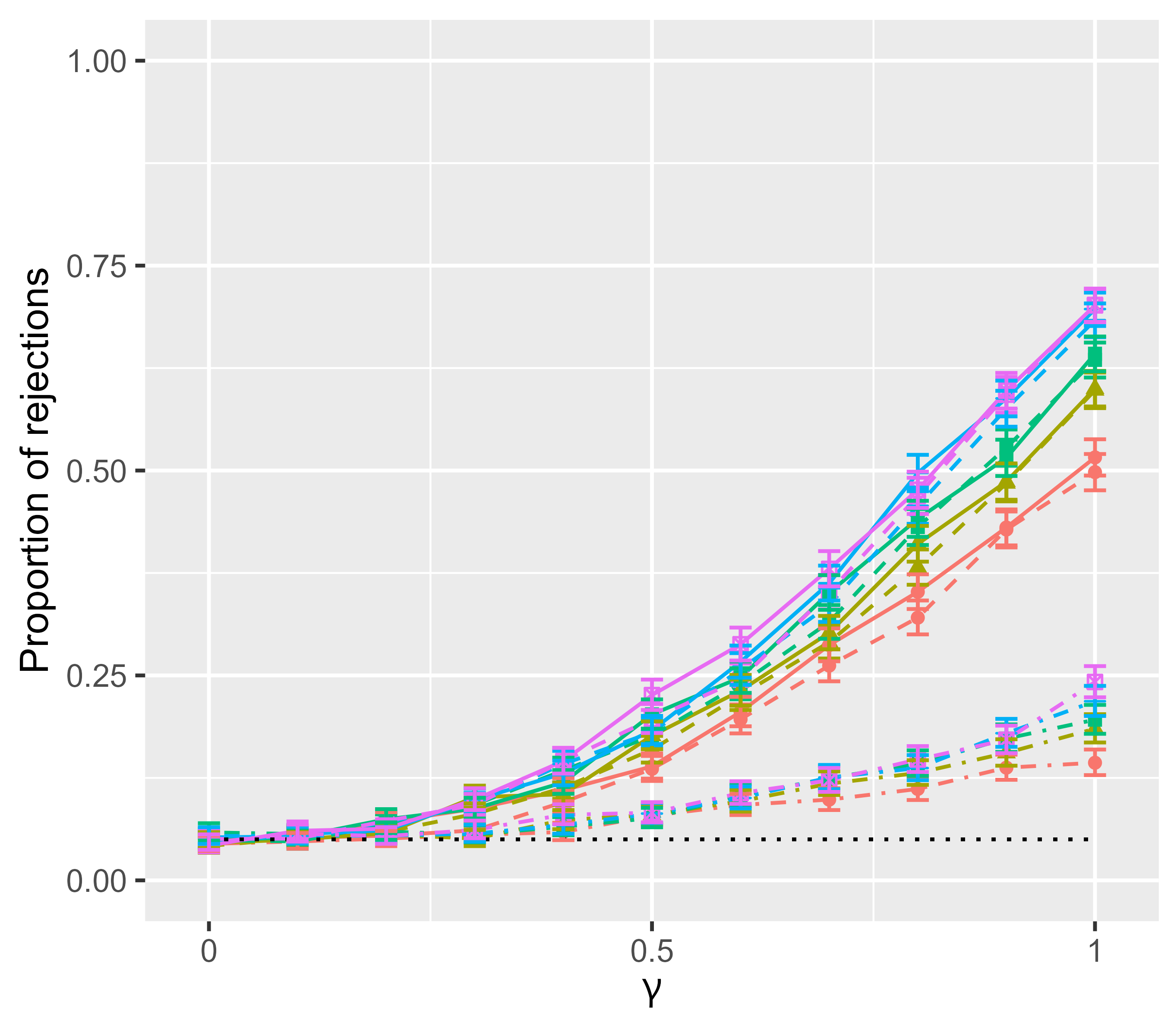}
                \end{subfigure}%
                \begin{subfigure}[t]{0.38\textwidth}
                    \centering
                    \includegraphics[height=0.16\textheight]{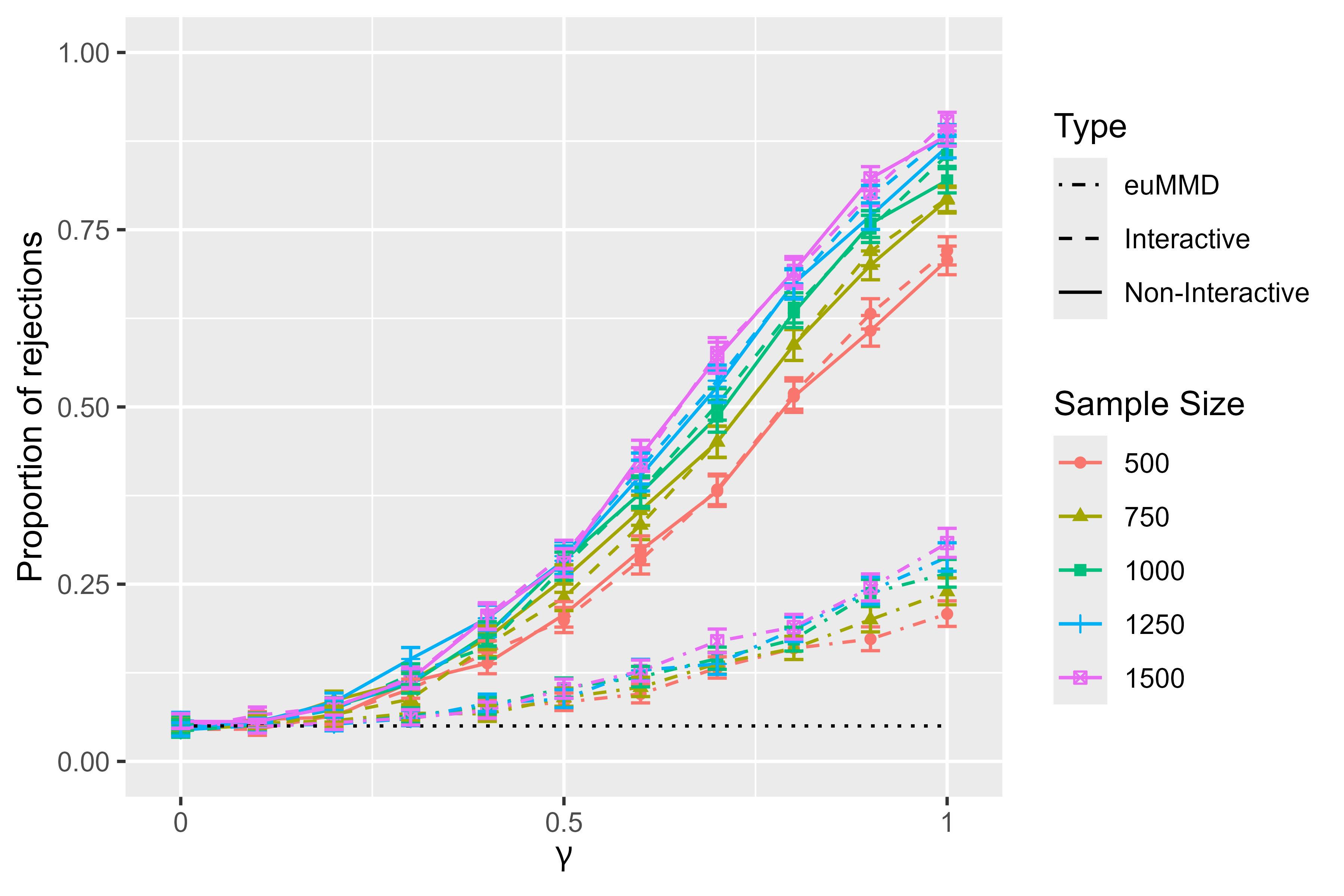}
                \end{subfigure}%

                \vspace{1em}

                \begin{subfigure}[t]{0.3\textwidth}
                    \centering
                    \includegraphics[height=0.16\textheight]{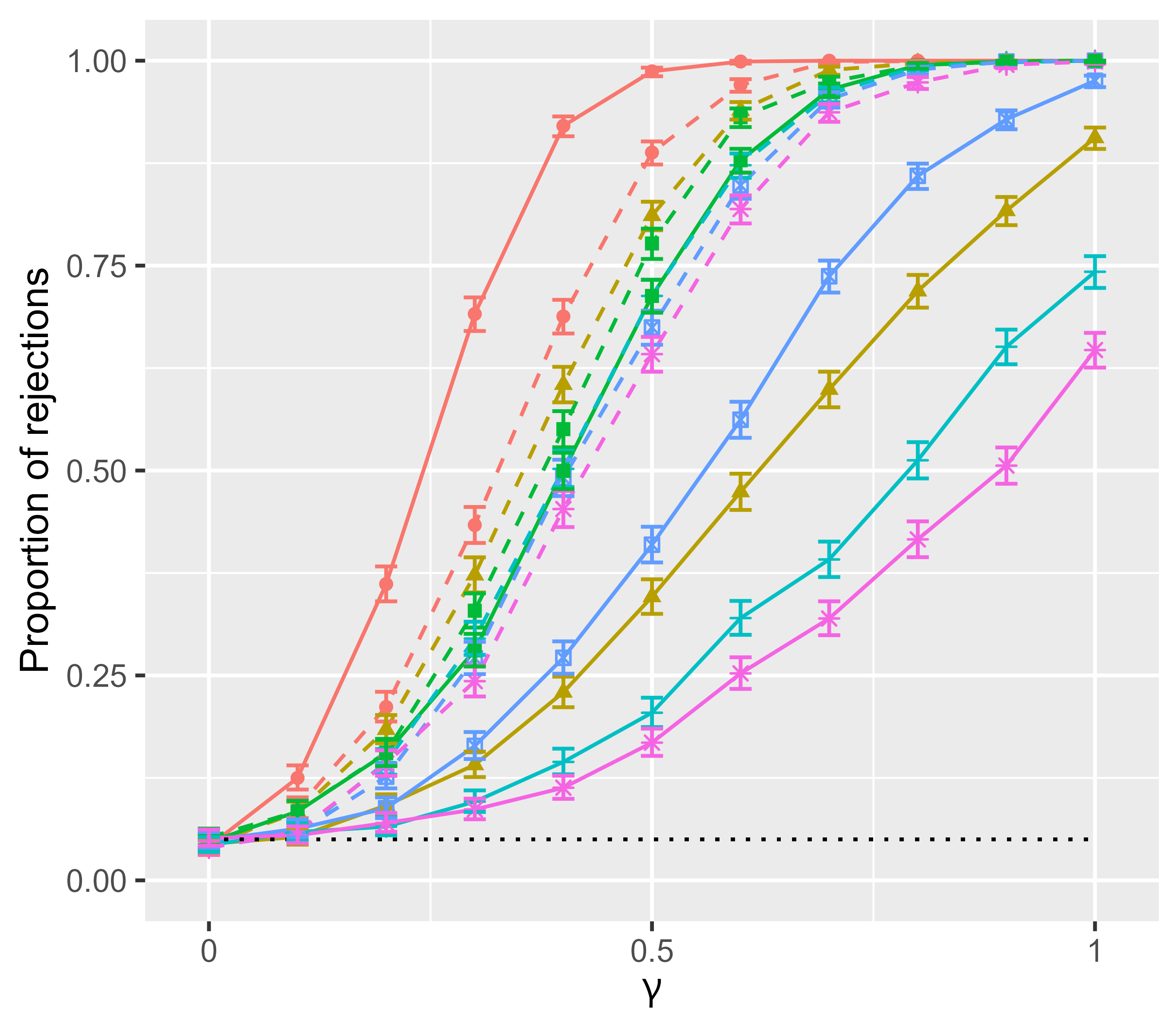}
                \end{subfigure}%
                \begin{subfigure}[t]{0.3\textwidth}
                    \centering
                    \includegraphics[height=0.16\textheight]{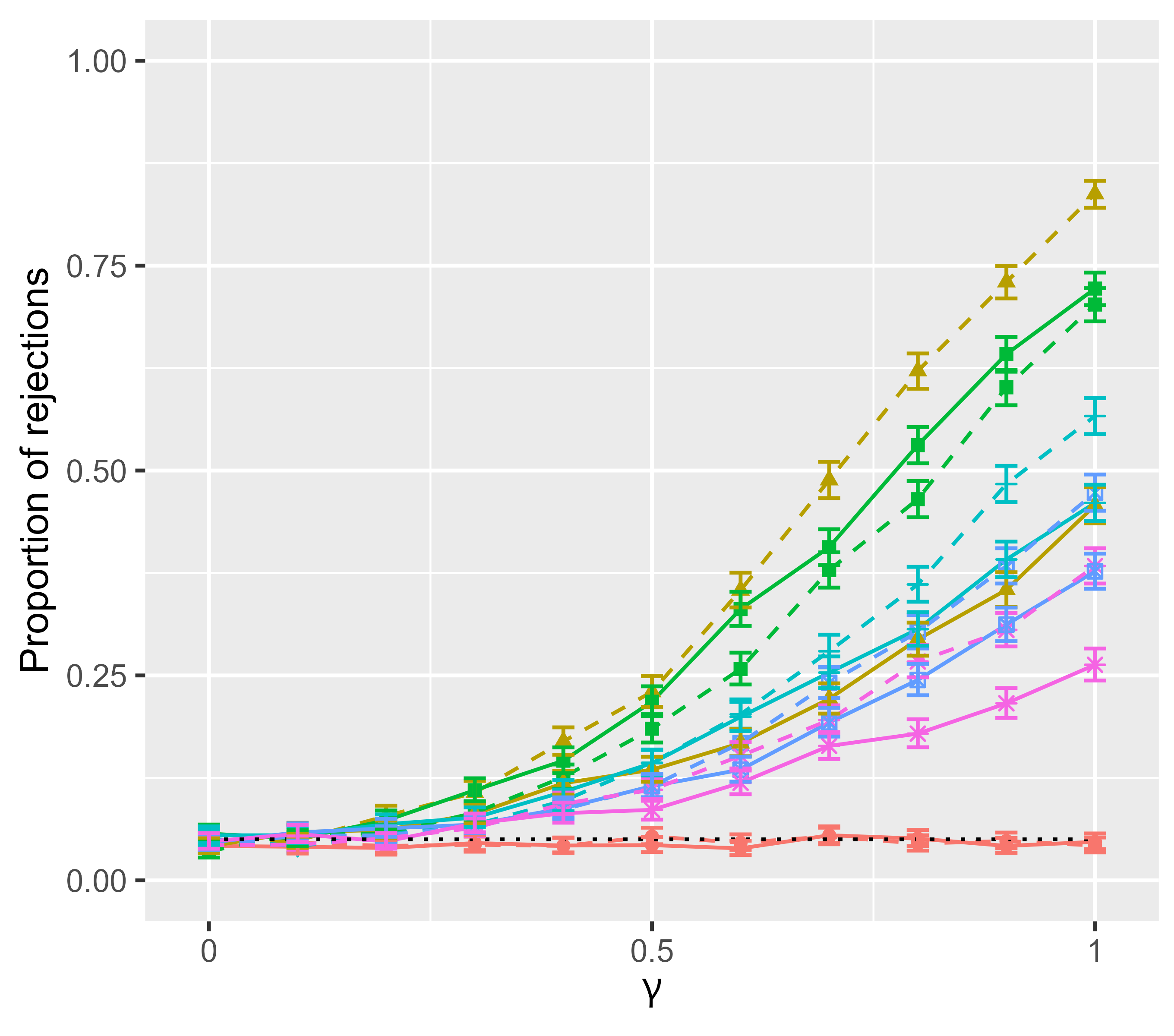}
                \end{subfigure}%
                \begin{subfigure}[t]{0.38\textwidth}
                    \centering
                    \includegraphics[height=0.16\textheight]{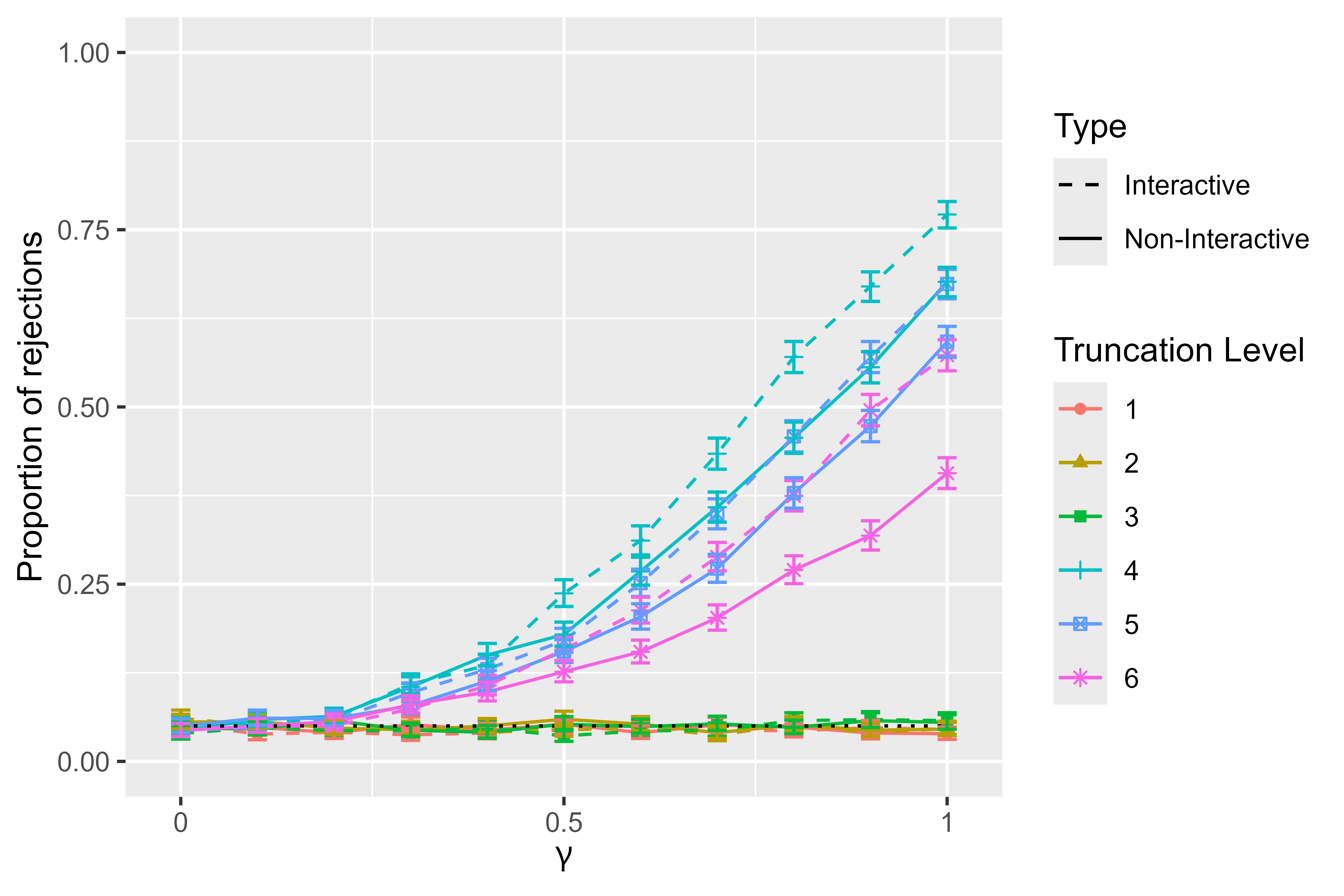}
                \end{subfigure}%
                
                \caption{Power curves for the testing problems with the distributions $P_{Y, \gamma}^{\mathrm{Beta}}$, $P_{Y, \gamma}^{\mathrm{Tri}}$ and $P_{Y, \gamma, k}^{\mathrm{Cos}}$ (first, second and third columns respectively), with $k = 1$ (first and second rows), and $k = 2$ (third row). First, second and third rows correspond to settings of varying privacy, sample size and manually specified truncation parameters respectively. Bars indicate pointwise 95\% confidence intervals. Privacy level $\varepsilon = \infty$ corresponds to an MMD-based test applied to the un\beame{privatised}{privatized}{privatized} data.}
                \label{sec6:fig:continuouspriv}
            \end{figure}

            \noindent
            \textbf{Discussion}: 
            The first row of \Cref{sec6:fig:continuouspriv} shows the performance as the privacy parameter $\varepsilon$ is varied. In summary, we observe two phenomena. Firstly, the relative performance of our methods improve against \texttt{euMMD} as the privacy requirement is strengthened. This is because, when \beame{privatising}{privatizing}{privatizing} by adding Laplace noise, the magnitude of the noise vanishes as $\varepsilon \rightarrow \infty$. On the other hand, our methods \beame{privatise}{privatize}{privatize} using the mechanisms of \cite{Duchi:2018:DJW}. These methods, whilst optimal in the high privacy regime where $\varepsilon$ is considered bounded from above by an absolute constant, introduce a non-zero magnitude of noise even as $\varepsilon \rightarrow \infty$, making them potentially less performant for larger values of $\varepsilon$. Secondly, we see that our method is superior when testing against $P_{Y, \gamma}^{\mathrm{Tri}}$ and $P_{Y, \gamma, 1}^{\mathrm{Cos}}$, whereas the MMD-based method outperforms for $P_{Y, \gamma}^{\mathrm{Beta}}$. We \beame{hypothesise}{hypothesize}{hypothesize} this is because $P_{Y, \gamma}^{\mathrm{Beta}}$ appears to be the smoothest of the three distributions considered, with deviations from the null being driven by lower frequency terms in the Fourier expansion. On the other hand, deviations are driven by higher frequency terms for the rougher distributions $P_{Y, \gamma}^{\mathrm{Tri}}$ and $P_{Y, \gamma, 1}^{\mathrm{Cos}}$. Hence, deviations are harder to detect after \beame{privatising}{privatizing}{privatizing} with additive Laplace noise which acts as a convolution, and thus attenuates the higher frequency terms.

            The second row of \Cref{sec6:fig:continuouspriv} shows the performance as the size of the second sample is grown. We observe, as in the discrete setting, that we obtain improvements in power, with these improvements holding even for an unbalanced sample many times larger than the smaller one.

            The third row of \Cref{sec6:fig:continuouspriv} shows the performance as the truncation level is manually specified. We observe that in the non-interactive setting, the power is highly sensitive to the chosen truncation. For example, considering the testing problem against $P_{Y, \gamma}^{\mathrm{Beta}}$, the power depends on the parity of the truncation level, with odd and small values being most powerful. This is likely due to the shape of the density corresponding to $P_{Y, \gamma}^{\mathrm{Beta}}$, which, due to the basis functions in \eqref{sec5:eq:trigbasis}, is better detected by sines corresponding to odd index basis functions, rather than the cosine terms. On the other hand, the interactive procedure is significantly more robust. This suggests a major benefit of the interactive procedure, and so it may be preferable to use the interactive procedure in practice, especially when there is uncertainly in the smoothness or shape of the distribution in question.
      
            For further comparison, we now consider the case of testing against $P_{Y, \gamma, 2}^{\mathrm{Cos}}$. The construction of $P_{Y, \gamma, 2}^{\mathrm{Cos}}$ is a perturbation of the uniform distribution such that the separation is concentrated at a specific frequency so that detection of deviation from the null is impossible unless the number of basis elements considered is sufficiently large, in this case at least four. Further, any choice larger than four provides no further information, and so will invariability result in a loss of power due to the additional variance. Such a construction is common when evaluating the performance of testing procedures under a Sobolev-type regularity assumption (e.g.~\citealt{Berrett:2021:IndepPerm}). As expected, neither procedure has any power for fewer than four basis elements. For larger than four considered basis elements, we see the interactive procedure preserves power better as this number increases. Hence, we can see that the interactive procedure is more robust even when the additional basis elements considered contain absolutely no further information.

        \subsection{Sensitivity Analyses} \label{sec6:sensitivity}
            We finish by considering the sensitivity to tuning parameters. First, we consider the sensitivity to the choice of $B$, the number of permutations. Throughout, we considered $B = 199$. The performance of the procedures is not particularly sensitive to $B$ once taken sufficiently large, and a value of this magnitude is common \citep[e.g.][]{Berrett:2021:USP}. We consider the discrete $L_1$-problem as in \Cref{sec6:disc}, with $n_1 = n_2 = 250$, $d = 4$, $\gamma = 0.75$, and consider $B \in \{10, 11, \hdots, 200\}$ with the power estimated for each value of $B$ across $10,000$ repetitions. We present in \Cref{fig:PermSens} the results of the sensitivity analysis, where we see that the power and type-I error guarantee is not particularly sensitive to the choice of $B$ once large enough. In particular, we see that the power jumps at every value one less of a multiple of $1/\alpha = 20$. This would suggest our value of $B = 199$ is competitive. Regardless, for large enough values of $B$ such as ours, the power does not appear to be particularly sensitive to the choice of $B$.

            \begin{figure}[H]
                \centering
                \includegraphics[width=0.45\linewidth]{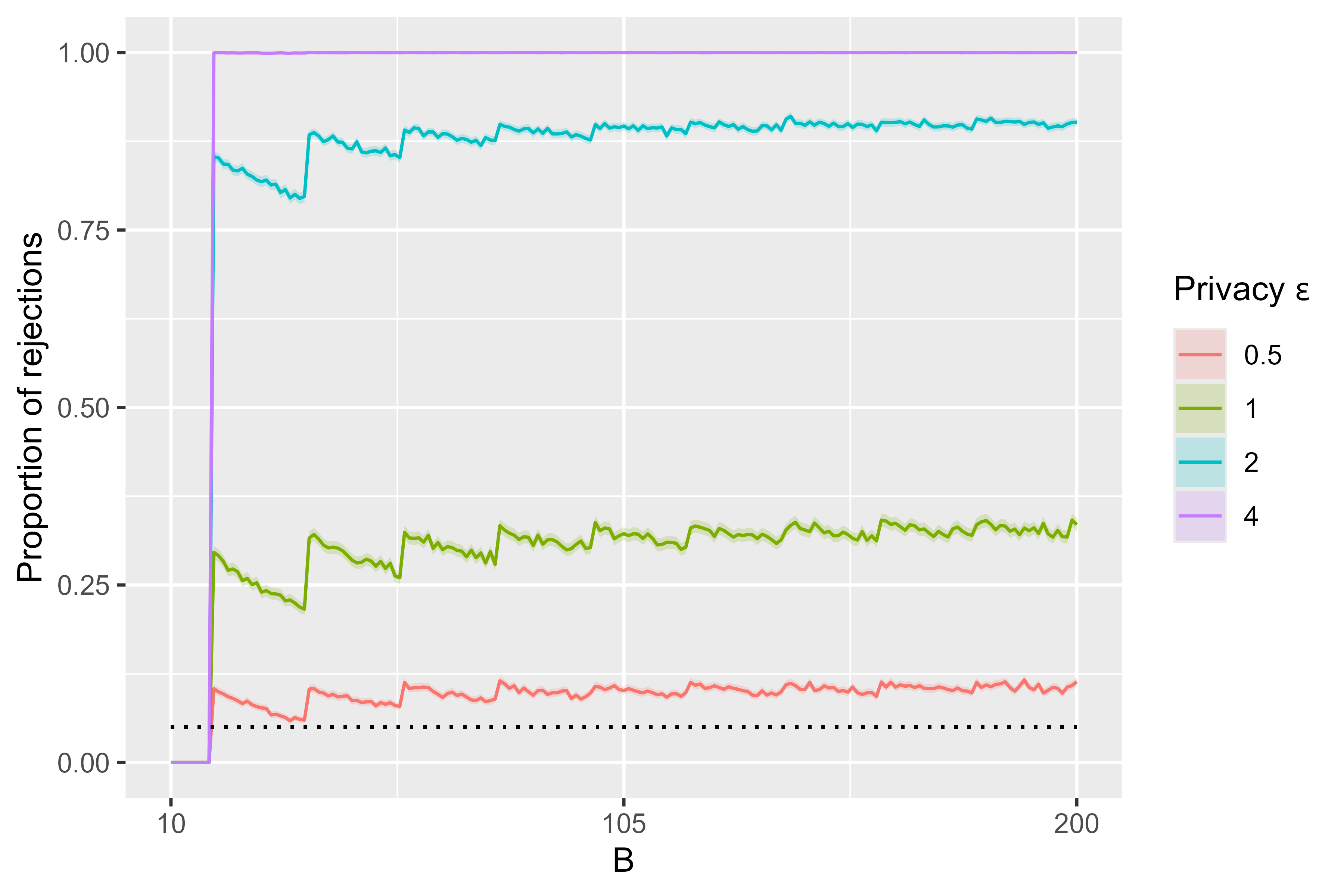}
                \caption{Power curves for the $L_1$-problem as the number of permutations $B$ is varied. Ribbons indicate pointwise $95$\% confidence intervals.}
                \label{fig:PermSens}
            \end{figure}            
            
            The interactive procedures for both discrete and continuous distribution problems involve truncating observations to facilitate the addition of noise for privacy, and the theoretical analysis permits a tunable constant in the truncation widths if desired. Given the performant nature of the interactive procedures over the non-interactive procedures previously observed, it is of interest to ensure that they are not overly sensitive to this choice in practice.

            In \Cref{fig:DISensitivityConstant}, we consider the $L_1$-problem of \Cref{sec6:disc}. Here, we vary the truncation width which is used in the sampling of the private views \eqref{sec4:eq:UEint2}, replacing $\tau$ therein with $\tau = c/(n_1\varepsilon^2)^{1/2}$ for a range of values of $c > 0$. We observe that the change in performance is minimal, suggesting the value of this tuning parameter is not of significant consequence.

            \begin{figure}[H]
                \centering
                \subfloat[Power curves for the $L_1$-problem. $d = 8$.]{\label{fig:DISensitivityConstant}
                \centering
                \includegraphics[width=0.45\linewidth]{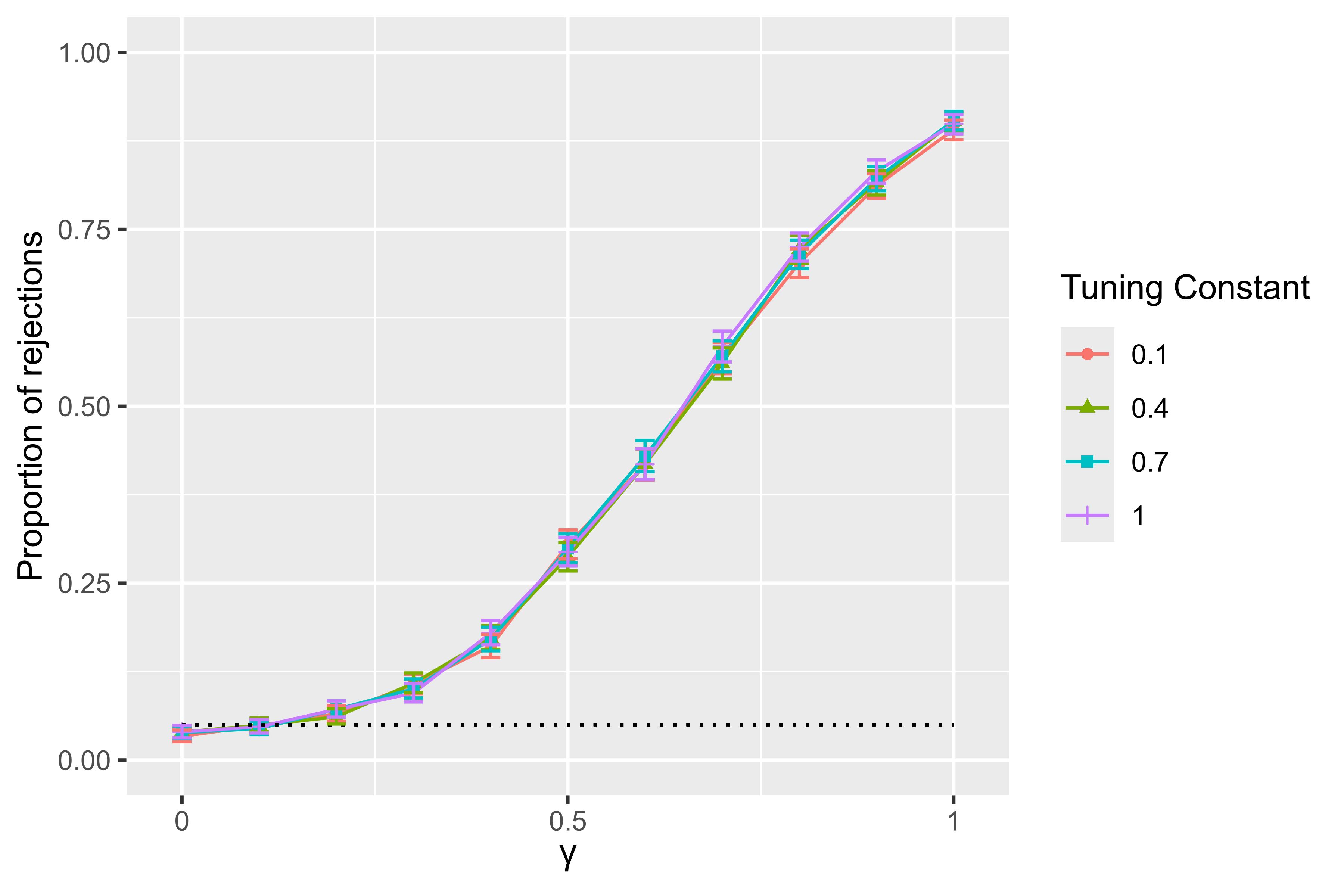}
                }
                \hfill
                \subfloat[Power curves for the smooth problem]{\label{fig:CISensitivityConstant}
                \centering
                
                \includegraphics[width=0.45\linewidth]{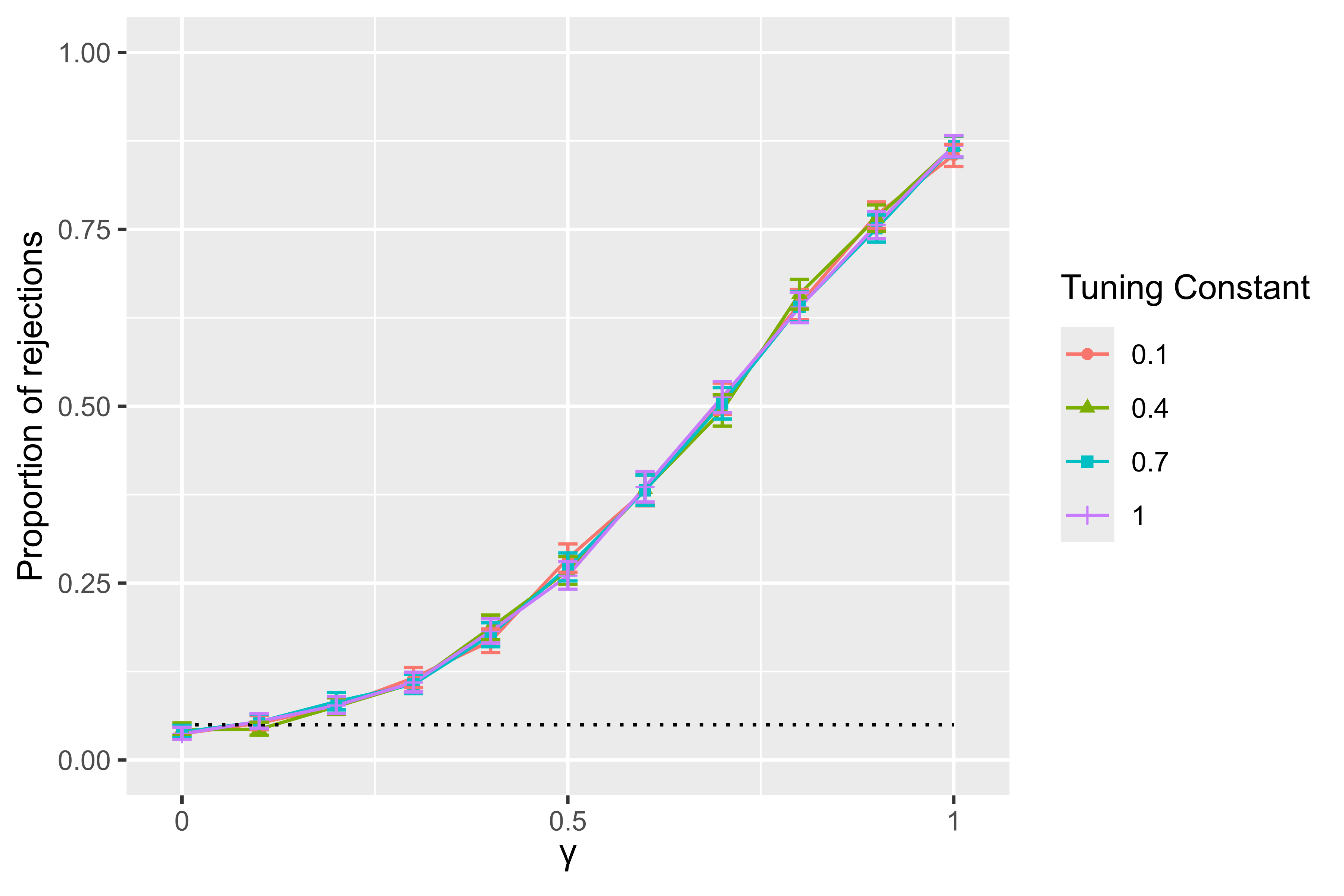}
                }
                
                \caption{Power curves for the interactive tests applied to the discrete $L_1$-problem (a) and the continuous smooth problem (b). $n_1 = n_2 = 250$. Bars indicate pointwise $95$\% confidence intervals.}
            \end{figure}

            For the interactive procedure for continuous distributions, we recall the smooth problem of \Cref{sec6:cont}. In particular, we consider the effect of the tuning parameter $C$ in the truncation widths $\{\eta_j\}_{j=1}^M$ given in \eqref{sec5:eq:truncvals} used when sampling of the private views \eqref{sec5:eq:contMethod1NotTrunc}. To ease the theoretical analysis, a sufficiently large value of $C$ is chosen, but such a large value might lead to poor performance in practice, and hence we set this constant to $1$ in our simulations. In \Cref{fig:CISensitivityConstant}, we consider the performance as we vary this constant, and observe that the performance of the algorithm is little changed.
            
    \section{Conclusion}

        In this paper, we considered two-sample distribution testing in classical settings, under the constraint of both non-interactive and interactive local differential privacy, \beame{utilising}{utilizing}{utilizing} permutation tests. We \beame{characterised}{characterized}{characterized} the minimax separation radii and developed further improvements, relaxing assumptions made in earlier permutation testing literature.

        To the best of our knowledge, there is no prior literature applying permutation testing procedures under interactive local privacy constraints. We demonstrate the improved theoretical and numerical properties of interactive procedures. Given it is not a-priori obvious whether the permutation procedure could still be applied to \beame{privatised}{privatized}{privatized} data under local privacy constraints, and especially the interactive setting where data are no longer necessarily independent, these positive results may help serve as a blueprint for further study into permutation testing under local privacy constraints such as independence testing or testing under distributional shifts.

        One note of interest is the cost of adapting to unknown smoothness parameters in the setting of continuous distributions. As shown in \cite{Ingster:2000:AdaptiveTests}, one must pay an extra iterated logarithmic cost $O([\log\{\log(n_1)\}]^{1/2})$ in the sample size to allow for adaptivity. In the private setting it is seen that this becomes $O(\log(n_1))$. This raises two questions: Whether the iterated logarithm can be recovered with more advanced techniques, and whether the square root can be recovered. The former question appears more difficult, as the common method of carrying out multiple tests to adapt to the unknown parameter runs the risk of privacy leakage, requiring the sample to be split $O(\log(n_1))$ times, suggesting an iterated logarithm is unattainable, but the possibility of a more advanced, possibly interactive, mechanism cannot be ruled out. In particular, the presence of a gap, if one exists, between the non-interactive and interactive settings in the context of adaptive tests would be highly interesting. The latter question seems more attainable, and in fact would simply require sub-Gaussian concentration rather than the sub-exponential concentration arguments used in our results. In particular, our interactive mechanisms enable stronger concentration results, and it seems possible to adapt the arguments to attain sub-Gaussian concentration.

        Another avenue for further research would be to further develop our interactive testing procedure in the setting of continuous distributions. As seen in simulations, the interactive test demonstrates greater power, and is more robust to tuning parameters, than the non-interactive test. Nevertheless, the test is more delicate than interactive goodness-of-fit tests under, for example, H\"{o}lder assumptions, requiring the combination of multiple tests to handle different regimes. The complexity seems to arise due to our choice of a Sobolev smoothness assumption, which is less commonly considered in these settings. Hence, further work to see whether the interactive test could be simplified, or more generally to understand the impact of a Sobolev smoothness assumption in locally private testing problems, could be of interest.

        We note that in this paper, we assumed the privacy level $\varepsilon$ is the same for both samples. However, \cite{Canonne:2024:Heterogeneous} instead consider heterogeneous privacy levels between the two samples, $\varepsilon_1$ and $\varepsilon_2$, where a greater magnitude of noise may be injected into the observations of one sample than the other. Applying heterogeneous privacy constraints to the permutation tests we consider is an interesting direction but introduces difficulties, as the differing privacy budgets means the distribution of the two \beame{privatised}{privatized}{privatized} samples will differ, even if the distribution of the underlying samples is the same, breaking the requirement of permutation invariance under the null. Seeing whether this issue can be overcome, such as by using a \beame{Studentising}{Studentizing}{Studentizing} procedure, is an interesting direction for future research.

        Finally, we note that this work and prior work on permutation tests under Central Differential Privacy (CDP) \citep{Kim:2023:PrivatePermutationTests} may potentially be combined for the setting of hypothesis testing under Federated Differential Privacy (FDP). The FDP setting in \cite{Cai:2024:FDPHypothesisTesting} supposes multiple independent data silos, each requiring a CDP guarantee for their data. Here, LDP and CDP testing strategies are combined, suggesting possible future directions of research.

\section*{Acknowledgements}
The second author was supported by European Research Council Starting Grant 101163546. The third author was partially funded by the EPSRC (EP/Z531327/1) and the Leverhulme Trust (Philip Leverhulme Prize).

\bibliographystyle{plainnat}
\bibliography{bibliography}

\appendix

\section*{Appendices}

The appendices are \beame{organised}{organized}{organized} as follows. In \Cref{app:sec:separationcriteria} we collect the proofs of the separation criteria stated in \Cref{sec3:sec:separationcriteria}. In \Cref{app:sec:UBDiscrete} and \Cref{app:sec:ContUB} we collect the proofs of the theoretical properties of our constructed hypothesis tests, yielding the upper bounds for the minimax separation radii results in \Cref{sec4} and \Cref{sec5} respectively. In \Cref{app:sec:lower}, we provide the proof of \Cref{sec3:lem:lowerboundlem} for obtaining lower bounds for the two-sample testing problem from the analogous goodness-of-fit testing problem. In Appendices \ref{app:sec:tailbounds1}, \ref{app:sec:tailbounds2} and \ref{app:sec:tailbounds3}, we collect the proofs for numerous tail bound results for, respectively, the $U$-statistics considered for our non-interactive tests; for the test statistics considered in the interactive test for discrete data; and the test statistics considered for the interactive tests for continuous data. Finally, in \Cref{app:misc} we collected miscellaneous technical lemmata used throughout the main text. 
\section{Separation Criteria} \label{app:sec:separationcriteria} \label{Sec:appA}
    \subsection{Proof of \texorpdfstring{\Cref{sec3:prop:permtestcontrol}}{Proposition 1}}
        \begin{proof}[Proof of \Cref{sec3:prop:permtestcontrol}]
            Consider a sample of permutations $\{\pi_1, \hdots, \pi_B\}$ of size $B$ drawn, with replacement, uniformly from $S_n$, the group of permutations on $[n]$. We start by noting that for any $t > 0$, $p_B$ as defined in \eqref{sec2:def:permpval} satisfies
            \begin{align}
                \mathbb{P}(p_B > \alpha)
                &= \mathbb{P}\bigg(1 + \sum_{b = 1}^B \mathbbm{1}\{T_n \leq T_n^{\pi_b}\} \geq (1+B)\alpha\bigg) \nonumber \\
                &\leq \mathbb{P}\bigg( \bigg\{1 + \sum_{b = 1}^B \mathbbm{1}\{T_n \leq T_n^\pi\} \geq (1+B)\alpha \bigg\} \cap \{T_n \geq t\}\bigg) + \mathbb{P}(T_n < t) \nonumber \\
                &\leq \mathbb{P}\bigg( 1 + \sum_{b = 1}^B \mathbbm{1}\{T_n^{\pi_b} \geq t\} \geq (1+B)\alpha\bigg) + \mathbb{P}(T_n < t) \nonumber \\
                &\leq \frac{1+B\mathbb{P}(T_n^{\pi_1} \geq t)}{(1+B)\alpha} + \mathbb{P}(T_n < t), \label{app:eq:generaltype2errorbound}
            \end{align}
            where the final inequality follows from Markov's inequality. 
            
            First, consider the case where $T_n^{\pi_1}$ is $\mathrm{SE}(\widetilde{\Sigma})$. By Proposition~2.7.1 in \cite{Vershynin:2018:HDPBook}, there exists an absolute constant $C > 0$, possibly differing from that in the proposition statement, such that $\mathbb{P}[T_n^{\pi_1} \geq C\widetilde{\Sigma} \log\{8/(\alpha\beta)\} ] \leq \alpha\beta/4$, where we use the fact that $\mathbb{E}[T_n^{\pi_1}] = 0$. We also have
            \begin{align*}
                \mathbb{P}\big[T_n < C\widetilde{\Sigma}\log\{8/(\alpha\beta)\}\big]
                &= \mathbb{P}\big[T_n - \mathbb{E}[T_n] < C\widetilde{\Sigma}\log\{8/(\alpha\beta)\} - \mathbb{E}[T_n]\big] \\
                &\leq \mathbb{P}\{T_n - \mathbb{E}[T_n] < -C\Sigma \log(4/\beta)\}
                \leq \beta/2,
            \end{align*}
            where in the first inequality we use the fact that $\mathbb{E}[T_n] \geq C\widetilde{\Sigma}\log\{8/(\alpha\beta)\} + C\Sigma \log(4/\beta)$ as in the proposition statement, and the final by the fact that $T_n$ is $\mathrm{SE}(\Sigma)$ and Proposition~2.7.1 in \cite{Vershynin:2018:HDPBook}.

            Thus, setting $t = C\widetilde{\Sigma}\log\{8/(\alpha\beta)\}$ in \eqref{app:eq:generaltype2errorbound} and using the above inequalities, we obtain
            \begin{align*}
                \mathbb{P}(p_B > \alpha)
                \leq \frac{1 +\alpha\beta B/4}{(1+B)\alpha} + \beta/2
                \leq \beta,
            \end{align*}
            where the final inequality uses the fact that $B \geq 4/(\alpha\beta) - 1$.
            
            In the case $T_n^{\pi_1}$ is $\mathrm{SG}(\widetilde{\Sigma}^2)$, the proof follows similarly as $\mathbb{P}[T_n^{\pi_1} \geq 2^{1/2}\widetilde{\Sigma} [\log\{8/(\alpha\beta)\}]^{1/2} ] \leq \alpha\beta/4$ by \citet[e.g.~Proposition~2.5][]{Wainwright:2019:HDSBook}.
        \end{proof}

    \subsection{Proof of \texorpdfstring{\Cref{sec3:thm:sepcondU}}{Theorem 3}}
        The proof relies on the following tail bounds for the original and permuted version of the $U$-statistic, the proofs of which are deferred to \Cref{app:sec:tailbounds1}.
    
        \begin{proposition} \label{app:prop:Ustatconc}
            Using the notation and setting of \Cref{sec3:thm:sepcondU}, the $U$-statistic $U_{n_1, n_2}$ is $\mathrm{SE}(\Sigma)$  where
            \begin{equation} \label{app:eq:Ustatconcsigma}
                \Sigma = C\max\bigg\{ \frac{d^{1/2}\sigma^2}{n_1}, \frac{(\mathbb{E}[U_{n_1, n_2}])^{1/2} \sigma}{n_1^{1/2}}\bigg\} ,
            \end{equation}
            for $C > 0$ some absolute constant.
        \end{proposition}
        
        \begin{proposition} \label{app:prop:Ustatpermconc} 
            Using the notation and setting of \Cref{sec3:thm:sepcondU}, the permuted $U$-statistic $U_{n_1, n_2}^\pi$ is $\mathrm{SE}(\widetilde{\Sigma})$ where
            \begin{equation} \label{app:eq:Ustatpermconcsigma}
                \widetilde{\Sigma} = \widetilde{C} \max\bigg\{ \frac{d^{1/2}\sigma^2}{n_1}, \frac{(\mathbb{E}[U_{n_1, n_2}])^{1/2}\sigma}{n_1^{1/2}}, \frac{\mathbb{E}[U_{n_1, n_2}]}{n_1} \bigg\}
            \end{equation}
            for $\widetilde{C} > 0$ some absolute constant.
        \end{proposition}

        With these results in hand, the proof of \Cref{sec3:thm:sepcondU} follows simply from an application of \Cref{sec3:prop:permtestcontrol}.
        \begin{proof}[Proof of \Cref{sec3:thm:sepcondU}]
            So that we may apply \Cref{sec3:prop:permtestcontrol}, we first verify $\mathbb{E}[U_{n_1, n_2}^\pi] = 0$. Indeed, denote $\tau_{i, n_1 + k}$ the transposition on $(i, n_1 + k)$, that is, the permutation which swaps these two symbols and keeps the remaining symbols fixed. We note that for any $(i,j) \in \mathcal{I}_{2}^{n_1}$ and $(k,l) \in \mathcal{I}_{2}^{n_2}$, that
            \begin{align*}
                \mathbb{E}&[(\vecbf{D}_{\pi(i)} - \vecbf{D}_{\pi(n_1 + k)})^T(\vecbf{D}_{\pi(j)} - \vecbf{D}_{\pi(n_1 + l)})] \\
                &= \mathbb{E}[(\vecbf{D}_{\pi \circ \tau_{i, n_1 + k}(i)} - \vecbf{D}_{\pi \circ \tau_{i, n_1 + k}(n_1 + k)})^T(\vecbf{D}_{\pi \circ \tau_{i, n_1 + k}(j)} - \vecbf{D}_{\pi \circ \tau_{i, n_1 + k}(n_1 + l)})] \\
                &= -\mathbb{E}[(\vecbf{D}_{\pi(i)} - \vecbf{D}_{\pi(n_1 + k)})^T(\vecbf{D}_{\pi(j)} - \vecbf{D}_{\pi(n_1 + l)})],                
            \end{align*}
            where the first equality is by the fact that $\pi$ and the composition $\pi \circ \tau_{i, n_1 + k}$ are equal in distribution. Hence, we have that $\mathbb{E}[(\vecbf{D}_{\pi(i)} - \vecbf{D}_{\pi(n_1 + k)})^T(\vecbf{D}_{\pi(j)} - \vecbf{D}_{\pi(n_1 + l)})] = 0$ and consequently, by the linearity of expectation, that $\mathbb{E}[U_{n_1, n_2}^\pi] = 0$.

            Then, combining \Cref{app:prop:Ustatconc} and \Cref{app:prop:Ustatpermconc}, we have by the condition \eqref{sec3:eq:permtestcontrolSE} of \Cref{sec3:prop:permtestcontrol} that the type-II error is controlled provided
            \begin{align*}
                \mathbb{E}[U_{n_1, n_2}] \geq C\widetilde{\Sigma}\log\{1/(\alpha\beta)\} + C\Sigma\log(1/\beta),
            \end{align*}
            for $C > 0$ some sufficiently large constant, where $\Sigma$ and $\widetilde{\Sigma}$ are as in \eqref{app:eq:Ustatconcsigma} and \eqref{app:eq:Ustatpermconcsigma} respectively. Simplifying the resulting expression shows that the criteria given in \eqref{sec2:eq:sepcriteria} are sufficient, which completes the proof.
        \end{proof}

    \section{Lower Bounds} \label{app:sec:lower}
        \subsection{Proof of \texorpdfstring{\Cref{sec3:lem:lowerboundlem}}{Lemma 4}}
            Before providing the proof of \Cref{sec3:lem:lowerboundlem}, we first formally define the goodness-of-fit testing setting.
        
            Consider the set-up introduced in \Cref{sec2:sec:problemdef}. For $\mathcal{D}_{X, n_1}$ and some distribution $P_0 \in \mathcal{P}$, we are interested in testing
            \begin{equation} \label{app:eq:GOFtestdef}
                H_0': P_X = P_0 \quad \mbox{vs.} \quad H_1': D'(P_X, P_0) \geq \rho',
            \end{equation}
            for some metric $D'$ on the space of distributions, and some level of separation $\rho' > 0$.
        
            We fix a collection of privacy mechanisms $Q$, observing $\tilde{\mathcal{D}}_{X, n_1} = \{ Z_i \}_{i \in [n_1]}$ the induced sample  of $\varepsilon$-LDP views satisfying \eqref{sec2:eq:LDPdef}.
            
            Consider a test $\phi' : \tilde{\mathcal{D}}_{X, n_1} \rightarrow \{0, 1\}$ where a value of $\phi = 1$ indicates that we reject the null hypothesis of \eqref{app:eq:GOFtestdef}. We then respectively define the space of distributions satisfying the alternative, and the set of level-$\alpha$ tests for $P_0$ and a fixed $\varepsilon$-LDP mechanism $Q$ as
            \begin{equation*}
                \mathcal{P}_1'(D', \rho') = \big\{P \in \mathcal{P} : D'(P, P_0) \geq \rho' \big\}, \quad \Phi_{n_1, P_0, Q}'(\alpha) = \bigg\{\phi': \mathbb{P}_{P_0, Q} (\phi' = 1) \leq \alpha \bigg\}.
            \end{equation*}
            For a fixed level $\alpha \in (0, 1)$, privacy level $\varepsilon \in (0, 1]$ and separation $\rho' > 0$, the private minimax testing risk is
            \begin{equation} \label{app:eq:GOFTestingRisk}
                 \mathcal{R}_{n_1, P_0, \varepsilon, \alpha, D', \rho'}'
                 = \alpha + \inf_{Q \in \mathcal{Q}_\varepsilon}\inf_{\phi' \in \Phi_{n_1, P_0, Q}(\alpha)} \sup_{P_X \in \mathcal{P}_1(D', \rho')} \mathbb{P}_{P_X, Q}(\phi' = 0),
            \end{equation}
            and, fixing a desired type-II error probability $\beta \in (0, 1 - \alpha)$, we define the (private) minimax separation radius 
            \begin{equation*}
                \rho^\ast(n_1, P_0, \varepsilon, \alpha, \beta, D') = \inf\{\rho' > 0 : \mathcal{R}_{n_1, P_0, \varepsilon, \alpha, D', \rho'}' \leq \alpha + \beta\}.
            \end{equation*}
            Again, when $\mathcal{Q}_\varepsilon$ is the class of all non-interactive or interactive mechanisms, we may write $\rho_{\mathrm{NI}}^\ast$ and $\rho_{\mathrm{I}}^\ast$ respectively.
                
            \begin{proof}[Proof of \Cref{sec3:lem:lowerboundlem}]
               The proof adapts \citet[Lemma~1]{Arias:2018:GoodnessofFitDimensionality} to account for \beame{privatisation}{privatization}{privatization}. We focus on two-sample sequentially interactive mechanisms, with the non-interactive case following as a special case.
        
                First, recall the goodness-of-fit testing problem where $P_0$ denotes the distribution under the null. Fix some separation $\tilde{\rho} > 0$ and let $Q$ be an $\varepsilon$-LDP interactive privacy mechanism and $\phi \in \Phi_{n_1, n_2, Q}(\alpha)$ be a level-$\alpha$ test, both for the two-sample testing problem. Lastly, denote the joint distribution of private marginals so that $(Z_1, \hdots, Z_{n_1}, W_1, \hdots, W_{n_2}) \sim M_{Z_{1:n_1}, W_{1:n_2}}$. Define a test $\phi'$ for the goodness-of-fit testing problem via 
                \begin{align*}
                    \phi'(z_1, \hdots, z_{n_1}) =
                    \begin{cases}
                        1, & \text{with prob. } \mathbb{E}[\phi(Z_1, \hdots, Z_{n_1}, W_1, \hdots, W_{n_2}) \mid Z_1 = z_1, \hdots, Z_{n_1} = z_{n_1}], \\
                        0, & \text{otherwise.}
                    \end{cases}
                \end{align*}
                
                Let $(Z_1, \hdots, Z_{n_1}) \sim M_{Z_{1:n_1}}$ where $M_{Z_{1:n_1}}$ is the corresponding marginal distribution of $M_{Z_{1:n_1}, W_{1:n_2}}$. Importantly, note that this corresponds to an induced one-sample privacy mechanism obtained by drawing surrogate raw values from $P_0$ and applying the two-sample privacy mechanism $Q$. We then see by the tower property of conditional expectation that
                \begin{align*}
                    \mathbb{E}[\phi'(Z_1, \hdots, Z_{n_1})]
                    &= \mathbb{E}[ \mathbb{E}\{\phi(Z_1, \hdots, Z_{n_1}, W_1, \hdots, W_{n_2}) \mid Z_1, \hdots, Z_{n_1}\}] \\
                    &= \mathbb{E}[ \phi(Z_1, \hdots, Z_{n_1}, W_1, \hdots, W_{n_2})] \leq \alpha
                \end{align*}
                where the inequality follows from the fact that $\phi \in \Phi_{n_1, n_2, Q}(\alpha)$. Hence, it holds that $\phi' \in \Phi_{n_1, P_0, Q}'(\alpha)$.
                
                With the type-I error guarantee in hand, it remains to consider the risk. In particular, as $\mathcal{R}_{n_1, P_0, \varepsilon, D, \tilde{\rho}}'$ lower bounds the goodness-of-fit testing risk for any one-sample privacy mechanism $Q$ and test $\phi'$, we have that
                \begin{align*}
                    \mathcal{R}_{n_1, P_0, \varepsilon, D, \tilde{\rho}}'
                    &\leq 
                    \alpha + \sup_{P_X \in \mathcal{P}_1'(\tilde{\rho})} \mathbb{P}_{P_X, Q}(\phi' = 0) \\
                    &=
                    \alpha + \sup_{P_X \in \mathcal{P}_1'(\tilde{\rho})} \mathbb{P}_{P_X, P_0, Q}(\phi = 0) \\
                    &\leq 
                    \alpha + \sup_{(P_X, P_Y) \in \mathcal{P}_1(\tilde{\rho})} \mathbb{P}_{P_X, P_Y, Q}(\phi = 0).
                \end{align*}
                Taking the infimum over $\phi \in \Phi_{n_1, n_2, Q}(\alpha)$ yields
                    \begin{align*}
                    \mathcal{R}_{n_1, P_0, \varepsilon, D, \tilde{\rho}}'
                    \leq \mathcal{R}_{n_1, n_2, Q, D, \tilde{\rho}},
                \end{align*}
                and lastly, taking the infimum over all non-interactive/interactive $\varepsilon$-LDP privacy mechanisms $Q$ completes the proof.
            \end{proof}

        \subsection{Proof of \texorpdfstring{\Cref{sec4:thm:main}}{Theorem 5} Lower Bound} \label{app:sec:disclbproof}

            In this section we prove lower bounds on the minimax separation radii for the two-sample testing problem for discrete multinomial distributions. We obtain results for two-sample testing from the lower bounds for the goodness-of-fit testing problem by applying \Cref{sec3:lem:lowerboundlem}. The non-interactive rates follow from the results of \cite{Berrett:2020:Faster}, which we extend to consider general $L_p$-separation, and we consider a more direct approach for the interactive setting to allow us to improve earlier results by a logarithmic factor.
        
            \begin{proof}[Proof of \Cref{sec4:thm:main} (Lower Bound)]

                So that we may apply \Cref{sec3:lem:lowerboundlem}, we first \beame{formalise}{formalize}{formalize} the conditions for the goodness-of-fit testing risk as in \eqref{app:eq:GOFTestingRisk}. We specify the relevant quantities therein, yielding:
                \begin{itemize}
                    \item the distribution under the null $P_0$ to be uniform over $[d]$ with corresponding probability vector $\vecbf{p}_0$; 
                    \item the metric $D'$ on the space of distributions to be the $L_p$-norm of the difference of the underlying probability vectors so that $D'(P, P_0) = \|\vecbf{p} - \vecbf{p}_0\|_p$; and
                    \item the separation $\rho'$ is set to be $\rho_{\mathrm{NI}} = d^{1/p - 1/4}/(n_1\varepsilon^2)^{1/2} \wedge d^{1/p - 1}/\{\log(d)\}^{1/2}$ and $\rho_{\mathrm{I}} = d^{1/p - 1/2}/(n_1\varepsilon^2)^{1/2} \wedge (d^{1/p - 1}\mathbbm{1}\{p \notin \{1, 2\}\} + \mathbbm{1}\{p \in \{1, 2\}\})$ in the non-interactive and interactive settings respectively.
                \end{itemize}

                For these choices, we show that $\mathcal{R}_{n_1, P_0, \varepsilon, D', \rho_{\mathrm{NI}}}', \mathcal{R}_{n_1, P_0, \varepsilon, D', \rho_{\mathrm{I}}}' \geq 1/2$ for the testing risks over non-interactive and interactive privacy mechanisms respectively, from which it will follow that the minimax separation radii satisfy $\rho_{\mathrm{NI}}^\ast(n, \varepsilon, L_1 , \alpha, \beta) \geq d^{1/p - 1/4}/(n_1\varepsilon^2)^{1/2} \wedge d^{1/p - 1}/\{\log(d)\}^{1/2}$ and $\rho_{\mathrm{I}}^\ast(n, \varepsilon, L_1 , \alpha, \beta) \geq d^{1/p - 1/2}/(n_1\varepsilon^2)^{1/2}$.
                
                We note that it is sufficient to consider the $L_1$-norm, from which the result for the $L_p$-norms with $p \in (1,2]$ will follow, though we will also consider the case of $L_2$-norm in the interactive setting directly to ameliorate the requirement on the range of $n\varepsilon^2$. Indeed, we note by H\"{o}lder's inequality that $\|\vecbf{p} - \vecbf{p}_0\|_1 \leq d^{1 - 1/p} \|\vecbf{p} - \vecbf{p}_0\|_p$ and hence we have the containment
                \begin{equation*}
                    \{\vecbf{p} : \|\vecbf{p} - \vecbf{p}_0\|_1 \geq \rho\}
                    \subseteq \{\vecbf{p} : \|\vecbf{p} - \vecbf{p}_0\|_p \geq d^{1/p - 1}\rho\}.
                \end{equation*}
                Thus, we see that the testing risks satisfy
                \begin{align*}
                    \mathcal{R}_{n_1, P_0, \varepsilon, L_p, d^{1/p - 1}\rho}'
                    &= \inf_{Q \in \mathcal{Q}_\varepsilon}\inf_{\phi \in \Phi_{n_1, P_0, Q}(\alpha)} \sup_{\vecbf{p} : \| \vecbf{p} - \vecbf{p}_0\|_p \geq d^{1/p - 1}\rho} \big\{\mathbb{P}_{P_0, Q}(\phi = 1) + \mathbb{P}_{P_X, Q}(\phi = 0) \big\} \\
                    &\geq \inf_{Q \in \mathcal{Q}_\varepsilon}\inf_{\phi \in \Phi_{n_1, P_0, Q}(\alpha)} \sup_{\vecbf{p} : \| \vecbf{p} - \vecbf{p}_0\|_1 \geq \rho} \big\{\mathbb{P}_{P_0, Q}(\phi = 1) + \mathbb{P}_{P_X, Q}(\phi = 0) \big\} \\
                    &= \mathcal{R}_{n_1, P_0, \varepsilon, L_1, \rho}'.
                \end{align*}
                Hence, if $\rho$ satisfies $\mathcal{R}_{n_1, P_0, \varepsilon, L_1, \rho}' \geq 1/2$, then the separation radius (non-interactive or interactive suitably) with respect to the $L_p$-norm satisfies
                \begin{equation} \label{app:eq:L1SufficientForDisc}
                    \rho^\ast(n, \varepsilon, L_p , \alpha, \beta) \geq d^{1/p -1}\rho.
                \end{equation}

                \medskip
                \noindent
                \textbf{Non-interactive Mechanisms}:
                We first consider the rates for the non-interactive settings, which we will obtain from \cite{Berrett:2020:Faster}. By considering the lower bound for $L_1$-separation in \citet[Supplement, p.~10]{Berrett:2020:Faster} with $j^\ast \asymp d$ and $p_0(j^\ast + 1) = 1/d$, corresponding to the uniform distribution, therein, we obtain
                \begin{equation*}
                    \rho_{\mathrm{NI}}^\ast(n, \varepsilon, L_1 , \alpha, \beta)
                    \gtrsim \min\bigg\{\frac{d^{3/4}}{(n\varepsilon^2)^{1/2}}, \frac{1}{\{\log(d)\}^{1/2}}, d^{1/2} \bigg\}
                    \gtrsim \min\bigg\{\frac{d^{3/4}}{(n\varepsilon^2)^{1/2}}, \frac{1}{\{\log(d)\}^{1/2}} \bigg\}.
                \end{equation*}
                Hence, for general $L_p$-separation with $p \in [1,2]$, we have by \eqref{app:eq:L1SufficientForDisc} that
                \begin{align*}
                    \rho_{\mathrm{NI}}^\ast(n, \varepsilon, L_p , \alpha, \beta)
                    &\geq d^{1/p - 1}\rho_{\mathrm{NI}}^\ast(n, \varepsilon, L_1 , \alpha, \beta) \\
                    &\gtrsim \min\bigg\{\frac{d^{1/p - 1/4}}{(n\varepsilon^2)^{1/2}}, \frac{d^{1/p - 1}}{\{\log(d)\}^{1/2}} \bigg\}.
                \end{align*}
                This completes the proof in the non-interactive case.

                \medskip
                \noindent
                \textbf{Interactive Mechanisms}:
                We now focus on the interactive setting. Here, we will use the fact that in the interactive case the information theoretic bounds we consider enable us to consider a simpler construction, allowing us to avoid the logarithm factors incurred in \cite{Berrett:2020:Faster}.
            
                Given sample size $n \in \mathbb{N}$, fix an interactive $\varepsilon$-LDP privacy mechanism $(Q_1, \hdots, Q_n) \in \mathcal{Q}_\varepsilon$ where, for $i \in [n]$, $Q_i:\sigma(\mathcal{Z}) \times [d] \times \mathcal{Z}^{i-1} \rightarrow [0,1]$. Denoting, for each $i \in [n]$, a measure $\nu_i$ on $\mathcal{Z}$, by \citet[Lemma~B.3]{Butucea:2023:Interactive} we can write the densities $\{q_i(\cdot \mid x_i, z_{1:i-1}) : i \in [n], x_i \in [d], z_{1:i-1} \in \mathcal{Z}^{i-1}\}$ where $\diff{Q}_i(\cdot \mid x_i, z_{1:i-1}) = q_i(\cdot \mid x_i, z_{1:i-1}) \diff{\nu_i}$ and $\exp(-\varepsilon) \leq q_i(\cdot \mid x_i, z_{1:i-1}) \leq \exp(\varepsilon)$.
                
                We write $m_{i, 0}(\cdot \mid z_{1:i-1}) = d^{-1}\sum_{a = 1}^d q_i(\cdot \mid a, z_{1:i-1})$ for the private marginal arising from privatising data distributed according to the uniform distribution on $[d]$ conditional on $\{Z_{1:i-1} = z_{1:i-1}\}$. For $i \in [n]$ and fixed $z_{1:i-1} \in \mathcal{Z}^{i-1}$, we then define the positive semi-definite matrix $\Omega_i(z_{1:i-1}) \in \mathbb{R}^{d \times d}$ via
                \begin{equation*}
                    (\Omega_i(z_{1:i-1}))_{a, b} = \int_{\mathcal{Z}} m_{i, 0}(z_i \mid z_{1:i-1}) \bigg( \frac{q_i(z_i \mid a, z_{1:i-1}) }{m_{i, 0}(z_i \mid z_{1:i-1})} - 1 \bigg)\bigg( \frac{q_i(z_i \mid b, z_{1:i-1}) }{m_{i, 0}(z_i \mid z_{1:i-1})} - 1 \bigg) \diff{\nu}_i(z_i).
                \end{equation*}
                We denote $\Omega = \sum_{i = 1}^n \mathbb{E}_{m_{1, 0}, \hdots, m_{i-1, 0}}[\Omega_i(Z_{1:i-1})]$. We have for all $i, j \in [d]$ that
                \begin{align}
                    |\Omega_{i,j}|
                    &\leq \sum_{i = 1}^n\mathbb{E}\bigg[ \int_{\mathcal{Z}} m_{i, 0}(z_i \mid Z_{1:i-1}) \biggl| \frac{q_i(z_i \mid a, Z_{1:i-1}) }{m_{i, 0}(z_i \mid Z_{1:i-1})} - 1 \biggr| \biggl| \frac{q_i(z_i \mid b, Z_{1:i-1}) }{m_{i, 0}(z_i \mid Z_{1:i-1})} - 1 \biggr| \diff{\nu}_i(z_i) \bigg] \nonumber \\
                    &\leq n\{\exp(2\varepsilon) - 1\}^2 \leq 50n\varepsilon^2, \label{app:eq:DiscMatTraceBound}
                \end{align}
                where the first inequality follows from the LDP condition, and the second by the fact $\varepsilon^2 \in (0,1]$.

                Denote $\vecbf{p}_0 = (1/d, \hdots, 1/d)^T$ the probability vector for the uniform distribution on $[d]$. For $\eta \in \{-1, 1\}^{d/2}$, denote the perturbed vector
                \begin{equation} \label{app:eq:NIDiscAlternatives}
                    \vecbf{p}_\eta
                    = (1/d + \delta \eta_1, 1/d - \delta \eta_1, \hdots, 1/d + \delta \eta_{d/2}, 1/d - \delta \eta_{d/2})
                    = \vecbf{p}_0 + \delta \vecbfm{\zeta},
                \end{equation}
                where we may assume $d$ is even which only changes the final rates by at most an absolute constant.

                We now consider testing between $\vecbf{p}_0$ and the mixture $\mathbb{E}_\eta[p_\eta]$ where $\eta = (\eta_1, \hdots, \eta_{d/2})^T$ consists of i.i.d.~Rademacher random variables. We can lower bound the testing risk as
                \begin{align*}
                    &\mathcal{R}_{n_1, P_0, \varepsilon, D', \rho'}' \\
                    &= \inf_{Q \in \mathcal{Q}_\varepsilon}\inf_{\phi \in \Phi_{n_1, P_0, Q}(\alpha)} \sup_{P_X \in \mathcal{P}_1(D', \rho')} \big\{\mathbb{P}_{P_0, Q}(\phi = 1) + \mathbb{P}_{P_X, Q}(\phi = 0) \big\} \\
                    &\geq \inf_{Q \in \mathcal{Q}_\varepsilon}\inf_{\phi \in \Phi_{n_1, P_0, Q}(\alpha)} \big\{\mathbb{P}_{P_0, Q}(\phi = 1) + \sup_{\eta \in \{-1, 1\}^{d/2}} \mathbb{P}_{P_\eta, Q}(\phi = 0) \big\} \\
                    &\geq \inf_{Q \in \mathcal{Q}_\varepsilon}\inf_{\phi \in \Phi_{n_1, P_0, Q}(\alpha)} \big\{\mathbb{P}_{P_0, Q}(\phi = 1) + \mathbb{E}_\eta[\mathbb{P}_{P_\eta, Q}(\phi = 0)] \big\} \\
                    &\geq \inf_{Q \in \mathcal{Q}_\varepsilon}\inf_{\phi \in \Phi_{n_1, P_0, Q}(\alpha)} \big\{\mathbb{P}_{P_0, Q}(\phi = 1) + \mathbb{E}_\eta[\mathbb{P}_{P_\eta, Q}(\phi = 0)]\big\} \\
                    &\geq \inf_{Q \in \mathcal{Q}_\varepsilon} \big\{ 1 - \|Q(P_0^n) - \mathbb{E}_\eta[Q(P_\eta^n)]\|_{\mathrm{TV}}\big\}.
                \end{align*}
                Hence, it remains to prove that $\|Q(P_0^n) - \mathbb{E}_\eta[Q(P_\eta^n)]\|_{\mathrm{TV}}\leq 1/2$. By Pinsker's inequality \citep[e.g.][Lemma~2.5]{Tsybakov:2009:Book} and \citet[Theorem~5]{Berrett:2020:Faster}, it holds that
                \begin{equation} \label{app:eq:IntDiscTVBound}
                    \begin{aligned}
                        \|Q(P_0^n) - \mathbb{E}_\eta[Q(P_\eta^n)]\|_{\mathrm{TV}}
                        &\leq \{\mathrm{KL}(Q(P_0^2), \mathbb{E}_\eta[Q(P_\eta^n)])/2\}^{1/2} \\
                        &\leq \{\mathbb{E}_\eta[(\vecbf{p}_\eta - \vecbf{p}_0)^T\Omega(\vecbf{p}_\eta - \vecbf{p}_0)]/2\}^{1/2}.                        
                    \end{aligned}
                \end{equation}
                In particular, we have by the form $\eqref{app:eq:NIDiscAlternatives}$ that
                \begin{align*}
                    \mathbb{E}_\eta[(\vecbf{p}_\eta - \vecbf{p}_0)^T\Omega(\vecbf{p}_\eta - \vecbf{p}_0)]
                    &= \sum_{i,j = 1}^d \mathbb{E}_\eta[\zeta_i\zeta_j \Omega_{i, j} ] \\
                    &= \sum_{i = 1}^{d/2} \big\{ \zeta_i^2 (\Omega_{2i - 1, 2i - 1} + \Omega_{2i, 2i}) - \zeta_i^2 (\Omega_{2i - 1, 2i} + \Omega_{2i, 2i - 1}) \big\} \\
                    &\leq 100\delta^2dn\varepsilon^2,
                \end{align*}
                where we use the fact that $|\Omega_{i,j}| \leq 50n\varepsilon^2$ for all $i, j \in [d]$. Setting $\delta = \{1/(200dn\varepsilon^2)\}^{1/2} \wedge 200^{-1/2}/d$ gives $\mathbb{E}_\eta[(\vecbf{p}_\eta - \vecbf{p}_0)^T\Omega(\vecbf{p}_\eta - \vecbf{p}_0)] \leq 1/2$, where taking the minimum with $200^{-1/2}/d$ ensures \eqref{app:eq:NIDiscAlternatives} defines a valid probability distribution. Hence, combining with \eqref{app:eq:IntDiscTVBound}, we obtain $\|Q(P_0^n) - \mathbb{E}_\eta[Q(P_\eta^n)]\|_{\mathrm{TV}}\leq 1/2$. We then calculate
                \begin{equation*}
                    \|\vecbf{p}_\eta - \vecbf{p}_0\|_p
                    = \bigg( \sum_{i = 1}^d \delta^p\bigg)^{1/p}
                    = \delta d^{1/p}
                    = \frac{d^{1/p - 1/2}}{(200n_1\varepsilon^2)^{1/2}} \wedge \frac{d^{1/p - 1}}{200^{1/2}},
                \end{equation*}
                and so the separation radius satisfies $\rho_{\mathrm{I}}^\ast(n, \varepsilon, L_p , \alpha, \beta) \gtrsim d^{1/p - 1/2}/(n_1\varepsilon^2)^{1/2} \wedge d^{1/p - 1}$ for $p \in [1, 2]$.

                Finally, to relax the $d^{1/p - 1}$ term in the case of $L_2$-separation, we note the construction in \citet[Supplement, p.~11]{Berrett:2020:Faster}, which immediately yields $\rho_{\mathrm{I}}^\ast(n, \varepsilon, L_2 , \alpha, \beta) \gtrsim 1/(n_1\varepsilon^2)^{1/2}$, completing the proof.
            \end{proof}
            
        \subsection{Proof of \texorpdfstring{\Cref{sec5:thm:main}}{Theorem 6} Lower Bound} \label{app:sec:contlbproof}
            In this section we prove lower bounds on the minimax separation radii for the two-sample testing problem for continuous distributions satisfying the Sobolev regularity assumption \eqref{sec5:eq:sobolev}. Our technique is inspired by that of existing literature on testing continuous densities under privacy such as \cite{Dubois:2019:GOFtest} and \cite{Lam-Weil:2022:GOFtest}, using H\"{o}lder and Besov regularity assumptions respectively. However, we require a different construction and analysis compared to this prior literature due to our unique choice of a Sobolev regularity assumption. For example, one property the prior literature use is the disjoint support of the considered basis functions, which does not hold for the Fourier basis we consider.

            This is further complicated by the fact that the Sobolev regularity assumption with the Fourier basis does not lend itself readily to obtaining results in terms of $L_p$-separation for $p \neq 2$. In comparison to the H\"{o}lder regularity assumption, the disjoint support of the basis functions greatly simplifies attempts to calculate $L_1$-norms. Instead, we leverage arguments of empirical process theory, guided by \cite{Berrett:2021:IndepPerm}, to help us obtain results for general $p \in [1,2]$. 
            
            \begin{proof}[Proof of \Cref{sec5:thm:main} (Lower Bound)]
            
                To enable us to consider the classic reduction strategy of testing between the uniform distribution and the distribution obtained by taking a mixture over possible alternatives, we start by constructing a suitable collection of alternatives.
            
                Given sample size $n \in \mathbb{N}$, fix an interactive $\varepsilon$-LDP privacy mechanism $(Q_1, \hdots, Q_n) \in \mathcal{Q}_\varepsilon$ where, for $i \in [n]$, $Q_i:\sigma(\mathcal{Z}) \times [0, 1]^d \times \mathcal{Z}^{i-1} \rightarrow [0,1]$. Recall that this includes non-interactive mechanisms as a special case. Denoting, for each $i \in [n]$, a measure $\nu_i$ on $\mathcal{Z}$, by \citet[Lemma~B.3]{Butucea:2023:Interactive} we can write the densities $\{q_i(\cdot \mid \vecbf{x}_i, z_{1:i-1}) : i \in [n], \vecbf{x}_i \in [0,1]^d, z_{1:i-1} \in \mathcal{Z}^{i-1}\}$ where $\diff{Q}_i(\cdot \mid \vecbf{x}_i, z_{1:i-1}) = q_i(\cdot \mid \vecbf{x}_i, z_{1:i-1}) \diff{\nu_i}$ and $\exp(-\varepsilon) \leq q_i(\cdot \mid \vecbf{x}_i, z_{1:i-1}) \leq \exp(\varepsilon)$. Letting $f_0 \equiv 1$ the uniform density on $[0,1]^d$, we also write $m_{i, 0}(\cdot \mid z_{1:i-1}) = \int_{[0,1]^d} q_i(\cdot \mid \vecbf{x}, z_{1:i-1}) f_0(\vecbf{x}) \diff{\vecbf{x}}$ for the private marginal arising from privatising data distributed according to $f_0$ conditional on $\{Z_{1:i-1} = z_{1:i-1}\}$. Denote the corresponding distributions $M_{i, 0}$.
    
                For $i \in [n]$ and fixed $z_{1:i-1} \in \mathcal{Z}^{i-1}$, we define the positive semi-definite linear operator $\Omega_i(z_{1:i-1}): L_2([0,1]^d) \rightarrow L_2([0,1]^d)$ via, for all $f \in L_2([0,1]^d)$,
                \begin{equation*}
                    \Omega_i(z_{1:i-1}) f = \int_{[0,1]^d} \int_{\mathcal{Z}} \frac{q_i(z_i \mid \vecbf{y}, z_{1:i-1})q_i(z_i \mid \cdot, z_{1:i-1})}{m_{i, 0}(z_i \mid z_{1:i-1})}f(\vecbf{y}) \diff{\nu_i}(z_i) \diff{\vecbf{y}},
                \end{equation*}
                where the order of integration is interchangeable via Fubini's theorem. We then denote $\Omega = \sum_{i = 1}^n \mathbb{E}_{M_{1, 0}, \hdots, M_{i - 1, 0}}[\Omega_i(Z_{1:i-1})]$. Denote $\Phi_R = \mathrm{span}(\{\varphi_\vecbf{l} : \vecbf{l} \in \mathbb{N}_0^d(R)\})$ a space spanned by a subset of Fourier basis functions of $L_2([0,1]^d)$ where we recall $\mathbb{N}_0^d(R)$ as defined in \eqref{sec5:eq:VectorBall} and $R$ to be specified later. Letting $\Omega$ act on the $V = |\mathbb{N}_0^d(R)|$ dimensional linear subspace $\Phi_R$ of $L_2([0,1]^d)$, denote the eigenfunctions and corresponding eigenvalues by $\{\psi_v : v \in [V]\}$ and $\{\lambda_v^2 : v \in [V]\}$ respectively, where we note the eigenvalues are non-negative as $\Omega$ is positive semi-definite. Note that in particular each $\psi_v$ can be written as a linear combination of the $\varphi_{\vecbf{l}}$ and so are orthogonal to $1$. Observe also that we have the equality $V = |\mathbb{N}_0^d(R)|$ as the kernel of the operator is trivial as can be shown by the LDP condition via the result of \citet[Lemma~B.3]{Butucea:2023:Interactive}.
    
                For $\eta \in \{-1, 1\}^V$ and $\tilde{\lambda}_v = \max\{\lambda_v/(n^{1/2}z_\varepsilon), V^{-1/2}\}$, where $z_\varepsilon = \exp(2\varepsilon) - \exp(-2\varepsilon)$, denote the perturbed function
                \begin{equation} \label{app:eq:ContAlternatives}
                    f_\eta = f_0 + \delta\sum_{v = 1}^V \frac{\eta_v}{\tilde{\lambda}_v} \psi_v.
                \end{equation}
                As $\{\psi_v : v \in V\} \subset \Phi_R$ forms an orthonormal basis of $\Phi_R$, we can further expand
                \begin{equation*}
                    f_\eta = f_0 + \delta\sum_{v = 1}^V \sum_{\vecbf{l} \in \mathbb{N}_0^d(R)} \frac{\eta_v \theta_{v, \vecbf{l}}}{\tilde{\lambda}_v} \varphi_\vecbf{l},
                \end{equation*}
                where we note that the matrix $(\theta_{v, \vecbf{l}})_{v, \vecbf{l}}$ is orthogonal.

                We now have the following lemma, valid for both the non-interactive and interactive cases, which states that the constructed functions \eqref{app:eq:ContAlternatives} are valid density functions satisfying the Sobolev regularity assumption, with an additional result allowing us to control the $L_1$-norm of deviations from the uniform distribution. The proof is deferred to the end of this subsection.
    
                \begin{lemma} \label{app:lem:ConstructedFunctions}
                    Let $\eta$ be uniform on the set $\{-1, 1\}^V$. Take $\delta$ such that
                    \begin{equation} \label{app:eq:ContFunctionDeltaReq}
                        \delta \leq \bigg(\frac{r^2}{\widetilde{C}_{s, d}VR^{2s + d}\log(4V/a)} \bigg)^{1/2},
                    \end{equation}
                    for $\widetilde{C}_{s,d} > 0$ some constant depending on $s$ and $d$ and $a \in (0, 1)$.
                    
                    There exists a subset $E \subseteq \{-1, 1\}^V$ satisfying $\mathbb{P}(\eta \in E) \geq 1 - a$, such that for all $\eta \in E$, we have
                    \begin{enumerate}
                        \item $f_\eta$ satisfies the Sobolev regularity condition \eqref{sec5:eq:sobolev}, yielding $f_\eta \in \mathcal{S}_{s, r}$;
                        \item $\|f_\eta - f_0\|_\infty \leq \widetilde{C}_{s,d}\|f_\eta - f_0\|_2$, and in particular $f_\eta \geq 0$, yielding $f_\eta \in \mathcal{F}_{s, r}$; and
                        \item $\|f_\eta - f_0\|_1 \geq \widetilde{C}_{s,d}^{-1}\delta V/\{\log(R)\}^{1/2}$.
                    \end{enumerate}
                \end{lemma}
    
                With the constructed family and this lemma in hand, we consider testing between $f_0$ and the mixture $E_\eta[f_\eta]$ where $\eta = (\eta_1, \hdots, \eta_V)^T$ consists of i.i.d.~Rademacher random variables. Recall the goodness-of-fit testing risk, as defined in \eqref{app:eq:GOFTestingRisk}. We specify the relevant quantities therein, yielding:
                \begin{itemize}
                    \item the distribution under the null $P_0$ to be that induced by the uniform density $f_0$; 
                    \item the metric $D'$ on the space $L_2([0,1]^d)$ to be the $L_1$-norm so that $D'(f, g) = \|f -g \|_1$; and
                    \item the separation $\rho'$ is set to be either $\rho_{\mathrm{NI}} = c_{r, s, d}(n\varepsilon^2)^{-s/(2s+3d/2)} \{\log(n\varepsilon^2)\}^{-1}$ or $\rho_{\mathrm{I}} = c_{r, s, d}'(n\varepsilon^2)^{-s/(2s+d)} \{\log(n\varepsilon^2)\}^{-1}$ in the non-interactive and interactive settings respectively, with $c_{r, s, d}$ and $c_{r, s, d}'$ sufficiently small constants depending on $r, s$ and $d$.
                \end{itemize}

                For these choices, we show that $\mathcal{R}_{n_1, P_0, \varepsilon, D', \rho_{\mathrm{NI}}}', \mathcal{R}_{n_1, P_0, \varepsilon, D', \rho_{\mathrm{I}}}' \geq 1/2$ for the testing risks over non-interactive and interactive privacy mechanisms respectively, from which it will follow that the minimax separation radii satisfy $\rho_{\mathrm{NI}}^\ast(n, \varepsilon, L_1 , \alpha, \beta) \geq c_{r, s, d}(n\varepsilon^2)^{-s/(2s+3d/2)} \{\log(n\varepsilon^2)\}^{-1}$ and $\rho_{\mathrm{I}}^\ast(n, \varepsilon, L_1 , \alpha, \beta) \geq c_{r, s, d}'(n\varepsilon^2)^{-s/(2s+d)} \{\log(n\varepsilon^2)\}^{-1}$.

                Defining $E \subseteq \{\eta : f_\eta \in \mathcal{P}_1(D', \rho')\}$ and letting $a \leq \mathbb{P}(E^c)$, we can lower bound the testing risk as
                \begin{align*}
                    &\mathcal{R}_{n_1, P_0, \varepsilon, D', \rho'}' \\
                    &= \inf_{Q \in \mathcal{Q}_\varepsilon}\inf_{\phi \in \Phi_{n_1, P_0, Q}(\alpha)} \sup_{P_X \in \mathcal{P}_1(D', \rho')} \big\{\mathbb{P}_{P_0, Q}(\phi = 1) + \mathbb{P}_{P_X, Q}(\phi = 0) \big\} \\
                    &\geq \inf_{Q \in \mathcal{Q}_\varepsilon}\inf_{\phi \in \Phi_{n_1, P_0, Q}(\alpha)} \big\{\mathbb{P}_{P_0, Q}(\phi = 1) + \sup_{\eta \in E} \mathbb{P}_{P_\eta, Q}(\phi = 0) \big\} \\
                    &\geq \inf_{Q \in \mathcal{Q}_\varepsilon}\inf_{\phi \in \Phi_{n_1, P_0, Q}(\alpha)} \big\{\mathbb{P}_{P_0, Q}(\phi = 1) + \mathbb{E}_\eta[\mathbbm{1}\{\eta \in E\}\mathbb{P}_{P_\eta, Q}(\phi = 0)] \big\} \\
                    &\geq \inf_{Q \in \mathcal{Q}_\varepsilon}\inf_{\phi \in \Phi_{n_1, P_0, Q}(\alpha)} \big\{\mathbb{P}_{P_0, Q}(\phi = 1) + \mathbb{E}_\eta[\mathbb{P}_{P_\eta, Q}(\phi = 0)] - \mathbb{E}_\eta[\mathbbm{1}\{\eta \notin E\}\mathbb{P}_{P_\eta, Q}(\phi = 0)]\big\} \\
                    &\geq \inf_{Q \in \mathcal{Q}_\varepsilon}\inf_{\phi \in \Phi_{n_1, P_0, Q}(\alpha)} \big\{\mathbb{P}_{P_0, Q}(\phi = 1) + \mathbb{E}_\eta[\mathbb{P}_{P_\eta, Q}(\phi = 0)]\big\} - a \\
                    &\geq \inf_{Q \in \mathcal{Q}_\varepsilon} \big\{ 1 - \|Q(P_0^n) - \mathbb{E}_\eta[Q(P_\eta^n)]\|_{\mathrm{TV}}\big\} - a.
                \end{align*}
                Hence, it remains to prove that $\|Q(P_0^n) - \mathbb{E}_\eta[Q(P_\eta^n)]\|_{\mathrm{TV}}\leq 1/2 - a$. We bound this total variation term in two different ways in the case of non-interactive and interactive privacy mechanisms.

                \medskip
                \noindent
                \textbf{Non-interactive Mechanisms}:
                For the non-interactive setting, we use the bound $\|Q(P_0^n) - \mathbb{E}_\eta[Q(P_\eta^n)]\|_{\mathrm{TV}}\leq \{\chi^2(\mathbb{E}_\eta[Q(P_\eta^n)], Q(P_0^n))\}^{1/2}/2$ (see, for example, Lemma~2.5 and Lemma~2.7 in \citealt{Tsybakov:2009:Book}). We then proceed to bound this chi-squared divergence.
                
                We observe some simplifications which occur in the non-interactive setting. In particular, the \beame{privatisation}{privatization}{privatization} of each observation occurs independently of previous observations. Hence, we write $m_{i, 0}(\cdot) = \int_{[0, 1]^d} q_i(\cdot \mid \vecbf{x})f_0(\vecbf{x}) \diff{\vecbf{x}}$ and $m_{i, \eta}(\cdot) = \int_{[0, 1]^d} q_i(\cdot \mid \vecbf{x})f_0(\vecbf{x}) \diff{\vecbf{x}}$ for the private marginals when the raw data are distributed according to $P_0$ and $P_\eta$ respectively. Further, we have that the operator $\Omega_i$ takes the form
                \begin{equation} \label{app:eq:NIOmegaOp}
                    \Omega_i f
                    = \int_{[0,1]^d} \int_{\mathcal{Z}_i} \frac{q_i(z \mid \vecbf{y})q_i(z \mid \cdot)}{m_{i, 0}(z)}f(\vecbf{y}) \diff{\nu_i}(z_i) \diff{\vecbf{y}},
                \end{equation}
                and we denote $\Omega = \sum_{i = 1}^n \Omega_i$.
                
                With these simplifications in hand, we denote $\eta'$ an i.i.d.~copy of $\eta$, and obtain
                \begin{align*}
                    &\chi^2(E_\eta[Q(P_\eta^n)], Q(P_0^n)) + 1 \\
                    &= E_{Q(P_0^n)}\bigg[\bigg(\frac{E_\eta[\prod_{i = 1}^n m_{\eta}(Z_i))]}{\prod_{i = 1}^n m_{0}(Z_i)}\bigg)^2 \bigg] \\
                    &= E_{\eta, \eta', Q(P_0)}\bigg[\prod_{i = 1}^n \bigg\{ \bigg( \sum_{v = 1}^V \frac{\delta\eta_v \int_{[0, 1]^d} q_i(Z_i \mid \vecbf{x}) \psi_v(\vecbf{x}) \diff{\vecbf{x}}}{\tilde{\lambda}_v m_0(Z_i)} + 1 \bigg) \\
                    &\hspace{7cm}\times\bigg(\sum_{v' = 1}^V \frac{\delta\eta_{v'}' \int_{[0, 1]^d} q_i(Z_i \mid \vecbf{x}) \psi_v(\vecbf{x}) \diff{\vecbf{x}}}{\tilde{\lambda}_{v'} m_0(Z_i)} + 1\bigg) \bigg\}\bigg] \\
                    &= E_{\eta, \eta'}\bigg[ \prod_{i = 1}^n \bigg(1
                    + \sum_{v = 1}^V \frac{\delta (\eta_v + \eta_v')}{\tilde{\lambda}_v} \int_{[0, 1]^d} \int_{\mathcal{Z}_i} q_i(Z_i \mid \vecbf{x}) \psi_v(\vecbf{x}) \diff{\vecbf{x}} \diff{\nu}_i(z_i) \\
                    &\hspace{2cm}+ \sum_{v, v' \in [V]} \frac{\delta^2\eta_v \eta_{v'}'}{\tilde{\lambda}_v\tilde{\lambda}_{v'}} \int_{\mathcal{Z}_i}\int_{[0, 1]^d}\int_{[0, 1]^d} \frac{q_i(Z_i \mid \vecbf{x})q_i(Z_i \mid \vecbf{y}) \psi_v(\vecbf{x})\psi_{v'}(\vecbf{y})}{m_0(z_i)} \diff{\vecbf{x}}\diff{\vecbf{y}} \diff{\nu_i}(z_i)\bigg)\bigg] \\
                    &= E_{\eta, \eta'}\bigg[\prod_{i = 1}^n \bigg( 1 + \sum_{v, v' \in [V]} \!\!\!\frac{\delta^2\eta_v \eta_{v'}'}{\tilde{\lambda}_v\tilde{\lambda}_{v'}} \int_{\mathcal{Z}_i}\int_{[0, 1]^d}\int_{[0, 1]^d} \frac{q_i(Z_i \mid \vecbf{x})q_i(Z_i \mid \vecbf{y}) \psi_v(\vecbf{x})\psi_{v'}(\vecbf{y})}{m_0(z_i)} \diff{\vecbf{x}}\diff{\vecbf{y}} \diff{\nu_i}(z_i)\bigg)\bigg] \\
                    &\leq E_{\eta, \eta'}\bigg[\prod_{i = 1}^n \exp\bigg(\sum_{v, v' \in [V]} \!\!\!\frac{\delta^2\eta_v \eta_{v'}'}{\tilde{\lambda}_v\tilde{\lambda}_{v'}} \int_{\mathcal{Z}_i}\int_{[0, 1]^d}\int_{[0, 1]^d} \frac{q_i(Z_i \mid \vecbf{x})q_i(Z_i \mid \vecbf{y}) \psi_v(\vecbf{x})\psi_{v'}(\vecbf{y})}{m_0(z_i)} \diff{\vecbf{x}}\diff{\vecbf{y}} \diff{\nu_i}(z_i)\bigg)\bigg] \\                    
                    &= E_{\eta, \eta'}\bigg[ \exp\bigg(\sum_{v, v' \in [V]} \frac{\delta^2\eta_v \eta_{v'}'}{\tilde{\lambda}_v\tilde{\lambda_{v'}}} \int_{[0, 1]^d} \psi_{v}(\vecbf{x})(\Omega \psi_{v'})(\vecbf{x}) \diff{\vecbf{x}} \bigg) \bigg] \\
                    &= E_{\eta, \eta'}\bigg[\exp\bigg(\sum_{v = 1}^V \frac{\delta^2\eta_v \eta_{v}'\lambda_v^2}{\tilde{\lambda}_v^2}  \bigg)\bigg] \\
                    &= E_{\eta, \eta'}\bigg[ \prod_{v = 1}^V \exp\big(nz_\varepsilon^2\delta^2\eta_v \eta_v'\big)\bigg] \\
                    &\leq \exp\big(Vn^2z_\varepsilon^4\delta^4/2 \big),
                \end{align*}
                where the fourth equality uses the fact that the functions $\psi_v$ are orthogonal to $1$; the fifth equality by the form of $\Omega = \sum_{i = 1}^n \Omega_i$ and \eqref{app:eq:NIOmegaOp}; the sixth by recalling that $\psi_v$ are the eigenfunctions of $\Omega$; and the second inequality the fact that $\lambda_v/\tilde{\lambda}_v \leq n^{1/2}z_\varepsilon$.
    
                Hence, for $\delta$ satisfying
                \begin{equation} \label{app:eq:NIContDeltaReq}
                    \delta \leq \bigg( \frac{2\log\{1 + (1 - 2a)^2\}}{Vn^2z_\varepsilon^4} \bigg)^{1/4},
                \end{equation}
                we have that $\mathcal{R}_{n_1, P_0, \varepsilon, D', \rho'}' \geq 1/2$.

                We are now ready to conclude the proof. In particular, we will specify the value of $R$, and hence $V$, from which we will obtain the separation radius given in the theorem statement. We first observe that
                \begin{align}
                    V
                    = |\mathbb{N}_0^d(R)|
                    = \sum_{\vecbf{l} \in \mathbb{N}_0^d(R)} 1
                    &\leq \int_{\vecbf{x} \in \mathbb{R}^d:\| \vecbf{x} \|_2 \leq R + d^{1/2}} \diff{\vecbf{x}} \nonumber \\
                    &= \frac{(R + d^{1/2})^{d} \mathrm{Vol}(\{\vecbf{x}: \|\vecbf{x}\|_2 \leq 1\})}{d} \nonumber \\
                    &\leq C_{s,d}^{(V)} R^d, \label{app:eq:Vupperbound}
                \end{align}
                for $C_{s,d}^{(V)} > 0$ a constant depending on $s$ and $d$, where the first inequality is by noting that for all $\vecbf{l} \in \mathbb{N}_0^d(R)$ the shifted hypercubes $\vecbf{l} + [0, 1]^d$ are contained in the positive orthant of the $\ell_2$-ball of radius $R + d^{1/2}$ in $\mathbb{R}^d$, and the final equality by evaluating the integral via a polar co-ordinate transformation. We can similarly obtain the lower bound
                \begin{align} \label{app:eq:Vlowerbound}
                    V
                    \geq \int_{\vecbf{x} \in \mathbb{R}^d:\| \vecbf{x} \|_2 \leq R} \diff{\vecbf{x}}
                    = \frac{R^{d} \mathrm{Vol}(\{\vecbf{x}: \|\vecbf{x}\|_2 \leq 1\})}{d}
                    \geq c_{s,d}^{(V)} R^d,
                \end{align}
                for $c_{s,d}^{(V)} > 0$ a constant depending on $s$ and $d$.
                
                We now set
                \begin{equation*}
                    R = \bigg( \frac{n z_\varepsilon^2 r^2}{\widetilde{C}_{s, d} [2C_{s, d}^{(V)} \log\{1 + (1 - 2a)^2\}]^{1/2}} \bigg)^{1/(2s + 3d/2)},
                \end{equation*}
                and take
                \begin{align*}
                    \delta = &\bigg[\log\bigg\{\frac{4C_{s, d}^{(V)}}{a}\bigg(\frac{n z_\varepsilon^2 r^2}{\widetilde{C}_{s, d} [2C_{s, d}^{(V)} \log\{1 + (1 - 2a)^2\}]^{1/2}}\bigg)^{d/(2s + 3d/2)}\bigg\}\bigg]^{-1/2} \\
                    &\hspace{7cm} \times \Bigg(\frac{(\widetilde{C}_{s, d})^{d/4}}{(C_{s, d}^{(V)})^{(2s + d)/4}r^{d/2}(n z_\varepsilon^2)^{s+d}}\Bigg)^{1/(2s+3d/2)}.
                \end{align*}
                Taken together, these values of $R$ and $\delta$ satisfy the conditions \eqref{app:eq:ContFunctionDeltaReq} and \eqref{app:eq:NIContDeltaReq}. Further, by \Cref{app:lem:ConstructedFunctions} and the bound \eqref{app:eq:Vlowerbound}, we have that the $L_1$-separation is lower bounded as
                \begin{align*}
                    \|f_\eta - f_0\|_1
                    \geq \widetilde{C}_{s, d}^{-1} c_{s, d}^{(V)} \delta R^{d}/\{\log(R)\}^{1/2} 
                    \geq c_{r, s, d}(n\varepsilon^2)^{-s/(2s+3d/2)} \{\log(n\varepsilon^2)\}^{-1}
                \end{align*}
                for $c_{r, s, d} > 0$ some constant depending on $r, s$ and $d$, where the second follows from noting $z_\varepsilon^2 \asymp \varepsilon^2$ as $\varepsilon \in (0, 1]$ and that $n\varepsilon^2 \gtrsim 1$ by assumption. Thus, we see that we have the separation radius of $\rho_{\mathrm{NI}}^\ast(n, \varepsilon, L_1 , \alpha, \beta) \geq c_{r, s, d}(n\varepsilon^2)^{-s/(2s+3d/2)} \{\log(n\varepsilon^2)\}^{-1}$. Finally, we note by H\"{o}lder's inequality that $\|f\|_1 \leq \mu([0, 1]^d)^{1 - 1/p} \|f\|_p = \|f\|_p$ where $\mu$ denotes the Lebesgue measure on $\mathbb{R}^d$. Hence, we have that the prior separation radius is in fact valid for all $p \in [1,2]$, giving $\rho_{\mathrm{NI}}^\ast(n, \varepsilon, L_p , \alpha, \beta) \geq c_{r, s, d}(n\varepsilon^2)^{-s/(2s+3d/2)} \{\log(n\varepsilon^2)\}^{-1}$, which completes the proof in the non-interactive case.

                \medskip
                \noindent
                \textbf{Interactive Mechanisms}:
    
                For the interactive setting. We start by using Pinsker's inequality \citep[e.g.][Lemma~2.5]{Tsybakov:2009:Book} to bound the total-variation distance via the inequality $\|Q(P_0^n) - \mathbb{E}_\eta[Q(P_\eta^n)]\|_{\mathrm{TV}} \leq \{\mathrm{KL}(Q(P_0^n), \mathbb{E}_\eta[Q(P_\eta^n)])/2\}^{1/2}$. For $i \in [n]$, we write $M_{0}(\cdot \mid Z_{1:i-1})$ and $M_{\eta}(\cdot \mid Z_{1:i-1})$ for the private marginal arising from privatising data distributed according to $f_0$ and $f_\eta$ conditional on $\{Z_{1:i-1} = z_{1:i-1}\}$ respectively. We have
                \begin{align}
                    \mathrm{KL}(Q(P_0^n), \mathbb{E}[Q(P_\eta^n)])
                    &\leq \mathbb{E}_{\eta}[\mathrm{KL}(Q(P_0^n), Q(P_\eta^n))] \nonumber \\
                    &= \sum_{i = 1}^n \mathbb{E}_{\eta, M_{1,\eta}, \hdots, M_{i-1,\eta}}[\mathrm{KL}(M_{0}(\cdot \mid Z_{1:i-1}), M_{\eta}(\cdot \mid Z_{1:i-1}))] \label{app:eq:KLContIntDecomp}
                \end{align}
                where the inequality is by the convexity of Kullback--Leibler divergence, which follows from the joint convexity of $f$-divergences \citep[e.g.][Proposition~6.1]{Goldfield:2020:fDivergences}, and the equality by the chain rule \citep[e.g.][Chapter~3]{Gray:2011:InformationTheory}.
    
                Focusing on a single Kullback--Leibler divergence term in the expectation, we have
                \begin{align}
                    &\mathrm{KL}(M_{0}(\cdot \mid Z_{1:i-1}), M_{\eta}(\cdot \mid Z_{1:i-1})) \nonumber \\
                    &\leq \int_{\mathcal{Z}} \frac{\{m_{i, \eta}(z_i \mid Z_{1:i-1}) - m_{i, 0}(z_i \mid Z_{1:i-1})\}^2}{m_{i, \eta}(z_i \mid Z_{1:i-1})} \diff{\nu}_i(z_i) \nonumber \\
                    &\leq \exp(2\varepsilon)\int_{\mathcal{Z}} \frac{\{m_{i, \eta}(z_i \mid Z_{1:i-1}) - m_{i, 0}(z_i \mid Z_{1:i-1})\}^2}{m_{i, 0}(z_i \mid Z_{1:i-1})} \diff{\nu}_i(z_i) \nonumber \\
                    &= \exp(2\varepsilon)\int_{\mathcal{Z}} \frac{\{\delta\sum_{v = 1}^V \eta_v\tilde{\lambda}_v^{-1} \int_{[0,1]^d} q_i(z_i \mid \vecbf{x}, Z_{1:i-1})\psi_v(\vecbf{x})\}^2}{m_{i, 0}(z_i \mid Z_{1:i-1})} \diff{\nu}_i(z_i) \nonumber \\
                    &= \exp(2\varepsilon)\!\!\!\sum_{v,v' \in [V]} \!\!\!\frac{\delta^2 \eta_v \eta_{v'} }{\tilde{\lambda}_v \tilde{\lambda}_{v'}} \int_{\mathcal{Z}} \int_{[0,1]^d}\int_{[0,1]^d} \!\!\!\frac{q_i(z_i \mid \vecbf{x}, Z_{1:i-1})q_i(z_i \mid \vecbf{y}, Z_{1:i-1})\psi_v(\vecbf{x})\psi_{v'}(\vecbf{y})}{m_{i, 0}(z_i \mid Z_{1:i-1})} \diff{\vecbf{x}}\diff{\vecbf{y}} \diff{\nu}_i(z_i) \nonumber \\
                    &= \exp(2\varepsilon) \!\!\!\sum_{v,v' \in [V]} \!\!\!\frac{\delta^2 \eta_v \eta_{v'} }{\tilde{\lambda}_v \tilde{\lambda}_{v'}}\int_{[0,1]^d} \psi_v(\vecbf{x})(\Omega_i(Z_{1:i-1})\psi_{v'})(\vecbf{x}) \diff{\nu}_i(z_i), \label{app:eq:KLContIntBound}
                \end{align}
                where the first inequality is by bounding the Kullback--Leibler divergence by the $\chi^2$-divergence \citep[e.g.][Lemma~2.7]{Tsybakov:2009:Book}, and the second inequality by the construction of the density $q_i$.
    
                Hence, combining \eqref{app:eq:KLContIntDecomp} and \eqref{app:eq:KLContIntBound}, and recalling $\Omega = \sum_{i = 1}^n \mathbb{E}_{M_{1, 0}, \hdots, M_{i-1, 0}}[\Omega_i(Z_{1:i-1})]$, we obtain
                \begin{align*}
                    \mathrm{KL}&(Q(P_0^n), \mathbb{E}[Q(P_\eta^n)]) \\
                    &= \exp(2\varepsilon) \mathbb{E}_\eta\bigg[\sum_{v,v' \in [V]} \frac{\delta^2 \eta_v \eta_{v'} }{\tilde{\lambda}_v \tilde{\lambda}_{v'}}\sum_{i = 1}^n \mathbb{E}_{m_{1, 0}, \hdots, m_{i-1, 0}}\bigg[\int_{[0,1]^d} \psi_v(\vecbf{x})(\Omega_i(Z_{1:i-1})\psi_{v'})(\vecbf{x}) \diff{\vecbf{x}}\bigg] \bigg]\\
                    &= \exp(2\varepsilon) \mathbb{E}_\eta\bigg[\sum_{v,v' \in [V]} \frac{\delta^2 \eta_v \eta_{v'} }{\tilde{\lambda}_v \tilde{\lambda}_{v'}}\int_{[0,1]^d} \psi_v(\vecbf{x})(\Omega\psi_{v'})(\vecbf{x}) \diff{\vecbf{x}} \bigg] \\
                    &= \exp(2\varepsilon) \sum_{v = 1}^V \frac{\delta^2 \lambda_v^2}{\tilde{\lambda}_v^2} \\
                    &\leq \exp(2\varepsilon) Vnz_\varepsilon^2\delta^2
                \end{align*}
                where the final equality is by the fact that the $\psi_v$ are the eigenfunctions of $\Omega$, and the inequality is by the fact that $\lambda_v/\tilde{\lambda}_v \leq n^{1/2}z_\varepsilon$. Hence, for $\delta$ satisfying
                \begin{equation} \label{app:eq:IContDeltaReq}
                    \delta \leq \bigg( \frac{2(1/2-a)^2}{\exp(2\varepsilon)Vnz_\varepsilon^2} \bigg)^{1/2},
                \end{equation}
                we have that $\mathcal{R}_{n_1, P_0, \varepsilon, D', \rho'}' \geq 1/2$.
                
                We now set
                \begin{equation*}
                    R = \bigg( \frac{n z_\varepsilon^2 r^2 \exp(2\varepsilon)}{2\widetilde{C}_{s, d} (1/2 - a)^2} \bigg)^{1/(2s + d)},
                \end{equation*}
                and take
                \begin{equation*}
                    \delta = \bigg[C_{s, d}^{(V)}\log\bigg\{\frac{4C_{s, d}^{(V)}}{a}\bigg( \frac{n z_\varepsilon^2 r^2 \exp(2\varepsilon)}{2\widetilde{C}_{s, d} (1/2 - a)^2} \bigg)^{d/(2s + d)}\bigg\}\bigg]^{-1/2}\Bigg(\frac{\widetilde{C}_{s, d}^{d/2}\{2(1/2 - a)^2\}^{s+d}}{\{\exp(2\varepsilon)\}^{s+d}(nz_\varepsilon^2)^{s+d}r^{d}}\Bigg)^{1/(2s + d)}.
                \end{equation*}
                Taken together, these values of $R$ and $\delta$ satisfy the conditions \eqref{app:eq:ContFunctionDeltaReq} and \eqref{app:eq:IContDeltaReq}. Further, by \Cref{app:lem:ConstructedFunctions} and the bound \eqref{app:eq:Vlowerbound}, we have that the $L_1$-separation is lower bounded as
                \begin{align*}
                    \|f_\eta - f_0\|_1
                    \geq \widetilde{C}_{s, d}^{-1} c_{s, d}^{(V)} \delta R^{d}/\{\log(R)\}^{1/2} 
                    \geq c_{r, s, d}'(n\varepsilon^2)^{-s/(2s+d)} \{\log(n\varepsilon^2)\}^{-1}
                \end{align*}
                for $c_{r, s, d}' > 0$ some constant depending on $r, s$ and $d$, where the second follows from noting $z_\varepsilon^2 \asymp \varepsilon^2$ as $\varepsilon \in (0, 1]$ and that $n\varepsilon^2 \gtrsim 1$ by assumption. Thus, we see that we have the separation radius of $\rho_{\mathrm{I}}^\ast(n, \varepsilon, L_1 , \alpha, \beta) \geq c_{r, s, d}'(n\varepsilon^2)^{-s/(2s+d)} \{\log(n\varepsilon^2)\}^{-1}$. As in the non-interactive case, we recall that by H\"{o}lder's inequality that $\|f\|_1 \leq \|f\|_p$. Hence, we have that the prior separation radius is valid for all $p \in [1,2]$, giving $\rho_{\mathrm{I}}^\ast(n, \varepsilon, L_p , \alpha, \beta) \geq c_{r, s, d}'(n\varepsilon^2)^{-s/(2s+d)} \{\log(n\varepsilon^2)\}^{-1}$, which completes the proof in the interactive case. Hence, the proof of the theorem is concluded.
            \end{proof}

            We now provide the proof of \Cref{app:lem:ConstructedFunctions}.
            
            \begin{proof}[Proof of \Cref{app:lem:ConstructedFunctions}]
                Throughout, for $\eta \in \{-1, 1\}^V$, we denote $\Psi_{\eta} = \delta \sum_v^V \eta_v \psi_v/\tilde{\lambda}_v$ such that $f_v = f_0 + \Psi_{\eta}$, noting that $\| \Psi_\eta \|_2^2 = \delta^2 \widetilde{\Lambda}$. We denote the set
                \begin{equation*}
                    E_1 = \bigg\{\eta \in \{-1, 1\}^V : \biggl|\sum_{v = 1}^V \frac{\delta\eta_v\theta_{v, \vecbf{l}}}{\tilde{\lambda}_v} \biggr| \leq \bigg\{\frac{8 \| \Psi_\eta \|_2^2}{3V}\log\bigg( \frac{4V}{a} \bigg) \bigg\}^{1/2} \text{ for all } \vecbf{l} \in \mathbb{N}_0^d(R)\bigg\}.
                \end{equation*}
                By Hoeffding's inequality \citep[e.g.~Proposition~2.5][]{Wainwright:2019:HDSBook}, it holds for $x > 0$ that
                \begin{align*}
                    \sum_{\vecbf{l} \in \mathbb{N}_0^d(R)} \mathbb{P}\bigg(\biggl| \sum_{v = 1}^V \frac{\delta \eta_v\theta_{v, \vecbf{l}}}{\tilde{\lambda}_v} \biggr| > x \bigg)
                    &\leq 2\sum_{\vecbf{l} \in \mathbb{N}_0^d(R)} \exp\bigg( -\frac{x^2}{2\delta^2\sum_{v = 1}^V (\theta_{v, \vecbf{l}}/\tilde{\lambda}_v)^2} \bigg) \\
                    &\leq 2\sum_{\vecbf{l} \in \mathbb{N}_0^d(R)} \exp\bigg( -\frac{x^2}{2\delta^2V} \bigg) \\
                    &= 2\sum_{\vecbf{l} \in \mathbb{N}_0^d(R)} \exp\bigg( -\frac{\widetilde{\Lambda}x^2}{2\| \Psi_\eta\|_2^2 V} \bigg) \\
                    &\leq 2\sum_{\vecbf{l} \in \mathbb{N}_0^d(R)} \exp\bigg( -\frac{3Vx^2}{8\| \Psi_\eta\|_2^2} \bigg)
                \end{align*}
                where the second inequality is by the fact that $\tilde{\lambda}_v^2 \geq 1/V$ for all $v \in [V]$, and the orthonormality of $(\theta_{v, \vecbf{l}})_{v, \vecbf{l}}$, and the final inequality by \eqref{app:eq:BigLambdaBound}. Hence, applying the union bound, we obtain that $\mathbb{P}(E_1^c) \leq a/2$.

                We now consider the three items in the lemma statement for $\eta \in E_1$, beginning with item~1. We have that
                \begin{align}
                    \sum_{\vecbf{l} \in \mathbb{N}_0^d(R)} \bigg\{ \| \vecbf{l} \|_2^{2s} \bigg( \sum_{v = 1}^V \frac{\delta\eta_v\theta_{v, \vecbf{l}}}{\tilde{\lambda}_v} \bigg)^2 \bigg\}
                    &\leq \frac{8 \| \Psi_\eta \|_2^2}{3V}\log\bigg( \frac{4V}{a} \bigg) \sum_{\vecbf{l} \in \mathbb{N}_0^d(R)} \| \vecbf{l} \|_2^{2s} \nonumber \\
                    &\leq \frac{2^{-d + 3}\| \Psi_\eta \|_2^2}{3V}\log\bigg( \frac{4V}{a} \bigg) \int_{\vecbf{x} \in \mathbb{R}^d : \|\vecbf{x}\|_2 \leq R + d^{1/2}} \| \vecbf{x} \|_2^{2s} \diff{\vecbf{x}} \nonumber \\
                    &\leq \frac{2^{-d + 3}\| \Psi_\eta \|_2^2 (R + d^{1/2})^{2s + d} \mathrm{Vol}(\{\vecbf{x}: \|\vecbf{x}\|_2 \leq 1\})}{3V(2s+d)}\log\bigg( \frac{4V}{a} \bigg) \nonumber \\
                    &\leq C_{s, d} \delta^2 V R^{2s + d}\log\bigg( \frac{4V}{a} \bigg), \label{app:eq:BallIntegralBound}
                \end{align}
            for $C_{s, d} > 0$ a constant depending on $s$ and $d$, where the first inequality is by the condition of $E_1$; the second inequality is by noting that for all $\vecbf{l} \in \mathbb{N}_0^d(R)$ the shifted hypercubes $\vecbf{l} + [0, 1]^d$ are contained in the positive orthant of the $\ell_2$-ball of radius $R + d^{1/2}$ in $\mathbb{R}^d$; the third by computing the integral via a polar co-ordinate transform; and the last by noting that $\|\Psi_\eta\|_2^2 = \delta^2 \widetilde{\Lambda} \leq \delta^2 V^2$. Hence, as $\delta$ satisfies
            \begin{equation*}
                \delta \leq \bigg(\frac{r^2}{\widetilde{C}_{s, d}VR^{2s + d}\log(4V/a)} \bigg)^{1/2},
            \end{equation*}
            and we may take $\widetilde{C}_{s, d}$ sufficiently large, it holds that $f_\eta \in \mathcal{S}_{s, r}$.

            For item~2, we proceed by aiming to bound
            \begin{equation*}
                \mathbb{P}\bigg( \sup_{\vecbf{x} \in [0,1]^d} f_\eta(\vecbf{x}) \geq 1 + C_{s, d}'''\|f_\eta - f_0 \|_2 \bigg)
                = \mathbb{P}\bigg( \sup_{\vecbf{x} \in [0,1]^d} \Psi_{\eta} \geq C_{s, d}'''\|\Psi_\eta\|_2 \bigg),
            \end{equation*}
            for $C_{s, d}''' > 0$ and constant depending on $s$ and $d$ to be introduced later.
            
            We start by considering a pseudo-metric $\tau$ on $[0, 1]^d$ via
            \begin{equation*}
                \tau(\vecbf{x}, \vecbf{y})
                = \bigg( \delta^2\sum_{v = 1}^V \frac{\{\psi_v(\vecbf{x}) - \psi_v(\vecbf{y})\}^2}{\tilde{\lambda}_v^2}\bigg)^{1/2}.
            \end{equation*}
            The maximum value the pseudo-metric can take is bounded as
            \begin{align}
                \sup_{\vecbf{x}, \vecbf{y} \in [0, 1]^d} \tau(\vecbf{x}, \vecbf{y}) 
                &\leq \sup_{\vecbf{x}, \vecbf{y} \in [0, 1]^d} \delta V^{1/2} \bigg(\sum_{v = 1}^V \sum_{\vecbf{l}, \vecbf{l}' \in \mathbb{N}_0^d(R)} \theta_{v, \vecbf{l}} \theta_{v, \vecbf{l}'} \{\varphi_{\vecbf{l}}(\vecbf{x}) - \varphi_{\vecbf{l}}(\vecbf{y})\}\{\varphi_{\vecbf{l}'}(\vecbf{x}) - \varphi_{\vecbf{l}'}(\vecbf{y})\} \bigg)^{1/2} \nonumber \\
                &\leq \sup_{\vecbf{x}, \vecbf{y} \in [0, 1]^d} \delta V^{1/2} \bigg(\sum_{\vecbf{l} \in \mathbb{N}_0^d(R)} \{\varphi_{\vecbf{l}}(\vecbf{x}) - \varphi_{\vecbf{l}}(\vecbf{y})\}^2 \bigg)^{1/2} \nonumber \\
                &\leq 2^{d/2 + 1} \delta V = C_{d}' \| \Psi_\eta \|_2 = \Delta \label{app:eq:PseudoMetricMax}
            \end{align}
            where the first inequality uses the fact that $\tilde{\lambda}_v^2 \geq 1/V$ for all $v \in [V]$; the second from the orthonormality of $(\theta_{v, \vecbf{l}})_{v, \vecbf{l}}$; and the final line from the facts that $\| \varphi_{\vecbf{l}} \|_\infty \leq 2^{d/2}$ and that $\widetilde{\Lambda} \geq 3V^2/4$.

            We consider, for $\nu_i \in \mathbb{R}$, the moment generating function
            \begin{align*}
                \mathbb{E}\bigg[\exp\bigg\{ \nu_i \bigg( \delta\sum_{v = 1}^V \frac{\eta_v}{\tilde{\lambda}_v} \psi_v(\vecbf{x}) - \delta\sum_{v = 1}^V \frac{\eta_v}{\tilde{\lambda}_v} \psi_v(\vecbf{y}) \bigg) \bigg\} \bigg]
                &= \mathbb{E}\bigg[\exp\bigg( \nu_i \delta\sum_{v = 1}^V \frac{\eta_v}{\tilde{\lambda}_v} \{\psi_v(\vecbf{x}) - \psi_v(\vecbf{y})\} \bigg) \bigg] \\
                &\leq \exp\bigg( \frac{\nu_i^2 \delta^2}{2}\sum_{v = 1}^V \frac{\{\psi_v(\vecbf{x}) - \psi_v(\vecbf{y})\}^2}{\tilde{\lambda}_v^2}\bigg) \\
                &\leq \exp\bigg( \frac{\nu_i^2}{2}\{\tau(\vecbf{x}, \vecbf{y})\}^2\bigg),
            \end{align*}
            and hence the difference $\Psi_\eta(\vecbf{x}) - \Psi_\eta(\vecbf{y})$ is, almost-surely in $\eta$, $\mathrm{SG(\{\tau(\vecbf{x}, \vecbf{y})\}^2)}$.

            We now apply a chaining argument. For $t \in \mathbb{N}$, let $\Delta_t = 2^{-t}\Delta$ and write $\mathcal{X}_t$ for a $\Delta_t$-net of $[0,1]^d$ with respect to $\tau$, and let $\mathcal{X}_0 = \{\vecbf{x}_0\}$ for some arbitrary point $\vecbf{x}_0 \in [0, 1]^d$. Now, for $t \in \{0\} \cup \mathbb{N}$, define a map $\Pi_t : [0, 1]^d \rightarrow \mathcal{X}_t$ such that $\tau(\vecbf{x}, \Pi_t(\vecbf{x})) \leq \Delta_t$. Write $\Psi_{t, \eta} = \Psi_{\eta} \circ \Pi_t$. We have for any $T \in \mathbb{N}$ that
            \begin{align*}
                \mathbb{E}\bigg[\sup_{\vecbf{x} \in [0,1]^d} \Psi_\eta(\vecbf{x}) \bigg]
                &\leq \mathbb{E}\bigg[\sup_{\vecbf{x} \in [0,1]^d} \Psi_{T, \eta}(\vecbf{x}) \bigg]
                + \mathbb{E}\bigg[\sup_{\vecbf{x} \in [0,1]^d} \{\Psi_{\eta}(\vecbf{x}) - \Psi_{T, \eta}(\vecbf{x})\} \bigg] \\
                &\leq \sum_{t = 1}^T \mathbb{E}\bigg[ \sup_{\vecbf{x} \in [0, 1]^d}\{\Psi_{t, \eta}(\vecbf{x}) - \Psi_{t-1, \eta}(\vecbf{x})\}\bigg]
                + \mathbb{E}\bigg[\sup_{\vecbf{x} \in [0,1]^d} \{\Psi_\eta(\vecbf{x}) - \Psi_{T, \eta}(\vecbf{x})\} \bigg],
            \end{align*}
            where we note that $\mathbb{E}[\Psi_{0, \eta}(\vecbf{x})] = \mathbb{E}[\Psi_\eta(\vecbf{x}_0)] = 0$.
            Further, $\tau(\Pi_t(\vecbf{x}), \Pi_{t-1}(\vecbf{x})) \leq 3\Delta_t$, and using the fact that $\Psi(\vecbf{x}) - \Psi(\vecbf{y})$ is $\mathrm{SG}(\{\tau(\vecbf{x}, \vecbf{y})\}^2)$, we have by an inequality on the expectation of a maximum of sub-Gaussian random variables \citep[e.g.][Theorem~2.5]{Boucheron:2013:ConcentrationIneqBook} that 
            \begin{align}
                \mathbb{E}\bigg[\sup_{\vecbf{x} \in [0,1]^d} \Psi(\vecbf{x}) \bigg]
                &\leq \bigg( \sum_{t = 1}^T\Delta_t\{18\log(N([0, 1]^d; \Delta_t))\}^{1/2} \bigg) \nonumber \\
                &\hspace{1.5cm} + \sup_{\vecbf{x} \in [0, 1]^d}\sum_{v = 1}^V \frac{\delta |\psi_v(\vecbf{x}) - \psi_v(\Pi_T(\vecbf{x}))|}{\tilde{\lambda}_v} \nonumber \\
                &\leq \bigg( \sum_{t = 1}^T \Delta_t\{18\log(N([0, 1]^d; \Delta_t))\}^{1/2} \bigg) + V^{1/2}\sup_{\vecbf{x} \in [0, 1^d]}\{\tau(\vecbf{x}, \Pi_T(\vecbf{x}))\}^{1/2} \nonumber \\
                &\leq \int_0^{\Delta/2} \{18\log(N([0, 1]^d; u))\}^{1/2} \diff{u} + V^{1/2}\Delta_T^{1/2}, \label{app:eq:DudleyBound}
            \end{align}
            where the first inequality uses the fact that $\eta_v \in \{-1, 1\}$; the second by the Cauchy--Schwarz inequality; and the final by facts that the metric entropy is a decreasing function of the ball radius and that $\Delta_1 = \Delta/2$.

            It remains to bound the metric entropy of the space $[0, 1]^d$ with respect to $\tau$. For $j \in [d]$, we write $\tilde{l}_j = \lceil l_j/2 \rceil$, and define the vector $\tilde{\vecbf{l}} = (\tilde{l}_j)_{j \in [d]}$. We have
            \begin{align*}
                \varphi_{l_j}(x_j) = c_{l_j, +} \exp(2\pi i \tilde{l}_j x_j) + c_{l_j, -} \exp(-2\pi i \tilde{l}_j x_j),
            \end{align*}
            where
            \begin{align*}
                c_{l_j, \pm} =
                \begin{cases}
                    2^{-1/2} & \text{ for $l_j$ even}, \\
                    \pm (2i)^{-1/2} & \text{ for $l_j$ odd}.
                \end{cases}
            \end{align*}
            Hence, denoting $\odot$ for the element-wise product of vectors, we have
            \begin{align*}
                \varphi_\vecbf{l}(\vecbf{x})
                = \sum_{\vecbfm{\varsigma} \in \{-1, 1\}^{d}} c_{\vecbfm{\varsigma}} \exp(2\pi i \langle \vecbfm{\varsigma} \odot \tilde{\vecbf{l}}, \vecbf{x} \rangle),
            \end{align*}
            where $c_{\vecbfm{\varsigma}} = \prod_{j = 1}^d c_{l_j, \varsigma_j}$. Thus, we proceed to bound
            \begin{align*}
                \{\varphi_{\vecbf{l}}(\vecbf{x}) - \varphi_{\vecbf{l}}(\vecbf{y})\}^2
                &= \bigg\{\sum_{\vecbfm{\varsigma} \in \{-1, 1\}^d} c_{\vecbfm{\varsigma}} \{\exp(2\pi i \langle \vecbfm{\varsigma} \odot \tilde{\vecbf{l}}, \vecbf{x} \rangle) - \exp(2\pi i \langle \vecbfm{\varsigma} \odot \tilde{\vecbf{l}}, \vecbf{y} \rangle)\} \bigg\}^2 \\
                &\leq \bigg\{\sum_{\vecbfm{\varsigma} \in \{-1, 1\}^d} |c_{\vecbfm{\varsigma}}\exp(2\pi i \langle \vecbfm{\varsigma} \odot \tilde{\vecbf{l}}, \vecbf{y} \rangle)| |\exp(2\pi i \langle \vecbfm{\varsigma} \odot \tilde{\vecbf{l}}, \vecbf{x} - \vecbf{y} \rangle) - 1| \bigg\}^2 \\
                &= \bigg\{2^{-d/2}\sum_{\vecbfm{\varsigma} \in \{-1, 1\}^d} |\exp(2\pi i \langle \vecbfm{\varsigma} \odot \tilde{\vecbf{l}}, \vecbf{x} - \vecbf{y} \rangle) - 1| \bigg\}^2 \\
                &\leq \bigg\{2^{-d/2}\sum_{\vecbfm{\varsigma} \in \{-1, 1\}^d} \min\{1, |\langle \vecbfm{\varsigma} \odot \tilde{\vecbf{l}}, \vecbf{x} - \vecbf{y} \rangle|\} \bigg\}^2 \\
                &\leq \bigg\{2^{-d/2}\sum_{\vecbfm{\varsigma} \in \{-1, 1\}^d} |\langle \vecbfm{\varsigma} \odot \tilde{\vecbf{l}}, \vecbf{x} - \vecbf{y} \rangle|^{s \wedge 1} \bigg\}^2 \\
                &\leq 2^{d} (\| \tilde{\vecbf{l}}\|_1 \|\vecbf{x} - \vecbf{y}\|_\infty)^{2(s \wedge 1)}.
            \end{align*}
            Hence we have
            \begin{align*}
                \{\tau(\vecbf{x}, \vecbf{y})\}^2
                &= \delta^2\sum_{v = 1}^V \frac{\{\psi_v(\vecbf{x}) - \psi_v(\vecbf{y})\}^2}{\tilde{\lambda}_v^2} \\
                &\leq \delta^2 V\sum_{v = 1}^V \sum_{\vecbf{l}, \vecbf{l}' \in \mathbb{N}_0^d(R)} \theta_{v, \vecbf{l}}\theta_{v, \vecbf{l}'}\{\varphi_{\vecbf{l}}(\vecbf{x}) - \varphi_{\vecbf{l}}(\vecbf{y})\}\{\varphi_{\vecbf{l}'}(\vecbf{x}) - \varphi_{\vecbf{l}'}(\vecbf{y})\} \\
                &\leq \delta^2V \bigg\{  \sum_{\vecbf{l}, \vecbf{l}' \in \mathbb{N}_0^d(R)} \bigg( \{\varphi_{\vecbf{l}}(\vecbf{x}) - \varphi_{\vecbf{l}}(\vecbf{y})\}\{\varphi_{\vecbf{l}'}(\vecbf{x}) - \varphi_{\vecbf{l}'}(\vecbf{y})\} \sum_{v = 1}^V \theta_{v, \vecbf{l}}\theta_{v, \vecbf{l}'} \bigg) \bigg\} \\
                &= \delta^2V \sum_{\vecbf{l} \in \mathbb{N}_0^d(R)} \{\varphi_{\vecbf{l}}(\vecbf{x}) - \varphi_{\vecbf{l}}(\vecbf{y})\}^2 \\
                &\leq \delta^2V \max_{\vecbf{l} \in \mathbb{N}_0^d(R)}\{[\varphi_{\vecbf{l}}(\vecbf{x}) - \varphi_{\vecbf{l}}(\vecbf{y})]^2\} \sum_{\vecbf{l} \in \mathbb{N}_0^d(R)} 1 \\
                &\leq 2^d\delta^2V^2 \|\vecbf{x} - \vecbf{y}\|_\infty^{2(s \wedge 1)} \max_{\vecbf{l} \in \mathbb{N}_0^d(R)}\{(\| \tilde{\vecbf{l}}\|_1^{2(s \wedge 1)}\} \\
                &\leq (C_{s, d}'')^2 \| \Psi_\eta\|_2^2 \|\vecbf{x} - \vecbf{y} \|_\infty^{2(s \wedge 1)}R^2,
            \end{align*}
            for $C_{s, d}'' > 0$ an= constant depending on $s$ and $d$, where the first inequality uses the fact that $\tilde{\lambda}_v^2 \geq 1/V$ for all $v \in [V]$; the second by the orthonormality of $(\theta_{v, \vecbf{l}})_{v, \vecbf{l}}$; and the final inequality by the Cauchy--Schwarz inequality and noting $\delta^2V^2 \lesssim \| \Psi_\eta \|_2^2$ by \Cref{app:lem:Eigenbound}.
            
            Hence, we can bound the metric entropy with respect to $\tau$ by the metric entropy with respect to the supremum metric. Denoting the covering number with respect to the supremum metric by $N_\infty$, we have
            \begin{align}
                \log(N([0, 1]^d; u))
                &\leq \log\bigg(N_\infty\bigg([0, 1]^d; \bigg\{\frac{u}{C_{s, d}'' \| \Psi_\eta \|_2 R}\bigg\}^{1/(s \wedge 1)} \bigg)\bigg) \nonumber \\
                &\leq d \log\bigg(1 + \bigg\{\frac{C_{s, d}'' \| \Psi_\eta \|_2 R}{u} \bigg\}^{1/(s \wedge 1)} \bigg) \nonumber \\
                &\leq \frac{d}{s \wedge 1} \log(1 + C_{s, d}'' \| \Psi_\eta \|_2R/u) \label{app:eq:MetricEntropyBound}
            \end{align}

            Finally, we note that \eqref{app:eq:DudleyBound} holds for all $T \in \mathbb{N}$ with the integral term independent of $T$, and so we can remove the $V^{1/2} \Delta_T^{1/2}$ by considering the limit $T \rightarrow \infty$. Hence, combining \eqref{app:eq:DudleyBound} and \eqref{app:eq:MetricEntropyBound}, we have
            \begin{align}
                \mathbb{E}\bigg[\sup_{\vecbf{x} \in [0,1]^d} \Psi(\vecbf{x}) \bigg]
                &\leq \int_0^{\Delta/2}\{18\log(N([0, 1]^d; u))\}^{1/2} \diff{u} \nonumber \\
                &\leq \frac{9d^{1/2} C_{s, d}'' \| \Psi_\eta \|_2 R}{(s \wedge 1)^{1/2}}\int_0^{\Delta/(2C_{s, d}'' \| \Psi_\eta \|_2 R)} \{\log(1 + 1/v)\}^{1/2} \diff{v} \nonumber \\
                &\leq \frac{18d^{1/2}\Delta}{(s \wedge 1)^{1/2}} \bigg\{ \log\bigg(\frac{8C_{s, d}'' \| \Psi_\eta \|_2 R}{\Delta } \bigg) \bigg\}^{1/2}
                \leq C_{s, d}''' \|\Psi_\eta\|_2 \{\log(R)\}^{1/2}/2 \label{app:eq:ESupBound}
            \end{align}
            for $C_{s, d}''' > 0$ a constant depending on $s$ and $d$, where the penultimate inequality is by the fact that $\int_0^x \{\log(1 + 1/a)\}^{1/2} \diff{a} \leq x\{\log(4/x)\}^{1/2}$ for $x \in (0, 1]$ by \Cref{app:lem:logintbound} and the final by recalling $\Delta = C_{d}' \|\Psi_\eta\|_2$. Hence, we have that $\mathbb{E}[\sup_{\vecbf{x} \in [0,1]^d} \Psi(\vecbf{x})] \leq C_{s, d}''' \|\Psi_\eta\|_2$.

            By a form of Hoeffding's inequality for empirical processes\footnote{The theorem requires the empirical process to be indexed cover a countable set, which we may assume as the basis functions are continuous on $[0, 1]^d$ allowing us to consider the countable subset $(\mathbb{Q} \cap [0,1])^d$.} \citep[e.g.][Theorem~12.1]{Boucheron:2013:ConcentrationIneqBook}, we have that $\sup_{\vecbf{x} \in [0,1]^d} \Psi(\vecbf{x})$ is sub-Gaussian with variance proxy
            \begin{align*}
                \sup_{\vecbf{x}, \vecbf{y} \in [0, 1]^d} \delta^2 V &\bigg(\sum_{v = 1}^V \sum_{\vecbf{l}, \vecbf{l}' \in \mathbb{N}_0^d(R)} \theta_{v, \vecbf{l}} \theta_{v, \vecbf{l}'} \{\varphi_{\vecbf{l}}(\vecbf{x}) - \varphi_{\vecbf{l}}(\vecbf{y})\}\{\varphi_{\vecbf{l}'}(\vecbf{x}) - \varphi_{\vecbf{l}'}(\vecbf{y})\} \bigg) \\
                &\leq \sup_{\vecbf{x}, \vecbf{y} \in [0, 1]^d} \delta V^{1/2} \bigg(\sum_{\vecbf{l} \in \mathbb{N}_0^d(R)} \{\varphi_{\vecbf{l}}(\vecbf{x}) - \varphi_{\vecbf{l}}(\vecbf{y})\}^2 \bigg)^{1/2} \\
                &\leq 2^{d + 2}\delta^2 V^2
                \leq 2^{d + 4}\| \Psi_\eta \|_2^2/3
                = \sigma_\Psi^2
            \end{align*}
            which follows from the same calculations as \eqref{app:eq:PseudoMetricMax}. Combining the above with \eqref{app:eq:ESupBound}, we obtain
            \begin{align*}
                \mathbb{P}\bigg( \sup_{\vecbf{x} \in [0,1]^d} f_\eta(\vecbf{x}) \geq 1 +  C_{s, d}'''\|f_\eta - &f_0 \|_2\{\log(R)\}^{1/2} \bigg) \\
                &= \mathbb{P}\bigg( \sup_{\vecbf{x} \in [0,1]^d} \Psi_{\eta} \geq  C_{s, d}'''\|\Psi_\eta\|_2\{\log(R)\}^{1/2} \bigg) \\
                &\leq \mathbb{P}\bigg( \sup_{\vecbf{x} \in [0,1]^d} \Psi_{\eta} - \mathbb{E}\bigg[\sup_{\vecbf{x} \in [0,1]^d} \Psi_{\eta}\bigg] \geq  C_{s, d}'''\|\Psi_\eta\|_2\{\log(R)\}^{1/2}/2 \bigg) \\
                &\leq \exp\bigg( -\frac{(C_{s, d}''')^2\|\Psi_\eta\|_2^2\log(R)}{8\sigma_\Psi^2} \bigg) \\
                &= \exp\bigg( -\frac{3(C_{s, d}''')^2 \log(R)}{2^{d+7}}\bigg)
                \leq a/4,
            \end{align*}
            where the final inequality comes from taking $C_{s, d}'''$ sufficiently large.

            Finally, noting that, as $\eta$ consists of Rademacher random variables, we have that $\eta$ is equal to $-\eta$ in distribution, immediately giving the result
            \begin{equation*}
                \mathbb{P}\bigg( \inf_{\vecbf{x} \in [0,1]^d} f_\eta(\vecbf{x}) \leq 1 - C_{s, d}'''\|f_\eta - f_0 \|_2\{\log(R)\}^{1/2} \bigg) \leq a/4.
            \end{equation*}
            Hence, we define the set $E_2 = \{\eta \in \{-1, 1\}^V : \|f_\eta - f_0\|_\infty \leq C_{s, d}'''\|f_\eta - f_0\|_2 \{\log(R)\}^{1/2}\}$, which is such that $\mathbb{P}(E_2^c) \leq a/2$.
            
            We recall that for $\eta \in E_1$, we have $f_\eta \in \mathcal{S}_{r, s}$. Further, for $\eta \in E_2$, we have that $\|f_\eta - f_0\|_\infty \leq C_{s, d}'''\| \Psi_\eta \|_2 \leq C_{s, d}''' \delta V$. Since $\delta$ satisfies
            \begin{equation*}
                \delta \leq \bigg(\frac{r^2}{\widetilde{C}_{s, d}VR^{2s + d}\log(2V/a)} \bigg)^{1/2},
            \end{equation*}
            and that $R^d \gtrsim C_{s, d}''''V$ for some constant depending on $s$ and $d$ which may be seen from the calculations \eqref{app:eq:BallIntegralBound} with setting $s$ to zero therein, we may take $\widetilde{C}_{s, d}$ large enough so that $\|f_\eta - f_0\|_\infty \leq 1$ (note that $C_{s, d}'''$ and $C_{s, d}''''$ do not depend on $\widetilde{C}_{s, d}$ beyond the shared dependence through $s$ and $d$, and so we may make their ratio small). Hence, combining $E = E_1 \cap E_2$ satisfying $\mathbb{P}(E^c) \leq a$, for $\eta \in E$, it holds that $f_\eta \in \mathcal{F}_{s, r}$.

            For item~3, we note that for $\eta \in E$, we have $\|f_\eta - f_0\|_\infty \leq C_{s, d}'''\|f_\eta - f_0\|_2$ and hence it holds that
            \begin{align*}
                \|f_\eta - f_0\|_1
                &= \int_{[0,1]^d} |f_\eta(\vecbf{x}) - f_0(\vecbf{x})| \diff{\vecbf{x}} \\
                &= \frac{1}{\|f_\eta - f_0\|_\infty} \int_{[0,1]^d} \|f_\eta - f_0\|_\infty |f_\eta(\vecbf{x}) - f_0(\vecbf{x})| \diff{\vecbf{x}} \\
                &\geq \frac{1}{\|f_\eta - f_0\|_\infty} \int_{[0,1]^d} |f_\eta(\vecbf{x}) - f_0(\vecbf{x})|^2 \diff{\vecbf{x}} \\
                &= \frac{\| f_\eta(\vecbf{x}) - f_0(\vecbf{x}) \|_2^2}{\|f_\eta - f_0\|_\infty}
                \geq (C_{s, d}''')^{-1}\| f_\eta(\vecbf{x}) - f_0(\vecbf{x}) \|_2/\{\log(R)\}^{1/2}.
            \end{align*}
            Thus, up to a constant depending on $s$ and $d$, it suffices to consider $\| f_\eta(\vecbf{x}) - f_0(\vecbf{x}) \|_2$, whence we have
            \begin{align*}
                \| f_\eta(\vecbf{x}) - f_0(\vecbf{x}) \|_2^2
                = \delta^2 \widetilde{\Lambda} \geq \frac{3\delta^2V^2}{4}.
            \end{align*}
            Finally, taking $\widetilde{C}_{s, d}$ sufficiently large to ensure the lemma statements are satisfied for this choice of constant completes the proof.
        \end{proof}

            \begin{lemma} \label{app:lem:Eigenbound}
                Denoting $\widetilde{\Lambda} = \sum_{v = 1}^V \tilde{\lambda}_v^{-2}$, we have the bound that
                \begin{equation} \label{app:eq:BigLambdaBound}
                    3V^2/4 \leq \widetilde{\Lambda} \leq V^2.
                \end{equation}
            \end{lemma}
            \begin{proof}[Proof of \Cref{app:lem:Eigenbound}]
                The upper bound follows immediately from noting that $\widetilde{\lambda}_v^{-2} \leq V$, and so it remains to prove the lower bound. For the lower bound, we first prove that $\mathrm{Trace}(\Omega) = \sum_{v = 1}^V \lambda_v^2 \leq nz_\varepsilon^2$. Indeed, we have
                \begin{align*}
                    &\sum_{v = 1}^V \lambda_v^2 \\
                    &= \sum_{v = 1}^V \int_{[0, 1]^d} \psi_v(\vecbf{x}) (\Omega \psi_v)(\vecbf{x}) \diff{\vecbf{x}} \\
                    &= \sum_{v = 1}^V \sum_{i = 1}^n\mathbb{E}\bigg[\int_{\mathcal{Z}} \int_{[0,1]^d}\int_{[0,1]^d} \frac{q_i(z_i \mid \vecbf{x}, Z_{1:i-1})q_i(z_i \mid \vecbf{y}, Z_{1:i-1})\psi_v(\vecbf{x})\psi_{v}(\vecbf{y})}{m_{i, 0}(z_i \mid Z_{1:i-1})} \diff{\vecbf{x}}\diff{\vecbf{y}} \diff{\nu}_i(z_i) \bigg] \\
                    &= \sum_{v = 1}^V \sum_{i = 1}^n\mathbb{E}\bigg[\int_{\mathcal{Z}} \bigg(\int_{[0,1]^d} \frac{q_i(z_i \mid \vecbf{x}, Z_{1:i-1})\psi_v(\vecbf{x})}{m_{i, 0}(z_i \mid Z_{1:i-1})} \diff{\vecbf{x}}\bigg)^2 m_{i, 0}(z_i \mid Z_{1:i-1}) \diff{\nu}_i(z_i) \bigg] \\
                    &= \sum_{v = 1}^V \sum_{i = 1}^n\mathbb{E}\bigg[\int_{\mathcal{Z}} \bigg\{\int_{[0,1]^d} \bigg(\frac{q_i(z_i \mid \vecbf{x}, Z_{1:i-1})}{m_{i, 0}(z_i \mid Z_{1:i-1})} - \exp(2\varepsilon)\bigg)\psi_v(\vecbf{x})\diff{\vecbf{x}}\bigg\}^2 m_{i, 0}(z_i \mid Z_{1:i-1}) \diff{\nu}_i(z_i) \bigg],
                \end{align*}
                where the fourth equality is by the fact $\psi_v$ is orthogonal to $1$. Hence, writing $h_i(z_i) = q_i(z_i \mid \vecbf{x}, Z_{1:i-1})/m_{i, 0}(z_i \mid Z_{1:i-1}) - \exp(2\varepsilon)$ we have, almost surely, that
                \begin{align*}
                    \sum_{v = 1}^V \bigg\{\int_{[0,1]^d} \bigg(\frac{q_i(z_i \mid \vecbf{x}, Z_{1:i-1})}{m_{i, 0}(z_i \mid Z_{1:i-1})} - \exp(2\varepsilon)\bigg)\psi_v(\vecbf{x})\diff{\vecbf{x}}\bigg\}^2
                    \leq \|h_i\|_2^2
                    \leq z_\varepsilon^2,
                \end{align*}
                where the first inequality follows from Bessel's inequality \citep[e.g.][p.~316]{Royden:2010}, and the second by the LDP condition. Hence, combining, we obtain
                \begin{equation*}
                    \sum_{v = 1}^V \lambda_v^2
                    \leq \sum_{i = 1}^n z_\varepsilon^2 = nz_\varepsilon^2.
                \end{equation*}
                
                We now proceed to prove the lower bound $\widetilde{\Lambda}$. Denoting $\mathcal{J} = \{v \in V : \tilde{\lambda}_v^2 = \lambda_v^2/(nz_\varepsilon^2)\}$ we have that
                \begin{align*}
                    \widetilde{\Lambda}
                    &= V^2 - |\mathcal{J}|V + nz_\varepsilon^2 \sum_{v \in \mathcal{J}}\lambda_v^{-2} \\
                    &\geq V^2 - |\mathcal{J}|V + |\mathcal{J}|^2 nz_\varepsilon^2 \bigg( \sum_{v \in \mathcal{J}} \lambda_v^2 \bigg)^{-1} \\
                    &\geq V^2 - |\mathcal{J}|V + |\mathcal{J}|^2 \\
                    &\geq V^2\bigg( 1 - \frac{|\mathcal{J}|}{V} + \frac{|\mathcal{J}|^2}{V^2}\bigg)
                    \geq 3V^2/4
                \end{align*}
                where the first inequality is by the inequality between arithmetic and harmonic means, and the second by the fact that we showed $\sum_{v \in \mathcal{J}} \lambda_v^2 \leq \sum_{v = 1}^V \lambda_v^2 \leq nz_\varepsilon^2$. This completes the proof.
            \end{proof}

    \section{Upper Bounds - Discrete} \label{app:sec:UBDiscrete}
        \subsection{Non-Interactive Method} \label{app:sec:multinomialUpperNonIntProof}
            \begin{proof}[Proof of \Cref{sec4:thm:main} (Upper Bound - Non-Interactive)]
                Under the null hypothesis, permutation invariance holds as all the \beame{privatised}{privatized}{privatized} observations are independent and have the same distribution in both samples. Hence, we have the desired type-I error control. We proceed to prove the power guarantee of the test constructed in \Cref{sec4:UpperBoundNonInt} via \Cref{sec3:thm:sepcondU}.
        
                \noindent
                \textbf{Step 1 (Sub-Gaussianity of \beame{Privatised}{Privatized}{Privatized} Data)}:
                    We note, by \Cref{app:lem:UESG}, that for $i \in [n_1]$ and $i' \in [n_2]$, the random variables $\vecbf{Z}_{i}$ and $\vecbf{W}_{i'}$ obtained via the \beame{privatisation}{privatization}{privatization} scheme \eqref{sec4:eq:UE} are $\mathrm{SG}(17/4)$. Further, they are independent.

                \noindent
                \textbf{Step 2 (Applying \Cref{sec3:thm:sepcondU})}:
                    By the construction \eqref{sec4:eq:UE}, we have
                    \begin{equation*}
                        \mathbb{E}[\vecbf{Z}_1] = (2\omega_{\varepsilon/2} - 1) \vecbf{p}_X + (1 - \omega_{\varepsilon/2})\vecbfm{1}, \quad \mbox{and{}} \quad \mathbb{E}[\vecbf{W}_{1}] = (2\omega_{\varepsilon/2} - 1) \vecbf{p}_Y + (1 - \omega_{\varepsilon/2})\vecbfm{1},
                    \end{equation*}
                    where $\vecbfm{1}$ is the vector of ones and $\omega_{\varepsilon/2} = \exp(\varepsilon/2)/\{ \exp(\varepsilon/2) + 1\}$. Hence, $\mathbb{E}[U_{n_1, n_2}] = (2\omega_{\varepsilon/2}-1)^2 \|\vecbf{p}_X - \vecbf{p}_Y\|_2^2$. It remains to check the criterion~\eqref{sec2:eq:sepcriteria} of \Cref{sec3:thm:sepcondU} with $\sigma^2 = 17/4$ therein, from which we deduce that we have the desired type-II error control when
                    \begin{equation} \label{app:eq:discNIl2cond}
                        \|\vecbf{p}_X - \vecbf{p}_Y\|_2
                        \geq \frac{Cd^{1/4}\log\{1/(\alpha\beta)\}}{(n_1\varepsilon^2)^{1/2}}, \quad \mbox{and} \quad n_1 \geq C\log\{1/(\alpha\beta)\},
                    \end{equation}
                    where we use the fact that $2\omega_{\varepsilon/2} - 1 \lesssim \varepsilon$ for $\varepsilon \leq 1$ and where $C > 0$ is an absolute constant, potentially differing from that in \eqref{sec2:eq:sepcriteria}.
                    
                \noindent
                \textbf{Step 3 (Extending Results to $L_p$-norms)}:
                With the result for $L_2$-norm in hand, to obtain the separation with respect to $L_p$-norms with $p \in [1,2]$. Given a vector $\vecbf{u} \in \mathbb{R}^d$, and writing $\vecbf{u}^{(p)} = (|u_1|^p, \hdots, |u_d|^p)^T$, we apply H\"{o}lder's inequality to $\vecbf{u}^{(p)}$ to obtain
                \begin{align*}
                    \|\vecbf{u}^{(p)}\|_1 \leq \|\vecbf{u}^{(p)}\|_{2/p} d^{1 - p/2}.
                \end{align*}
                Hence, taking the $p$-th root and re-arranging, and putting $\vecbf{p}_X - \vecbf{p}_Y$ in place of $\vecbf{u}$, we obtain
                \begin{equation} \label{app:eq:lpbound}
                    \|\vecbf{p}_X - \vecbf{p}_Y\|_p d^{1/2 - 1/p} \leq \|\vecbf{p}_X - \vecbf{p}_Y\|_2.
                \end{equation}
                Hence, we see that
                \begin{equation*}
                    \| \vecbf{p}_X - \vecbf{p}_Y \|_p
                    \geq \frac{Cd^{1/p - 1/4}\log\{4/\alpha\beta\}}{(n_1\varepsilon^2)^{1/2}}, \quad \mbox{and} \quad n_1 \geq C\log\{1/(\alpha\beta)\},
                \end{equation*}
                is sufficient for \eqref{app:eq:discNIl2cond} to hold. This concludes the proof.
            \end{proof}

    \subsection{Interactive Method}
        The proof relies on the following sub-exponentianity results for the original and permuted version of the test statistic, the proofs of which are deferred to \Cref{app:sec:tailbounds2}.
        \begin{lemma} \label{app:lem:linstatSE}
            The test statistic $T_{n_1, n_2}$ as defined in \eqref{sec4:eq:intstatdisc} is $\mathrm{SE}(\Sigma)$ where
            \begin{equation*}
                \Sigma = C\max\bigg\{ \frac{D_\tau^{1/2}}{(n_1\varepsilon^2)^{1/2}}, \frac{1}{n_1\varepsilon^2}\bigg\} ,
            \end{equation*}
            for $C > 0$ some absolute constant and where
            \begin{equation} \label{app:eq:Dtau}
                D_\tau = D_\tau(\vecbf{p}_X, \vecbf{p}_Y)
                = \sum_{j=1}^d (p_{X,j} - p_{Y,j})^2 \min\{1, \tau/|p_{X,j} - p_{Y,j}|\},
            \end{equation}
            for $\tau = 1/(n_1\varepsilon^2)^{1/2}$.
        \end{lemma}
        \begin{lemma} \label{app:lem:linstatpermSG}
            The test statistic $T_{n_1, n_2}^{\pi}$ as defined in \eqref{sec4:eq:intstatdiscperm} is $\mathrm{SG}(\widetilde{\Sigma}^2)$ where
            \begin{equation*}
                \widetilde{\Sigma}^2 = \widetilde{C} \max\bigg\{\frac{\mathbb{E}[T_{n_1, n_2}]}{(n_1\varepsilon^2)^{1/2}}, \frac{1}{n_1\varepsilon^2}\bigg\},
            \end{equation*}
            for $\widetilde{C} > 0$ some absolute constant.
        \end{lemma}
        With these results in hand, we now prove the theorem.
        \begin{proof}[Proof of \Cref{sec4:thm:main} (Upper Bound - Interactive)]
            Under the null hypothesis, the \beame{privatised}{privatized}{privatized} observations used in calculating the test statistic have the same marginal distribution, and are exchangeable. These properties together ensure that permutation invariance holds and hence we have the desired type-I error control. We now proceed to prove the power guarantee of the test constructed in \Cref{sec4:UpperBoundInt} via \Cref{sec3:prop:permtestcontrol}.

            \noindent
            \textbf{Step 1 (Applying \Cref{sec3:prop:permtestcontrol}):}
            We first inspect $\mathbb{E}[T_{n_1, n_2}]$ so that we may apply \Cref{sec3:prop:permtestcontrol}. We employ a similar argument as in the proof of Theorem 3 in the supplement of \cite{Berrett:2020:Faster}, where a Berry--Esseen theorem is used to control the bias introduced by truncation.

            We first denote $d_j = \hat{p}_{X,j} - \hat{p}_{Y,j} - (p_{X,j} - p_{Y,j})$, $\Delta_j = p_{X,j} - p_{Y,j}$ and $\sigma_j^2 = \mathrm{Var}(d_j) = c_{\varepsilon/2}^2/(4n_1) + c_{\varepsilon/2}^2/(4n_2) \leq 15\tau^2$, where we use the fact that $c_{\varepsilon/2} \leq 5/\varepsilon$ for $\varepsilon \in (0,1]$. Letting $G \sim \mathrm{N}(0, 1)$ a standard Gaussian random variable independent of the data, we have
            \begin{align*}
                |\mathbb{E}[&\Pi_{[-\tau, \tau]}(\hat{p}_{X,j} - \hat{p}_{Y,j})] - \mathbb{E}[\Pi_{[-\tau, \tau]}(\Delta_j - \sigma_j G)]| \\
                &= \bigg| \int_{-\Delta_j}^{\tau - \Delta_j} \mathbb{P}(d_j \geq x) - \mathbb{P}(\sigma_j G \geq x) \diff{x} - \int_{\Delta_j}^{\tau + \Delta_j} \mathbb{P}(d_j \leq -x) - \mathbb{P}(\sigma_j G \leq -x) \diff{x} \bigg|\\
                &\leq 2\tau \sup_{x \in \mathbb{R}}|\mathbb{P}(d_j \leq x) - \mathbb{P}(\sigma_j G \leq x)| \\
                &\leq 12\tau \frac{n_1^{-2}\mathbb{E}[|c_{\varepsilon/2}\{Z_{i, j}' - 1/(\exp(\varepsilon/2) + 1)\} - p_{X,j}|^3]}{\{n_1^{-1}c_{\varepsilon/2}^2 \mathbb{E}[Z_{1,j}'](1 - \mathbb{E}[Z_{1,j}']) + n_2^{-1}c_{\varepsilon/2}^2 \mathbb{E}[W_{1,j}'](1 - \mathbb{E}[W_{1,j}'])\}^{3/2}} \\
                &\hspace{2cm}+ 12\tau \frac{n_2^{-2}\mathbb{E}[|c_{\varepsilon/2}\{W_{i, j}' - 1/(\exp(\varepsilon/2) + 1)\} - p_{Y,j}|^3]}{\{n_1^{-1}c_{\varepsilon/2}^2 \mathbb{E}[Z_{1,j}'](1 - \mathbb{E}[Z_{1,j}']) + n_2^{-1}c_{\varepsilon/2}^2 \mathbb{E}[W_{1,j}'](1 - \mathbb{E}[W_{1,j}'])\}^{3/2}} \\
                &\lesssim \tau \frac{n_1^{-2} + n_2^{-2}}{n_1^{-3/2} + n_2^{-3/2}}
                \lesssim \frac{\tau}{n_1^{1/2}},
            \end{align*}
            where the second inequality is by a Berry--Esseen theorem for independent but non-identically distributed random variables  \cite[e.g.~Chapter~16.5~Theorem~2][]{Feller:1991:introductionVol2}. Proceeding as in the proof of Theorem 3 in the supplement of \cite{Berrett:2020:Faster}, we obtain
            \begin{align*}
                \mathbb{E}[T_{n_1, n_2}]
                \geq \frac{D_\tau}{6} - c\frac{\tau}{n_1^{1/2}}
                \geq \frac{D_\tau}{6} - c\frac{1}{n_1\varepsilon}
            \end{align*}
            for $c > 0$ some absolute constant, where we recall $\tau = (n_1\varepsilon^2)^{-1/2}$ for the last inequality.

            Hence, applying the criterion \eqref{sec2:eq:permtestcontrolSG} of \Cref{sec3:prop:permtestcontrol} and the variance proxies given in \Cref{app:lem:linstatSE} and \Cref{app:lem:linstatpermSG}, we have the desired type-II error control when $\mathbb{E}[T_{n_1, n_2}] \geq C\widetilde{\Sigma}[\log\{1/(\alpha\beta)\}]^{1/2} + C\Sigma\log(1/\beta)$, that is
            \begin{equation*}
                \mathbb{E}[T_{n_1, n_2}]
                \geq C\bigg( \frac{\mathbb{E}[T_{n_1, n_2}][\log\{1/(\alpha\beta)\}]^{1/2} + D_\tau^{1/2}\log(1/\beta)}{(n_1\varepsilon^2)^{1/2}}
                + \frac{[\log\{1/(\alpha\beta)\}]^{1/2} + \log(1/\beta)}{n_1\varepsilon^2} \bigg),
            \end{equation*}
            for $C > 0$ a sufficiently large absolute constant. Then, noting as in \cite{Berrett:2020:Faster} that $D_\tau \geq \|\vecbf{p}_X - \vecbf{p}_Y \|_2^2 \wedge \tau\|\vecbf{p}_X - \vecbf{p}_Y \|_2$, we can simplify the above to the sufficient conditions
            \begin{equation} \label{app:eq:discIl2cond}
                    \|\vecbf{p}_X - \vecbf{p}_Y\|_2
                    \geq C'\frac{\{\log(1/\beta)\}^2 + [\log\{1/(\alpha\beta)\}]^{1/2}}{(n_1\varepsilon^2)^{1/2}}, \quad \mbox{and} \quad n_1\varepsilon^2 \geq C'\log\{1/(\alpha\beta)\},
            \end{equation}
            for $C' > 0$ an absolute constant.

            \noindent
            \textbf{Step 2 (Extending Results to $L_p$-norms):}
            With the result for $L_2$-norm in hand, to obtain the separation with respect to $L_p$-norms with $p \in [1,2]$. By the inequality \eqref{app:eq:lpbound}, we see that
            \begin{equation*}
                \| \vecbf{p}_X - \vecbf{p}_Y \|_p
                \geq C''d^{1/p - 1/2}\frac{\{\log(1/\beta)\}^2 + [\log\{1/(\alpha\beta)\}]^{1/2}}{(n_1\varepsilon^2)^{1/2}}, \quad \mbox{and} \quad n_1\varepsilon^2 \geq C''\log\{1/(\alpha\beta)\}
            \end{equation*}
            is sufficient for \eqref{app:eq:discIl2cond} to hold. This concludes the proof.
        \end{proof}

    \section{Upper Bounds - Continuous} \label{app:eq:ContUB}
        \label{app:sec:ContUB}  
        \subsection{Non-Interactive Method} 
            \begin{proof}[Proof of \Cref{sec5:thm:main} (Upper Bound - Non-Interactive)]
                Under the null hypothesis, permutation invariance holds as all observations are independent and have the same distribution in both samples. Hence, we have the desired type-I error control. We proceed to prove the power guarantee of the test constructed in \Cref{sec5:UpperBoundNonInt} via \Cref{sec3:thm:sepcondU}.

                \noindent
                \textbf{Step 1 (Sub-Gaussianity of \beame{Privatised}{Privatized}{Privatized} Data)}:
                    For a choice of $R$ to be specified later, $i \in [n_1]$ and $i' \in [n_2]$, the random variables $\vecbf{Z}_{i}$ and $\vecbf{W}_{i'}$ obtained by the sampling mechanisms \citet[Equation~26]{Duchi:2018:DJW} and \citet[Equation~3.18]{Li:2023:Thesis}, are $\mathrm{SG}(C_dV/\varepsilon^2)$, for $C_d > 0$ some constant depending on $d$, by \citet[Lemma~3.6][]{Li:2023:Thesis}\footnote{The result therein considers sub-Gaussian random variables without \beame{centring}{centring}{centering}, for which the same sub-Gaussian parameter holds for the \beame{centred}{centred}{centered} case up to an absolute constant via Exercise~2.7.10 in \cite{Vershynin:2018:HDPBook}}, where we note $B'$ therein satisfies $B' \asymp C_dV^{1/2}/\varepsilon$. Further, they are independent.

                \noindent
                \textbf{Step 2 (Applying \Cref{sec3:thm:sepcondU})}:
                    By the construction \eqref{sec4:eq:basistermestimate} and the unbaisedness of the private sampling procedures, we have that
                    \begin{equation*}
                        \mathbb{E}[U_{n_1, n_2}] = \|\vecbfm{\theta}_{X, 1:V} - \vecbfm{\theta}_{Y, 1:V} \|_2^2 \geq \|f_X - f_Y\|_2^2 - \frac{4r^2}{R^{2s}}
                    \end{equation*}
                    where the inequality is by the bound \eqref{sec5:eq:SobolevError}. It remains to check the criteria \eqref{sec2:eq:sepcriteria} of \Cref{sec3:thm:sepcondU} with $\sigma^2 = C_dV/\varepsilon^2$ therein. We briefly note that \Cref{sec3:thm:sepcondU} does not account for privacy, and instead the effect of privacy in the final result follows from the presence of $\varepsilon$ in the sub-Gaussian variance proxy.
                    
                    By \eqref{sec2:eq:sepcriteria}, we deduce that we have the desired type-II error control when $n_1 \geq C_1\log\{1/(\alpha\beta)\}$ for $C_1 > 0$ some absolute constant, and
                    \begin{align}
                        \|f_X - f_Y\|_2^2
                        &\geq \frac{C_d'V^{3/2}[\log\{1/(\alpha\beta)\}]^2}{n_1\varepsilon^2}
                        + \frac{r^2}{R^{2s}} \nonumber \\
                        &\geq C_{r, s, d} \bigg(\frac{[\log\{1/(\alpha\beta)\}]^2}{n_1\varepsilon^2}\bigg)^{4s/(4s+3d)} \label{app:eq:contNIsepcond}
                    \end{align}
                    where $C_d' > 0$ is a constant depending on $d$, and the inequality \eqref{app:eq:Vlowerbound} on $V$ and the choice $R = \{n_1\varepsilon^2/[\log\{1/(\alpha\beta)\}]\}^{1/(2s + 3d/2)}$ with $C_{r, s, d} > 0$ a constant depending on $r, s$ and $d$ gives the second inequality.
                    
                \noindent
                \textbf{Step 3 (Extending Results to $L_p$-norm)}:
                With the result for $L_2$-norm in hand, we now obtain the analogous conditions for separation with respect to the $L_p$-norm for $p \in [1,2)$. We first note that, for a function $f$ supported on $[0,1]^d$, we have by H\"{o}lder's inequality that
                \begin{equation} \label{app:eq:LpNormRelation}
                    \|f\|_p
                    = \bigg(\int |f(\vecbf{x})|^p \diff{\vecbf{x}} \bigg)^{1/p}
                    \leq \mu([0, 1]^d)^{1/p - 1/2} \bigg(\int |f(\vecbf{x})|^2 \diff{\vecbf{x}} \bigg)^{1/2}
                    = \|f\|_2
                \end{equation}
                where $\mu$ denotes the Lebesgue measure on $\mathbb{R}^d$. Hence, as the $L_2$-norm is lower bounded by the $L_p$-norm for $p \in [1,2)$, we see that the condition of \eqref{app:eq:contNIsepcond} is still satisfied when $\|f_X - f_Y\|_p$ takes the place of $\|f_X - f_Y\|_2$, which completes the proof.
            \end{proof}

        \subsection{Interactive Method - Combined Procedure}
            \begin{proof}[Proof of \Cref{sec5:thm:main} (Upper Bound - Interactive)]
                Recall the combined test $\phi_{\mathrm{final}, \alpha}$ and the individual tests $\phi_{\mathrm{trunc}, \alpha/2}$, $\phi_{\mathrm{local}, \alpha/2}$ constructed in \Cref{sec5:UpperBoundInt}. A type-I error guarantee of $\alpha$ for $\phi_{\mathrm{final}, \alpha}$ holds immediately from the union bound, and so it remains to prove we have the desired power.

                To prove the power guarantee of the combined test, we use the following results providing sufficient conditions for the component tests to have power. In what follows, we fix a value of $R$ to be specified, and denote $L_{2, R} = \{\sum_{\vecbf{l} \in \mathbb{N}_0^d(R)} (\theta_{X, \vecbf{l}} - \theta_{Y, \vecbf{l}})^2\}^{1/2}$ for $\{\theta_{X, \vecbf{l}}\}_{\vecbf{l} \in \mathbb{N}_0^d(R)}, \{\theta_{Y, \vecbf{l}}\}_{\vecbf{l} \in \mathbb{N}_0^d(R)}$ the coefficients of the basis expansion of the densities $f_X, f_Y$ respectively, as per the construction \eqref{sec5:eq:sobolev}. We recall that we denote $V = |\mathbb{N}_0^d(R)|$ and $J = \{1, \hdots, \log_2(V)\}$.

                \begin{proposition} \label{app:prop:contintpart2}
                    Suppose $n_1 \leq n_2$ without loss of generality. Suppose also that
                    \begin{equation*} \label{app:eq:contintpart2cond}
                        \begin{aligned}
                            \bigg\{ \frac{C_{r, s, d}V\log(4n_2/\beta)[\log\{2|J|/(\alpha\beta)\}]^2}{n_1\varepsilon^2/(2|J|)} \bigg\}^{1/2}
                            \leq &L_{2, R} \\
                            &\leq \bigg\{ \frac{C_{r, s, d}V^2\log(4n_2/\beta)[\log\{2|J|/(\alpha\beta)\}]^2}{n_1\varepsilon^2/(2|J|)} \bigg\}^{1/2},                            
                        \end{aligned}
                    \end{equation*}
                    for $C_{r, s, d} > 0$ some sufficiently large constant depending on $r, s$ and $d$. Then, the type-II error of the test $\phi_{\mathrm{trunc}, \alpha/2}$ constructed in \Cref{sec5:UpperBoundInt} satisfies $\mathbb{P}(\phi_{\mathrm{trunc}, \alpha/2} = 0) \leq \beta$.
                \end{proposition}
                \begin{proposition} \label{app:prop:contintpart1}
                    Suppose $n_1 \leq n_2$ without loss of generality. Suppose also that
                    \begin{equation} \label{app:eq:contintpart1cond}
                        L_{2, R} \geq \bigg\{ \frac{C_{r, s, d}V^2(\log(4n_1/\beta) + [\log\{1/(\alpha\beta)\}]^2)}{n_1\varepsilon^2} \bigg\}^{1/2},
                    \end{equation}
                    for $C_{r, s, d} > 0$ some sufficiently large constant depending on $r, s$ and $d$. Then, the type-II error of the test $\phi_{\mathrm{local}, \alpha/2}$ constructed in \Cref{sec5:UpperBoundInt} satisfies $\mathbb{P}(\phi_{\mathrm{local}, \alpha/2} = 0) \leq \beta$.
                \end{proposition}
                The key idea of the proofs of these results, which are delayed to later in this appendix, is to consider an event where the truncation/co-ordinate \beame{localisation}{localization}{localization} is \beame{favourable}{favourable}{favorable} and allows us to prove the test statistics have the desired concentration properties around a suitable mean, and showing the complement occurs with low probability. In particular, bounding the probability of the complement by a term depending on the type-II error allows us to control the overall type-II error of the tests.
                
                To unify the two regions given by \eqref{app:eq:contintpart2cond} and \eqref{app:eq:contintpart1cond} into a single condition, we note that as we assume $\alpha + \beta \leq 1/2$, we have $\alpha\beta \leq e^{-1}$, and hence $\log\{1/(\alpha\beta)\} \geq 1$, and also that $\log(4n_2/\beta) \geq \log(4n_1/\beta) \geq 1$ as $n_2 \geq n_1$. Hence, using the fact that $2xy \geq x + y$ for $x,y \geq 1$, we have that
                \begin{equation*}
                     \log(4n_1/\beta) + [\log\{1/(\alpha\beta)\}]^2
                     \leq
                    \log(4n_2/\beta) + [\log\{1/(\alpha\beta)\}]^2
                    \leq 2\log(4n_2/\beta)[\log\{1/(\alpha\beta)\}]^2.
                \end{equation*}
                We also note by the bounds \eqref{app:eq:Vupperbound} and \eqref{app:eq:Vlowerbound}, that we have the inequality $|J| = \log_2(V) \leq C_{r, s, d}'\log(n_1\varepsilon^2) \leq C_{r, s, d}'\log(4n_2/\beta)$. Thus, combining \Cref{app:prop:contintpart2} and \Cref{app:prop:contintpart1}, we have for some constant $C_{r, s, d} > 0$, possibly differing from those in the proposition statements, that under the condition
                \begin{equation} \label{app:eq:contcombsuffcond}
                    L_{2, R}
                    \geq \bigg\{ \frac{C_{r, s, d}V[\log(4n_2/\beta)\log\{1/(\alpha\beta)\}]^2}{n_1\varepsilon^2} \bigg\}^{1/2},
                \end{equation}
                we have that $\min\{\mathbb{P}(\phi_{\mathrm{trunc}, \alpha/2} = 0), \mathbb{P}(\phi_{\mathrm{local}, \alpha/2} = 0)\} \leq \beta$, and hence 
                \begin{align*}
                    \mathbb{P}(\phi_{\mathrm{final}, \alpha} = 0)
                    &= \mathbb{P}(\{\phi_{\mathrm{trunc}, \alpha/2} = 0\} \cap \{\phi_{\mathrm{local}, \alpha/2} = 0\}) \\
                    &\leq \min\{\mathbb{P}(\phi_{\mathrm{trunc}, \alpha/2} = 0), \mathbb{P}(\phi_{\mathrm{local}, \alpha/2} = 0)\}
                    \leq \beta.
                \end{align*}
                
                It remains to connect the condition on $L_{2,R}$ to one on $\|f_X - f_Y\|_2$. In particular, by \eqref{sec5:eq:SobolevError} we have
                \begin{equation*}
                    L_{2, R}^2 \geq \|f_X - f_Y\|_2^2 - \frac{r^2}{R^{2s}}.
                \end{equation*}
                Hence, we have that \eqref{app:eq:contcombsuffcond} holds under the condition
                \begin{align} \label{app:eq:contIntsepcond}
                    \|f_X - f_Y\|_2^2
                    &\geq \frac{C_{r, s, d}V[\log(4n_2/\beta)\log\{1/(\alpha\beta)\}]^2}{n_1\varepsilon^2} + \frac{r^2}{R^{2s}} \nonumber \\
                    &\geq C_{r, s, d}''\bigg(\frac{[\log(n_2/\beta)\log\{1/(\alpha\beta)\}]^2}{n_1\varepsilon^2} \bigg)^{2s/(2s+1)},
                \end{align}
                where the inequality \eqref{app:eq:Vlowerbound} on $V$ and the choice $R = (n_1\varepsilon^2/\{\log(4n_2/\beta)[\log\{1/(\alpha\beta)\}]^2\})^{1/(2s + d)}$ with $C_{r, s, d}'' > 0$ a constant depending on $r, s$ and $d$ gives the second inequality.
                
                To conclude the proof, we recall by \eqref{app:eq:LpNormRelation} that $\|f_X - f_Y\|_p \leq \|f_X - f_Y\|_2$ for $p \in [1,2)$. Hence, as the $L_2$-norm is lower bounded by the $L_p$-norm, we see that the condition of \eqref{app:eq:contIntsepcond} is still satisfied when $\|f_X - f_Y\|_p$ takes the place of $\|f_X - f_Y\|_2$, which completes the proof.
            \end{proof}
        \subsection{Interactive Method - First Procedure}
            \begin{proof}[Proof of \Cref{app:prop:contintpart2}]
                First recall the form of the test $\phi_{\mathrm{trunc}, \alpha/2} = \max_{j \in J} \{\phi_{\mathrm{trunc}, \alpha/(2|J|)}^{\eta_j}\}$. We prove the power guarantee by bounding the type-II error. To this end, we show that there exists some $j^\ast \in J$ such that $\phi_{\mathrm{trunc}, \alpha/(2|J|)}^{\eta_{j^\ast}}$ has type-II error at most $\beta$. In particular, by the construction of the set $J$ and the assumption on the range of $L_{2,R}$ as in \eqref{app:eq:contintpart2cond}, there exists $j^\ast$ such that
                \begin{equation} \label{app:eq:contintpart2etaval}
                    C V^{1/2} L_{2, R}\{\log(4n_2/\beta)\}^{1/2}
                    \leq \eta_{j^\ast}
                    \leq 2^{1/2}C V^{1/2} L_{2, R}\{\log(4n_2/\beta)\}^{1/2},
                \end{equation}
                for $C > 0$ some constant depending on $r, s$ and $d$. In what follows, we consider the test $\phi_{\mathrm{trunc}, \alpha/(2|J|)}^{\eta_{j^\ast}}$ and its corresponding test statistics $T_{n_1, n_2}$ and permuted counterpart $T_{n_1, n_2}^\pi$ for a permutation $\pi$.

                We recall that the two samples are split as in \eqref{sec5:eq:combsamplesplit}. Hence, in what follows we denote $n_1^\ast = n_1/(2|J|)$ and $n_2^\ast = n_2/(2|J|)$, assuming these take integer values for simplicity, and re-index so that we may consider $i \in [n_1^\ast]$ and $j \in [n_2^\ast]$. We then consider the following quantities: we recall the definition of the estimators $\hat{\theta}_{X, \vecbf{l}}, \hat{\theta}_{Y, \vecbf{l}}$ for $\vecbf{l} \in \mathbb{N}_0^d(R)$ as in \eqref{sec5:eq:RRsobcoeffestscombProc1} and the \beame{privatised}{privatized}{privatized} values \eqref{sec5:eq:contMethod1NotTrunc}. For $i \in [n_1^\ast]$, $i' \in [n_2^\ast]$, define the quantities
                \begin{equation} \label{app:eq:tildebeforeprivandtruncvals}
                    \tilde{X}_i = \sum_{\vecbf{l} \in \mathbb{N}_0^d(R)} (\hat{\theta}_{X,\vecbf{l}} - \hat{\theta}_{Y,\vecbf{l}})\varphi_\vecbf{l}(X_i), \quad \mbox{and} \quad \tilde{Y}_{i'} = \sum_{\vecbf{l} \in \mathbb{N}_0^d(R)} (\hat{\theta}_{X,\vecbf{l}} - \hat{\theta}_{Y,\vecbf{l}})\varphi_\vecbf{l}(Y_{i'}),
                \end{equation}
                that is, the values the users in the second stage of interactivity obtain before \beame{privatisation}{privatization}{privatization} and truncation. Then, for $i \in [n_1^\ast]$ and $i' \in [n_2^\ast]$, we define the events
                \begin{equation} \label{app:eq:intfirstprocevents}
                    A_i^{(X)} = \{|\tilde{X}_i| \leq \eta_{j^\ast}\} \mbox{ and } A_{i'}^{(Y)} = \{|\tilde{Y}_i| \leq \eta_{j^\ast}\}, \mbox{ and let } A = \{ \cap_{i \in [n_1]} A_i^{(X)} \} \cap \{ \cap_{i \in [n_2]} A_i^{(Y)} \}.
                \end{equation}
                Denote the combined values by $\tilde{D}_i = \tilde{X}_i$ for $i \in [n_1^\ast]$ and $\tilde{D}_{n_1 + i'} = \tilde{Y}_{i'}$ for $i' \in [n_2^\ast]$. 
                
                Consider a sample of permutations $\{\pi_1, \hdots, \pi_B\}$ of size $B$ drawn, with replacement, uniformly from $S_{n_1^\ast + n_2^\ast}$, the group of permutations on $[n_1^\ast + n_2^\ast]$.  We start by noting that for any $t > 0$ and the event $A$ to be defined shortly, $\mathbb{P}(p_B > \alpha/(2|J|)\}$ for $p_B$ as in \eqref{sec2:def:permpval} can be decomposed as
                \begin{align}
                    &\mathbb{P}(p_B > \alpha/(2|J|)) \nonumber \\
                    &= \mathbb{P}(\{p_B > \alpha/(2|J|)\} \cap A) + \mathbb{P}(\{p_B > \alpha/(2|J|)\} \cap A^c) \nonumber \\
                    &\leq \mathbb{P}(\{p_B > \alpha/(2|J|)\} \cap A) + \mathbb{P}(A^c)  \nonumber \\
                    &= \mathbb{P}\bigg( \bigg\{ 1 + \sum_{b=1}^B \mathbbm{1}\{T_{n_1^\ast, n_2^\ast} \leq T_{n_1^\ast, n_2^\ast}^{\pi_b}\} \geq (1+B)\alpha/(2|J|) \bigg\} \cap A\bigg) + \mathbb{P}(A^c)  \nonumber \\
                    &\leq \mathbb{P}\bigg( \bigg\{1 + \sum_{b=1}^B \mathbbm{1}\{T_{n_1^\ast, n_2^\ast} \leq T_{n_1^\ast, n_2^\ast}^{\pi_b}\} \geq (1+B)\alpha/(2|J|) \bigg\} \cap \{T_{n_1^\ast, n_2^\ast} \geq t\} \cap A \bigg) \nonumber \\
                    &\hspace{8cm}+ \mathbb{P}(\{T_{n_1^\ast, n_2^\ast} < t\} \cap A)
                    + \mathbb{P}(A^c)  \nonumber \\
                    &\leq \mathbb{P}\bigg(1 + \sum_{b=1}^B \mathbbm{1}\{A\} \mathbbm{1}\{T_{n_1^\ast, n_2^\ast}^{\pi_b} \geq t\} \geq (1+B)\alpha/(2|J|) \bigg)
                    + \mathbb{P}(\{T_{n_1^\ast, n_2^\ast} < t\} \cap A)
                    + \mathbb{P}(A^c)  \nonumber \\
                    &\leq \frac{1+B\mathbb{P}(\{T_{n_1^\ast, n_2^\ast}^{\pi_1} \geq t\} \cap A)}{(1+B)\alpha/(2|J|)} + \mathbb{P}(A^c) + \mathbb{P}(\{T_{n_1^\ast, n_2^\ast} < t\} \cap A)
                    \leq \mathrm{(I)} + \mathrm{(II)} + \mathrm{(III)}, \label{app:eq:type2errorboundFirst}
                \end{align}
                where the penultimate inequality is by Markov's inequality. We proceed by bounding each term individually.
                
                \noindent
                \textbf{Term} (I):
                Denote by $\widetilde{D}_i^\ast$ for $i \in [n_1^\ast + n_2^\ast]$ the combined sample of the values the users in the second stage of interaction obtain without truncation and \beame{privatisation}{privatization}{privatization}, that is
                \begin{align*}
                    \widetilde{D}_i^\ast = \widetilde{X}_i, \mbox{ for } i \in [n_1^\ast], \quad \mbox{and} \quad
                    \widetilde{D}_{n_1^\ast + i'}^\ast = \widetilde{Y}_{i'}, \mbox{ for } i' \in [n_2^\ast].
                \end{align*}
                Define the sums
                \begin{align*}
                    \Lambda = \frac{1}{n_1^\ast} \sum_{i=1}^{n_1^\ast} \frac{2\eta_{j^\ast}}{\varepsilon}\zeta_i - \frac{1}{n_2^\ast} \sum_{i'=1}^{n_2^\ast} \frac{2\eta_{j^\ast}}{\varepsilon}\xi_{i'}, \quad \mbox{and} \quad
                    T_{n_1^\ast, n_2^\ast}^{\ast, \pi_1} = \frac{1}{n_1^\ast} \sum_{i = 1}^{n_1^\ast} \tilde{D}_{\pi(i)}^\ast - \frac{1}{n_2^\ast} \sum_{i' = 1}^{n_2^\ast} \tilde{D}_{\pi(n_1^\ast + i')}^\ast,
                \end{align*}
                so that $T_{n_1^\ast, n_2^\ast}^{\pi_1} = T_{n_1^\ast, n_2^\ast}^{\ast, \pi_1} + \Lambda$ on $A$. Hence, we decompose the tail probability 
                \begin{align*}
                    \mathbb{P}(\{T_{n_1^\ast, n_2^\ast}^{\pi_1} \geq t\} \cap A)
                    &\leq \mathbb{P}(\{T_{n_1^\ast, n_2^\ast}^{\ast, \pi_1} \geq t/2\} \cap A)
                    + \mathbb{P}(\Lambda \geq t/2) \\
                    &\leq \mathbb{P}(\{T_{n_1^\ast, n_2^\ast}^{\ast, \pi_1} \geq t/2\} \cap A)
                    + 2\exp\bigg(\frac{-c_2(n_1^\ast\varepsilon^2)^{1/2}t}{\eta_{j^\ast}}\bigg) \\
                    &\leq 2\exp\bigg(\frac{-c_1n_1^\ast t^2}{\eta_{j^\ast}^2}\bigg)
                    + 2\exp\bigg(\frac{-c_2(n_1^\ast\varepsilon^2)^{1/2}t}{\eta_{j^\ast}}\bigg)
                \end{align*}
                where $c_1, c_2 > 0$ are absolute constants; the first inequality uses the fact that $T_{n_1^\ast, n_2^\ast}^{\pi_1} = T_{n_1^\ast, n_2^\ast}^{\ast, \pi_1} + \Lambda$ on $A$; the second inequality as $\Lambda$ is $\mathrm{SE}(2^{5/2}\eta_{j^\ast}/(n_1^\ast\varepsilon^2)^{1/2})$ by the independence of the summands of $\Lambda$ and the fact a Laplace random variable is $\mathrm{SE}(2^{1/2})$ by \Cref{app:lem:LaplaceSE}; and the final inequality is by \Cref{app:lem:contintpermSE}, the statement and proof of which is delayed to \Cref{app:sec:tailbounds3}.
                
                In particular, we have for $C_1 > 0$ a sufficiently large absolute constant that
                \begin{equation} \label{app:eq:generaltype2errorbound1}
                    \frac{2|J|}{(1+B)\alpha} + \frac{2B|J|}{(1+B)\alpha}\mathbb{P}\bigg(\bigg\{T_{n_1^\ast, n_2^\ast}^{\pi_1} \geq \frac{C_1\eta_{j^\ast}\log\{2|J|/(\alpha\beta)\}}{(n_1^\ast\varepsilon^2)^{1/2}} \bigg\} \cap A\bigg)
                    \leq \frac{\beta}{3},
                \end{equation}
                where we use the facts that $B \geq 8|J|/(\alpha\beta) - 1$ and $[\log\{2|J|/(\alpha\beta)\}]^2 \geq \log\{2|J|/(\alpha\beta)\}$ as $\alpha + \beta \leq 1/2$.

                \noindent
                \textbf{Term} (II):
                For $\vecbf{l} \in \mathbb{N}_0^d(R)$, denoting $\hat{\delta}_\vecbf{l} = \hat{\theta}_{X,\vecbf{l}} - \hat{\theta}_{Y,\vecbf{l}}$ and $\delta_\vecbf{l} = \mathbb{E}[\hat{\delta}_\vecbf{l}]$, we decompose and bound the moment generating function of $\tilde{X}_i - \mathbb{E}[\tilde{X}_i]$ as
                \begin{align*}
                    \mathbb{E}\bigg[\exp\bigg\{\lambda (&\tilde{X}_i - \mathbb{E}[\tilde{X}_i])\bigg\}\bigg] \\
                    &= \mathbb{E}\bigg[\exp\bigg\{\lambda\sum_{\vecbf{l} \in \mathbb{N}_0^d(R)}\{\hat{\delta}_\vecbf{l} \varphi_\vecbf{l}(X_i) - \delta_\vecbf{l} \theta_{X,\vecbf{l}}\}\bigg\}\bigg] \\
                    &= \mathbb{E}\bigg[ \mathbb{E}\bigg[\exp\bigg\{\lambda\sum_{\vecbf{l} \in \mathbb{N}_0^d(R)}(\hat{\delta}_\vecbf{l} - \delta_\vecbf{l}) \varphi_\vecbf{l}(X_i)\bigg\}\biggm| X_i\bigg] \exp\bigg\{\lambda\sum_{\vecbf{l} \in \mathbb{N}_0^d(R)} \delta_\vecbf{l} \{\varphi_\vecbf{l}(X_i) - \theta_{X, \vecbf{l}}\} \bigg\}\bigg] \\
                    &\leq \mathbb{E}\bigg[ \exp\bigg\{\frac{5\cdot2^d\lambda^2 V}{n_1^\ast\varepsilon^2}\sum_{\vecbf{l} \in \mathbb{N}_0^d(R)} \varphi_\vecbf{l}(X_i)^2\bigg\} \exp\bigg\{\lambda\sum_{\vecbf{l} \in \mathbb{N}_0^d(R)} \delta_\vecbf{l} \{\varphi_\vecbf{l}(X_i) - \theta_{X, \vecbf{l}}\} \bigg\}\bigg] \\
                    &\leq \exp\bigg\{\frac{5\cdot2^{2d}\lambda^2 V^2}{n_1^\ast\varepsilon^2}\bigg\} \mathbb{E}\bigg[ \exp\bigg\{\lambda\sum_{\vecbf{l} \in \mathbb{N}_0^d(R)} \delta_\vecbf{l} \{\varphi_\vecbf{l}(X_i) - \theta_{X, \vecbf{l}}\} \bigg\} \bigg],
                \end{align*}
                where the first inequality holds as each \beame{privatised}{privatized}{privatized} value $Z_i'$, $W_{i'}'$ takes values in $\{-2^{d/2}c_\varepsilon, 2^{d/2}c_\varepsilon\}$ and are hence $\mathrm{SG}(5\cdot2^d/\varepsilon^2)$ (e.g.~\citealt[Exercise~2.4]{Wainwright:2019:HDSBook}), where we recall $c_\varepsilon = \{\exp(\varepsilon) + 1\}/\{\exp(\varepsilon) - 1\}$ and hence use the fact that $c_\varepsilon^2 \leq 5/\varepsilon^2$ for $\varepsilon \in (0,1]$. Thus, the estimators \eqref{sec5:eq:RRsobcoeffestscombProc1} which are the average of at least $n_1^\ast/V$ i.i.d.~copies of such variables are $\mathrm{SG}(5\cdot2^{d}V/(n_1^\ast\varepsilon^2))$,
                and so the differences $\hat{\theta}_{X, \vecbf{l}} - \hat{\theta}_{Y, \vecbf{l}}$ are $\mathrm{SG}(5\cdot2^{d+1}V/(n_1^\ast\varepsilon^2))$ for all $\vecbf{l} \in \mathbb{N}_0^d(R)$. The second inequality then follows from the fact that $|\varphi_\vecbf{l}(\vecbf{x})| \leq 2^{d/2}$ for all $\vecbf{l} \in \mathbb{N}_0^d(R)$.

                Focusing on the remaining expectation term, we note the collection $\{\varphi_\vecbf{l}(\vecbf{X}_i) - \theta_{X, \vecbf{l}}\}_{\vecbf{l} \in \mathbb{N}_0^d(R)}$ consists of $\mathrm{SG}(2^d)$ (e.g.~\citealt[Exercise~2.4]{Wainwright:2019:HDSBook}) random variables. The collection is not independent, but one can show via H\"{o}lder's inequality that the sum $\sum_{\vecbf{l} \in \mathbb{N}_0^d(R)}\delta_\vecbf{l}\{\varphi_\vecbf{l}(\vecbf{X}_i) - \theta_{X, \vecbf{l}}\}$ is $\mathrm{SG}(2^d\{\sum_{\vecbf{l} \in \mathbb{N}_0^d(R)}|\delta_\vecbf{l}|\}^2)$. Hence, combining, we have
                \begin{align*}
                    \mathbb{E}\bigg[\exp\bigg\{\lambda (\tilde{X}_i - \mathbb{E}[\tilde{X}_i])\bigg\}\bigg]
                    &\leq \exp\bigg\{\frac{5\cdot2^{2d}\lambda^2 V^2}{n_1^\ast\varepsilon^2}\bigg\} \exp\bigg\{\lambda^2\bigg(\sum_{\vecbf{l} \in \mathbb{N}_0^d(R)}|\delta_\vecbf{l}|\bigg)^2\bigg\} \\
                    &\leq \exp\bigg\{\frac{5\cdot2^{2d}\lambda^2V^2}{n_1^\ast\varepsilon^2}\bigg\} \exp\bigg\{\lambda^2 V L_{2, R}^2\bigg\} \\
                    &\leq \exp\bigg\{2\lambda^2 V L_{2, R}^2\bigg\},
                \end{align*}
                where the second inequality is by the Cauchy--Schwarz inequality, and the final inequality follows from the assumed range on the values of $L_{2, R}$ as in \eqref{app:eq:contintpart2cond}. Hence, $\tilde{X}_i$ is $\mathrm{SG}(4VL_{2,R}^2)$, yielding
                \begin{align*}
                    \mathbb{P}\{(A_i^{(X)})^c\}
                    &=\mathbb{P}(|\tilde{X}_i| > \eta_{j^\ast}) \\
                    &\leq \mathbb{P}(|\tilde{X}_i - \mathbb{E}[\tilde{X}_i]| > \eta_{j^\ast}/2) + \mathbb{P}(|\mathbb{E}[\tilde{X}_i]| > \eta_{j^\ast}/2) \\
                    &= \mathbb{P}(|\tilde{X}_i - \mathbb{E}[\tilde{X}_i]| > \eta_{j^\ast}/2) \\
                    &\leq 2\exp\bigg(\frac{-\eta_{j^\ast}^2}{16VL_{2, R}^2}\bigg) \\
                    &\leq 2\exp\bigg(\frac{-C^2\log(4n_2/\beta)}{16}\bigg) \leq 2\bigg(\frac{\beta}{4n_2}\bigg)^{C_2}, 
                \end{align*}
                for $C_2 > 0$ some absolute constant which we may take sufficiently large by taking $C$ large; where the second equality holds because $\eta_{j^\ast}/2 > |\mathbb{E}[\tilde{X}_i]|$ for $C$ large enough which follows from the bound $|\mathbb{E}[\tilde{X}_i]| = |\sum_{\vecbf{l} \in \mathbb{N}_0^d(R)} \delta_\vecbf{l} \theta_{X,\vecbf{l}}| \leq 2^{d/2}V^{1/2}L_{2, R}$ due to the Cauchy--Schwarz inequality; the second inequality is by, for example, \citet[Proposition~2.5][]{Wainwright:2019:HDSBook}; and the penultimate inequality is by \eqref{app:eq:contintpart2etaval}. A similar argument holds for $\tilde{Y}_{i'}$ yielding the same bound for $\mathbb{P}\{(A_{i'}^{(Y)})^c\}$ and hence we have by the union bound and taking $C_2$ sufficiently large that
                \begin{equation} \label{app:eq:contintproc1probbound}
                    \mathbb{P}(A^c)
                    \leq \sum_{i = 1}^{n_1^\ast} \mathbb{P}\{(A_i^{(X)})^c\} + \sum_{i' = 1}^{n_2^\ast} \mathbb{P}\{(A_i^{(Y)})^c\}
                    \leq 4n_2^\ast\bigg(\frac{\beta}{4n_2}\bigg)^{C_2}
                    \leq \frac{\beta}{3}.
                \end{equation}

                \noindent
                \textbf{Term} (III):
                Denote the sum
                \begin{equation*}
                    T_{n_1^\ast, n_2^\ast}^\ast = \frac{1}{n_1^\ast} \sum_{i = 1}^{n_1^\ast} \tilde{X}_i - \frac{1}{n_2^\ast} \sum_{i' = 1}^{n_2^\ast} \tilde{Y}_{i'},
                \end{equation*}
                where we have $T_{n_1^\ast, n_2^\ast} = T_{n_1^\ast, n_2^\ast}^\ast + \Lambda$ on $A$. We thus consider the decomposition
                \begin{align*}
                    \mathbb{P}(\{T_{n_1^\ast, n_2^\ast} < t\}& \cap A) \\
                    &= \mathbb{P}(\{T_{n_1^\ast, n_2^\ast}^\ast - \mathbb{E}[T_{n_1^\ast, n_2^\ast}^\ast] + \Lambda < t - \mathbb{E}[T_{n_1^\ast, n_2^\ast}^\ast]\} \cap A)  \\
                    &\leq \mathbb{P}(\{T_{n_1^\ast, n_2^\ast}^\ast - \mathbb{E}[T_{n_1^\ast, n_2^\ast}^\ast] < t/2 - \mathbb{E}[T_{n_1^\ast, n_2^\ast}^\ast]/2 \} \cap A) + \mathbb{P}(\Lambda < t/2 - \mathbb{E}[T_{n_1^\ast, n_2^\ast}^\ast]/2).
                \end{align*}
                
                Note $\mathbb{E}[T_{n_1^\ast, n_2^\ast}^\ast] \geq C_1\eta_{j^\ast}\log\{2|J|/(\alpha\beta)\}/(n_1^\ast\varepsilon^2)^{1/2} + C_{s, r, d}'\eta_{j^\ast}\log(1/\beta)/(n_1^\ast\varepsilon^2)^{1/2}$ for $C_{s, r, d}' > 0$ a sufficiently large constant depending on $s, r$ and $d$, which holds by \eqref{app:eq:contintpart2cond}, the value of $\eta_{j^\ast}$ as in \eqref{app:eq:contintpart2etaval}, and the fact that $\mathbb{E}[T_{n_1^\ast, n_2^\ast}^\ast] = L_{2, R}^2$. Hence, we obtain
                \begin{align}
                    \mathbb{P}\bigg(\bigg\{&T_{n_1^\ast, n_2^\ast} < \frac{C_1\eta_{j^\ast}\log\{2|J|/(\alpha\beta)\}}{(n_1^\ast\varepsilon^2)^{1/2}} \bigg\} \cap A \bigg) \nonumber \\
                    &\leq \mathbb{P}\bigg(\bigg\{T_{n_1^\ast, n_2^\ast}^\ast - \mathbb{E}[T_{n_1^\ast, n_2^\ast}^\ast] < -\frac{C_{s, r, d}'\eta_{j^\ast}\log(1/\beta)}{2(n_1^\ast\varepsilon^2)^{1/2}} \bigg\} \cap A\bigg) + \mathbb{P}\bigg(\Lambda < -\frac{C_{s, r, d}'\eta_{j^\ast}\log(1/\beta)}{2(n_1^\ast\varepsilon^2)^{1/2}} \bigg) \nonumber \\
                    &\leq \mathbb{P}\bigg(\bigg\{T_{n_1^\ast, n_2^\ast}^\ast - \mathbb{E}[T_{n_1^\ast, n_2^\ast}^\ast] < -\frac{C_{s, r, d}'\eta_{j^\ast}\log(1/\beta)}{2(n_1^\ast\varepsilon^2)^{1/2}} \bigg\} \cap A\bigg) + \frac{\beta}{12} \nonumber \\
                    &\leq \frac{\beta}{12} + \frac{3\cdot2^{d+5/2}Vc_\varepsilon(n_1^\ast\varepsilon^2)^{1/2}}{C_{s, r, d}'\eta_{j^\ast}\log(1/\beta)}\bigg(\frac{\beta}{4n_2}\bigg)^{C_2/2} + \frac{\beta}{12} \leq \frac{\beta}{3} \label{app:eq:generaltype2errorbound2}
                \end{align}
                where the second inequality is by fact that $\Lambda$ is $\mathrm{SE}(2^{5/2}\eta_{j^\ast}/(n_1^\ast\varepsilon^2)^{1/2})$; the third by \Cref{app:lem:contintSE}, the statement and proof of which is delayed to \Cref{app:sec:tailbounds3}; and the final inequality as we may take the constants $C_3$ and $C_{s, r, d}'$ large enough so that the second term in the penultimate inequality is bounded by $\beta/6$, noting the upper bound on $V$ in \eqref{app:eq:Vupperbound} and value $R = [n_1^\ast\varepsilon^2/\{\log(4n_2/\beta)[\log\{2|J|/(\alpha\beta)\}]^2\}]^{1/(2s + 1)}$.
                
                To conclude, setting $t = C_1\eta_{j^\ast}\log\{2|J|/(\alpha\beta)\}/\{(n_1^\ast\varepsilon^2)^{1/2}\}$ in \eqref{app:eq:type2errorboundFirst}, and combining \eqref{app:eq:generaltype2errorbound1}, \eqref{app:eq:contintproc1probbound} and \eqref{app:eq:generaltype2errorbound2}, we obtain the power guarantee $\mathbb{P}(p_B > \alpha) \leq \beta$, which completes the proof.

            \end{proof}
        \subsection{Interactive Method - Second Procedure}
            \begin{proof}[Proof of \Cref{app:prop:contintpart1}]
                In what follows, we assume the size of the two samples are $2n_1$ and $2n_2$ for simplicity, with each sample split into two folds as in \eqref{sec4:eq:sampledefs}. The fact that the samples are also split in the combined test of \Cref{sec5:UpperBoundInt} before being passed to the component test $\phi_{\mathrm{local}, \alpha/2}$ will only change the final rate up to constants. We first note by the assumption \eqref{app:eq:contintpart1cond} there exists $\vecbf{l}^\dagger \in \mathbb{N}_0^d(R)$ such that
                \begin{equation*}
                    |\theta_{X, \vecbf{l}^\dagger} - \theta_{Y, \vecbf{l}^\dagger}| \geq \bigg(\frac{C_{r, s, d}V(\log(4n_1/\beta)+[\log\{1/(\alpha\beta)\}]^2)}{n_1\varepsilon^2}\bigg)^{1/2}.
                \end{equation*}
                denote
                \begin{equation*}
                    \mathcal{L} = \big\{\vecbf{l} \in \mathbb{N}_0^d(R) : |\theta_{X, \vecbf{l}} - \theta_{Y, \vecbf{l}}| < \{C_{r, s, d} V/(4n_1\varepsilon^2)\}^{1/2}\log\{1/(\alpha\beta)\}\big\},
                \end{equation*}
                and consider the event
                \begin{equation*}
                    A = \big\{ |\hat{\theta}_{X, \vecbf{l}^\dagger} - \hat{\theta}_{Y, \vecbf{l}^\dagger}| \geq \max_{\vecbf{l} \in \mathcal{L}} |\hat{\theta}_{X, \vecbf{l}} - \hat{\theta}_{Y, \vecbf{l}}| \big\}.
                \end{equation*}
                On $A$, we have that the chosen index $\vecbf{l}^\ast \notin \mathcal{L}$.
                
                Consider a sample of permutations $\{\pi_1, \hdots, \pi_B\}$ of size $B$ drawn, with replacement, uniformly from $S_{n_1 + n_2}$, the group of permutations on $[n_1 + n_2]$.  For any collection $\{t_1, \hdots, t_V\}$ with $t_\vecbf{l} > 0$ for all $\vecbf{l} \in \mathbb{N}_0^d(R)$, $\mathbb{P}(p_B > \alpha/2)$ for $p_B$ as in \eqref{sec2:def:permpval} can be decomposed as
                \begin{align}
                    \mathbb{P}(&p_B > \alpha/2) \\
                    &= \mathbb{P}(\{p_B > \alpha/2\} \cap A) + \mathbb{P}(\{p_B > \alpha/2a\} \cap A^c) \nonumber \\
                    &\leq \mathbb{P}(\{p_B > \alpha/2\} \cap A) + \mathbb{P}(A^c) \nonumber \\
                    &= \mathbb{P}\bigg( \bigg\{ 1 + \sum_{b = 1}^B \mathbbm{1}\{U_{n_1, n_2} \leq U_{n_1, n_2}^{\pi_b}\} \geq (1+B)\alpha/2 \bigg\} \cap A\bigg) + \mathbb{P}(A^c) \nonumber \\
                    &\leq \sum_{\vecbf{l} \in \mathbb{N}_0^d(R) \setminus \mathcal{L}} \bigg[ \mathbb{P}\bigg( \bigg\{1 + \sum_{b = 1}^B \mathbbm{1}\{U_{n_1, n_2} \leq U_{n_1, n_2}^{\pi_b}\} \geq (1+B)\alpha/2 \bigg\} \cap \{U_{n_1, n_2} \geq t_\vecbf{l}\} \mid \vecbf{l}^\ast = \vecbf{l} \bigg) \nonumber\\
                    &\hspace{4cm}+ \mathbb{P}(\{U_{n_1, n_2} < t_\vecbf{l}\} \mid \vecbf{l}^\ast = \vecbf{l}) \bigg]\mathbb{P}(\vecbf{l}^\ast = \vecbf{l} ) 
                    + \mathbb{P}(A^c) \nonumber \\
                    &\leq \sum_{\vecbf{l} \in \mathbb{N}_0^d(R) \setminus \mathcal{L}} \bigg[ \mathbb{P}\bigg( \bigg\{1 + \sum_{b = 1}^B \mathbbm{1}\{U_{n_1, n_2}^{\pi_b} \geq t_\vecbf{l}\} \geq (1+B)\alpha/2 \bigg\} \mid \vecbf{l}^\ast = \vecbf{l} \bigg) \nonumber\\
                    &\hspace{4cm}+ \mathbb{P}(U_{n_1, n_2} < t_\vecbf{l} \mid \vecbf{l}^\ast = \vecbf{l}) \bigg]\mathbb{P}(\vecbf{l}^\ast = \vecbf{l} ) 
                    + \mathbb{P}(A^c) \nonumber \\
                    &\leq \sum_{\vecbf{l} \in \mathbb{N}_0^R(R) \setminus \mathcal{L}} \bigg[ \frac{1+B\mathbb{P}(U_{n_1, n_2}^{\pi_1} \geq t_\vecbf{l} \mid \vecbf{l}^\ast = \vecbf{l})}{(1+B)\alpha/2} + \mathbb{P}(U_{n_1, n_2} < t_\vecbf{l} \mid \vecbf{l}^\ast = \vecbf{l}) \bigg]\mathbb{P}(\vecbf{l}^\ast = \vecbf{l}) + \mathbb{P}(A^c), \nonumber \\
                    &\leq (\mathrm{I}) + (\mathrm{II}) \label{app:eq:contintsecondterms}
                \end{align}
                where the second inequality uses the law of total probability and the fact that the chosen $\vecbf{l}^\ast \notin \mathcal{L}$ on $A$, and the final inequality is by Markov's inequality. We now consider the terms individually.
                
                \noindent
                \textbf{Term} (I):                
                Conditional on a fixed value of $\vecbf{l}^\ast$, we have that the data $\{Z_i\}_{i \in [n_1]},\{W_{i'}\}_{i' \in [n_2]}$ as in \eqref{sec5:eq:localProcSelected} are independent and take values in $\{-2^{d/2}c_\varepsilon, 2^{d/2}c_\varepsilon\}$ and are hence $\mathrm{SG}(5\cdot2^d/\varepsilon^2)$ (e.g.~\citealt[Exercise~2.4]{Wainwright:2019:HDSBook}), where we used the fact that $c_\varepsilon^2 \leq 5/\varepsilon^2$ for $\varepsilon \in (0,1]$. Further, the test statistic $U_{n_1, n_2}$ is an unbiased estimator of $(\theta_{X, \vecbf{l}^\ast} - \theta_{Y, \vecbf{l}^\ast})^2$. Further, by the same argument as in the proof of \Cref{sec3:thm:sepcondU}, the permuted test statistic $U_{n_1, n_2}^\pi$ is mean zero. Hence, we may appeal to the results \Cref{app:prop:Ustatconc} and \Cref{app:prop:Ustatpermconc} on the concentration properties of $U$-statistics to see that $U_{n_1, n_2}$ and $U_{n_1, n_2}^{\pi}$ are $\mathrm{SE}(\Sigma_{\vecbf{l}^\ast})$ and $\mathrm{SE}(\widetilde{\Sigma}_{\vecbf{l}^\ast})$ respectively, where
                \begin{equation*}
                    \Sigma_{\vecbf{l}^\ast} = C_d'\max\bigg\{\frac{1}{n_1\varepsilon^2}, \frac{|\theta_{X, \vecbf{l}^\ast} - \theta_{Y, \vecbf{l}^\ast}|}{(n_1\varepsilon^2)^{1/2}} \bigg\}, \;\;  \widetilde{\Sigma}_{\vecbf{l}^\ast} = \widetilde{C}_d'\max\bigg\{\frac{1}{n_1\varepsilon^2}, \frac{|\theta_{X, \vecbf{l}^\ast} - \theta_{Y, \vecbf{l}^\ast}|}{(n_1\varepsilon^2)^{1/2}}, \frac{(\theta_{X, \vecbf{l}^\ast} - \theta_{Y, \vecbf{l}^\ast})^2}{n_1} \bigg\},
                \end{equation*}
                for $C_d', \widetilde{C}_d' > 0$ some constants depending on $d$. In particular, we note for $\vecbf{l} \notin \mathcal{L}$ that $(\theta_{X, \vecbf{l}} - \theta_{Y, \vecbf{l}})^2 = \mathbb{E}[U_{n_1, n_2} \mid \vecbf{l}^\ast = \vecbf{l}] \geq C_d''\widetilde{\Sigma}_{\vecbf{l}}\log\{16/(\alpha\beta)\} + C_d''\Sigma_{\vecbf{l}}\log(8/\beta)$ with $C_d'' > 0$ some constant for $C_{r, s, d}$ taken sufficiently large. Proceeding similarly to the proof of \Cref{sec3:prop:permtestcontrol}, setting $t_\vecbf{l} = \widetilde{\Sigma}_{\vecbf{l}}\log\{16/(\alpha\beta)\}$, we obtain
                \begin{align}
                    \sum_{\vecbf{l} \in \mathbb{N}_0^d(R) \setminus \mathcal{L}}\bigg[&\frac{1 + B \mathbb{P}(U_{n_1, n_2}^{\pi_1} \geq \widetilde{\Sigma}_{\vecbf{l}}\log\{16/(\alpha\beta)\} \mid \vecbf{l}^\ast = \vecbf{l})}{(1+B)\alpha/2} \nonumber \\
                    &\hspace{2cm}+ \mathbb{P}(U_{n_1, n_2} - \mathbb{E}[U_{n_1, n_2} \mid \vecbf{l}^\ast = \vecbf{l}]< -\Sigma_{\vecbf{l}}\log(8/\beta) \mid \vecbf{l}^\ast = \vecbf{l})\bigg] \mathbb{P}(\vecbf{l}^\ast = \vecbf{l}) \nonumber \\
                    &\leq \frac{1 +\alpha\beta B/16}{(1+B)\alpha/2} + \beta/4
                    \leq 3\beta/8, \label{app:eq:contintsecondterms1}
                \end{align}
                where the second inequality is by \Cref{app:prop:Ustatconc}, \Cref{app:prop:Ustatpermconc} and Proposition~2.7.1~$(i)$ in \cite{Vershynin:2018:HDPBook}, and the last inequality from the fact that $B \geq 4/(\alpha\beta/2) - 1 = 8/(\alpha\beta) - 1$.

                \noindent
                \textbf{Term} (II):  
                Recall that each \beame{privatised}{privatized}{privatized} value $Z_i$ and $W_{i'}$ is $\mathrm{SG}(5\cdot2^d/\varepsilon^2)$. Hence, the estimators $\hat{\theta}_{X, \vecbf{l}}, \hat{\theta}_{Y, \vecbf{l}}$ which are the average of at least $n_1/V$ i.i.d.~copies of such variables are $\mathrm{SG}(5\cdot2^dV/(n_1\varepsilon^2))$,
                and so the differences $\hat{\theta}_{X, \vecbf{l}} - \hat{\theta}_{Y, \vecbf{l}}$ are $\mathrm{SG}(5\cdot2^{d+1}V/(n_1\varepsilon^2))$ for all $\vecbf{l} \in \mathbb{N}_0^d(R)$. Hence, we have for all $\vecbf{l} \in \mathbb{N}_0^d(R)$ that
                \begin{align*}
                    \mathbb{P}(|\hat{\theta}_{X, \vecbf{l}} - \hat{\theta}_{Y, \vecbf{l}} - (\theta_{X, \vecbf{l}} - \theta_{Y, \vecbf{l}})| > x) \leq 2\exp\bigg( -\frac{n_1\varepsilon^2 x^2}{5\cdot2^{d+2}V}\bigg).
                \end{align*}
                Setting $\delta = \{C_{r, s, d} V\log(4n_1/\beta))/(16n_1\varepsilon^2)\}^{1/2}$, we have that
                \begin{equation*}
                    \mathbb{P}(|\hat{\theta}_{X, \vecbf{l}} - \hat{\theta}_{Y, \vecbf{l}} - (\theta_{X, \vecbf{l}} - \theta_{Y, \vecbf{l}})| > \delta) \leq 2\{\beta/(4n_1)\}^{C_{r, s, d}/(5\cdot2^{d+6})}.
                \end{equation*}
                Using the inequality $(x+y)^{1/2} \geq (x^{1/2}+y^{1/2})/2$ for $x,y > 0$, we note that for $\vecbf{l} \in \mathcal{L}$ we have that $|\theta_{X, \vecbf{l}^\dagger} - \theta_{Y, \vecbf{l}^\dagger}| - |\theta_{X, \vecbf{l}} - \theta_{Y, \vecbf{l}}| > 2\delta$ and so by the union bound
                \begin{equation}
                    \mathbb{P}(A^c)
                    \leq 4|\mathcal{L}|\{\beta/(4n_1)\}^{C_{r, s, d}/(5\cdot2^{d+6})}
                    \leq 4V\{\beta/(4n_1)\}^{C_{r, s, d}/(5\cdot2^{d+6})} \leq \beta/2, \label{app:eq:contintsecondterms2}
                \end{equation}
                where the final inequality comes from taking $C_{r, s, d}$ large enough.
                
                Hence, combining \eqref{app:eq:contintsecondterms}, \eqref{app:eq:contintsecondterms1} and \eqref{app:eq:contintsecondterms2}, we have $\mathbb{P}(p_B > \alpha/2) \leq \beta$, which completes the proof.
            \end{proof}

        \subsection{Adaptive Tests} \label{app:sec:adaptivetestproofs}
        
            \begin{proof}[Proof of \Cref{sec5:thm:adaptive}]
                Recall the adaptive test $\phi_{\mathrm{adapt}}$ and the tests $\{\phi_k\}_{k \in [k_{\mathrm{max}}]}$ as constructed in \Cref{sec5:UpperBoundNonIntAdapt}. A type-I error guarantee of $\alpha$ holds immediately from the union bound, and so it remains to prove we have the desired power. It suffices to show that there exists some $k^\ast \in [k_{\mathrm{\max}}]$ such that $\phi_{k^\ast}$ has type-II error at most $\beta$, and hence $\mathbb{P}(\phi_{\mathrm{adaptNI}} = 0) \leq \mathbb{P}(\phi_{k^\ast} = 0) \leq \beta$. We consider the non-interactive and interactive cases separately as follows.

                \noindent
                \textbf{Non-interactive Case}:
                We start by noting
                \begin{equation*}
                    k_{\mathrm{max}} = (2/3)\log_2(n_1\varepsilon^2) + 1
                    \geq \frac{1}{2s+3d/2}\log_2\bigg( \frac{n_1\varepsilon^2}{k_\mathrm{\max}[\log\{k_\mathrm{\max}/(\alpha\beta)\}]^2}\bigg)  + 1.
                \end{equation*}
                Hence, there exists $k^\ast \in [k_{\mathrm{max}}]$ such that
                \begin{equation} \label{app:eq:NIkstar}
                    \bigg(\frac{n_1\varepsilon^2}{k_\mathrm{\max}[\log\{k_\mathrm{\max}/(\alpha\beta)\}]^2}\bigg)^{1/(2s+3d/2)}
                    \leq 2^{k^\ast}
                    \leq 2\bigg(\frac{n_1\varepsilon^2}{k_\mathrm{\max}[\log\{k_\mathrm{\max}/(\alpha\beta)\}]^2}\bigg)^{1/(2s+3d/2)}.
                \end{equation}
                
                By the proof of \Cref{sec5:thm:main} for the non-interactive case, we see that the test $\phi_{k^\ast}$ has type-II error control $\beta$ when
                \begin{align*}
                    \|f_X - f_Y\|_2^2
                    &\geq \frac{C_{r, s, d}(2^{k^\ast})^{3d/2}[\log\{k_{\mathrm{max}}/(\alpha\beta)\}]^2}{(n_1/k_{\mathrm{max}})\varepsilon^2} + \frac{r^2}{(2^{k^\ast})^{2s}},
                \end{align*}
                for $C_{r, s, d} > 0$ some constant depending on $d$. Using \eqref{app:eq:NIkstar} to upper bound the right-hand-side, we obtain
                \begin{align*}
                    \frac{C_{r, s, d}(2^{k^\ast})^{3d/2}[\log\{k_{\mathrm{max}}/(\alpha\beta)\}]^2}{(n_1/k_{\mathrm{max}})\varepsilon^2} &+ \frac{r^2}{(2^{k^\ast})^{2s}} \\
                    &\leq 2(2C_{r, s, d}+r^2)\bigg(\frac{n_1\varepsilon^2}{k_{\mathrm{max}}[\log\{k_\mathrm{max}/(\alpha\beta)\}]^2} \bigg)^{-2s/(2s+3d/2)}\\
                    &\leq C_{r, s, d}'\bigg(\frac{n_1\varepsilon^2}{\log(n_1\varepsilon^2)[\log\{\log(n_1\varepsilon^2)/(\alpha\beta)\}]^2}\bigg)^{-2s/(2s+3d/2)},
                \end{align*}
                where $C_{r, s, d}' > 0$ is some constant depending on $r, s$ and $d$. As with the proof of \Cref{sec5:thm:main}, the result for $L_p$-norm with $p \in [1, 2]$ then follows from H\"{o}lder's inequality. This completes the proof for the non-interactive case.

                \noindent
                \textbf{Interactive case}:
                We start by noting
                \begin{equation*}
                    k_{\mathrm{max}} = \log_2(n_1\varepsilon^2) + 1
                    \geq \frac{1}{2s+d}\log_2\bigg( \frac{n_1\varepsilon^2}{k_\mathrm{\max}[\log(n_2/\beta)\log\{k_\mathrm{\max}/(\alpha\beta)\}]^2}\bigg)  + 1.
                \end{equation*}
                Hence, there exists $k^\ast \in [k_{\mathrm{max}}]$ such that
                \begin{equation} \label{app:eq:Ikstar}
                    \begin{aligned}
                        \bigg(\frac{n_1\varepsilon^2}{k_\mathrm{\max}[\log(n_2/\beta)\log\{k_\mathrm{\max}/(\alpha\beta)\}]^2}\bigg)^{1/(2s+d)}
                        &\leq 2^{k^\ast} \\
                        &\leq 2\bigg(\frac{n_1\varepsilon^2}{k_\mathrm{\max}[\log(n_2/\beta)\log\{k_\mathrm{\max}/(\alpha\beta)\}]^2}\bigg)^{1/(2s+d)}.
                    \end{aligned}
                \end{equation}
                
                By the proof of \Cref{sec5:thm:main} for the interactive case, we see that the test $\phi_{k^\ast}$ has type-II error control $\beta$ when
                \begin{align*}
                    \|f_X - f_Y\|_2^2
                    &\geq \frac{\widetilde{C}_{r, s, d}2^{dk^\ast}[\log(n_2/\beta)\log\{k_{\mathrm{max}}/(\alpha\beta)\}]^2}{(n_1/k_{\mathrm{max}})\varepsilon^2} + \frac{r^2}{(2^{k^\ast})^{2s}},
                \end{align*}
                for $\widetilde{C}_{r, s, d} > 0$ some constant depending on $d$. Using \eqref{app:eq:Ikstar} to upper bound the right-hand-side,
                \begin{align*}
                    &\frac{\widetilde{C}_{r, s, d}2^{dk^\ast}[\log(n_2/\beta)\log\{k_{\mathrm{max}}/(\alpha\beta)\}]^2}{(n_1/k_{\mathrm{max}})\varepsilon^2} + \frac{r^2}{(2^{k^\ast})^{2s}} \\
                    &\leq 2(2\widetilde{C}_{r, s, d}+r^2)\bigg(\frac{n_1\varepsilon^2}{k_{\mathrm{max}}[\log(n_2/\beta)\log\{k_\mathrm{max}/(\alpha\beta)\}]^2} \bigg)^{-2s/(2s+d)}\\
                    &\leq \widetilde{C}_{r, s, d}' \bigg(\frac{n_1\varepsilon^2}{\log(n_1\varepsilon^2)[\log(n_2/\beta)\log\{\log(n_1\varepsilon^2)/(\alpha\beta)\}]^2}\bigg)^{-2s/(2s+d)},
                \end{align*}
                where $\widetilde{C}_{r, s, d}' > 0$ is some constant depending on $r, s$ and $d$. As with the proof of \Cref{sec5:thm:main}, the result for $L_p$-norm with $p \in [1, 2]$ then follows from H\"{o}lder's inequality. This completes the proof for the interactive case.
            \end{proof}

\section{Tail Bound Proofs (\texorpdfstring{$U$}{U}-Statistic)} \label{app:sec:tailbounds1}
    \subsection{Tail Bound for \texorpdfstring{$U$}{U}-Statistic}
        \begin{proof}[Proof of \Cref{app:prop:Ustatconc}]
            Throughout, we assume $n_1 \geq 2$, and so use the fact that $n_1 - 1 \geq n_1/2$. We first reduce to the case with equal sample sizes: Consider $\vecbf{L} = (l_1, \hdots, l_{n_1}) \subseteq [n_2]$ an $n_1$-tuple drawn without replacement uniformly over the set $[n_2]$, and define
            \begin{align*}
                U_{n_1}^\vecbf{L} 
                = \frac{1}{n_1(n_1 - 1)}\sum_{(i,j) \in \mathcal{I}_2^{n_1}} (\vecbf{Z}_i - \vecbf{W}_{l_i})^T(\vecbf{Z}_j - \vecbf{W}_{l_j}).
            \end{align*}
            Hence, we decompose the moment generating function as
            \begin{align*}
                \mathbb{E}[\exp\{\theta (U_{n_1, n_2} - \mathbb{E}[U_{n_1, n_2}])\}]
                &= \mathbb{E}\big[\exp\{\theta (\mathbb{E}[U_{n_1}^\vecbf{L} \mid \mathcal{D}_{n_1, n_2}] - \mathbb{E}[U_{n_1, n_2}])\}\big] \\
                &\leq \mathbb{E}\big[\exp\{\theta (U_{n_1}^\vecbf{L} - \mathbb{E}[U_{n_1}^\vecbf{L}]\})\big]
            \end{align*}
            where the equality follows from the fact that $U_{n_1, n_2} = \mathbb{E}[U_{n_1}^\vecbf{L} \mid \mathcal{D}_{n_1, n_2}]$ by construction, and the inequality is from Jensen's inequality and the fact that $\mathbb{E}[U_{n_1}^\vecbf{L}] = \mathbb{E}[U_{n_1, n_2}]$. Hence, if $U_{n_1}^\vecbf{L}$ is $\mathrm{SE}(\Sigma)$, then $U_{n_1, n_2}$ is $\mathrm{SE}(\Sigma)$ also.

            It remains to prove that $U_{n_1}^\vecbf{L}$ is $\mathrm{SE}(\Sigma)$ for some $\Sigma > 0$. Denote $\vecbfm{\mu} = \mathbb{E}[\vecbf{Z}_1 - \vecbf{W}_1]$, and $\vecbf{V}_i^\vecbf{L} = \vecbf{Z}_i - \vecbf{W}_{l_i} - \vecbfm{\mu}$. We decompose the tail probability into two terms
            \begin{equation} \label{app:eq:Ustattaildecomp}
                \mathbb{P}\big( \big| U_{n_1}^\vecbf{L} - \mathbb{E}[U_{n_1}^\vecbf{L}] \big| \geq x \big)
                \leq \mathbb{P}\bigg( \frac{1}{n_1(n_1-1)}  \bigg| \sum_{i \neq j} (\vecbf{V}_i^\vecbf{L})^T \vecbf{V}_j^\vecbf{L} \bigg| \geq x/2 \bigg)
                + \mathbb{P}\bigg( \frac{2}{n_1}  \bigg|\sum_{j = 1}^{n_1} \vecbfm{\mu}^T \vecbf{V}_j^\vecbf{L} \bigg| \geq x/2 \bigg),
            \end{equation}
            which follows from the decomposition
            \begin{equation*}
                U_{n_1}^\vecbf{L} - \mathbb{E}[U_{n_1}^\vecbf{L}] = \frac{1}{n_1(n_1-1)}\sum_{i \neq j} (\vecbf{V}_i^\vecbf{L})^T \vecbf{V}_j^\vecbf{L}  + 2n_1^{-1} \sum_{j = 1}^{n_1} \vecbfm{\mu}^T \vecbf{V}_j^\vecbf{L}.
            \end{equation*}

            For the first term, conditional on $\vecbf{L}$, the collection $\{\vecbf{V}_j^\vecbf{L}\}_{j \in [n_1]}$ consists of independent mean-zero $\mathrm{SG}(2\sigma^2)$ random vectors by the proposition statement and \Cref{app:lem:IndSGSum}. Hence, by Exercise~6.2.7 in \cite{Vershynin:2018:HDPBook}, we have, for $c > 0$ some absolute constant, that
            \begin{align}
                \mathbb{P}\bigg( \frac{1}{n_1(n_1-1)} \bigg| \sum_{i \neq j} (\vecbf{V}_i^\vecbf{L})^T \vecbf{V}_j^\vecbf{L} \bigg| \geq x/2 &\biggm| \vecbf{L} \bigg) \nonumber \\
                &\leq 2\exp\bigg( -c\min\bigg\{ \frac{n_1(n_1-1)x^2}{d\sigma^4}, \frac{\{n_1(n_1-1)\}^{1/2}x}{\sigma^2}\bigg\}\bigg) \nonumber \\
                &\leq 2\exp\bigg( -c\min\bigg\{ \frac{n_1(n_1-1)x^2}{d\sigma^4}, \frac{\{n_1(n_1-1)\}^{1/2}x}{d^{1/2}\sigma^2}\bigg\}\bigg) \nonumber \\
                &\leq 4\exp\bigg( -\frac{c n_1 x}{d^{1/2}\sigma^2}\bigg), \label{app:eq:Ustattaildecompterm1}
            \end{align}
            where the final inequality follows from the fact that $\exp(-x^2) \leq 2\exp(-x)$ for all $x \geq 0$.
            
            For the second term, conditional on $\vecbf{L}$, we note that $\{\vecbfm{\mu}^T \vecbf{V}_j^\vecbf{L}\}_{j \in [n_1]}$ is a collection of univariate independent mean-zero $\mathrm{SG}(2\|\vecbfm{\mu}\|_2^2\sigma^2)$ random variables. Hence, by \citet[e.g.~Proposition~2.5][]{Wainwright:2019:HDSBook}, we have that
            \begin{equation} \label{app:eq:Ustattaildecompterm2}
                \mathbb{P}\bigg( \frac{2}{n_1}  \bigg|\sum_{j = 1}^{n_1} \vecbfm{\mu}^T \vecbf{V}_j^\vecbf{L} \bigg| \geq x/2 \biggm| \vecbf{L} \bigg)
                \leq 2\exp\bigg( -\frac{c'n_1x^2}{\|\vecbfm{\mu}\|_2^2 \sigma^2}\bigg)
                \leq 4\exp\bigg( -\frac{c'n_1^{1/2}x}{\|\vecbfm{\mu}\|_2 \sigma} \bigg),
            \end{equation}
            for $c' > 0$ some absolute constant.
            
            Combining \eqref{app:eq:Ustattaildecomp}, \eqref{app:eq:Ustattaildecompterm1} and \eqref{app:eq:Ustattaildecompterm2}, and noting $\| \vecbfm{\mu} \|_2^2 = \mathbb{E}[U_{n_1, n_2}]$ we see that 
            \begin{align*}
                 \mathbb{P}\big( \big| U_{n_1}^\vecbf{L} - \mathbb{E}[U_{n_1}^\vecbf{L}] \big| \geq x \big)
                &\leq 4\exp\bigg( -\frac{c n_1 x}{d^{1/2}\sigma^2}\bigg)
                + 4\exp\bigg( -\frac{c'n_1^{1/2}x}{(\mathbb{E}[U_{n_1, n_2}])^{1/2} \sigma}\bigg) \\
                &\leq 8\exp\bigg( -\min\bigg\{ \frac{cn_1}{d^{1/2}\sigma^2}, \frac{c'n_1^{1/2}}{(\mathbb{E}[U_{n_1, n_2}])^{1/2} \sigma} \bigg\}x \bigg).
            \end{align*}
            Finally, appealing to \Cref{app:lem:SETailtoMGF} with the above tail bound, we see that
            $U_{n_1}^\vecbf{L}$ is $\mathrm{SE}(\Sigma)$ where
            \begin{equation*}
                \Sigma = C'\max\bigg\{ \frac{d^{1/2}\sigma^2}{n_1}, \frac{(\mathbb{E}[U_{n_1, n_2}])^{1/2} \sigma}{n_1^{1/2}}\bigg\} 
            \end{equation*}
            for $C' > 0$ an absolute constant. This completes the proof.
        \end{proof}
        
    \subsection{Tail Bound for Permuted \texorpdfstring{$U$}{U}-Statistic}
        \begin{proof}[Proof of \Cref{app:prop:Ustatpermconc}]
            Throughout, we assume $n_1 \geq 2$, and so use the fact that $n_1 - 1 \geq n_1/2$. We first recall the notation we use. Denote the combined sample $\mathcal{D}_{n_1, n_2} = \{\vecbf{D}_{1}, \hdots, \vecbf{D}_{n_1 + n_2}\}$ where $\vecbf{D}_i = \vecbf{Z}_i$ for $i \in [n_1]$ and $\vecbf{D}_{n_1 + i'} = \vecbf{W}_{i'}$ for $i' \in [n_2]$, and we denote by $\pi \in S_{n_1 + n_2}$ a permutation sampled uniformly at random from the collection of permutations on $n_1 + n_2$ symbols.
            
            We reduce to the case with equal sample sizes. Consider $\vecbf{L} = (l_1, \hdots, l_{n_1}) \subseteq [n_2]$ an $n_1$-tuple drawn without replacement uniformly over the set $[n_2]$, and define
            \begin{equation} \label{app:eq:UpermMGFbound}
                U_{n_1}^{\vecbf{L}, \pi} 
                = \frac{1}{n_1(n_1 - 1)}\sum_{(i,j) \in \mathcal{I}_2^{n_1}} (\vecbf{D}_{\pi(i)} - \vecbf{D}_{\pi(n_1 + l_i)})^T(\vecbf{D}_{\pi(j)} - \vecbf{D}_{\pi(n_1 + l_j)}).
            \end{equation}
            Note in particular that $\pi$ and $\vecbf{L}$ are independent and that any of $\pi(i), \pi(n_1 + l_i), \pi(j), \pi(n_1 + l_j)$ for $i \in [n_1]$ may map to an index outside those considered in the sum defining $U_{n_1}^{\vecbf{L}, \pi}$. Noting that $\mathbb{E}[U_{n_1, n_2}^{\pi}] = 0$ as shown in the proof of \Cref{sec3:thm:sepcondU}, we consider the moment generating function
            \begin{align*}
                \mathbb{E}[\exp\{\lambda U_{n_1, n_2}^{\pi}\}]
                &= \mathbb{E}\big[\exp\{\lambda \mathbb{E}[U_{n_1}^\vecbf{L, \pi} \mid \mathcal{D}_{n_1, n_2}, \pi]\}\big]\\
                &\leq \mathbb{E}\big[ \mathbb{E}[\exp\{\lambda (U_{n_1}^\vecbf{L, \pi} - E_{n_1}^{\vecbf{L}, \pi}\}) \mid \vecbf{L}, \pi] \exp(\lambda E_{n_1}^{\vecbf{L}, \pi})\big]
            \end{align*}
            for $\lambda \in \mathbb{R}$, where we denote $E_{n_1}^{\vecbf{L}, \pi} = \mathbb{E}[U_{n_1}^{\vecbf{L}, \pi} \mid \vecbf{L}, \pi]$ for brevity; the equality holds as $U_{n_1, n_2}^{\pi} = \mathbb{E}[U_{n_1}^{\vecbf{L}, \pi} \mid \mathcal{D}_{n_1, n_2}, \pi]$ by construction; and the inequality by an application of Jensen's inequality and the tower property.

            To bound the two expectation terms in the previous display, we use two lemmata we delay to the end of this appendix, and use them to complete the present proof. Indeed, combining \eqref{app:eq:UpermMGFbound}, \Cref{app:lem:centredpermstattail}, \Cref{app:lem:permstattail}, and the respective conditions on the range of $|\lambda|$ by the sub-exponential conditions, we obtain for $|\lambda| \leq (C'' \max\{ d^{1/2}\sigma^2/n_1, \mathbb{E}[U_{n_1, n_2}]^{1/2}\sigma/n_1^{1/2}, \mathbb{E}[U_{n_1, n_2}]/n_1\})^{-1}$ that
            \begin{align*}
                \mathbb{E}[\exp\{\lambda U_{n_1, n_2}^{\pi}\}]
                &\leq \exp\bigg( \lambda^2 C^2\max\bigg\{ \frac{d\sigma^4}{n_1^2}, \frac{\mathbb{E}[U_{n_1, n_2}]\sigma^2}{n_1} \bigg\} \bigg) \exp\bigg( \lambda^2(C')^2\frac{\mathbb{E}[U_{n_1, n_2}]^2}{n_1^2} \bigg)\\ 
                &\leq \exp\bigg( \lambda^2 (C'')^2 \max\bigg\{ \frac{d\sigma^4}{n_1^2}, \frac{\mathbb{E}[U_{n_1, n_2}]\sigma^2}{n_1}, \frac{\mathbb{E}[U_{n_1, n_2}]^2}{n_1^2} \bigg\} \bigg),
            \end{align*}
            for $C'' > 0$ some absolute constant. This completes the proof.
        \end{proof}

        We now provide the statement and proofs of \Cref{app:lem:centredpermstattail} and \Cref{app:lem:permstattail}, using the same notation and definitions as in the previous proof.

        \begin{lemma} \label{app:lem:centredpermstattail}
            Conditional on $\vecbf{L}$ and $\pi$, $U_{n_1}^{\vecbf{L}, \pi}$ is $\mathrm{SE}(C\max\{d^{1/2}\sigma^2/n_1, \mathbb{E}[U_{n_1, n_2}]^{1/2}\sigma/n_1^{1/2}\})$ almost surely for $C > 0$ some absolute constant.
        \end{lemma}

        \begin{proof}
            Denote $\vecbfm{\mu}_i^{\vecbf{L}, \pi} = \mathbb{E}[\vecbf{D}_{\pi(i)} - \vecbf{D}_{\pi(n_1 + l_i)} \mid \vecbf{L}, \pi]$ and $\vecbf{V}_i^{\vecbf{L}, \pi} = \vecbf{D}_{\pi(i)} - \vecbf{D}_{\pi(n_1 + l_i)} - \vecbfm{\mu}_i^{\vecbf{L}, \pi}$. We have the decomposition
            \begin{align*}
                &U_{n_1}^{\vecbf{L}, \pi} - \mathbb{E}[U_{n_1}^{\vecbf{L}, \pi} \mid \vecbf{L}, \pi] \\
                &= \frac{1}{n_1(n_1-1)} \sum_{i \neq j} (\vecbf{V}_i^{\vecbf{L}, \pi})^T \vecbf{V}_j^{\vecbf{L}, \pi}
                + \frac{2}{n_1(n_1-1)} \sum_{i \neq j} (\vecbfm{\mu}_i^{\vecbf{L}, \pi})^T \vecbf{V}_j^{\vecbf{L}, \pi} \\
                &= \frac{1}{n_1(n_1-1)} \sum_{i \neq j} (\vecbf{V}_i^{\vecbf{L}, \pi})^T \vecbf{V}_j^{\vecbf{L}, \pi}
                + \frac{2}{n_1-1} \bigg(\frac{1}{n_1}\sum_{i = 1}^{n_1} \vecbfm{\mu}_i^{\vecbf{L}, \pi} \bigg)^T \bigg(\sum_{j = 1}^{n_1} \vecbf{V}_j^{\vecbf{L}, \pi}\bigg) \\
                &\hspace{9cm}- \frac{2}{n_1(n_1-1)} \sum_{i = 1}^{n_1} (\vecbfm{\mu}_i^{\vecbf{L}, \pi})^T \vecbf{V}_i^{\vecbf{L}, \pi}.
            \end{align*}

            We thus consider the tail probability using this decomposition, obtaining
            \begin{align}
                \mathbb{P}\big( \big| U_{n_1}^{\vecbf{L}, \pi} - \mathbb{E}[U_{n_1}^{\vecbf{L}, \pi} \mid \vecbf{L}, \pi] \big| &\geq x \mid \vecbf{L}, \pi \big) \nonumber \\
                &\leq \mathbb{P}\bigg( \frac{1}{n_1(n_1-1)}  \bigg| \sum_{i \neq j} (\vecbf{V}_i^{\vecbf{L}, \pi})^T \vecbf{V}_j^{\vecbf{L}, \pi} \bigg| \geq x/3 \biggm| \vecbf{L}, \pi \bigg) \nonumber \\
                &\hspace{1cm}+ \mathbb{P}\bigg( \frac{2}{n_1-1} \bigg| \bigg(\frac{1}{n_1}\sum_{i = 1}^{n_1} \vecbfm{\mu}_i^{\vecbf{L}, \pi} \bigg)^T \bigg(\sum_{j = 1}^{n_1} \vecbf{V}_j^{\vecbf{L}, \pi}\bigg) \bigg| \geq x/3 \biggm| \vecbf{L}, \pi\bigg) \nonumber \\
                &\hspace{1cm}+ \mathbb{P}\bigg( \frac{2}{n_1(n_1-1)} \bigg| \sum_{i = 1}^{n_1} (\vecbfm{\mu}_i^{\vecbf{L}, \pi})^T \vecbf{V}_i^{\vecbf{L}, \pi} \bigg| \geq x/3 \biggm| \vecbf{L}, \pi \bigg) \nonumber \\
                &= (\mathrm{I}) + (\mathrm{II}) + (\mathrm{III}). \label{app:eq:centredtermsum}
            \end{align}

            We now \beame{analyse}{analyse}{analyze} the three terms separately to obtain upper bounds, where we omit the statement that they hold almost surely for brevity. Before proceeding, we first make note of the following observation used in all three cases: Conditional on a fixed tuple $\vecbf{L}$ and permutation $\pi$, we have that for $i, j \in [n_1]$ with $i \neq j$, the random variables $\vecbf{D}_{\pi(i)}, \vecbf{D}_{\pi(n_1 + l_j)}$ are independent, and the collection $\{\vecbf{V}_i^{\vecbf{L}, \pi}\}_{i \in [n_1]}$ consists of $\mathrm{SG}(2\sigma^2)$ random vectors by \Cref{app:lem:IndSGSum}.

            \noindent
            \textbf{Term} I:
                As the collection $\{\vecbf{V}_i^{\vecbf{L}, \pi}\}_{i \in [n_1]}$ consists of independent $\mathrm{SG}(2\sigma^2)$ random vectors, we can appeal to the same argument leading to \eqref{app:eq:Ustattaildecompterm1}, from which we obtain
                \begin{equation} \label{app:eq:centredterm1}
                    \mathbb{P}\bigg( \frac{1}{n_1(n_1-1)}  \bigg| \sum_{i \neq j} (\vecbf{V}_i^{\vecbf{L}, \pi})^T \vecbf{V}_j^{\vecbf{L}, \pi} \bigg| \geq x/3 \biggm| \vecbf{L}, \pi \bigg)
                    \leq 4\exp\bigg( -\frac{cn_1x}{d^{1/2}\sigma^2}\bigg).
                \end{equation}
                for $c > 0$ some absolute constant.

            \noindent
            \textbf{Term} II:
                Let $\bar{\vecbfm{\mu}}^{\vecbf{L}, \pi} = n_1^{-1}\sum_{i = 1}^{n_1} \vecbfm{\mu}_i^{\vecbf{L}, \pi}$ and let $\vecbfm{\mu} = \mathbb{E}[\vecbf{Z}_1 - \vecbf{W}_1]$. Noting that the trivial bound $\| \bar{\vecbfm{\mu}}^{\vecbf{L}, \pi} \|_2 \leq \| \vecbfm{\mu} \|_2$ holds almost surely, we have that conditional on $\vecbf{L}$ and $\pi$, that the sum $(\bar{\vecbfm{\mu}}^{\vecbf{L}, \pi})^T (\sum_{j = 1}^{n_1} \vecbf{V}_j^{\vecbf{L}, \pi})$ is $\mathrm{SG}(2n_1\| \vecbfm{\mu} \|_2^2\sigma^2)$ by \Cref{app:lem:IndSGSum}. This yields,
                \begin{align} 
                    \mathbb{P}\bigg( \frac{2}{n_1-1} \bigg| \bigg(\frac{1}{n_1}\sum_{i = 1}^{n_1} \vecbfm{\mu}_i^{\vecbf{L}, \pi} \bigg)^T \bigg(\sum_{j = 1}^{n_1} \vecbf{V}_j^{\vecbf{L}, \pi}\bigg) \bigg| \geq x/3 \biggm| \vecbf{L}, \pi \bigg)
                    &\leq 2\exp\bigg( - \frac{c'n_1x^2}{\| \vecbfm{\mu} \|_2^2\sigma^2} \bigg) \nonumber \\
                    &\leq 4\exp\bigg( - \frac{c''n_1^{1/2}x}{\| \vecbfm{\mu} \|_2\sigma} \bigg) \label{app:eq:centredterm2}
                \end{align}
                where $c', c'' > 0$ are absolute constants, and where the final inequality follows from the bound $\exp(-x^2) \leq 2\exp(-x)$ valid for all $x \geq 0$.
                
            \noindent
            \textbf{Term} III:
                We first note that we have the trivial bound that $\| \vecbfm{\mu}_i^{\vecbf{L}, \pi} \|_2 \leq \| \vecbfm{\mu} \|_2$ almost surely. Thus, conditional on $\vecbf{L}$ and $\pi$, we have that the sum $\sum_{i = 1}^{n_1} (\vecbfm{\mu}_i^{\vecbf{L}, \pi})^T \vecbf{V}_i^{\vecbf{L}, \pi}$ is $\mathrm{SG}(n_1 \| \vecbfm{\mu} \|_2^2 \sigma^2)$ by \Cref{app:lem:IndSGSum}. This yields
                \begin{align}
                    \mathbb{P}\bigg( \frac{2}{n_1(n_1-1)} \bigg| \sum_{i = 1}^{n_1} (\vecbfm{\mu}_i^{\vecbf{L}, \pi})^T \vecbf{V}_i^{\vecbf{L}, \pi} \bigg| \geq x/3 \biggm| \pi, \vecbfm{L} \bigg)
                    &\leq 2\exp\bigg( - \frac{c'''n_1(n_1 - 1)^2x^2}{\| \vecbfm{\mu} \|_2^2 \sigma^2}\bigg) \nonumber \\
                    &\leq 4\exp\bigg( -\frac{c''''n_1^{3/2}x}{\| \vecbfm{\mu} \|_2 \sigma}\bigg),\label{app:eq:centredterm3}
                \end{align}
                for $c''', c'''' > 0$ absolute constants, where the final inequality follows from the bound $\exp(-x^2) \leq 2\exp(-x)$ for all $x \geq 0$.

            \noindent Combining \eqref{app:eq:centredtermsum}, \eqref{app:eq:centredterm1}, \eqref{app:eq:centredterm2} and \eqref{app:eq:centredterm3}, we obtain
            \begin{align*} 
                \mathbb{P}\big( \big| U_{n_1}^{\vecbf{L}, \pi} &- \mathbb{E}[U_{n_1}^{\vecbf{L}, \pi} \mid \vecbf{L}, \pi] \big| \geq x \mid \vecbf{L}, \pi \big) \nonumber \\
                &\leq 4\exp\bigg( -\frac{cn_1x}{d^{1/2}\sigma^2}\bigg)
                + 4\exp\bigg( - \frac{c''n_1^{1/2}x}{\| \vecbfm{\mu} \|_2\sigma} \bigg)
                + 4\exp\bigg( -\frac{c''''n_1^{3/2}x}{\| \vecbfm{\mu} \|_2 \sigma}\bigg), \\
                &\leq 12\exp\bigg( -\frac{x}{C'\max\{d^{1/2}\sigma^2/n_1, \|\vecbfm{\mu}\|_2\sigma/n_1^{1/2}\}}\bigg),
            \end{align*}
            for $C' > 0$ some absolute constant. Finally, invoking \Cref{app:lem:SETailtoMGF} and noting $\mathbb{E}[U_{n_1, n_2}] = \|\vecbfm{\mu}\|_2^2$ concludes the proof.
        \end{proof}

        \begin{lemma} \label{app:lem:permstattail}
            The random variable $E_{n_1}^{\vecbf{L}, \pi}$
            is $\mathrm{SE}(C'\mathbb{E}[U_{n_1, n_2}]/n_1)$  for $C' > 0$ some absolute constant.
        \end{lemma}

        \begin{proof}
            Denoting $\vecbfm{\mu} = \mathbb{E}[\vecbf{Z}_1 - \vecbf{W}_1]$ and defining the quantities
            \begin{equation} \label{app:eq:permmeanterms}
                \begin{aligned}
                    a_i^{\vecbf{L}, \pi} &= \mathbbm{1}\{\pi(i) \in [n_1],\; \pi(n_1 + l_i) \in [n_1 + n_2] \setminus [n_1]\}, \\
                    b_i^{\vecbf{L}, \pi} &= \mathbbm{1}\{\pi(i) \in [n_1 + n_2] \setminus [n_1],\; \pi(n_1 + l_i) \in [n_1]\}, \mbox{ and} \\
                    d_i^{\vecbf{L}, \pi} &= a_i^{\vecbf{L}, \pi} - b_i^{\vecbf{L}, \pi},
                \end{aligned}
            \end{equation}
            we have
            \begin{align}
                \frac{1}{n_1} \sum_{i = 1}^{n_1} \mathbb{E}[\vecbf{D}_{\pi(i)} &- \vecbf{D}_{\pi(n_1 + l_i)} \mid \vecbf{L}, \pi] \nonumber \\
                &= \frac{1}{n_1}\sum_{i = 1}^{n_1} \big\{ a_i^{\vecbf{L}, \pi} \mathbb{E}[\vecbf{D}_{\pi(i)} - \vecbf{D}_{\pi(n_1 + l_i)} \mid \vecbf{L}, \pi] + b_i^{\vecbf{L}, \pi} \mathbb{E}[\vecbf{D}_{\pi(i)} - \vecbf{D}_{\pi(n_1 + l_i)} \mid \vecbf{L}, \pi] \big\} \nonumber \\
                &= \frac{1}{n_1}\sum_{i = 1}^{n_1} \big\{ a_i^{\vecbf{L}, \pi}\mathbb{E}[\vecbf{Z}_1 - \vecbf{W}_1] + b_i^{\vecbf{L}, \pi}\mathbb{E}[\vecbf{W}_1 - \vecbf{Z}_1] \big\}
                = \bigg( \frac{1}{n_1} \sum_{i = 1}^{n_1} d_i^{\vecbf{L}, \pi} \bigg) \vecbfm{\mu} = \frac{\psi_{n_1}^{\vecbf{L}, \pi}}{n_1} \vecbfm{\mu}. \label{app:eq:permmeanform}
            \end{align}
            where we write $\psi_{n_1}^{\vecbf{L}, \pi} = \sum_{i = 1}^{n_1} d_i^{\vecbf{L}, \pi}$. We now decompose the moment generating function of $E_{n_1}^{\vecbf{L}, \pi}$ as
            \begin{align}
                &\mathbb{E}\big[\exp(\lambda E_{n_1}^{\vecbf{L}, \pi}) \big] \nonumber \\
                &= \mathbb{E}\bigg[\exp\bigg(\frac{\lambda n_1}{n_1 - 1} \bigg\| \frac{1}{n_1} \sum_{i = 1}^{n_1} \mathbb{E}[\vecbf{D}_{\pi(i)} - \vecbf{D}_{\pi(n_1 + l_i)} \mid \vecbf{L}, \pi] \bigg\|_2^2  \bigg) \nonumber \\
                &\hspace{6cm} \cdot \exp\bigg( -\frac{\lambda}{n_1(n_1 - 1)}\sum_{i = 1}^{n_1} \| \mathbb{E}[\vecbf{D}_{\pi(i)} - \vecbf{D}_{\pi(n_1 + l_i)} \mid \vecbf{L}, \pi] \|_2^2 \bigg)\bigg] \nonumber \\
                &= \mathbb{E}\bigg[\exp\bigg(  \frac{\lambda \| \vecbfm{\mu} \|_2^2 (\psi_{n_1}^{\vecbf{L}, \pi})^2}{n_1(n_1-1)}  \bigg)
                \exp\bigg( -\frac{\lambda}{n_1(n_1 - 1)}\sum_{i = 1}^{n_1} \| \mathbb{E}[\vecbf{D}_{\pi(i)} - \vecbf{D}_{\pi(n_1 + l_i)} \mid \vecbf{L}, \pi ]\|_2^2 \bigg)\bigg] \nonumber \\
                &\leq \mathbb{E}\bigg[\exp\bigg(  \frac{2\lambda \| \vecbfm{\mu} \|_2^2(\psi_{n_1}^{\vecbf{L}, \pi})^2}{n_1(n_1-1)}  \bigg) \bigg]^{1/2} \mathbb{E}\bigg[
                \exp\bigg( -\frac{2\lambda}{n_1(n_1 - 1)}\sum_{i = 1}^{n_1} \| \mathbb{E}[\vecbf{D}_{\pi(i)} - \vecbf{D}_{\pi(n_1 + l_i)} \mid \vecbf{L}, \pi] \|_2^2 \bigg)\bigg]^{1/2} \nonumber \\
                &= \mathbb{E}\bigg[ \mathbb{E}\bigg[\exp\bigg(  \frac{2\lambda \| \vecbfm{\mu} \|_2^2 (\psi_{n_1}^{\vecbf{L}, \pi})^2}{n_1(n_1-1)}  \bigg) \biggm| \vecbf{L} \bigg] \bigg]^{1/2} \nonumber \\
                &\hspace{3cm} \cdot\mathbb{E}\bigg[
                \exp\bigg( -\frac{2\lambda}{n_1(n_1 - 1)}\sum_{i = 1}^{n_1} \| \mathbb{E}[\vecbf{D}_{\pi(i)} - \vecbf{D}_{\pi(n_1 + l_i)} \mid \vecbf{L}, \pi] \|_2^2 \bigg)\bigg]^{1/2} \label{app:eq:secondMGFbound}
            \end{align}
            where the first equality is by the definition of $E_{n_1}^{\vecbf{L}, \pi}$; the second equality is by \eqref{app:eq:permmeanform}; the inequality by H\"{o}lder's inequality; and the final line by the tower property. We now aim to bound these two expectation terms.
            
            For the first term, we have for $|\lambda| \leq cn_1/\|\vecbfm{\mu}\|_2^2$ that
            \begin{equation}
                \mathbb{E}\bigg[\exp\bigg(  \frac{2\lambda \| \vecbfm{\mu} \|_2^2 }{n_1(n_1-1)}(\psi_{n_1}^{\vecbf{L}, \pi})^2  \bigg) \biggm| \vecbf{L} \bigg]
                \leq \exp\bigg( \frac{C_1 |\lambda| \|\vecbfm{\mu}\|_2^2}{n_1 - 1} \bigg)
                \leq 2\exp\bigg( \frac{C_2 \lambda^2 \|\vecbfm{\mu}\|_2^4}{n_1^2} \bigg)
                \label{app:eq:secondMGFboundterm1}
            \end{equation}
            for $C_1, C_2 > 0$ absolute constants and $c > 0$ a sufficiently small absolute constant, where the first inequality follows from the fact that, conditional on $\vecbf{L}$, $\psi_{n_1}^{\vecbf{L}, \pi}$ is $\mathrm{SG}(2n_1)$ by \Cref{app:lem:permcontrol} and then using Proposition~2.5.2~$(iii)$ from \cite{Vershynin:2018:HDPBook} to bound the moment generating function of the square of a sub-Gaussian random variable. The second inequality then follows from the bound $\exp(|x|) \leq 2\exp(x^2)$.

            For the second term,
            \begin{align}
                \mathbb{E}\bigg[\exp\bigg( -\frac{2\lambda}{n_1(n_1 - 1)}\sum_{i = 1}^{n_1} \| \mathbb{E}[\vecbf{D}_{\pi(i)} - \vecbf{D}_{\pi(n_1 + l_i)} \mid \vecbf{L}, \pi ] \|_2^2 \bigg)\bigg]
                &\leq \exp\bigg( \frac{C_3|\lambda|\|\vecbfm{\mu}\|_2^2}{n_1 - 1} \bigg) \nonumber \\
                &\leq 2\exp\bigg( \frac{C_4\lambda^2\|\vecbfm{\mu}\|_2^4}{n_1^2} \bigg)
                \label{app:eq:secondMGFboundterm2}
            \end{align}
            for $C_3, C_4 > 0$ absolute constants, where the first inequality is by the fact $\| \mathbb{E}[\vecbf{D}_{\pi(i)} - \vecbf{D}_{\pi(n_1 + l_i)} \mid \vecbf{L}, \pi] \|_2^2 \leq \|\vecbfm{\mu}\|_2^2$ almost surely, and the second inequality from the bound $\exp(|x|) \leq 2\exp(x^2)$.

            Finally, combining \eqref{app:eq:secondMGFbound}, \eqref{app:eq:secondMGFboundterm1}, and \eqref{app:eq:secondMGFboundterm2}
            \begin{align*}
                \mathbb{E}\big[\exp(\lambda E_{n_1}^{\vecbf{L}, \pi}) \big]
                \leq 2 \exp \bigg( \frac{C_5\lambda^2\|\vecbfm{\mu}\|_2^4}{n_1^2} \bigg)
            \end{align*}
            for $|\lambda| \leq n_1/(C_5\|\vecbfm{\mu}\|_2^2)$ where $C_5 > 0$ is some sufficiently large absolute constant. Finally, invoking \Cref{app:lem:SEMGFprefactor} completes the proof.
        \end{proof}

    \section{Tail Bound Proofs (Interactive Procedure - Discrete Distributions)} \label{app:sec:tailbounds2}
        \subsection{Tail Bound for Linear Statistic}
            \begin{proof}[Proof of \Cref{app:lem:linstatSE}]
                We first reduce to the case with equal sample sizes. Let $\vecbf{L} = (l_1, \hdots, l_{n_1}) \subseteq [n_2]$ an $n_1$-tuple drawn without replacement uniformly over the set $[n_2]$, and define
                \begin{align*}
                    T_{n_1}^\vecbf{L} 
                    = \frac{1}{n_1} \sum_{i = 1}^{n_1} (Z_i - W_{l_i}).
                \end{align*}
                Recalling the samples used in the first and second stages of the interactive procedure \eqref{sec4:eq:sampledefs}, denote the combination of the two \beame{privatised}{privatized}{privatized} samples used in the first stage of the estimator via $\mathcal{D}' = \widetilde{\mathcal{D}}_{X, n_1}' \cup \widetilde{\mathcal{D}}_{Y, n_2}'$, and the combination with the second stage via
                $\mathcal{D} = \widetilde{\mathcal{D}}_{X, n_1} \cup \widetilde{\mathcal{D}}_{X, n_1}' \cup \widetilde{\mathcal{D}}_{Y, n_2} \cup \widetilde{\mathcal{D}}_{Y, n_2}'$.
                We decompose the moment generating function of $T_{n_1, n_2}$ as
                \begin{align}
                    \mathbb{E}[&\exp\{\lambda (T_{n_1, n_2} - \mathbb{E}[T_{n_1, n_2}])\}] \nonumber \\
                    &= \mathbb{E}\big[\exp\{\lambda (\mathbb{E}[T_{n_1}^\vecbf{L} \mid \mathcal{D}] - \mathbb{E}[T_{n_1, n_2}])\}\big] \nonumber \\
                    &\leq \mathbb{E}\big[\exp\{\lambda (T_{n_1}^\vecbf{L} - \mathbb{E}[T_{n_1}^\vecbf{L}]\})\big] \nonumber \\
                    &= \mathbb{E}\big[ \mathbb{E}[\exp\{\lambda(T_{n_1}^\vecbf{L} - \mathbb{E}[T_{n_1}^\vecbf{L} \mid \mathcal{D}']\} \mid \mathcal{D}' ] \exp\{\lambda(\mathbb{E}[T_{n_1}^\vecbf{L} \mid \mathcal{D}'] - \mathbb{E}[T_{n_1}^\vecbf{L}])\} \big] \nonumber \\
                    &= \mathbb{E}[\exp\{\lambda(Z_1 - W_{l_1} - \mathbb{E}[Z_1 - W_{l_1} \mid \mathcal{D}'])/n_1\} \mid \mathcal{D}' ]^{n_1} \exp\{\lambda(\mathbb{E}[T_{n_1}^\vecbf{L} \mid \mathcal{D}'] - \mathbb{E}[T_{n_1}^\vecbf{L}])\} \big] \nonumber \\
                    &\leq \exp(2\lambda^2c_\varepsilon^2 \tau^2/n_1)\exp\{\lambda(\mathbb{E}[T_{n_1}^\vecbf{L} \mid \mathcal{D}'] - \mathbb{E}[T_{n_1}^\vecbf{L}])\} \big] \nonumber \\
                    &\leq \exp\{10\lambda^2/(n_1^2\varepsilon^4))\} \exp\{\lambda(\mathbb{E}[T_{n_1}^\vecbf{L} \mid \mathcal{D}'] - \mathbb{E}[T_{n_1}^\vecbf{L}])\} \big], \label{app:eq:DiscLinMGF1}
                \end{align}
                where the first equality is by the fact $T_{n_1, n_2} = \mathbb{E}[T_{n_1}^\vecbf{L} \mid \mathcal{D}]$; the first inequality by Jensen's inequality and the fact that $\mathbb{E}[T_{n_1}^\vecbf{L}] = \mathbb{E}[T_{n_1, n_2}]$; the second equality by the tower property; the third equality by the fact that the collection $\{Z_i - W_{l_i}\}_{i \in [n_1]}$ is conditionally independent and identically distributed given $\mathcal{D}'$; and the final inequality by the fact that $Z_1 - W_{l_1} \in \{0, \pm 2 c_\varepsilon \tau\}$ are bounded random variables and hence $\mathrm{SG}(4c_\varepsilon^2\tau^2)$ (e.g.~Exercise~2.4~in \citealt{Wainwright:2019:HDSBook}), the fact that $\tau = (n_1 \varepsilon^2)^{-1/2}$ and that $c_\varepsilon^2 \leq 5/\varepsilon^2$ for $\varepsilon \in (0,1]$.

                We now consider the remaining expectation term in \eqref{app:eq:DiscLinMGF1}. In particular, we rely on \Cref{app:lem:linearmgfboundlemma}, which we defer to the end of this appendix, to bound the moment generating function. We now complete the present proof. Indeed, combining \eqref{app:eq:DiscLinMGF1} and \Cref{app:lem:linearmgfboundlemma}, we obtain
                \begin{align*}
                    \mathbb{E}[\exp\{\lambda (T_{n_1, n_2} - \mathbb{E}[T_{n_1, n_2}])\}]
                    \leq \exp\bigg\{\frac{10\lambda^2}{n_1^2\varepsilon^4}\bigg\} \exp\bigg\{ \frac{\lambda^2 D_\tau}{c_1n_1\varepsilon^2} \bigg\}
                    \leq \exp\bigg( \frac{C_1^2\lambda^2}{\min\{n_1^2 \varepsilon^4, n_1 \varepsilon^2 / D_\tau\}} \bigg)
                \end{align*}
                for $|\lambda| \leq c_1 n_1 \varepsilon^2$ and $C_1 > 0$ some sufficiently large absolute constant. Lastly, further restricting the range of $\lambda$ as $|\lambda| \leq \min\{n_1 \varepsilon^2, (n\varepsilon^2/D_\tau)^{1/2}\}/C_2$, we see that $T_{n_1, n_2}$ is $\mathrm{SE}(C_2\max\{(n_1 \varepsilon^2)^{-1},\\ D_\tau^{1/2}/(n_1 \varepsilon^2)^{1/2}\})$ for some absolute constant $C_2 > 0$. This completes the proof.
            \end{proof}

            We now provide the following lemma used to control the moment generating function term in \eqref{app:eq:DiscLinMGF1}.
            \begin{lemma} \label{app:lem:linearmgfboundlemma}
                For $|\lambda| \leq c_1 n_1 \varepsilon^2$, the moment generating function of $\mathbb{E}[T_{n_1}^{\vecbf{L}} \mid \mathcal{D}'] - \mathbb{E}[T_{n_1}^\vecbf{L}])$ is bounded as 
                \begin{equation*}
                    \mathbb{E}[\exp\{\lambda(\mathbb{E}[T_{n_1}^\vecbf{L} \mid \mathcal{D}_{n_1, n_2}] - \mathbb{E}[T_{n_1}^\vecbf{L}])\}]
                    \leq \exp\bigg\{ \frac{\lambda^2 D_\tau}{c_1n_1\varepsilon^2} \bigg\},
                \end{equation*}
                where $D_\tau$ is as in \eqref{app:eq:Dtau}, and $c_1 > 0$ is some sufficiently small absolute constant.
            \end{lemma}

            \begin{proof}[Proof of \Cref{app:lem:linearmgfboundlemma}]
                Recalling the quantities defined in \Cref{sec4:UpperBoundInt}, denote $\delta_j = p_{X,j} - p_{Y,j}, \hat{\delta}_{j} = \hat{p}_{X,j} - \hat{p}_{Y,j}$ and $\hat{\delta}_{j, \tau} = \Pi_{[-\tau, \tau]}(\hat{p}_{X,j} - \hat{p}_{Y,j})$. We first show that the collection $\{\hat{\delta}_{j, \tau}\}_{j \in [d]}$ is negatively-associated (NA) as per \Cref{app:def:NA}. By \Cref{app:lem:UENA}, for all $i \in [n_1]$, the collection $\{Z_{i,j}'\}_{j \in [d]}$ is NA, and likewise for $\{W_{i',j}'\}_{j \in [d]}$ with $i' \in [n_2]$. By \Cref{app:lem:NALemma}, we have that the collection $\{-W_{i,j}'\}_{j \in [d]}$ is NA. Hence, by \citet[Property~7]{Joag:1983:NA}, we have that the combined collection $\{Z_{i,j}'\}_{i \in [n_1],\; j \in [d]} \cup \{-W_{i',j}'\}_{i' \in [n_2],\; j \in [d]}$ is NA by the independence of the $\vecbf{Z}'_i$ and $\vecbf{W}'_{i'}$. Lastly, noting the form
                \begin{align*}
                    \hat{\delta}_j
                    = \frac{1}{n_1} \sum_{i = 1}^{n_1} Z_{i,j}' - \frac{1}{n_2} \sum_{i' = 1}^{n_2} W_{i',j}',
                \end{align*}
                each $\hat{\delta}_j$ is a sum, that is, a non-decreasing function, on disjoint subsets of the collection $\{Z_{i,j}'\}_{i \in [n_1],\; j \in [d]} \cup \{-W_{i',j}'\}_{i' \in [n_2],\; j \in [d]}$, and hence by \citet[Property~6]{Joag:1983:NA} the collection $\{\hat{\delta}_j\}_{j \in [d]}$ is NA. We then obtain that the collection of the truncated and \beame{centred}{centred}{centered} values $\{\hat{\delta}_{j, \tau } - \mathbb{E}[\hat{\delta}_{j, \tau }]\}_{j \in [d]}$ is similarly NA by noting that the truncation operator is a non-decreasing function.
                
                Hence, we can decompose the moment generating function as
                \begin{align}
                    \mathbb{E}[\exp\{&\lambda(\mathbb{E}[T_{n_1}^\vecbf{L} \mid \mathcal{D}'] - \mathbb{E}[T_{n_1}^\vecbf{L}])\}]
                    = \mathbb{E}\bigg[\exp \bigg\{\lambda \sum_{j=1}^d \delta_j (\hat{\delta}_{j, \tau} - \mathbb{E}[\hat{\delta}_{j, \tau}])  \bigg\}\bigg] \nonumber \\
                    &\leq  \mathbb{E}\bigg[ \prod_{j : \delta_j > 0} \exp \bigg\{2\lambda \delta_j (\hat{\delta}_{j, \tau} - \mathbb{E}[\hat{\delta}_{j, \tau}])  \bigg\}\bigg]^{1/2} 
                    \mathbb{E}\bigg[ \prod_{j : \delta_j < 0} \exp \bigg\{2\lambda \delta_j (\hat{\delta}_{j, \tau} - \mathbb{E}[\hat{\delta}_{j, \tau}])  \bigg\}\bigg]^{1/2} \nonumber \\
                    &\leq \bigg( \prod_{j : \delta_j > 0} \mathbb{E}\bigg[ \exp \bigg\{2\lambda \delta_j (\hat{\delta}_{j, \tau} - \mathbb{E}[\hat{\delta}_{j, \tau}])  \bigg\}\bigg] \bigg)^{1/2} \bigg( \prod_{j : \delta_j < 0} \mathbb{E}\bigg[ \exp \bigg\{2\lambda \delta_j (\hat{\delta}_{j, \tau} - \mathbb{E}[\hat{\delta}_{j, \tau}])  \bigg\}\bigg] \bigg)^{1/2} \nonumber \\
                    &= \prod_{j=1}^d \mathbb{E}\bigg[ \exp \bigg\{2\lambda \delta_j (\hat{\delta}_{j, \tau} - \mathbb{E}[\hat{\delta}_{j, \tau}])  \bigg\}\bigg]^{1/2}, \label{app:eq:linearmgfboundbyNA}
                \end{align}
                where the first inequality is by H\"{o}lder's inequality, and the second inequality follows by negative association properties as follows: indeed, for $j \in [d]$ such that $\delta_j > 0$, the mapping $(\hat{\delta}_{j, \tau} - \mathbb{E}[\hat{\delta}_{j, \tau}]) \mapsto \delta_j (\hat{\delta}_{j, \tau} - \mathbb{E}[\hat{\delta}_{j, \tau}])$ corresponds to a non-decreasing function and hence the $\{\delta_j (\hat{\delta}_{j, \tau} - \mathbb{E}[\hat{\delta}_{j, \tau}])\}_{j : \delta_j > 0}$ are NA \citep[Property~6]{Joag:1983:NA}. Hence, we may appeal to \citep[Property~2]{Joag:1983:NA} to upper bound the expectation of the product by the product of expectations. A similar argument with non-increasing functions via \Cref{app:lem:UENA} for $j \in [d]$ such that $\delta_j < 0$ yields an analogous result for the second expectation term.  
                
                We focus on an individual term of the product in \eqref{app:eq:linearmgfboundbyNA} for a fixed $j \in [d]$, \beame{analysing}{analysing}{analyzing} the behaviour for two different regimes of the magnitude of $|\delta_j|$.
                
                \noindent
                \textbf{Case 1 ($|\delta_j| \geq 2\tau$):}
                Denote by $\hat{\delta}_{j, \tau}'$ an i.i.d.~copy of $\hat{\delta}_{j, \tau}$ and the event $A_j = \{\hat{\delta}_{j} \leq \tau\} \cup \{\hat{\delta}_{j}' \leq \tau\}$. We have the bound
                \begin{align}
                    \mathbb{E}\bigg[ \exp \bigg\{2\lambda \delta_j (\hat{\delta}_{j, \tau} - \mathbb{E}[\hat{\delta}_{j, \tau}])  \bigg\}\bigg]
                    &\leq \mathbb{E}\bigg[ \exp \bigg\{2\lambda \delta_j (\hat{\delta}_{j, \tau} - \hat{\delta}_{j, \tau}')  \bigg\}\bigg] \nonumber \\
                    &\leq 1 - \mathbb{P}(A_j) + \mathbb{P}(A_j)\mathbb{E}\bigg[ \exp \bigg\{2\lambda \delta_j (\hat{\delta}_{j, \tau} - \hat{\delta}_{j, \tau}')  \bigg\} \biggm| A_j \bigg] \nonumber \\
                    &\leq 1 - \mathbb{P}(A_j) + \mathbb{P}(A_j) \exp(8\lambda^2 \delta_j^2 \tau^2) \nonumber \\
                    &\leq 1 + 8\lambda^2\delta_j^2\tau^2 \exp\{c_1\delta_j^2/(2\tau^2)\} \mathbb{P}(A_j) \nonumber \\
                    &\leq \exp[8\lambda^2\delta_j^2\tau^2 \exp\{c_1\delta_j^2/(2\tau^2)\} \mathbb{P}(A_j)], \label{app:eq:mgfdoublestepbound}
                \end{align}
                where the first inequality is by Jensen's inequality; the third inequality is by the fact that, conditional on $A_j$, $\hat{\delta}_{j, \tau} - \hat{\delta}_{j, \tau}' \in [-2\tau, 2\tau]$ is a bounded mean-zero random variable and hence $\mathrm{SG}(4\tau^2)$ (e.g.~\citealt[Exercise~2.4]{Wainwright:2019:HDSBook}); and the penultimate inequality by the fact that $\exp(x) \leq 1 + xc$ for $x \leq \log(c)$ where $c > 1$, which can be verified by checking the derivative, which holds as $\lambda^2 \leq c_2 n_1^2 \varepsilon^4$ by the lemma statement which in turn implies $8\lambda^2 \delta_j^2 \tau^2 \leq c_1 \delta_j^2/(2\tau^2)$ provided $c_2 > 0$, some absolute constant, is sufficiently small. 
                
                We then note the bound $\mathbb{P}(A_j) \leq 2\exp( -c_1 \delta_j^2/\tau^2 )$ which follows from noting the inclusion of events $\{\hat{\delta}_{j} \leq \tau \} \subseteq \{\hat{\delta}_{j} - \delta_j \leq -\delta_j/2 \}$ and applying the concentration inequality
                \begin{equation*}
                    \mathbb{P}(\hat{\delta}_{j} \leq \tau)
                    \leq \mathbb{P}(\hat{\delta}_{j} - \delta_j \leq -\delta_j/2)
                    \leq \exp( -c_1n_1\varepsilon^2\delta_j^2 )
                    = \exp( -c_1\delta_j^2/\tau^2 ),
                \end{equation*}
                where the second inequality holds due to \citet[e.g.~Proposition~2.5][]{Wainwright:2019:HDSBook}; the fact that $\hat{\delta}_j$ is $\mathrm{SG}(10/\{n_1\varepsilon^2\})$ by \Cref{app:lem:UEEstSG}; and for $c_1 > 0$ taken sufficiently small.
                
                Hence, combining with \eqref{app:eq:mgfdoublestepbound}, we obtain the bound
                \begin{align*}
                    \mathbb{E}\bigg[ \exp \bigg\{2\lambda \delta_j (\hat{\delta}_{j, \tau} - \mathbb{E}[\hat{\delta}_{j, \tau}])  \bigg\}\bigg]
                    \leq \exp[16\lambda^2\delta_j^2\tau^2 \exp\{-c_1\delta_j^2/(2\tau^2)\}].
                \end{align*}
                The case $\delta_j \leq -2\tau$ follows similarly, yielding the bound for all $|\delta_j| \geq 2\tau$.
                
                \noindent
                \textbf{Case 2 ($|\delta_j| < 2\tau$):}
                For this case, we immediately appeal to the fact that $\hat{\delta}_{j, \tau}$ takes values in $[-\tau, \tau]$ to obtain
                \begin{equation*}
                    \mathbb{E}\bigg[ \exp \bigg\{2\lambda \delta_j (\hat{\delta}_{j, \tau} - \mathbb{E}[\hat{\delta}_{j, \tau}])  \bigg\}\bigg]
                    \leq \exp(2\lambda^2 \delta_j^2 \tau^2)
                    \leq \exp[16\lambda^2\delta_j^2\tau^2 \exp\{-c_1\delta_j^2/(2\tau^2)\}],
                \end{equation*}
                where the first inequality is via e.g.~\citet[Exercise~2.4]{Wainwright:2019:HDSBook}, and the second as $|\delta_j| \leq 2\tau$ ensures $\exp(-c_1n_1\varepsilon^2\delta_j^2/2) \geq 1/8$ for $c_1 > 0$ taken small enough.
                
                Hence, we have
                \begin{equation} \label{app:eq:CombinedCasesIntBound}
                    \mathbb{E}\bigg[ \exp \bigg\{2\lambda \delta_j (\hat{\delta}_{j, \tau} - \mathbb{E}[\hat{\delta}_{j, \tau}])  \bigg\}\bigg]
                    \leq \exp[16\lambda^2\delta_j^2\tau^2 \exp\{-c_1\delta_j^2/(2\tau^2)\}],
                \end{equation}
                for all values of $\delta_j$.

                We conclude by combining \eqref{app:eq:linearmgfboundbyNA} and \eqref{app:eq:CombinedCasesIntBound} to obtain
                \begin{align*}
                    \mathbb{E}[\exp\{\lambda(\mathbb{E}[T_{n_1}^\vecbf{L} \mid \mathcal{D}_{n_1, n_2}] - \mathbb{E}[T_{n_1}^\vecbf{L}])\}]
                    &\leq \exp\bigg\{ \frac{16\lambda^2}{n_1\varepsilon^2} \sum_{j=1}^d \delta^2_j \exp(-c_1 n_1 \varepsilon^2 \delta_j^2/2) \bigg\} \\
                    &\leq \exp\bigg\{ \frac{16(c_1)^{-1/2}\lambda^2 D_\tau}{n_1\varepsilon^2} \bigg\},
                \end{align*}
                where the final inequality follows from the fact $\sup_{x > 0} x\exp(-c x^2)/(x \wedge 1) = \exp(-1/2)/(2c)^{1/2}$ for $c > 0$ sufficiently small, and recalling $D_\tau$ from \eqref{app:eq:Dtau}. This completes the proof.
            \end{proof}
            
        \subsection{Tail Bound for Permuted Linear Statistic}
            \begin{proof}[Proof of \Cref{app:lem:linstatpermSG}]
                We first reduce to the case with equal sample sizes. Let $\vecbf{L} = (l_1, \hdots, l_{n_1}) \subseteq [n_2]$ a $n_1$-tuple drawn without replacement uniformly over the set $[n_2]$, and define
                \begin{align*}
                    T_{n_1}^{\vecbf{L}, \pi} 
                    = \frac{1}{n_1} \sum_{i = 1}^{n_1} (D_{\pi(i)} - D_{\pi(n_1 + l_i)}).
                \end{align*}
                where we recall $\pi$ is a permutation on $n_1 + n_2$ symbols drawn uniformly from $S_{n_1 + n_2}$. Note in particular that $\pi$ and $\vecbf{L}$ are independent and that $\pi(i), \pi(n_1 + l_i)$ may map to an index outside those considered in $T_{n_1}^{\vecbf{L}, \pi}$. 
                Recalling the samples used in the first and second stages of the interactive procedure \eqref{sec4:eq:sampledefs}, denote the combination of the two \beame{privatised}{privatized}{privatized} samples used in the first stage via $\mathcal{D}' = \widetilde{\mathcal{D}}'_{X, n_1} \cup \widetilde{\mathcal{D}}'_{Y, n_2}$ and the further combination with the second stage via
                $\mathcal{D} = \widetilde{\mathcal{D}}_{X, n_1} \cup \widetilde{\mathcal{D}}_{X, n_1}' \cup \widetilde{\mathcal{D}}_{Y, n_2} \cup \widetilde{\mathcal{D}}_{Y, n_2}'$. We decompose the moment generating function as
                \begin{align}
                    \mathbb{E}[&\exp(\lambda T_{n_1, n_2}^{\pi})] \nonumber \\
                    &\leq \mathbb{E}[\exp(\mathbb{E}[T_{n_1}^{\vecbf{L}, \pi} \mid \mathcal{D}, \pi])]
                    \leq \mathbb{E}[\exp(\lambda T_{n_1}^{\vecbf{L}, \pi})] \nonumber \\
                    &= \mathbb{E}\big[\exp\{\lambda (T_{n_1}^{\vecbf{L}, \pi} - \mathbb{E}[T_{n_1}^{\vecbf{L}, \pi} \mid \vecbf{L}, \pi])\} \exp(\lambda \mathbb{E}[T_{n_1}^{\vecbf{L}, \pi} \mid \vecbf{L}, \pi])\big] \nonumber \\
                    &= \mathbb{E}\big[ \mathbb{E}[\exp\{\lambda (T_{n_1}^{\vecbf{L}, \pi} - \mathbb{E}[T_{n_1}^{\vecbf{L}, \pi} \mid \vecbf{L}, \pi])\} \mid \vecbf{L}, \pi] \exp(\lambda \mathbb{E}[T_{n_1}^{\vecbf{L}, \pi} \mid \vecbf{L}, \pi])\big] \nonumber \\
                    &= \mathbb{E}\Big[ \mathbb{E}\big[\mathbb{E}[\exp\{\lambda (T_{n_1}^{\vecbf{L}, \pi} - \mathbb{E}[T_{n_1}^{\vecbf{L}, \pi} \mid \vecbf{L}, \pi])\} \mid \mathcal{D}', \vecbf{L}, \pi] \mid \vecbf{L}, \pi \big] \exp(\lambda \mathbb{E}[T_{n_1}^{\vecbf{L}, \pi} \mid \vecbf{L}, \pi])\Big] \nonumber \\
                    &\leq \exp\{2c_\varepsilon^2 \tau^2 \lambda^2/n_1\} \mathbb{E}[\exp(\lambda \mathbb{E}[T_{n_1}^{\vecbf{L}, \pi} \mid \vecbf{L}, \pi])]
                    \leq \exp\{10 \lambda^2/(n_1^2 \varepsilon^4)\} \mathbb{E}[\exp(\lambda \mathbb{E}[T_{n_1}^{\vecbf{L}, \pi} \mid \vecbf{L}, \pi])] \label{app:eq:linearmgfperm1}
                \end{align}
                where the first and second inequalities are by the fact that $T_{n_1, n_2}^\pi = \mathbb{E}[T_{n_1}^{\vecbf{L}, \pi} \mid \mathcal{D}, \pi]$ and by applying Jensen's inequality; the second and third equalities by the tower property; the penultimate inequality by the fact that for a fixed $\vecbf{L}, \pi$, and conditioning on the sample from the first step of the interactive procedure $\mathcal{D}'$, the quantity $T_{n_1}^{\vecbf{L}, \pi}$ is the average of $n_1$ conditionally independent random variables taking values in $\{0, \pm 2c_\varepsilon\tau\}$ and hence $\mathrm{SG}(4c_\varepsilon^2\tau^2)$
                (e.g.~\citealt[Exercise~2.4]{Wainwright:2019:HDSBook}); and the final inequality by the fact that $\tau = (n_1 \varepsilon^2)^{-1/2}$ and that $c_\varepsilon^2 \leq 5/\varepsilon^2$ for $\varepsilon \in (0,1]$.
                
                We now focus on the remaining expectation term in \eqref{app:eq:linearmgfperm1}. We recall the values $a_i^{\vecbf{L}, \pi}, b_i^{\vecbf{L}, \pi}$ and $d_i^{\vecbf{L}, \pi}$ defined in \eqref{app:eq:permmeanterms}, and, similarly to \eqref{app:eq:permmeanform}, we obtain
                \begin{align*}
                    \mathbb{E}[T_{n_1}^{\vecbf{L}, \pi} \mid \vecbf{L}, \pi]
                    &= \frac{1}{n_1} \sum_{i = 1}^{n_1} \big\{a_i^{\vecbf{L}, \pi}\mathbb{E}[D_{\pi(i)} - D_{\pi(n_1 + l_i)}] + b_i^{\vecbf{L}, \pi}\mathbb{E}[D_{\pi(i)} - D_{\pi(n_1 + l_i)}]\big\} \\
                    &= \frac{1}{n_1}\sum_{i = 1}^{n_1} \big\{ a_i^{\vecbf{L}, \pi}\mathbb{E}[Z_1 - W_{l_1}] + b_i^{\vecbf{L}, \pi}\mathbb{E}[W_{l_1} - Z_1] \big\}
                    = \bigg( \frac{1}{n_1} \sum_{i = 1}^{n_1} d_i^{\vecbf{L}, \pi} \bigg) \mathbb{E}[T_{n_1, n_2}]. \\
                \end{align*}
                Recalling the notation $\psi_{n_1}^{\vecbf{L}, \pi} = \sum_{i = 1}^{n_1} d_i^{\vecbf{L}, \pi}$, we note that conditional on $\vecbf{L}$, $\psi_{n_1}^{\vecbf{L}, \pi}$ is $\mathrm{SG}(2n_1)$ and mean-zero by \Cref{app:lem:permcontrol}, which yields
                \begin{equation} \label{app:eq:linearmgfperm2}
                    \mathbb{E}[\exp(\lambda \psi_{n_1}^{\vecbf{L}, \pi}\mathbb{E}[T_{n_1, n_2}]/n_1)]
                    = \mathbb{E}\big[\mathbb{E}[\exp(\lambda \psi_{n_1}^{\vecbf{L}, \pi}\mathbb{E}[T_{n_1, n_2}]/n_1) \mid \vecbf{L}]\big]
                    \leq \exp(\lambda^2 \mathbb{E}[T_{n_1, n_2}]^2/n_1).
                \end{equation}

                Combining \eqref{app:eq:linearmgfperm1} and \eqref{app:eq:linearmgfperm2} and noting $\varepsilon \in (0, 1]$, we see that
                \begin{equation*}
                    \mathbb{E}[\exp(\lambda T_{n_1, n_2}^{\pi})]
                    \leq \exp\bigg\{ \lambda^2 \bigg(\frac{\mathbb{E}[T_{n_1, n_2}]^2}{n_1} + \frac{10}{n_1^2 \varepsilon^4} \bigg)\bigg\}
                    \leq \exp\bigg( 2\lambda^2 \max\bigg\{\frac{\mathbb{E}[T_{n_1, n_2}]^2}{n_1\varepsilon^2}, \frac{10}{n_1^2 \varepsilon^4} \bigg\}\bigg).
                \end{equation*}
                Hence, we see that $T_{n_1, n_2}^{\pi}$ is $\mathrm{SG}(C \max\{\mathbb{E}[T_{n_1, n_2}]^2/(n_1\varepsilon^2), 1/(n_1\varepsilon^2)^2\})$, where $C > 0$ is some absolute constant, completing the proof.
            \end{proof}               

    \section{Tail Bound Proofs (Interactive Procedure - Continuous Distributions)} \label{app:sec:tailbounds3}
        \subsection{Tail Bound for Linear Statistic}
            \begin{lemma} \label{app:lem:contintSE}
                Recall the values $\tilde{X}_i$ and $\tilde{Y}_{i'}$ as in \eqref{app:eq:tildebeforeprivandtruncvals} for $i \in [n_1^\ast]$, $i' \in [n_2^\ast]$, and the event $A$ as in \eqref{app:eq:intfirstprocevents}.
                
                The sum
                \begin{equation*}
                    T_{n_1^\ast, n_2^\ast}^\ast = \frac{1}{n_1^\ast} \sum_{i = 1}^{n_1^\ast} \tilde{X}_i - \frac{1}{n_2^\ast} \sum_{i' = 1}^{n_2^\ast} \tilde{Y}_{i'},
                \end{equation*}
                satisfies the following tail bound
                \begin{equation*}
                    \mathbb{P}(\{ |T_{n_1^\ast, n_2^\ast}^\ast - \mathbb{E}[T_{n_1^\ast, n_2^\ast}^\ast]| \geq x \} \cap A) \leq 4\exp\bigg( -\frac{cn_1^\ast \varepsilon^2 x^2}{\eta_{j^\ast}^2} \bigg) + \frac{3\cdot2^{d+3/2}Vc_\varepsilon}{x}\bigg(\frac{\beta}{4n_2}\bigg)^{C_2/2},
                \end{equation*}
                where $c > 0$ is an absolute constant, $C_2 > 0$ is the absolute constant in \eqref{app:eq:contintproc1probbound}, $\eta_{j^\ast}$ is as in \eqref{app:eq:contintpart2etaval}, and where $c_\varepsilon = \{\exp(\varepsilon) + 1\}/\{\exp(\varepsilon) - 1\}$.
            \end{lemma}
            \begin{proof}
                Recalling the samples used in the first and second stages of the interactive procedure \eqref{sec4:eq:sampledefs}, denote the combination of the two \beame{privatised}{privatized}{privatized} samples used in the first stage of the estimator via $\mathcal{D}' = \widetilde{\mathcal{D}}_{X, n_1^\ast}' \cup \widetilde{\mathcal{D}}_{Y, n_2^\ast}'$, and the combination with the second stage via
                $\mathcal{D} = \widetilde{\mathcal{D}}_{X, n_1^\ast} \cup \widetilde{\mathcal{D}}_{X, n_1^\ast}' \cup \widetilde{\mathcal{D}}_{Y, n_2^\ast} \cup \widetilde{\mathcal{D}}_{Y, n_2^\ast}'$. Further, denote       
                \begin{equation*}
                    T_{n_1^\ast, n_2^\ast}^{\ast, A} = \frac{1}{n_1^\ast} \sum_{i = 1}^{n_1^\ast} \tilde{X}_i \mathbbm{1}\{A_i^{(X)}\} - \frac{1}{n_2^\ast} \sum_{i' = 1}^{n_2^\ast} \tilde{Y}_{i'}\mathbbm{1}\{A_{i'}^{(Y)}\},
                \end{equation*}
                where we recall $\tilde{X}_i$ and $\tilde{Y}_{i'}$ as in \eqref{app:eq:tildebeforeprivandtruncvals}.
                
                Then, noting $T_{n_1^\ast, n_2^\ast}^\ast = T_{n_1^\ast, n_2^\ast}^{\ast, A}$ on $A$, we decompose the tail probability as
                \begin{align}
                    &\mathbb{P}(\{ |T_{n_1^\ast, n_2^\ast}^\ast - \mathbb{E}[T_{n_1^\ast, n_2^\ast}^\ast]| \geq x \} \cap A) 
                    \leq \mathbb{P}(\{ |T_{n_1^\ast, n_2^\ast}^{\ast, A} - \mathbb{E}[T_{n_1^\ast, n_2^\ast}^{\ast, A} \mid \mathcal{D}']| \geq x/3 \} \cap A) \nonumber \\
                    &+ \mathbb{P}(\{ |\mathbb{E}[T_{n_1^\ast, n_2^\ast}^{\ast, A} \mid \mathcal{D}'] - \mathbb{E}[T_{n_1^\ast, n_2^\ast}^\ast \mid \mathcal{D}']| \geq x/3 \} \cap A)
                    + \mathbb{P}(\{ |\mathbb{E}[T_{n_1^\ast, n_2^\ast}^\ast \mid \mathcal{D}'] - \mathbb{E}[T_{n_1^\ast, n_2^\ast}^\ast]| \geq x/3 \} \cap A) \nonumber \\
                    &\leq (\mathrm{I}) + (\mathrm{II}) + (\mathrm{III}), \label{app:eq:contintnonpermtail}
                \end{align}
                and proceed by bounding each term individually.
                
                \noindent
                \textbf{Term} (I):
                We note that by the construction of the events $A_i^{(X)}$ and $A_{i'}^{(Y)}$, that the random variables $\tilde{X}_i \mathbbm{1}\{A_i^{(X)}\}$ and $\tilde{Y}_{i'}\mathbbm{1}\{A_{i'}^{(Y)}\}$ are bounded and take values in $[-\eta_{j^\ast}, \eta_{j^\ast}]$, and are hence $\mathrm{SG}(\eta_{j^\ast}^2)$ (e.g.~\citealt[Exercise~2.4]{Wainwright:2019:HDSBook}). Further, conditional on $\mathcal{D}'$, they are independent, and hence the average $T_{n_1^\ast, n_2^\ast}^{\ast, A}$ is $\mathrm{SG}(2\eta_{j^\ast}^2/n_1^\ast)$ almost surely on $\mathcal{D}'$. Thus, by \citet[Proposition~2.5][]{Wainwright:2019:HDSBook} we have the bound
                \begin{align}
                    \mathbb{P}(\{ |T_{n_1^\ast, n_2^\ast}^{\ast, A} - \mathbb{E}[T_{n_1^\ast, n_2^\ast}^{\ast, A} &\mid \mathcal{D}']| \geq x/3 \} \cap A) \nonumber \\
                    &= \mathbb{E}[\mathbb{P}(\{ |T_{n_1^\ast, n_2^\ast}^{\ast, A} - \mathbb{E}[T_{n_1^\ast, n_2^\ast}^{\ast, A} \mid \mathcal{D}']| \geq x/3 \} \cap A \mid \mathcal{D}')]
                    \leq 2\exp\bigg( \frac{-n_1^\ast x^2}{4 \eta_{j^\ast}^2} \bigg). \label{app:eq:contintnonpermtail1}
                \end{align}

                \noindent
                \textbf{Term} (II): Writing $c_\varepsilon = \{\exp(\varepsilon) + 1\}/\{\exp(\varepsilon) - 1\}$, we observe
                \begin{align}
                    \mathbb{P}(\{ &|\mathbb{E}[T_{n_1^\ast, n_2^\ast}^{\ast, A} \mid \mathcal{D}'] - \mathbb{E}[T_{n_1^\ast, n_2^\ast}^\ast \mid \mathcal{D}']| \geq x/3 \} \cap A) \nonumber \\
                    &\leq \mathbb{P}\bigg(\biggl| \frac{1}{n_1^\ast}\sum_{i = 1}^{n_1^\ast}\mathbb{E}[\tilde{X}_i \mathbbm{1}\{(A_i^{(X)})^c\} \mid \mathcal{D}'] - \frac{1}{n_2^\ast}\sum_{i' = 1}^{n_2^\ast}\mathbb{E}[\tilde{Y}_{i'} \mathbbm{1}\{(A_{i'}^{(Y)})^c\} \mid \mathcal{D}'] \biggr| \geq x/3  \bigg) \nonumber \\
                    &\leq \mathbb{P}\bigg( 2^{d+1}Vc_\varepsilon \biggl| \frac{1}{n_1^\ast}\sum_{i = 1}^{n_1^\ast}\mathbb{P}\{(A_i^{(X)})^c \mid \mathcal{D}'\}^{1/2} \biggr| + 2^{d+1}Vc_\varepsilon \biggl| \frac{1}{n_2^\ast}\sum_{i' = 1}^{n_2^\ast}\mathbb{P}\{(A_{i'}^{(Y)})^c \mid \mathcal{D}'\}^{1/2} \biggr|\geq x/3 \bigg) \nonumber \\
                    &\leq \frac{3\cdot2^{d+1}Vc_\varepsilon}{x}\bigg(\frac{1}{n_1^\ast}\sum_{i = 1}^{n_1^\ast} \mathbb{E}\big[\mathbb{P}\{(A_i^{(X)})^c \mid \mathcal{D}'\}^{1/2}\big] + \frac{1}{n_2^\ast}\sum_{i' = 1}^{n_2^\ast} \mathbb{E}\big[\mathbb{P}\{(A_{i'}^{(Y)})^c \mid \mathcal{D}'\}^{1/2}\big] \bigg) \nonumber \\
                    &\leq \frac{3\cdot2^{d+1}Vc_\varepsilon}{x}2^{1/2}\bigg(\frac{\beta}{4n_2}\bigg)^{C_2/2} \label{app:eq:contintnonpermtail2}
                \end{align}
                where the second inequality is by H\"{o}lder's inequality and the fact that the $\tilde{X}_i, \tilde{Y}_{i'}$ are bounded as $\tilde{X}_i, \tilde{Y}_{i'} \in [-2^{d+1}Vc_\varepsilon, 2^{d+1}Vc_\varepsilon]$ for all $i \in [n_1^\ast], i' \in [n_2^\ast]$; the penultimate inequality is by Markov's inequality, and the final by \eqref{app:eq:contintproc1probbound}, the tower property and Jensen's inequality applied with the concavity of the square root. 
                
                \noindent
                \textbf{Term} (III):
                Denote $\hat{\delta}_{\vecbf{l}} = \hat{\theta}_{X, \vecbf{l}} - \hat{\theta}_{Y, \vecbf{l}}$ and $\delta_\vecbf{l} = \mathbb{E}[\delta_\vecbf{l}]$. We have
                \begin{align*}
                    \mathbb{E}[\exp\{\lambda(\mathbb{E}[T_{n_1^\ast, n_2^\ast}^\ast \mid \mathcal{D}'] - \mathbb{E}[T_{n_1^\ast, n_2^\ast}^\ast])\}]
                    &= \mathbb{E}\bigg[\exp\bigg\{\lambda \sum_{\vecbf{l} \in \mathbb{N}_0^d(R)} (\hat{\delta}_\vecbf{l} - \delta_l)(\theta_{X, \vecbf{l}} - \theta_{Y, \vecbf{l}})\bigg\}\bigg] \\
                    &= \prod_{\vecbf{l} \in \mathbb{N}_0^d(R)} \mathbb{E}\bigg[\exp\bigg\{\lambda  (\hat{\delta}_\vecbf{l} - \delta_l)(\theta_{X, \vecbf{l}} - \theta_{Y, \vecbf{l}})\bigg\}\bigg] \\
                    &\leq \exp\bigg(\frac{2^d\lambda^2 V}{n_1c_\varepsilon^2} L_{2, R}^2 \bigg) \\
                    &\leq \exp\bigg(\frac{5\cdot2^d\lambda^2 V}{n_1\varepsilon^2} L_{2, R}^2 \bigg),
                \end{align*}
                where the second equality is by the independence of the $\hat{\delta}_\vecbf{l}$; the penultimate inequality as, by construction, the $\hat{\delta}_\vecbf{l}$ are $\mathrm{SG}(2^{d+1}V/(n_1^\ast c_\varepsilon^2))$; and the last inequality is by the fact that $c_\varepsilon^2 \leq 5/\varepsilon^2$ for $\varepsilon \in (0,1]$. Hence, we see that the $\mathbb{E}[T_{n_1^\ast, n_2^\ast}^\ast \mid \mathcal{D}']$ is $\mathrm{SG}(5\cdot2^{d+1}VL_{2, R}^2/(n_1^\ast\varepsilon^2))$ and we have the bound
                \begin{align}
                    \mathbb{P}(\{ |\mathbb{E}[T_{n_1^\ast, n_2^\ast}^\ast \mid \mathcal{D}'] - \mathbb{E}[T_{n_1^\ast, n_2^\ast}^\ast]| \geq x/3 \} \cap A)
                    &\leq \mathbb{P}(|\mathbb{E}[T_{n_1^\ast, n_2^\ast}^\ast \mid \mathcal{D}'] - \mathbb{E}[T_{n_1^\ast, n_2^\ast}^\ast]| \geq x/3) \nonumber \\
                    &\leq 2\exp\bigg(-\frac{n_1^\ast\varepsilon^2 x^2}{45\cdot2^{d+1}VL_{2, R}^2}\bigg). \label{app:eq:contintnonpermtail3}
                \end{align}
                Combining together \eqref{app:eq:contintnonpermtail}, \eqref{app:eq:contintnonpermtail1}, \eqref{app:eq:contintnonpermtail2} and \eqref{app:eq:contintnonpermtail3}, we obtain, for $c > 0$ some absolute constant,
                \begin{equation*}
                    \mathbb{P}(\{ |T_{n_1^\ast, n_2^\ast}^\ast - \mathbb{E}[T_{n_1^\ast, n_2^\ast}^\ast]| \geq x \} \cap A)
                    \leq 4\exp\bigg( -\frac{cn_1^\ast \varepsilon^2 x^2}{\eta_{j^\ast}^2} \bigg) + \frac{3\cdot2^{d+3/2}Vc_\varepsilon}{x}\bigg(\frac{\beta}{4n_2}\bigg)^{C_2/2},
                \end{equation*}
                where we use the fact that $C_{r, s, d} V^{1/2} L_{2, R}\{\log(4n_2/\beta)\}^{1/2} \leq \eta_{j^\ast}$ for $C_{r, s, d} > 0$ the constant in \eqref{app:eq:contintpart2etaval} which can be taken sufficiently large.
                
            \end{proof}
            
        \subsection{Tail Bound for Permuted Linear Statistic}
        \begin{lemma} \label{app:lem:contintpermSE}
            Recall the values $\tilde{D}_i = \tilde{X}_i$ and $\tilde{D}_{n_1^\ast + i'} = \tilde{Y}_{i'}$ with $\tilde{X}_i$ and $\tilde{Y}_{i'}$ as in \eqref{app:eq:tildebeforeprivandtruncvals} for $i \in [n_1^\ast]$, $i' \in [n_2^\ast]$. Recall also the event $A$ as in \eqref{app:eq:intfirstprocevents}.
        
            Let $\pi \in S_{n_1^\ast + n_2^\ast}$ be a uniformly sampled permutation on $n_1^\ast + n_2^\ast$ symbols. The sum
            \begin{equation*}
                T_{n_1^\ast, n_2^\ast}^{\ast, \pi} = \frac{1}{n_1^\ast} \sum_{i = 1}^{n_1^\ast} \tilde{D}_{\pi(i)} - \frac{1}{n_2^\ast} \sum_{i' = 1}^{n_2^\ast} \tilde{D}_{\pi(n_1^\ast + i')}.
            \end{equation*}
            satisfies the following tail bound
            \begin{equation*}
                \mathbb{P}(\{ |T_{n_1^\ast, n_2^\ast}^{\ast, \pi}| \geq x \} \cap A) \leq 2\exp\bigg(\frac{-n_1^\ast x^2}{24\eta_{j^\ast}^2}\bigg).
            \end{equation*}
        \end{lemma}
            \begin{proof}
                For $i \in [n_1^\ast + n_2^\ast]$, write $A_i = \{|\tilde{D}_i| \leq \eta_{j^\ast}\}$ for $\eta_{j^\ast}$ as in \eqref{app:eq:contintpart2etaval}. Denote the combination of the two \beame{privatised}{privatized}{privatized} samples used in the first stage of the estimator via $\mathcal{D}' = \widetilde{\mathcal{D}}_{X, n_1^\ast}' \cup \widetilde{\mathcal{D}}_{Y, n_2^\ast}'$, and the combination with the second stage via
                $\mathcal{D} = \widetilde{\mathcal{D}}_{X, n_1^\ast} \cup \widetilde{\mathcal{D}}_{X, n_1^\ast}' \cup \widetilde{\mathcal{D}}_{Y, n_2^\ast} \cup \widetilde{\mathcal{D}}_{Y, n_2^\ast}'$. Define       
                \begin{equation*}
                    T_{n_1^\ast, n_2^\ast}^{\ast, \pi, A} = \frac{1}{n_1^\ast} \sum_{i = 1}^{n_1^\ast} \tilde{D}_{\pi(i)} \mathbbm{1}\{A_{\pi(i)}\} - \frac{1}{n_2^\ast} \sum_{i' = 1}^{n_2^\ast} \tilde{D}_{\pi(n_1^\ast + i')}\mathbbm{1}\{A_{\pi(n_1^\ast + i')}\}.
                \end{equation*}

                We now reduce to the case with equal sample sizes. Consider $\vecbf{L} = (l_1, \hdots, l_{n_1}) \subseteq [n_2^\ast]$ an $n_1^\ast$-tuple drawn without replacement uniformly over the set $[n_2^\ast]$, and define
                \begin{equation*}
                    \begin{aligned}
                        T_{n_1^\ast}^{\ast, \vecbf{L}, \pi} 
                        &= \frac{1}{n_1^\ast} \sum_{i = 1}^{n_1^\ast} (D_{\pi(i)} - D_{\pi(n_1^\ast + l_i)}),\; \mbox{and} \\
                        T_{n_1^\ast}^{\ast, \vecbf{L}, \pi, A} 
                        &= \frac{1}{n_1^\ast} \sum_{i = 1}^{n_1^\ast} (D_{\pi(i)}\mathbbm{1}\{A_{\pi(i)}\} - D_{\pi(n_1^\ast + l_i)}\mathbbm{1}\{A_{\pi(n_1^\ast + l_i)}\}).
                    \end{aligned}
                \end{equation*}
                where we recall $\pi$ is a permutation on $n_1^\ast + n_2^\ast$ symbols drawn uniformly from $S_{n_1^\ast + n_2^\ast}$. Note in particular that $\pi$ and $\vecbf{L}$ are independent and that $\pi(i), \pi(n_1^\ast + l_i)$ may map to an index outside those considered in $T_{n_1^\ast}^{\ast, \vecbf{L}, \pi}$. We decompose the moment generating function as
                \begin{align}
                    \mathbb{E}[&\exp(\lambda T_{n_1^\ast, n_2^\ast}^{\ast, \pi, A})\mid \vecbf{L}, \pi] \nonumber \\
                    &= \mathbb{E}\big[\exp(\lambda \mathbb{E}[T_{n_1^\ast}^{\ast, \vecbf{L}, \pi, A} \mid \mathcal{D}, \pi] )\big] \nonumber \\
                    &\leq \mathbb{E}[\exp(\lambda T_{n_1^\ast}^{\ast, \vecbf{L}, \pi, A})] \nonumber \\
                    &= \mathbb{E}\Big[ \mathbb{E}\big[ \mathbb{E}[\exp\{\lambda (T_{n_1^\ast}^{\ast, \vecbf{L}, \pi, A} - \mathbb{E}[T_{n_1^\ast}^{\ast, \vecbf{L}, \pi, A} \mid \vecbf{L}, \pi])\} \mid \mathcal{D}', \vecbf{L}, \pi] \mid \vecbf{L}, \pi \big] \exp(\lambda \mathbb{E}[T_{n_1^\ast}^{\ast, \vecbf{L}, \pi, A} \mid \vecbf{L}, \pi]) \Big] \nonumber \\
                    &\leq \exp\{2\lambda^2\eta_{j^\ast}^2/n_1^\ast\} \mathbb{E}\big[ \exp(\lambda \mathbb{E}[T_{n_1^\ast}^{\ast, \vecbf{L}, \pi, A} \mid \vecbf{L}, \pi]) \big] \label{app:eq:contintpermterm1}
                \end{align}
                where the first equality is by that fact that $T_{n_1^\ast, n_2^\ast}^{\ast, \pi, A} = \mathbb{E}[T_{n_1^\ast}^{\ast, \vecbf{L}, \pi, A} \mid \mathcal{D}, \pi]$; the first inequality by Jensen's inequality; the second equality by the tower property; and the last inequality by the fact that for a fixed $\vecbf{L}, \pi$, and conditioning on the sample from the first step of the interactive procedure $\mathcal{D}'$, the quantity $T_{n_1}^{\ast, \vecbf{L}, \pi, A}$ is the average of $n_1^\ast$ conditionally independent random variables taking values in $\{0, \pm 2\eta_{j^\ast}\}$ and is hence $\mathrm{SG}(4\eta_{j^\ast}^2)$
                (e.g.~\citealt[Exercise~2.4]{Wainwright:2019:HDSBook})

                We now focus on the remaining expectation term in \eqref{app:eq:contintpermterm1}. We recall the values $a_i^{\vecbf{L}, \pi}, b_i^{\vecbf{L}, \pi}$ and $d_i^{\vecbf{L}, \pi}$ defined in \eqref{app:eq:permmeanterms}, and, similarly to \eqref{app:eq:permmeanform}, we obtain
                \begin{align*}
                    \mathbb{E}[T_{n_1^\ast}^{\ast, \vecbf{L}, \pi, A} \mid \vecbf{L}, \pi]
                    &= \frac{1}{n_1^\ast} \sum_{i = 1}^{n_1^\ast} \big\{a_i^{\vecbf{L}, \pi}\mathbb{E}[D_{\pi(i)}\mathbbm{1}\{A_{\pi(i)}\} - D_{\pi(n_1^\ast + l_i)}\mathbbm{1}\{A_{\pi(n_1^\ast + l_i)}\}]  \\
                    &\hspace{2cm}+ b_i^{\vecbf{L}, \pi}\mathbb{E}[D_{\pi(i)}\mathbbm{1}\{A_{\pi(i)}\} - D_{\pi(n_1^\ast + l_i)}\mathbbm{1}\{A_{\pi(n_1^\ast + l_i)}\} ]\big\}  \\
                    &= \frac{1}{n_1^\ast}\sum_{i = 1}^{n_1^\ast} \big\{ a_i^{\vecbf{L}, \pi}\mathbb{E}[\tilde{X}_1\mathbbm{1}\{A_1\} - \tilde{Y}_1\mathbbm{1}\{A_{n_1^\ast + 1}\}]  \\
                    &\hspace{2cm}+ b_i^{\vecbf{L}, \pi}\mathbb{E}[\tilde{Y}_1\mathbbm{1}\{A_{n_1^\ast + 1}\} - \tilde{X}_1\mathbbm{1}\{A_1\}] \big\} \\
                    &= \bigg( \frac{1}{n_1^\ast} \sum_{i = 1}^{n_1^\ast} d_i^{\vecbf{L}, \pi} \bigg) \mathbb{E}[\tilde{X}_1\mathbbm{1}\{A_1\} - \tilde{Y}_1\mathbbm{1}\{A_{n_1^\ast + 1}\}].
                \end{align*}
                Recalling the notation $\psi_{n_1^\ast}^{\vecbf{L}, \pi} = \sum_{i = 1}^{n_1^\ast} d_i^{\vecbf{L}, \pi}$, we note that $\psi_{n_1^\ast}^{\vecbf{L}, \pi}$ is $\mathrm{SG}(2n_1^\ast)$ and mean-zero by \Cref{app:lem:permcontrol}, which yields
                \begin{align}
                    \mathbb{E}[\exp(\lambda \mathbb{E}[T_{n_1^\ast}^{\ast, \vecbf{L}, \pi, A} \mid \vecbf{L}, \pi])]
                    &= \mathbb{E}\bigg[\exp\bigg( \frac{\lambda \psi_{n_1^\ast}^{\vecbf{L}, \pi}}{n_1^\ast} \mathbb{E}[\tilde{X}_1\mathbbm{1}\{A_1\} - \tilde{Y}_1\mathbbm{1}\{A_{n_1^\ast + 1}\}] \bigg) \bigg] \nonumber \\
                    &\leq \mathbb{E}\bigg[\exp\bigg( \frac{\lambda^2}{n_1^\ast} \mathbb{E}[\tilde{X}_1\mathbbm{1}\{A_1\} - \tilde{Y}_1\mathbbm{1}\{A_{n_1^\ast + 1}\}]^2 \bigg) \bigg] \nonumber \\
                    &\leq \exp\bigg( \frac{4\lambda^2 \eta_{j^\ast}^2}{n_1^\ast} \bigg), \label{app:eq:contintpermterm2}
                \end{align}
                where the final inequality comes from the bounds $|\tilde{X}_1\mathbbm{1}\{A_1\}|, |\tilde{Y}_1\mathbbm{1}\{A_{n_1^\ast + 1}\}| \leq \eta_j^\ast$. Combining \eqref{app:eq:contintpermterm1} and \eqref{app:eq:contintpermterm2}, we obtain
                \begin{align*}
                    \mathbb{E}[\exp(\lambda T_{n_1^\ast, n_2}^{\ast, \pi, A})\mid \vecbf{L}, \pi]
                    \leq \exp\bigg(\frac{6\lambda^2\eta_{j^\ast}^2}{n_1^\ast} \bigg),
                \end{align*}
                which shows that $T_{n_1^\ast, n_2}^{\ast, \pi, A}$ is $\mathrm{SG}(12\eta_{j^\ast}^2/n_1^\ast)$. Lastly, recalling $T_{n_1^\ast}^{\ast, \pi, A} = T_{n_1^\ast, n_2}^{\ast, \pi}$ on $A$ and appealing to \citet[e.g.~Proposition~2.5][]{Wainwright:2019:HDSBook}, gives the tail bound in the lemma statement, completing the proof.
            \end{proof}
        
    \section{Technical Lemmata} \label{app:misc}
        \subsection{Properties of Unary-Encoding}
            In this section we prove some useful properties of the unary-encoding mechanism. We first recall the unary-encoding mechanism: For a value $x \in [d]$ and a privacy parameter $\varepsilon > 0$, construct the \beame{privatised}{privatized}{privatized} vector $\vecbf{z} \in \{0, 1\}^d$ with entries given as
            \begin{equation} \label{app:eq:UE}
                z_{i} =
                \begin{cases}
                    \mathbbm{1}\{x = j\} \quad &\mbox{with probability } \omega_{\varepsilon/2}, \\
                    1 - \mathbbm{1}\{x = j\} \quad &\mbox{with probability } 1 - \omega_{\varepsilon/2}.
                \end{cases}
            \end{equation}
            where we recall $\omega_{\varepsilon/2} = \exp(\varepsilon/2)/\{\exp(\varepsilon/2) + 1\}$.
            
            \subsubsection{Unary-Encoding is Sub-Gaussian}
                \begin{lemma} \label{app:lem:UESG}
                    For a probability vector $\vecbf{p}_X \in [0,1]^d$, let the random variable $X \sim \mathrm{Multinom}(\vecbf{p}_X)$. The random variable $\vecbf{Z}$ arising from applying the unary-encoding \beame{privatisation}{privatization}{privatization} method \eqref{app:eq:UE} with privacy parameter $\varepsilon > 0$ is a $\mathrm{SG}(17/4)$ random vector.
                \end{lemma}
    
                \begin{proof}
                    For any unit vector $\vecbf{u} \in \mathbb{R}^d$ and any $\lambda > 0$, we have
                    \begin{align*} 
                        \mathbb{E}[\exp\{\lambda \vecbf{u}^T & (\vecbf{Z} - \mathbb{E}[\vecbf{Z}])\}] \\
                        &= \mathbb{E} \big[ \mathbb{E}[\exp\{\lambda \vecbf{u}^T (\vecbf{Z} - \mathbb{E}[\vecbf{Z} \mid X])\} \mid X] \exp\{\lambda \vecbf{u}^T (\mathbb{E}[\vecbf{Z} \mid X] - \mathbb{E}[\vecbf{Z}])\} \big] \\
                        &= \mathbb{E}\bigg[ \bigg( \prod_{j = 1}^d \mathbb{E}[\exp\{\lambda u_j(Z_j - \mathbb{E}[Z_j \mid X])\} \mid X] \bigg) \exp\{\lambda \vecbf{u}^T(\mathbb{E}[\vecbf{Z} \mid X] - \mathbb{E}[\vecbf{Z}])\} \bigg] \\
                        &\leq \exp(\lambda^2\| \vecbf{u} \|_2^2/8) \mathbb{E}[\exp\{\lambda \vecbf{u}^T (\mathbb{E}[\vecbf{Z} \mid X] - \mathbb{E}[\vecbf{Z}])\} ] \\
                        &= \exp(\lambda^2/8) \mathbb{E}[\exp\{\lambda \vecbf{u}^T (\mathbb{E}[\vecbf{Z} \mid X] - \mathbb{E}[\vecbf{Z}])\} ],
                    \end{align*}
                    where the first equality is by the tower property; the second as $\vecbf{Z}$ has independent entries conditional on $X$; and the first inequality as $Z_j \in \{0,1\}$ and is thus $\mathrm{SG}(1/4)$ \citep[e.g.~Exercise~2.4][]{Wainwright:2019:HDSBook}.

                    Focusing on the remaining expectation term, we note that $\mathbb{E}[\vecbf{Z} \mid X] - \mathbb{E}[\vecbf{Z}] = (2\omega_{\varepsilon/2} - 1)(\vecbf{e}_X - \vecbf{p}_X)$ where $\vecbf{e}_j$ for $j \in [d]$ is the $j$-th standard basis vector of $\mathbb{R}^d$. Then, by the facts that $2\omega_{\varepsilon/2} - 1 \in [0,1]$ and $\vecbf{u}^T(\vecbf{e}_X - \vecbf{p}_X) \in [-2, 2]$, we have that $\vecbf{u}^T (\mathbb{E}[\vecbf{Z} \mid X] - \mathbb{E}[\vecbf{Z}])$ is $\mathrm{SG}(4)$ \citep[e.g.~Exercise~2.4][]{Wainwright:2019:HDSBook}, and hence
                    \begin{align*} 
                        \mathbb{E}[\exp\{\lambda \vecbf{u}^T (\vecbf{Z} - \mathbb{E}[\vecbf{Z}])\}]
                        \leq \exp(\lambda^2/8) \exp(2\lambda^2)
                        \leq \exp(17\lambda^2/8),
                    \end{align*}
                    which completes the proof.
                \end{proof}

            \subsubsection{Negative Association of Unary-Encoding}
                We first introduce the definition of negative association (NA) for random variables.
                \begin{definition}[\citealt{Joag:1983:NA}] \label{app:def:NA}
                    A collection of random variables $\{X_1, \hdots, X_n\}$ is said to be \emph{negatively associated} if, for every pair of disjoint subsets $I_1, I_2 \subseteq [n]$ such that $I_1 \cup I_2 = [n]$, it holds that
                    \begin{equation*}
                        \mathrm{Cov}(f(X_i : i \in I_1), g(X_{i'} : i' \in I_2)) \leq 0
                    \end{equation*}
                    for all real-valued functions $f, g$ both non-decreasing\footnote{Note that in \citet{Joag:1983:NA}, increasing is used in place of non-decreasing, and decreasing in place of non-increasing.}, or equivalently both non-increasing, functions.
                \end{definition}
                With this definition in hand, we prove the following,.
                \begin{lemma} \label{app:lem:UENA}
                    For $X$ a random variable taking values in $[d]$, the co-ordinates $\{Z_1, \hdots, Z_d\}$ of the random variable $\vecbf{Z}$, arising from applying the unary-encoding \beame{privatisation}{privatization}{privatization} method \eqref{app:eq:UE} with privacy parameter $\varepsilon \geq 0$ to $X$ satisfies the negative association property.
                \end{lemma}
                \begin{proof}
                    Denote disjoint sets of indices $I_1, I_2 \subseteq[d]$ such that $I_1 \cup I_2 = [d]$. Consider non-decreasing real-valued functions $f(Z_i,\;i \in I_1)$ and $g(Z_i,\;i \in I_2)$ where we write $f(\vecbf{Z})$ and $g(\vecbf{Z})$ respectively for brevity. Denote the events $A_i = \{X = i\}$ and let $B_1 = \cup_{i \in I_1} A_i = \{X \in I_1\}$, $B_2 = \cup_{i \in I_2} A_i =\{X \in I_2\}$.

                    Define the quantities
                    \begin{align*}
                        &E_f^\ast = \mathbb{E}[f(\vecbf{Z}) \mid B_1]; \quad E_f = \mathbb{E}[f(\vecbf{Z}) \mid B_2]; \quad
                        E_g^\ast = \mathbb{E}[g(\vecbf{Z}) \mid B_2]; \quad E_g = \mathbb{E}[g(\vecbf{Z}) \mid B_1]; \\
                        &p_f = \mathbb{P}(B_1); \quad \mbox{and} \quad p_g = \mathbb{P}(B_2).
                    \end{align*}
                    Using the fact that $f(\vecbf{Z}), g(\vecbf{Z})$ are independent conditional on $B_1$ or $B_2$, we obtain
                    \begin{align}
                        \mathbb{E}[f(\vecbf{Z})g(\vecbf{Z})]
                        &= \mathbb{E}[f(\vecbf{Z})g(\vecbf{Z}) \mid B_1] p_f
                        + \mathbb{E}[f(\vecbf{Z})g(\vecbf{Z}) \mid B_2] p_g \nonumber \\
                        &= \mathbb{E}[f(\vecbf{Z})\mid B_1] \mathbb{E}[g(\vecbf{Z})\mid B_1] p_f
                        + \mathbb{E}[f(\vecbf{Z})\mid B_2] \mathbb{E}[g(\vecbf{Z})\mid B_2] p_g \nonumber \\
                        &= E_f^\ast E_g p_f + E_f E_g^\ast p_g. \label{app:eq:NAUEIntStep}
                    \end{align}
                    Similarly, noting $p_f + p_g = 1$, we can obtain
                    \begin{align}
                        \mathbb{E}[f(\vecbf{Z})]&\mathbb{E}[g(\vecbf{Z})] \nonumber \\
                        &= (E_f^\ast p_f + E_f p_g) (E_g^\ast p_g + E_g p_f) \nonumber \\
                        &= E_f^\ast E_g p_f + E_f E_g^\ast p_g + E_f^\ast E_g^\ast p_f p_g - E_f^\ast E_g (p_f - p_f^2) - E_f E_g^\ast (p_g - p_g^2) + E_f E_g p_f p_g \nonumber \\
                        &= \mathbb{E}[f(\vecbf{Z})g(\vecbf{Z})] + (E_f^\ast - E_f )(E_g^\ast - E_g)p_f p_g, \label{app:eq:NAUEIntStep2}
                    \end{align}
                    where the final equality is by \eqref{app:eq:NAUEIntStep}. 
                    
                    It remains to show that the second term in \eqref{app:eq:NAUEIntStep2} is non-negative. We first note on the event $A_i$ that $Z_i \sim \mathrm{Unif}(\omega_{\varepsilon/2})$ and $Z_j \overset{\mathrm{i.i.d.}}{\sim} \mathrm{Unif}(1 - \omega_{\varepsilon/2})$ for $j \neq i$, with the collection $\{Z_j\}_{j \in [d]}$ consisting of independent random variables and where $\omega_{\varepsilon/2} = \exp(\varepsilon/2)/\{\exp(\varepsilon/2) + 1\}$. Supposing without loss of generality $I_1 = \{1, 2, \hdots, |I_1|\}$ for simplicity, we have that
                    \begin{align*}
                        \mathbb{E}[f(\vecbf{Z}) \mid A_1]
                        &= \mathbb{E}[f(Z_1, Z_2, \hdots, Z_{|I_1|}) \mid A_1] \\
                        &= \mathbb{E}[f(0, Z_2, \hdots, Z_{|I_1|}) \mid A_1] (1 - \omega_{\varepsilon/2}) + \mathbb{E}[f(1, Z_2, \hdots, Z_{|I_1|}) \mid A_1] \omega_{\varepsilon/2} \\
                        &\geq \mathbb{E}[f(0, Z_2, \hdots, Z_{|I_1|}) \mid A_1] \omega_{\varepsilon/2} + \mathbb{E}[f(1, Z_2, \hdots, Z_{|I_1|}) \mid A_1] (1 - \omega_{\varepsilon/2}) \\
                        &= \mathbb{E}[f(Z_1, Z_2, \hdots, Z_{|I_1|}) \mid B_2] = \mathbb{E}[f(\vecbf{Z}) \mid B_2],
                    \end{align*}
                    where the inequality is by the fact that $f$ is non-decreasing and that $\omega_{\varepsilon/2} \geq 1/2$.
    
                    Applying a similar argument shows that $\mathbb{E}[f(\vecbf{Z}) \mid A_j] \geq \mathbb{E}[f(\vecbf{Z}) \mid B_2]$ for all $j \in I_1$, and hence $\mathbb{E}[f(\vecbf{Z}) \mid B_1] \geq \mathbb{E}[f(\vecbf{Z}) \mid B_2]$. We can then similarly show that $\mathbb{E}[g(\vecbf{Z}) \mid B_2] \geq \mathbb{E}[g(\vecbf{Z}) \mid B_1]$.
                    Hence, we have that $E_f^\ast \geq E_f, E_g^\ast \geq E_g$, and combining with \eqref{app:eq:NAUEIntStep2}, we obtain $\mathbb{E}[f(\vecbf{Z})]\mathbb{E}[g(\vecbf{Z})] \geq \mathbb{E}[f(\vecbf{Z}) g(\vecbf{Z})]$, which completes the proof.
                \end{proof}

                The following lemma (an analogue of \citealt[Property~6]{Joag:1983:NA} to non-increasing functions) will be useful following application of \Cref{app:lem:UENA}.

                \begin{lemma} \label{app:lem:NALemma}
                    Let the collection of random variables $X = \{X_1, \hdots, X_n\}$ be negatively associated. A collection of non-increasing functions $h_1, \hdots h_k$, for $k \in [n]$, defined on disjoint subsets of $X$ is negatively associated.
                \end{lemma}
                \begin{proof}
                    For negative association, the condition of \Cref{app:def:NA} must hold for all possible pairs $(f,g)$ of real-valued non-increasing functions on $\mathbb{R}^k$, or equivalently non-decreasing functions. Suppose without loss of generality $f$ and $g$ are non-increasing. As the functions $h_1, \hdots, h_k$ are non-increasing, the entry-wise compositions with $f$ and $g$ are non-decreasing. Hence, the condition of \Cref{app:def:NA} holds in this case, completing the proof.
                \end{proof}

            \subsubsection{Distribution Estimate obtained by Unary-Encoding is Sub-Gaussian}
                Recall the quantities $\hat{\vecbf{p}}_X, \hat{\vecbf{p}}_Y \in \mathbb{R}^d$ defined in \eqref{sec4:eq:UEpmfests}. For $j \in [d]$, denote the difference for the $j$-th co-ordinate as $\hat{\delta}_j = \hat{p}_{X,j} - \hat{p}_{Y,j}$.
                \begin{lemma} \label{app:lem:UEEstSG}
                    Assume without loss of generality that $n_1 \leq n_2$. For each $j \in [d]$, the quantity $\hat{\delta}_j = \hat{p}_{X,j} - \hat{p}_{Y,j}$ is $\mathrm{SG}(10/(n_1\varepsilon^2))$.
                \end{lemma}
                \begin{proof}
                    We have that
                    \begin{equation*}
                        \hat{p}_{X,j}
                        = \frac{c_{\varepsilon/2}}{n_1} \sum_{i = 1}^{n_1} \bigg( Z_{i, j}' - \frac{1}{\exp(\varepsilon/2) + 1} \bigg),
                    \end{equation*}
                    where the collection $\{Z_{i ,j}'\}_{i \in [n_1]}$ is independent and we denote $c_\varepsilon = \{\exp(\varepsilon) + 1\}/\{\exp(\varepsilon) - 1\}$. We have that $Z_{i ,j}'$ is $\mathrm{SG}(1/4)$ as each $Z_{i ,j}'$ takes values in $\{0, 1\}$ \citep[e.g.~Exercise~2.4][]{Wainwright:2019:HDSBook}, and hence the scaled and translated summands are $\mathrm{SG}(c_{\varepsilon/2}^2/\{4n_1^2\})$. Hence, by the independence of the collection $\{Z_{i ,j}'\}_{i \in [n_1]}$, the sum $\hat{p}_{X,j}$ is $\mathrm{SG}(c_{\varepsilon/2}^2/\{4n_1\})$. 

                    By a similar argument, we have that $\hat{p}_{Y,j}$ is $\mathrm{SG}(c_{\varepsilon/2}^2/\{4n_2\})$. Lastly, noting that $\hat{p}_{X,j}$ and $\hat{p}_{Y,j}$ are independent for a given $j \in [d]$, and that $c_{\varepsilon/2}^2 \leq 20/\varepsilon^2$ for $\varepsilon \in (0,1]$, we obtain that the difference $\hat{\delta}_j$ is $\mathrm{SG}(10/\{n_1\varepsilon^2\})$.
                \end{proof}

        \subsection{Controlling Mean of Permuted Sample}
            \begin{lemma} \label{app:lem:permcontrol}
                Consider, for positive integers $m \geq n$, a permutation $\pi$ drawn uniformly at random from $S_{n+m}$, the symmetric group on the symbols $[n + m]$. Further, fix an $n$-tuple $\vecbf{L} = \{l_1, \hdots, l_{n}\}$ drawn from $[m]$ without replacement. For $i \in [n]$, denote
                \begin{equation*}
                    d_i = \mathbbm{1}\{\pi(i) \in [n],\; \pi(n+l_i) \in [n+m]\setminus[n]\} - \mathbbm{1}\{\pi(i) \in [n+m] \setminus [n],\; \pi(n + l_i) \in [n]\}.
                \end{equation*}
                The sum $\sum_{i = 1}^n d_i$ is mean zero and $\mathrm{SG}(2n)$.
            \end{lemma}

            \begin{proof}
                First, for a fixed $i \in [n]$, consider the event $\{\pi(i) \in [n],\; \pi(n+l_i) \in [n+m]\setminus[n]\}$. We proceed by a combinatoric argument: There are $n$ valid choices for $\pi(i) \in [n]$, and $m$ valid choices for $\pi(n+l_i) \in [n+m]\setminus[n]$. The remaining $(n+m-2)$ entries may be freely permuted, giving a total of $nm(n+m-2)!$ valid permutations. By symmetry, there are also $nm(n+m-2)!$ valid permutations which satisfy the event $\{\pi(i) \in [n+m]\setminus[n],\; \pi(n+l_i) \in [n]\}$. Hence, as the permutation is uniformly selected, we have that
                \begin{align*}
                    \mathbb{P}\big( \pi(i) &\in [n],\; \pi(n+l_i) \in [n+m]\setminus[n] \big)\\
                    &= \mathbb{P}\big( \pi(i) \in [n+m]\setminus[n],\; \pi(n+l_i) \in [n] \big) 
                    = \frac{nm(n+m-2)!}{(n+m)!} = \frac{nm}{(n+m)(n+m-1)}.
                \end{align*}
                Hence, we see that $\mathbb{E}[d_i] = 0$ for all $i \in [n]$. Importantly, we note that regardless of the sizes of the two groups $[n]$ and $[n+m] \setminus [n]$, we have $\mathbb{P}(d_i = 1) = \mathbb{P}(d_i = -1)$.

                When considering the moment generating function of the sum $\sum_{i = 1}^{n} d_i$, we must note the dependence between the summands. Denoting $\vecbfm{\Pi}_i = \{\pi(1), \hdots, \pi(i-1), \pi(n + l_1), \hdots, \pi(n + l_{i-1})\}$, we observe that once the values of $\pi(1), \hdots, \pi(i-1), \pi(n+l_1), \hdots, \pi(n+l_{i-1})$ are fixed by conditioning on $\vecbfm{\Pi}_i$, the values of $\pi(i) \hdots, \pi(n), \pi(l_i), \hdots, \pi(l_n)$, are uniformly distributed as a permutation over the symbols $[n+m] \setminus \vecbfm{\Pi}_i$.
                
                Hence, defining the random integers $n^{(i)} = \big|[n] \setminus \vecbfm{\Pi}_i \big|$ and $m^{(i)} = \big|[n+m]\setminus ([n] \cup \vecbfm{\Pi}_i) \big|$, we have that
                \begin{align*}
                    \mathbb{P}\big\{\pi(i) \in [n] \setminus \vecbfm{\Pi}_i,\; \pi(n+l_i) &\in [n+m]\setminus ([n] \cup \vecbfm{\Pi}_i) \mid \vecbfm{\Pi}_i \big\} \\
                    &=\mathbb{P}\big\{\pi(i) \in [n+m]\setminus ([n] \cup \vecbfm{\Pi}_i) \setminus \vecbfm{\Pi}_i,\; \pi(n+l_i) \in [n] \setminus \vecbfm{\Pi}_i \mid \vecbfm{\Pi}_i \big\} \\
                    &= \frac{n^{(i)}m^{(i)}}{(n^{(i)} + m^{(i)})(n^{(i)} + m^{(i)} - 1)} = p^{(i)},
                \end{align*}
                where $p^{(i)}$ is a random probability. Hence, we have that
                \begin{align*}
                    &\mathbb{P}(d_i = 1 \mid \vecbfm{\Pi}_i)
                    = \mathbb{P}(d_i = -1 \mid \vecbfm{\Pi}_i) = p^{(i)}, \quad \mbox{and} \quad
                    \mathbb{P}\big(d_i = 0 \mid \vecbfm{\Pi}_i \big) = 1 - 2p^{(i)}.
                \end{align*}

                As the collection of random variables $\{d_1, \hdots, d_{i - 1}\}$ is completely determined by the collection $\vecbfm{\Pi}_i = \{\pi(1), \hdots, \pi(i - 1), \pi(n + l_1), \hdots, \pi(n + l_{i-1})\}$, by the tower property we obtain for all $i \in [n]$
                \begin{align*}
                    \mathbb{E}[\exp(\theta d_i) \mid d_1, \hdots, d_{i - 1}]
                    &= \mathbb{E}\big[ \mathbb{E}[\exp(\theta d_i) \mid \vecbfm{\Pi}_i )] \mid d_1, \hdots, d_{i - 1}] \big] \\
                    &\leq \mathbb{E}\big[p^{(i)}\exp(\theta) + p^{(i)}\exp(-\theta) + (1-2p^{(i)}) \mid d_1, \hdots, d_{i - 1}] \big] \\
                    &\leq e^{\theta^2} \quad \mbox{almost surely,}
                \end{align*}
                where the final inequality holds as $p^{(i)} \leq 1/2$ almost surely.
                
                Finally, again applying the tower property, we have that
                \begin{align*}
                    \mathbb{E}\bigg[&\exp\bigg(\theta \sum_{i=1}^n (d_i - \mathbb{E}[d_i]) \bigg)\bigg]
                    = \mathbb{E}\bigg[\exp\bigg(\theta \sum_{i=1}^n d_i\bigg)\bigg]
                    \\
                    &= \mathbb{E}\bigg[ \cdots \mathbb{E}\bigg[\mathbb{E}\bigg[\exp(\theta d_n) \biggm| d_{1}, \hdots, d_{n-1}\bigg] \exp(\theta d_{n-1}) \biggm| d_{1}, \hdots, d_{n-2} \bigg] \cdots \exp(\theta d_1) \bigg]
                    \leq e^{n\theta^2}.
                \end{align*}
                Hence, the sum $\sum_{i=1}^n d_i$ is $\mathrm{SG}(2n)$, which concludes the proof.
            \end{proof}

            \subsection{Private Sampling Methods for Continuous Non-interactive Procedure}
                The non-interactive testing procedure for continuous distributions developed in \Cref{sec5:UpperBoundNonInt} uses two different private sampling methods depending on the parity of $V$, the chosen truncation level of the series expansion of the corresponding density functions. We outline these sampling methods as follows:

                First, for $i \in [n_1]$, $i' \in [n_2]$ recall the quantities $\tilde{\vecbf{Z}}_i$, $\tilde{\vecbf{W}}_i$ constructed in \eqref{sec4:eq:nonintcoefviews}. Sample $T_{i}^{(X)}, T_{i'}^{(Y)} \overset{\mathrm{i.i.d.}}{\sim} \mathrm{Bernoulli}(\exp(\varepsilon)/\{\exp(\varepsilon) + 1\})$.

                \noindent
                \textbf{For odd $V$}: 
                Output independent
                \begin{equation} \label{app:eq:oddMsamplemech}
                    \begin{aligned}
                        \vecbf{Z}_i &\sim
                        \begin{cases}
                            \mathrm{Uniform}( \vecbf{a} \in \{-A, A\}^V \mid \vecbf{a}^T \tilde{\vecbf{Z}}_i \geq 0), \mbox{ if } T_i^{(X)} = 1,\\
                            \mathrm{Uniform}( \vecbf{a} \in \{-A, A\}^V \mid \vecbf{a}^T \tilde{\vecbf{Z}}_i \leq 0), \mbox{ if } T_i^{(X)} = 0,
                        \end{cases} \\
                        \vecbf{W}_{i'} &\sim
                        \begin{cases}
                            \mathrm{Uniform}( \vecbf{a} \in \{-A, A\}^V \mid \vecbf{a}^T \tilde{\vecbf{W}}_{i'} \geq 0), \mbox{ if } T_{i'}^{(Y)} = 1,\\
                            \mathrm{Uniform}( \vecbf{a} \in \{-A, A\}^V \mid \vecbf{a}^T \tilde{\vecbf{W}}_{i'} \leq 0), \mbox{ if } T_{i'}^{(Y)} = 0,
                        \end{cases}
                    \end{aligned}
                \end{equation}
                where
                \begin{equation*}
                    A = 2^{d/2} \bigg(\frac{\exp(\varepsilon) + 1}{\exp(\varepsilon) - 1} \bigg)\frac{2^{V-1} [\{(V-1)/2\}!]^2}{(V - 1)!}.
                \end{equation*}

                \noindent
                \textbf{For even $V$}:
                Sample independent
                \begin{align*}
                    \Breve{\vecbf{Z}}_i &\sim
                    \begin{cases}
                        \mathrm{Uniform}( \vecbf{a} \in \{-A', A'\}^V \mid \vecbf{a}^T \tilde{\vecbf{Z}}_i > 0, \mbox{ or } \vecbf{a}^T \tilde{\vecbf{Z}}_i = 0 \mbox{ with } a_1 = A'), \mbox{ if } T_i^{(X)} = 1,\\
                        \mathrm{Uniform}( \vecbf{a} \in \{-A', A'\}^V \mid \vecbf{a}^T \tilde{\vecbf{Z}}_i < 0, \mbox{ or } \vecbf{a}^T \tilde{\vecbf{Z}}_i = 0 \mbox{ with } a_1 = -A'), \mbox{ if } T_i^{(X)} = 0,
                    \end{cases} \\
                    \Breve{\vecbf{W}}_{i'} &\sim
                    \begin{cases}
                        \mathrm{Uniform}( \vecbf{a} \in \{-A', A'\}^V \mid \vecbf{a}^T \tilde{\vecbf{W}}_{i'} > 0, \mbox{ or } \vecbf{a}^T \tilde{\vecbf{W}}_{i'} = 0 \mbox{ with } a_1 = A'), \mbox{ if } T_{i'}^{(Y)} = 1,\\
                        \mathrm{Uniform}( \vecbf{a} \in \{-A', A'\}^V \mid \vecbf{a}^T \tilde{\vecbf{W}}_{i'} < 0, \mbox{ or } \vecbf{a}^T \tilde{\vecbf{W}}_{i'} = 0 \mbox{ with } a_1 = -A'), \mbox{ if } T_{i'}^{(Y)} = 0,
                    \end{cases}
                \end{align*}
                where
                \begin{equation*}
                    A' = 2^{1/2} \bigg( \frac{\exp(\varepsilon) + 1}{\exp(\varepsilon) - 1} \bigg) \frac{2^{V-1}(V/2 - 1)!(V/2)}{(V-2)!(V-2)}.
                \end{equation*}

                Output $\vecbf{Z}_i$, $\vecbf{W}_{i'}$ with $j$-th co-ordinates
                \begin{equation} \label{app:eq:evenMsamplemech}
                    \begin{aligned}
                        Z_{i, j} &=
                        \begin{cases}
                            \Breve{Z}_{i , j}(V-2)/\{2(V-1)\} &\mbox{ for } j = 1, \\
                            \Breve{Z}_{i, j} &\mbox{ otherwise},
                        \end{cases} \\
                        W_{i', j} &=
                        \begin{cases}
                            \Breve{W}_{i' , j}(V-2)/\{2(V-1)\} &\mbox{ for } j = 1, \\
                            \Breve{W}_{i', j} &\mbox{ otherwise}.
                        \end{cases}
                    \end{aligned}
                \end{equation}

            \subsection{Integral Bound}
            \begin{lemma}\label{app:lem:logintbound}
                For all $x \in (0,1]$, it holds that $\int_0^x \{\log(1 + 1/a)\}^{1/2} \diff{a} \leq x\{\log(4/x)\}^{1/2}$.
            \end{lemma}
            \begin{proof}
                By the Cauchy--Schwarz inequality, we have that
                \begin{equation*}
                    \int_0^x \{\log(1 + 1/a)\}^{1/2} \diff{a}
                    \leq x^{1/2} \bigg( \int_0^x \log(1 + 1/a) \diff{a} \bigg)^{1/2}.
                \end{equation*}
                Hence, it suffices to show $\int_0^x \log(1 + 1/a) \diff{a} \leq x\log(4/x)$. In particular, we have
                \begin{align*}
                    \int_0^x \log(1 + 1/a) \diff{a} - x\log(4/x)
                    &= x\log(1 + 1/x) + \log(x + 1) - x\log(4/x) \\
                    &= (x+1)\log(x+1) - x\log(4)
                    = h(x).
                \end{align*}
                Finally, we note $h(0) = h(1) = 0$ and $h''(x) = 1/(1+x) > 0$ for $x \in (0, 1]$, and thus $h(x) \leq 0$ for all $x \in (0, 1]$, which completes the proof.
            \end{proof}

            \subsection{Miscellaneous Results for Sub-Gaussian and Sub-exponential Random Variables}
                The concept of a sub-Gaussian random variable can be extended to vectors of random variables as follows.
                \begin{definition}[Sub-Gaussian Random Vectors]
                    A vector of random variables $\vecbf{X}$ taking values in $\mathbb{R}^d$ is termed \emph{sub-Gaussian} with variance proxy $\sigma^2$, denoted $\mathrm{SG}(\sigma^2)$, if, for any vector $\vecbf{v} \in \mathbb{R}^d$, the variable $\vecbf{v}^T \vecbf{X}$ is $\mathrm{SG}(\| \vecbf{v} \|_2^2 \sigma^2)$. That is,
                    \begin{equation*}
                        \mathbb{E}[\exp\{\theta \vecbf{v}^T (\vecbf{X} - \mathbb{E}[\vecbf{X}])\}] \leq \exp(\theta^2  \|\vecbf{v}\|_2^2 \sigma^2/2)
                    \end{equation*}
                    for all $\theta \in \mathbb{R}$.
                \end{definition}
                It is easy to verify the following.
                \begin{lemma} \label{app:lem:IndSGSum}
                    If the random vectors $\vecbf{X}$ and $\vecbf{Y}$ are independent and are $\mathrm{SG}(\sigma_X^2)$, $\mathrm{SG}(\sigma_Y^2)$ respectively, then the sum $\vecbf{X} + \vecbf{Y}$ is $\mathrm{SG}(\sigma_X^2 + \sigma_Y^2)$.
                \end{lemma}
                \begin{proof}
                    The proof follows from the assumed independence of $\vecbf{X}$, $\vecbf{Y}$. Indeed,
                    \begin{align*}
                        \mathbb{E}[\exp\{\theta \vecbf{v}^T (\vecbf{X} - \vecbf{Y} - \mathbb{E}[\vecbf{X} - \vecbf{Y}])\}] 
                        &= \mathbb{E}[\exp\{\theta \vecbf{v}^T (\vecbf{X} - \mathbb{E}[\vecbf{X}])\}]\mathbb{E}[\exp\{\theta \vecbf{v}^T (\vecbf{Y} - \mathbb{E}[\vecbf{Y}])\}] \\
                        &\leq \exp(\theta^2  \|\vecbf{v}\|_2^2 \sigma_X^2/2)\exp(\theta^2  \|\vecbf{v}\|_2^2 \sigma_Y^2/2) \\
                        &= \exp\{\theta^2  \|\vecbf{v}\|_2^2 (\sigma_X^2 + \sigma_Y^2)/2\}.
                    \end{align*}
                \end{proof}

                In particular, a random variable distributed according to the Laplace distribution is sub-exponential.
                \begin{lemma} \label{app:lem:LaplaceSE}
                    A random variable $X$ distributed according to a Laplace distribution is $\mathrm{SE}(2^{1/2})$.
                \end{lemma}
                \begin{proof}
                    The moment generating function satisfies
                    \begin{equation*}
                        \mathbb{E}[\exp(\lambda X)] = \frac{1}{1-\lambda} \leq \exp(2\lambda^2)
                    \end{equation*}
                    for $|\lambda| \leq 2^{-1/2}$.
                \end{proof}

                The presence of an absolute constant $C_1 \geq 1$ preceding the exponential term in the bound on the moment generating function results in a sub-exponential random variable with the same variance proxy up to scaling by some absolute constant. This is \beame{formalised}{formalized}{formalized} as follows.

                \begin{lemma} \label{app:lem:SEMGFprefactor}
                    If the random variable $X$ satisfies
                    \begin{equation*}
                        \mathbb{E}[\exp\{\theta (X - \mathbb{E}[X])\}] \leq C_1\exp(\theta^2 \sigma^2)
                    \end{equation*}
                    for all $\theta \in \mathbb{R}$ such that $|\theta| \leq 1/\sigma$ where $C_1 \geq 1$ is some absolute constant, then $X$ is $\mathrm{SE}(C_2\sigma)$ for some absolute constant $C_2 > 0$.
                \end{lemma}
                \begin{proof}
                    Following the proof of Proposition~2.7.1 \citep{Vershynin:2018:HDPBook}, one can see that the inequality $\mathbb{E}[\exp\{\theta (X - \mathbb{E}[X])\}] \leq C_1\exp(\theta^2 \sigma^2/2)$ implies property (b) therein up to some constant depending on $C_1$, which in turn implies property (e) therein, which demonstrates that $X$ is $\mathrm{SE}(C_2\sigma)$ for some absolute constant $C_2 > 0$.
                \end{proof}

                If a random variable $X$ satisfies a sub-exponential tail up to some sufficiently large constant, then it is sub-exponential.
                \begin{lemma} \label{app:lem:SETailtoMGF}
                    If the random variable $X$ satisfies the tail bound
                    \begin{equation*}
                        \mathbb{P}(|X - \mathbb{E}[X]| \geq x) \leq  c\exp\{-x/\sigma\},
                    \end{equation*}
                    for $x \geq 0$ and $c \geq 2$ some absolute constant, then $X$ is $\mathrm{SE}(C\sigma)$ for $C > 0$ some absolute constant.
                \end{lemma}
                \begin{proof}
                    We consider the quantity $2\exp\{-x/(2\sigma\log(c))\}$. For $0 \leq x \leq \sigma\log(c)$, we have
                    \begin{equation*}
                        \min\big\{2\exp\{-x/(2\sigma\log(c))\},\; c\exp\{-x/\sigma\}\big\} \geq 1,
                    \end{equation*}
                    and so both trivially satisfy the tail probability bound. For $x \geq \sigma\log(c)$, one can verify
                    \begin{equation*}
                        2\exp\{-x/(2\sigma\log(c))\} > c\exp\{-x/\sigma\}.
                    \end{equation*}
                    Hence, we have the tail bound $\mathbb{P}(|X - \mathbb{E}[X]| \geq x) \leq  2\exp\{-x/(2\sigma\log(c)\}$, and so we may appeal to Proposition~2.7.1 in \cite{Vershynin:2018:HDPBook} to obtain that $X$ is $\mathrm{SE}(C\sigma)$ for some absolute constant $C > 0$.
                \end{proof}

            \subsection{Binary Search Procedure} \label{app:sec:binarysearch}
                In this subsection we describe the binary search procedure we use in \Cref{sec6:disc} to empirically validate the minimax separation radii. First, set a desired tolerance $\delta \in [0,1]$ and the number of repetitions $r \in \mathbb{N}$. \beame{Initialise}{Initialize}{Initialize} $l_0 = 0$, $u_0 = 1$ and $J = \lceil 2\log_2(1/\delta) \rceil$. Then, for $j \in [J]$, set $\gamma_j = (u_{j-1} - l_{j-1})/2$ and for $i \in [r]$ implement the test of interest on simulated data drawn from the distributions $P_X$, $P_{Y, \gamma_j}^{(p)}$ as in \eqref{app:eq:discsimdistributions} with $p \in \{1,2\}$ depending on the setting. The simulated data are independent across the simulated tests. Denoting the outcome of each test as $\phi_j^{(i)}$ where we recall a $1$ corresponds to a rejection of the null hypothesis, calculate the estimated type-II error $\hat{\beta}_j = r^{-1}\sum_{i = 1}^r (1-\phi_j^{(i)})$. If $|\hat{\beta}_j - 1/2| \leq \delta$, then return $\gamma_j$ and terminate the procedure early. Otherwise, if $\hat{\beta}_j > 1/2$, set $l_j = l_{j-1}$, $u_j = (l_{j-1} + u_{j-1})/2$ or if $\hat{\beta}_j < 1/2$, set $l_j = (l_{j-1} + u_{j-1})/2$, $u_j = u_{j-1}$. Repeat until either the procedure terminates early or return $\gamma_J$ once complete.
\end{document}